\documentclass[bj]{imsart}

\RequirePackage{amsthm,amsmath,amsfonts,amssymb}
\RequirePackage[numbers]{natbib}

\startlocaldefs

\usepackage{mathrsfs}
\usepackage{xr}

\usepackage{accents}
\newcommand{\ubar}[1]{\underaccent{\bar}{#1}}
\usepackage{multirow}
\usepackage[normalem]{ulem}
\usepackage[linesnumbered,lined,boxed]{algorithm2e}

\makeatletter
\newcommand{\mylabel}[2]{#2\def\@currentlabel{#2}\label{#1}}
\makeatother
\usepackage{amsmath,amssymb,amsthm,amsfonts,amsbsy,latexsym,dsfont,tikz,bbm,verbatim}
\RequirePackage[colorlinks,citecolor=blue,urlcolor=blue]{hyperref}
\usepackage{subfig}

\usepackage{graphicx,color}

\usepackage{enumerate}

\newcommand{\Cov}[0]{\text{Cov}}
\newcommand{\Var}[0]{\text{Var}}

\newcommand{\calS}[0]{\mathcal{S}}

\newcommand{\R}[0]{\mathbb{R}}

\newcommand{\kjack}{I_{n-1,r-1}^{(k)}}

\newcommand{\Pro}{\mathbb{P}}
\newcommand{\Exp}{\mathbb{E}}

\theoremstyle{plain}
\newtheorem{theorem}{Theorem}[section]
\newtheorem{lemma}[theorem]{Lemma}
\newtheorem{corollary}[theorem]{Corollary}

\theoremstyle{definition}
\newtheorem{definition}[theorem]{Definition}
\newtheorem{remark}[theorem]{Remark}

\renewcommand{\leq}{\leqslant} 
\renewcommand{\geq}{\geqslant}

\newcommand{\ind}{\mathds{1}}

\def\qed{ \hfill $\blacksquare$}  

\newcommand{\cA}{\mathcal{A}}\newcommand{\cB}{\mathcal{B}}\newcommand{\cC}{\mathcal{C}}
\newcommand{\cD}{\mathcal{D}}\newcommand{\cE}{\mathcal{E}}\newcommand{\cF}{\mathcal{F}}
\newcommand{\cH}{\mathcal{H}}

\newcommand{\cM}{\mathcal{M}}\newcommand{\cN}{\mathcal{N}}
\newcommand{\cR}{\mathcal{R}}
\newcommand{\cS}{\mathcal{S}}\newcommand{\cT}{\mathcal{T}}\newcommand{\cU}{\mathcal{U}}
\newcommand{\cV}{\mathcal{V}}\newcommand{\cX}{\mathcal{X}}
  
\newcommand{\mvxi}{\boldsymbol{\xi}}

\newcommand{\bD}{\mathbb{D}}
\newcommand{\bG}{\mathbb{G}}

\newcommand{\bM}{\mathbb{M}}\newcommand{\bN}{\mathbb{N}}
\newcommand{\bP}{\mathbb{P}}\newcommand{\bR}{\mathbb{R}}
\newcommand{\bU}{\mathbb{U}}

\newcommand{\bMn}{\mathbb{M}_n}     

\newcommand{\sV}{\mathscr{V}}

\DeclareMathOperator{\E}{\mathds{E}}

\usepackage{enumitem}

\endlocaldefs

\begin{document}

\begin{frontmatter}
\title{Stratified incomplete local simplex tests for curvature of nonparametric multiple regression}
\runtitle{Testing for regression curvature}

\begin{aug}
\author[A]{\fnms{Yanglei} \snm{Song}  \ead[label=e1]{yanglei.song@queensu.ca}},
\author[B]{\fnms{Xiaohui} \snm{Chen}\ead[label=e2]{xhchen@illinois.edu}}
\and
\author[C]{\fnms{Kengo} \snm{Kato}\ead[label=e3]{kk976@cornell.edu}}
\address[A]{Department of Mathematics and Statistics, 
Queen's University, Jeffery Hall, Kingston, ON,  Canada, K7L 3N6, \printead{e1}}

\address[B]{Department of Statistics, 
University of Illinois at Urbana-Champaign,
725 S. Wright Street, Champaign, IL 61820, \printead{e2}}

\address[C]{Department of Statistics and Data Science,
Cornell University, 1194 Comstock Hall, Ithaca, NY 14853,
\printead{e3}
}
\end{aug}

\begin{abstract}
Principled nonparametric tests for regression curvature in $\bR^{d}$ are often statistically and computationally challenging. This paper introduces the stratified incomplete local simplex (SILS) tests for joint concavity of nonparametric multiple regression. The SILS tests with suitable bootstrap calibration are shown to achieve simultaneous guarantees on dimension-free computational complexity, polynomial decay of the uniform error-in-size, and power consistency for general (global and local) alternatives. To establish these results, we develop a general theory for incomplete $U$-processes with stratified random sparse weights.  Novel technical ingredients include maximal inequalities for the supremum of multiple incomplete $U$-processes.
\end{abstract}

\end{frontmatter}

\section{Introduction}

This paper concerns the hypothesis testing problem for curvature (i.e., concavity, convexity, or linearity) of a nonparametric multiple regression function. Testing the validity of such geometric hypothesis is important for performing high-quality subsequent shape-constrained statistical analysis. For instance, there has been considerable effort    in nonparametric estimation of a convex (concave) regression function, partly because estimation 
under convexity constraint requires no tuning parameter as opposed to e.g. standard kernel estimation whose performance depends critically on a user-chosen bandwidth parameter \cite{hildreth1954point,hanson1976,mammen1991,groeneboom2001,seijo2011nonparametric,lim2012consistency,chatterjee2014,
cai2015,guntuboyina2015,chen2016,chatterjee2016,han2016,kur2019optimality}. 
In empirical studies such as economics and finance, convex (concave) regressions have wide applications in modeling the relationship between wages and education \cite{MurphyWelch1990}, between firm value and product price \cite{BorensteinFarrell2007_RAND}, and between mutual fund return and multiple risk factors \cite{FAMA19933,abrevaya2005nonparametric}. 

Consider the nonparametric multiple regression model
\begin{equation}
\label{eqn:nonparametric_regression}
Y = f(V) + \varepsilon,
\end{equation}
where $Y$ is a scalar response variable, $V$ is a $d$-dimensional covariate vector, $\varepsilon$ is a random error term such that $\E[\varepsilon | V] = 0$ and $\Var(\epsilon) > 0$, and $f : \bR^{d} \to \bR$ is the conditional mean  (i.e., regression)  function.  Let $P$ be the joint distribution of $X = (V, Y) \in \R^{d+1}$ and $X_{i} := (V_{i}, Y_{i}), i \in [n] := \{1,\dots,n\}$ be a sample of independent random vectors with common distribution $P$. For a given 
convex,  compact subset $\cV \subset \R^{d}$, 
based on the observations $\{X_{i}\}_{i=1}^{n}$, 
we aim to test the following hypothesis:
\begin{equation}
\label{def:H_0}
H_{0}: f \text{ is concave on } \cV,
\end{equation}
against some (globally or locally) non-concave alternatives. In this work, we directly leverage the simplex characterization of concave functions, i.e., $f$ is concave on $\cV$ if and only if
\begin{equation}\label{def:concavity}
a_{1} f(v_{1}) + \cdots + a_{d+1} f(v_{d+1}) \leq f(a_{1} v_{1} + \cdots + a_{d+1} v_{d+1}),
\end{equation}
for any $v_{1},\dots,v_{d+1} \in \cV$ and  nonnegative reals $a_{1},\dots,a_{d+1}$ such that $a_{1}+\cdots+a_{d+1}=1$.
Working with this definition allows us to circumvent the need to estimate the regression function $f$,
and thus the resulting tests would be robust to model misspecification. Further, the concavity hypothesis can be quantitatively evaluated on the observed data, 
which is  the idea behind the {\it simplex statistic} in \cite{abrevaya2005nonparametric}.


Specifically, $d+1$ covariate vectors in $\bR^d$ form a simplex.
Consider $r := d+2$ data points $x_1:=(v_1,y_1),\ldots, x_{r}:=(v_{r}, y_{r}) \in \bR^{d+1}$  generated from the model \eqref{eqn:nonparametric_regression}. If for all $j \in [r]$,
$v_j$ is not in the simplex spanned by $\{v_i: i \neq j\}$, for example the vectors   $\{v_1,v_2,v_3,v_5\}$ in Figure \ref{fig:main_idea},  then we set $w(x_1,\ldots,x_{r}) = 0$. Otherwise, there exists a {unique} $j$ such that $v_j$ can be written as a convex combination of other covariate vectors, i.e., $v_j = \sum_{i\neq j} a_i v_i$ for some $a_i \geq 0, \sum_{i\neq j} a_i = 1$; in this case, we compare the response $y_j$ with the same  combination  of others, $\{y_i: i \neq j\}$, i.e., setting
$$w(x_1,\ldots,x_{r}) =  \sum\limits_{i\neq j} a_i y_i - y_j.
$$
For example, in Figure \ref{fig:main_idea},  $v_4$ is in the simplex spanned by $\{v_1,v_2,v_3\}$. We note that the index $j$ and the coefficients $\{a_i: i \neq j\}$ are functions of $\{v_i: i \in [r]\}$, and defer the precise definitions to Section \ref{sec:local_simplex_stat}.

If $f$ is indeed concave (i.e., in $H_{0}$) and $\varepsilon$ is symmetric about zero, then 
$\Exp\left[\text{sign}(w(X_1,\ldots,X_{r})\right] \leq 0$ due to \eqref{def:concavity}, 
where $\text{sign}(t) := \mathbbm{1}(t > 0) - \mathbbm{1}(t < 0)$ is the sign function. Thus \cite{abrevaya2005nonparametric} proposes to use the following global $U$-statistic of all $r$-tuples from $\{X_i: i \in [n]\}$ and reject the null if the  statistic is large:
\begin{equation*}
|I_{n,r}|^{-1}\sum_{\iota \in I_{n,r}} \text{sign}\left(w(X_\iota) \right), \text{ with } X_{\iota} = (X_{i_1},\ldots, X_{i_r}),
\end{equation*}
where  $I_{n,r} := \{\iota = (i_1,\ldots, i_r): 1 \leq i_1 < \ldots < i_r \leq n\}$, and $|\cdot|$  denotes the set cardinality.

\begin{figure}[htbp!]
\includegraphics[width=0.75\textwidth]{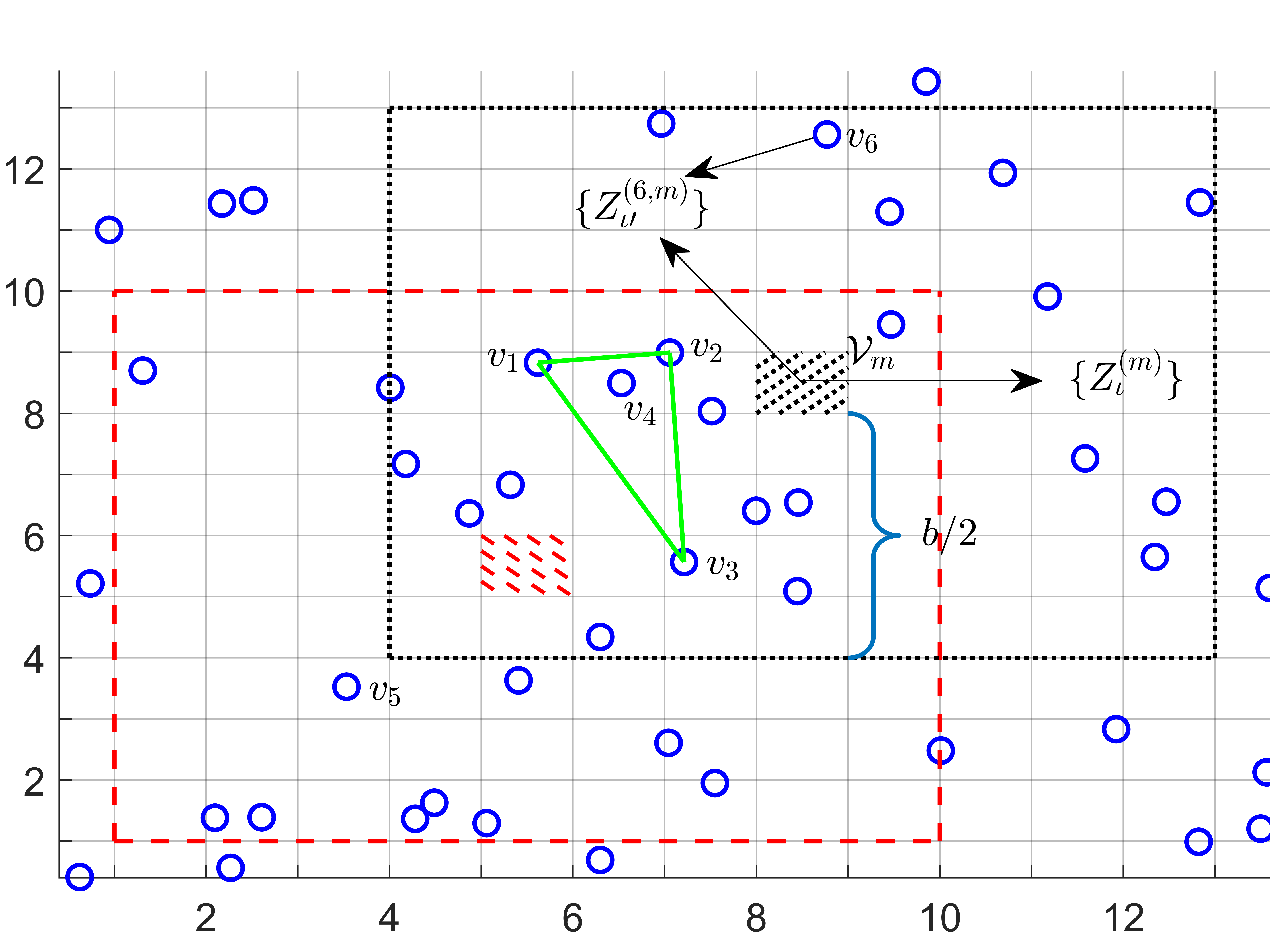}
  \captionsetup{width=\linewidth}
\caption{(i). Each circle represents a two-dimensional feature vector (i.e., $d=2$). For each query point $v \in \cV$, a sampling plan is a collection of Bernoulli random variables $\{Z_\iota(v): \iota \in I_{n,r}\}$, one for each subset of $r = d+2$ data points. If $Z_{\iota}(v) = 1$, then {$h_v^{\text{sg}}(X_{\iota})$} contributes to  the average in \eqref{def:preliminary_test_proc}. (ii).
 The space $\cV$ is {stratified} into disjoint regions (e.g., $1$-by-$1$ squares above). The query points in each region share the same sampling plan (e.g. {$\{Z_{\iota}^{(m)}\}$} for the dotted region $\cV_m$), while different regions have independent sampling plans. For example, 
the indicator for $\iota = (v_1,v_2,v_3,v_4)$ may be one for $\cV_m$, but zero for the dashed region. (iii). 
Due to the localizing kernel \eqref{def:sg_kernel}, for each query point $v \in \cV$, it suffices to consider data points that are within $b/2$ distance to $v$ in each coordinate. Thus for $\cV_m$ above, e.g., it suffices to consider  data points within the  dotted square ($b=8$). The key idea is that query points in a small region share similar nearby data points, and  the region-specific sampling plan allows us to  allocate the ``limited resources" only in ``important areas".
}
\label{fig:main_idea}
\end{figure}

\subsection{Local simplex statistics}
Since the above ``global" $U$-statistic  is not consistent against general alternatives, e.g., when $f$ is only non-concave in a small region, \cite{abrevaya2005nonparametric} also proposes the {\it localized simplex statistics}. Specifically,  let $L:\bR^{d} \to \bR$ be a function such that $L(z) = 0$ if $\|z\|_{\infty} := \max_{j \in [d]}|z_j| > 1/2$,
and $L_{b}(\cdot) := b^{-d} L(\cdot/b)$ for $b > 0$.
For $x_i := (v_i,y_i) \in \bR^{d+1}, i \in [r]$,  and a bandwidth parameter $b_n > 0$, define
\begin{equation}
\label{def:sg_kernel}
h_{v}^{\text{sg}}(x_{1},\dots,x_{r}) := \text{sign}\left(w(x_1, \dots, x_{r})\right) b_n^{d/2}\prod_{k=1}^{r} L_{b_n}(v-v_{k}), \quad v \in \cV.
\end{equation}
Thus for each $v \in \cV$, only nearby data points are utilized in constructing a local statistic. Note that $h_v^{\text{sg}}$ depends on $b_n$, which we omit in most places for simplicity of notations.

Given  a {finite} collection of {query (or design) points} $\sV_n \subset \cV$,  \cite{abrevaya2005nonparametric} proposes to reject the null if
\begin{align}\label{def:preliminary_test_proc}
\sup_{v \in \sV_n} U_n(h_v^{\text{sg}}) \text{ is large},\;\; \text{ where }\;\;
U_n(h_v^{\text{sg}}) := |I_{n,r}|^{-1}\sum_{\iota \in I_{n,r}} 
h_v^{\text{sg}}(X_{\iota}). 
\end{align}
In \cite{abrevaya2005nonparametric}, it  requires the query points in $\sV_n$ to be well separately, i.e., $\|v - v'\|_{\infty} > b_n$ for each pair of distinct $v,v' \in \sV_n$, which is restrictive when $d \geq 2$ and $b_n$ cannot be too small. Such a requirement is imposed since \cite{abrevaya2005nonparametric} uses extreme value theory to obtain the asymptotic distribution of the supremum, for which the convergence of approximation error is known to be logarithmically slow \cite{hall1991convergence}. 

In \cite{chen2017jackknife}, a valid jackknife multiplier bootstrap (JMB) is proposed to calibrate the distribution of the supremum of the (local) $U$-{process}, $\sup_{v \in \cV} U_n(h_v^{\text{sg}})$. 
Even though JMB tailored to the concavity test problem is statistically consistent, it requires tremendous, if not prohibitive, resources  to compute $\sup_{v \in \cV} U_n(h_v^{\text{sg}}) $, as well as calibrating its distribution via bootstrap, for $d \geq 2$. For instance, suppose that $V$ has a Lebesgue density that is bounded away from zero on $\cV$. Then  the number of data points within the $b_n$-neighbourhood of $v \in \cV$ is on average $O(n b_n^{d})$. Thus to compute $U_n(h_v^{\text{sg}})$ for a fixed $v \in \cV$, the required number of evaluations of $w(\cdot)$ is on average $O( (n b_n^{d} )^r)$,
which is computationally intensive, if $d \geq 2$ (thus $r=4$), and the bandwidth $b_n$ is not too small. In fact, in the numerical study (Section \ref{sec:simulation}), we estimate that for $d=3,n=1000,b_n = 0.6$ ($b_n/2$ is the half width), it would take more than $7$ days to use bootstrap for calibration even with $40$ computer cores.

It is tempting to break the computational bottleneck by using the incomplete version of the $U$-process $\{ U_n(h_v^{\text{sg}}) : v \in \cV \}$, which has been studied for high-dimensional $U$-statistics \cite{chen2017randomized,song2019approximating}. Specifically, we may associate each subset of $r$ data points, $\iota \in I_{n,r}$, 
with an independent Bernoulli random variable $Z_{\iota}$, and only include $h_v^{\text{sg}}(X_{\iota})$ in the average in \eqref{def:preliminary_test_proc} if $Z_{\iota} = 1$. 
Note that this is a ``centralized" sampling plan, in the sense that $\{Z_{\iota}: \iota \in I_{n,r}\}$ is shared by each $v \in \cV$. Here, we explain intuitively why such a plan does not solve the computational challenge, and postpone the detailed discussion until Section \ref{sub:comp_complex}. First, for each $\iota = (i_1,\ldots,i_r) \in I_{n,r}$, if $\|v_{i_j} - v_{i_k}\|_{\infty} > b_n$ for some $j,k \in [r]$ (e.g., $v_5,v_6$ in Figure \ref{fig:main_idea}), then $h_v^{\text{sg}}(X_{\iota}) = 0$ {for each} $v \in \cV$. As a result, with a very high probability,  a randomly selected $r$-tuples $X_{\iota}$ is ``wasted". Second, if two query points $v,v'$ are not close, in the sense that $\|v-v'\|_{\infty} > b_n$ (e.g. $v_5,v_6$ in Figure \ref{fig:main_idea} if they are used as query points),  then they share no nearby data points as defined by the localizing kernel in \eqref{def:sg_kernel}, which is a property {ignored} by the centralized sampling.

\subsection{Our contributions}

In this paper, we introduce the {\it stratified incomplete local simplex} (SILS) statistics for testing the concavity assumption in nonparametric multiple regression. 
We  show that SILS tests have simultaneous guarantees on {\it dimension-free} computational complexity, {\it polynomial decay} of the uniform error-in-size, and {\it power consistency} against  general alternatives. We elaborate below our contributions, and also refer readers to  Figure \ref{fig:main_idea} for a pictorial illustration of key ideas.\\

\noindent \underline{\bf Computational contributions.} The SILS test is proposed to address the computational issue with  the test statistic \eqref{def:preliminary_test_proc}, as well as
calibrating its distribution.
Specifically, we first partition the space $\cV$ into disjoint regions $\{\cV_m: m \in [M]\}$ for some integer $M \geq 1$. Let $N := n^{\kappa} b_n^{-dr}$ be a computational parameter for some $\kappa > 0$, and for each $m \in [M]$, let $\{Z_{\iota}^{(m)}: \iota \in I_{n,r}\}$ be a  collection of independent Bernoulli random variables with success probability $p_n := N/|I_{n,r}|$, which is called a {\it sampling plan}. For different regions, the sampling plans are {independent}. Then we consider the stratified, incomplete version of \eqref{def:preliminary_test_proc} as our statistic for testing the hypothesis \eqref{def:H_0}:
\begin{align*}
\sup_{m \in [M]} \; \sup_{v \in \cV_m}\;\;
\left(\sum\nolimits_{\iota \in I_{n,r}} Z_{\iota}^{(m)}  h_v^{\text{sg}}(X_{\iota})  
\right)/\left( \sum\nolimits_{\iota \in I_{n,r}} Z_{\iota}^{(m)} \right).
\end{align*}
Similar idea is applied to bootstrap calibration (see Subsection \ref{sec:bootstrap_uproc}), which involves another computational parameter $N_2 := n^{\kappa'} b_n^{-dr}$ for some $\kappa' > 0$. 
Due to the localization by the kernel \eqref{def:sg_kernel} and the stratification (see Figure \ref{fig:main_idea}), we show in Section \ref{subsec:implementation} that the overall  computational cost is $O( M n^{\kappa} \log(n) + M n^{1+\kappa'} b_n^{-d}\log(n) + BMn)$, where $B$ is the number of bootstrap iterations. 
Our theory allows  $\kappa, \kappa'$ to be arbitrarily small, but due to power analysis, we recommend $\kappa = \kappa' = 1$. In addition, $M$ is usually chosen so that $M = O(b_n^{-d})$, and to ensure a non-vanishing number of local data points, we must have 
$b_n^{-d} = O(n)$; thus the cost is independent of the dimension $d$.  

Further, to alleviate the burden of selecting a single bandwidth, we propose  to use the supremum of  the statistics  associated with \textit{multiple} $b_n$ (Subsection \ref{subsec:unif_bandwidth}). Finally, we conduct  extensive simulations to demonstrate the computational feasibility  of the proposed method, and to corroborate  our theory.\footnote{The implementation can be found on the  github (\url{https://github.com/ysong44/Stratified-incomplete-local-simplex-tests}).}\\


\noindent \underline{\bf Statistical contributions.} 
In addition to the function class $\cH^{\text{sg}} := \{h_v^{\text{sg}}\}$, which uses the sign of simplex statistics, we also consider another class of functions  $\cH^{\text{id}} := \{h_v^{\text{id}}\}$, where $h_v^{\text{id}}$ uses $w(\cdot)$ instead of its sign (see \eqref{def:h_x_concavity});
note that $h_v^{\text{id}}$ is unbounded unless $\varepsilon$ has bounded support. On one hand, $\cH^{\text{sg}}$ requires the observation noise $\varepsilon$ to be conditionally symmetric about  zero  \cite{abrevaya2005nonparametric}, but otherwise is robust to heavy tailed $\varepsilon$. On the other hand, $\cH^{\text{id}}$ requires $\varepsilon$ to have a light tail, but otherwise imposes no restrictions \cite{chen2017jackknife}.
For both classes of functions, we establish the {size validity}, as well as {power consistency} against general alternatives, for the proposed procedure, under no  smoothness assumption on the regression function.

In fact, under fairly general moment assumptions, we derive a unified Gaussian approximation and bootstrap theory for stratified, incomplete $U$-processes (Section \ref{sec:prelim} and \ref{sec:GAR_Bootstrap}), associated with a general function class $\cH$, where the SILS test for regression concavity  is an application of the general results.\\

\noindent \underline{\bf Technical contributions.}
The analysis of the stratified, incomplete $U$-processes requires a strategy different from the coupling approach used for  complete $U$-processes \cite{chen2017jackknife}: (i) we establish corresponding results for high dimensional stratified, incomplete  $U$-statistics (Appendix \ref{sec:GAR_incomp_U_stat}); (ii) we show that the supremum of the  process is well approximated by the supremum over a finite, but diverging, collection of $v \in \cV$. The main novelty are  {local and non-local maximal inequalities} to bound the supremum difference between a complete $U$-process and its
stratified, incomplete version
 (Appendix \ref{app:local_max_ineq_incom}), which can also be useful for other applications involving sampling, such as estimating the density of functions of several random  variables \cite{gine2007local}.
 
We note that the developed maximal inequalities are novel compared to \cite{chen2017jackknife} and \cite{chen2017randomized}. First, \cite{chen2017jackknife} studies {complete} $U$-processes, and neither stratification nor sampling is involved. Second, \cite{chen2017randomized} establishes inequalities for {incomplete} high dimensional $U$-{vectors}, 
whose proofs are fundamentally different from those for {processes}, and which does not have the  stratification component.  See also Remark \ref{rk:main_challenge} for technical challenges associated with {local} $U$-processes.

\subsection{Related work}
\label{subsec:related_work}

Regression under concave/convex restrictions has a long and rich history dating back to \cite{hildreth1954point}. 
Traditionally, the literature  focused on the univariate ($d=1$) case \cite{hanson1976,mammen1991,groeneboom2001,cai2015,guntuboyina2015,chen2016}, but there is a significant recent theoretical progress in the multivariate case \cite{seijo2011nonparametric,lim2012consistency,han2016,kur2019optimality}; see also \cite{matzkin1991,kuosmanen2008representation,hannah2013multivariate,Mazumder2019}. We refer readers to \cite{chetverikov2018econometrics,guntuboyina2018nonparametric} for a review on estimation and inference under shape constraints including concave/convexity constraints. 

The literature on testing the hypotheses of regression concavity is relatively scarce, especially for multiple regression, i.e., $d \geq 2$. Simplex statistic and its local version are introduced in \cite{abrevaya2005nonparametric}, and the bootstrap calibration (without computational concerns) is investigated in \cite{chen2017jackknife}. 
Several testing procedures  based on splines \cite{diack1998nonparametric,wang2011testing,komarova2019testing} have been proposed, 
which, however, are only proven to work 
for the univariate case since they are essentially second-derivative tests at the spline knots. Thus such methods can only test marginal concavity in the presence of multiple covariates, and multi-dimensional spline interpolation is much less understood in the nonparametric regression setting.
Further, in the univariate case with a white-noise model, multi-scale testing for qualitative hypotheses  is considered in \cite{dumbgen2001multiscale}, and minimax risks for estimating the $L^{q}$ distance ($1 \leq q < \infty$)  between an unknown signal and the cones of positive/monotone/convex functions are established in \cite{juditsky2002}. 

A very recent work by \cite{fang2020projection} proposes a projection framework for testing shape restrictions including concavity, which we call ``FS" test. Specifically, the FS test \cite{fang2020projection} first  estimates the regression function $f$ using {unconstrained, nonparametric} methods (e.g. by sieved splines), and then evaluate and calibrate the $L^2$ distance between the estimator and the space of concave functions. 
As discussed in Appendix \ref{subsec:minimax_FS}, 
the FS test is expected to achieve descent power, but fails to control the size properly when $f$ is not smooth; this is because if $f$ is not smooth enough,  there is {no choice} of tuning parameter (e.g., the number of terms in sieved B-splines) that can meet its two requirements simultaneously: under-smoothing and strong approximation (see Appendix \ref{subsec:minimax_FS}). In simulation studies (Section \ref{sec:simulation}), we observe that the FS test rejects $H_0$  with a very large probability when $f$ is concave, piecewise affine. In contrast, for our procedure,
the probability of rejecting the null  attains the maximum when $f$ is an affine function, as the equality in \eqref{def:concavity} is achieved if and only if $f$ is affine; 
thus, the size validity requires no additional assumption on $f$. Finally, we show in Section \ref{sec:simulation} that the proposed method achieves a comparable power to the FS test.


We postpone the discussion of related work  on the distribution approximation and bootstrap for $U$-processes until Subsection \ref{subsec:related_U_proc}.

\subsection{Organization of the paper} In Section \ref{sec:prelim}, we introduce stratified, incomplete $U$-processes, as well as bootstrap calibration, for a general function class $\cH$. In Section \ref{sec:local_simplex_stat}, we apply the general theory to the concavity test application, and  establish its size validity and power consistency.
In Section \ref{sub:comp_complex}, we  discuss the computational complexity of the proposed procedure as well as its implementation. In Section \ref{sec:simulation}, we present simulation results for $d=2$, with the cases of $d=3,4$ presented in Appendix \ref{app:simulations}. In Section \ref{sec:GAR_Bootstrap}, we establish the validity of Gaussian approximation and bootstrap for stratified, incomplete $U$-processes. 
The additional results, proofs, and discussions  are presented in Appendix.


\subsection{Notation} \label{subsec:notation}

We denote $X_i,  \ldots X_{i'}$ by $X_i^{i'}$ for $i \leq i'$. 
For any integer $n$, we denote by $[n]$ the set $\{1,2,\ldots,n\}$.
 For $a,b \in \bR$, let 
$\lfloor a \rfloor$ denote the largest integer that does not exceed $a$,
$a \vee b = \max\{a,b\}$ and 
$a \wedge b = \min\{a,b\}$.
For $a \in \bR^d$ and $q \in [1,\infty)$, denote
$\|a\|_{q} = \left( \sum_{i=1}^{d} |a_i|^q\right)^{1/q}$,
and $\|a\|_{\infty} = \max_{i \in [d]} |a_i|$.
 For $a,b \in \bR^d$, we write $a \leq b$ if
$a_j \leq b_j$ for $1 \leq j \leq d$, and write $[a,b]$ for the hyperrectangle $\prod_{j=1}^{d}[a_j,b_j]$ if $a \leq b$. 
For $\beta > 0$, let $\psi_{\beta}: [0,\infty) \to \bR$ be a function defined by
$\psi_{\beta}(x) = e^{x^\beta}-1$, and for any real-valued random variable $\xi$, define
$\|\xi\|_{\psi_{\beta}} = \inf\{C > 0: \Exp[\psi_{\beta}(|\xi|/C)] \leq 1\}$. Denote by $I_{n,r} := \{\iota = (i_1,\ldots, i_r): 1 \leq i_1 < \ldots < i_r \leq n\}$ the set of all ordered $r$-tuples of $[n]$ and denote by $|\cdot|$ the set cardinality.


For a nonempty set $T$, denote $\ell^{\infty}(T)$ the Banach space of real-valued functions $f:T\to \bR$ equipped with the sup norm $\|f\|_T := \sup_{t \in T}|f(t)|$. For a semi-metric space $(T,d)$, denote by $N(T,d,\epsilon)$ its $\epsilon$-covering number, i.e., the minimum number of closed $d$-balls with radius $\epsilon$ that cover $T$; see \cite[Section 2.1]{van1996weak}.  For a probability space $(T,\cT,Q)$ and a measurable function $f: T \to \bR$, denote 
$Qf = \int f dQ$ whenever it is well defined. For $q \in [1,\infty]$, denote by $\|\cdot\|_{Q,q}$ the $L^q(Q)$-seminorm, i.e., $\|f\|_{Q,q} = \left(Q|f|^q\right)^{1/q}$ for $q < \infty$ and
$\|f\|_{Q,\infty}$ for the essential supremum.

For $k = 0,1,\ldots,r$ and a measurable function $f:(S^r,\cS^r) \to (\bR,\cB(\bR))$, let $P^{r-k} f$ denote the function on $S^k$ such that
$P^{r-k}f(x_1,\ldots,x_k)\; =\; \Exp[f(x_1,\ldots,x_k, X_{k+1},\ldots,X_{r})],
$,
whenever it is well defined. For a generic random variable $Y$ , let $\Pro_{\vert Y}(\cdot)$ and $\Exp_{\vert Y}[\cdot]$ denote the conditional probability and expectation given $Y$, respectively. 
Throughout the paper, we assume that 
\begin{equation*}
r \geq 2,  \;\; n \geq 4, \;\; N\geq 4, \;\;
p_n := N/|I_{n,r}|\leq 1/2, \;\; N \geq n/r \geq 1.
\end{equation*}
Also, we assume the probability space  is rich enough in the sense that there exists a random variable that has the uniform distribution on $(0,1)$ and is independent of all other random variables.

\section{Stratified incomplete U-processes}
\label{sec:prelim}

In this section, we introduce stratified, incomplete $U$-processes, as well as bootstrap calibration, for a general function class $\cH$. For intuitions, it may help to think $\cH$ as the collection of functions 
$h_{v}^{\text{sg}}$ in \eqref{def:sg_kernel} indexed by $v \in \cV$, and refer to Figure \ref{fig:main_idea}.

Thus, let $X_1^n := \{X_{1},\dots,X_{n}\}$ be independent and identically distributed (i.i.d.) random variables taking value in a measurable space $(S,\cS)$ with common distribution $P$. Fix $r \geq 2$, and let $\cH$ be a collection of symmetric, measurable functions $h: (S^{r},\cS^r) \to (\bR,\cB(\bR))$.
Define the $U$-process and its standardized version as follows: for $h \in \cH$,
\begin{align*}
U_n(h) := {|I_{n,r}|^{-1}} \sum_{\iota \in I_{n,r}} h(X_{\iota}), \quad
\bU_n(h) := \sqrt{n}\left(U_n(h) - \Exp\left[U_n(h) \right] \right),
\end{align*}
where recall that  $X_{\iota} = (X_{i_1},\ldots,X_{i_r})$ if $\iota=(i_1,\ldots,i_r) \in I_{n,r}$. 
The summation in the above complete $U$-process involves $\sim n^r$ terms, and thus is computationally expensive even for a moderate $r$ (say $\geq 3$), which motivates  its stratified incomplete version.

\subsection{Test statistics}

Let $\{\cH_m: m \in [M]\}$ be a partition of $\cH$, i.e., $\cH_{m_1} \cap \cH_{m_2} = \emptyset$ for $m_1 \neq m_2$, and $\cup_{m=1}^{M} \cH_m = \cH$. The partition,
and thus $M$, may  depend on the sample size $n$. Given a positive integer $N$, which represents a computational parameter, define 
\[
\left\{ Z^{(m)}_{\iota}:   m \in [M], \;\; \iota \in I_{n,r} 
\right\} \;\;\;  \overset {i.i.d.}{\mathlarger{\sim}} \;\;\; \text{Bernoulli}(p_n), \quad \text{ with } p_n := N/|I_{n,r}|,
\] 
which are independent of the data $X_1^n$. For $m \in [M]$, denote by
$\widehat{N}^{(m)} := \sum_{\iota \in I_{n,r}} Z^{(m)}_{\iota}$ the total number of sampled $r$-tuples for the subclass $\cH_m$.
Further, define a function $\sigma: \cH \to \{1,\ldots,M\}$ that maps $h \in \cH$ to the index of the partition to which $h$ belongs, i.e.,
$
\sigma(h) = m  \Leftrightarrow  h \in \cH_m.
$
Finally, we define
the \textit{stratified, incomplete $U$-process} and its standardized version: for $h \in \cH$, if $\sigma(h)=m$,
\begin{align}\label{def:incomplete_U_proc}
U_{n,N}'(h) := \left(\widehat{N}^{(m)}\right)^{-1} \sum_{\iota \in I_{n,r}} Z_{\iota}^{(m)} h(X_{\iota}), \quad
\bU_{n,N}'(h) := \sqrt{n}\left(U_{n,N}'(h) - \Exp\left[U_{n,N}'(h) \right] \right).
\end{align}

An important goal of the paper is to develop bootstrap methods to calibrate the distribution of the supremum of the stratified incomplete $U$-process, i.e., $\bMn := \sup_{h \in \cH} \bU_{n,N}'(h)$. 

\noindent \textbf{\underline{Statistical tests.}} We will use $\sqrt{n}\sup_{h \in \cH}{U_{n,N}'(h)}$ as the \textit{test statistic}, which can be evaluated given the data $X_1^n$ and sampling plans $\{Z_\iota^{(m)}\}$. If under the null, $P^r h \leq 0$ for each $h \in \cH$, then
\[
\sqrt{n}\sup_{h \in \cH}{U_{n,N}'(h)} \;\leq\; \sup_{h \in \cH} \bU_{n,N}'(h) \;=\;\bMn.
\]
Thus a test based on the $\alpha$-th upper quantile of $\bMn$ controls the size below $\alpha$. If, in addition, under certain configuration in the null,  $P^r h = 0$ for each $h \in \cH$, then the test is \textit{non-conservative}, i.e., controlling the size at $\alpha$.

\begin{remark}
A stratification of 
$\cH=\{h_{v}^{\text{sg}}: v \in \cV\}$ is equivalent to partitioning $\cV$ into sub-regions $\{\cV_m: m \in [M]\}$ and letting $\cH_m = \{h_{v}^{\text{sg}}: v \in \cV_m\}$ (see Figure \ref{fig:main_idea}). Query points in $\cV_m$ share the same sampling plan $\{Z_{\iota}^{(m)}:\iota \in I_{n,r} \}$. 
%
As we shall see in Section \ref{sub:comp_complex}, it is {computationally} important to partition the function class $\cH$ so that each partition has its individual sampling plan. Our analysis is non-asymptotic, so no stratification ($M=1$) is also allowed.
\end{remark}

\subsection{Bootstrap calibration}\label{sec:bootstrap_uproc}

To operationalize the above test, we  use multiplier bootstrap to calibrate the distribution of $\bMn$. To gain intuition, assume for a moment $P^r h = 0$ for $h \in \cH$, and observe that
\begin{equation}\label{bootstrap_intuition}
({\widehat{N}^{\sigma(h)}}{N}^{-1})  \bU_{n,N}'(h) =   \bU_{n}(h) + \alpha_n^{1/2}  N^{-1/2} \sum_{\iota \in I_{n,r}} \left(Z_{\iota}^{(\sigma(h))} -p_n\right) h(X_{\iota}),
\end{equation}
where $\alpha_n := n/N$. The first term on the right is a complete $U$-statistic, and thus is  approximated by its H\'ajek projection $r n^{-1/2}\sum_{k \in [n]} P^{r-1}h(X_k)$. The second term is due to stratified sampling:  {conditional} on data $X_1^r$,  it is a sum of independent centered Bernoulli random variables, 
 with variance approximately given by $\alpha_n U_n(h^2)$. We will handle these two sources of  variation.

 
The H\'ajek projection part requires  additional notations. 
Let $\cD_{n} := X_1^n \cup \{Z_{\iota}^{(m)}: \iota \in I_{n,r}, m \in [M]\}$ be the data involved in the definition of $\bU'_{n,N}$ in  \eqref{def:incomplete_U_proc}. For each $k \in [n]$, denote by 
\begin{align*}
\kjack:=
\{\;(i_1,\ldots,i_{r-1}): 1 \leq i_1 < \cdots < i_{r-1} \leq n, \quad {i_j} \neq k \;\text{ for } 1 \leq j \leq r-1 \;\},
\end{align*}
the collection of all ordered $r-1$ tuples in the set $\{1,\ldots,n\}\setminus \{k\}$. Let $N_2$ be another computational budget,
and define
\[
\left\{Z^{(k, m)}_{\iota}: \;k \in [n], \; m \in [M], \; \iota \in  \kjack  \right\}  \;  \overset {i.i.d.}{\mathlarger{\sim}} \; \text{Bernoulli}(q_n), \;\;\; q_n := N_2/|I_{n-1,r-1}|,
\]
that are independent of $\cD_n$. For example, if $\cH=\{h_{v}^{\text{sg}}: v \in \cV\}$, {each pair} of data point $X_k$ and region $\cV_m$ is associated with an independent sampling plan $\{Z^{(k, m)}_{\iota}:\iota \in \kjack  \}$; see Figure \ref{fig:main_idea}.

For $k \in [n]$ and $m \in [M]$, define
$
\widehat{N}_2^{(k, m)} := \sum_{\iota \in \kjack} Z^{(k, m)}_{\iota}$,  the number of selected $r-1$ tuples from $[n]\setminus \{k\}$ for $X_k$ and $\cH_{m}$. Further, 
for $h \in \cH$ with $\sigma(h) = m$,
\begin{equation}
\label{def:G_k_h}
\bG^{(k)}(h) := \left(\widehat{N}_2^{(k, m)}\right)^{-1} \sum_{\iota \in \kjack} 
Z^{(k, m)}_{\iota} h( X_{\iota^{(k)}}), \quad
\overline{\bG}(h) := n^{-1} \sum_{k = 1}^{n} \bG^{(k)}(h),
\end{equation}
where  $\iota^{(k)} := \{ k \} \cup \iota$. 
Here, $\bG^{(k)}(h)$ is intended as an estimator for the $k^{th}$ term  in the H\'ajek projection, since by definition
$\Exp\left[h( X_{\iota^{(k)}}) \, \vert X_k \right]=P^{r-1}h(X_k)$.

\noindent  \textbf{\underline{Multiplier bootstrap.}}
Now let  $\left\{\xi_k,\; \xi^{(m)}_{\iota}: \;k \in [n],\;m \in [M], \iota \in I_{n,r} \right\}  \;  \overset {i.i.d.}{\mathlarger{\sim}} \; N(0,1)$ be independent standard Gaussian multipliers, independent from the data $X_1^n$ and the sampling plans, i.e., 
 \begin{equation}
 \label{def:Dn_prime}
 \cD_n' := \cD_n \cup \{Z^{(k, m)}_{\iota}: k \in [n],\; m \in [M], \;  \iota \in  \kjack   \}.
 \end{equation}
Define for $h \in \cH$ with $\sigma(h) = m$,
\begin{align} \label{def:bootstrap_A_and_B}
\begin{split}
\bU_{n,A}^{\#}(h) &:= n^{-1/2}\sum_{k =1}^{n} \xi_{k}\left(\bG^{(k)}(h) - \overline{\bG}(h) \right),\\
\bU_{n,B}^{\#}(h) &:= \left(\widehat{N}^{(m)}\right)^{-1/2} \sum_{\iota \in I_{n,r}} \xi^{(m)}_{\iota} \sqrt{Z_{\iota}^{(m)}} \left(
h(X_{\iota}) - U'_{n,N}(h)
\right),
\end{split}
\end{align}
where $0/0$ is interpreted as $0$. Note that the multipliers $\{\xi_k\}$ are shared across regions, while $\{\xi^{(m)}_{\iota}\}$ are region-specific. Further,  conditional on $\cD'_n$, $\bU^{\#}_{n,A}$ and $\bU^{\#}_{n,B}$ are centered Gaussian processes with covariance functions  $\widehat{\gamma}_A(h,h') := {n}^{-1}\sum_{k = 1}^{n}  (\bG^{(k)}(h) - \overline{\bG}(h))(\bG^{(k)}(h') - \overline{\bG}(h'))$ and
$\widehat{\gamma}_B(h,h') := (\widehat{N}^{(\sigma(h))})^{-1} 
\sum_{\iota \in I_{n,r}} Z_{\iota}^{(\sigma(h))} (
h(X_{\iota}) - U'_{n,N}(h)
)(
h'(X_{\iota}) - U'_{n,N}(h')
) \mathbbm{1}\{\sigma(h) = \sigma(h')\}$ for any $h, h' \in \cH$.
In view of \eqref{bootstrap_intuition}, we  combine these two processes and define 
\begin{equation}
\label{def:bootstrap_combined}
\bU_{n,*}^{\#}(h) := r \bU_{n,A}^{\#}(h) + \alpha_n^{1/2} \bU_{n,B}^{\#}(h)\;\; \text{ for } h \in \cH,  \qquad
\bMn^{\#} := \sup_{h \in \cH} \bU_{n,*}^{\#}(h).
\end{equation}
Finally, we estimate the conditional (given $\cD_n'$) distribution of $\bM_n^{\#}$ by  bootstrap, i.e., by repeatedly generating independent realizations of
the multipliers $\{\xi_k,   \xi^{(m)}_{\iota}\}$ with the data $X_1^n$ and the sampling plans  $\{Z_{\iota}^{(m)}, Z_{\iota}^{(k,m)}\}$  fixed, and   obtain the critical value for the previous test statistic from the conditional distribution of $\bMn^{\#}$.



\subsection{A simplified version of approximation results}\label{subsec:simplified_results}
To justify the bootstrap procedure, we need to show that conditional on $\cD_n'$, the distribution of $\bMn^{\#}$ is close to that of $\bMn$, which is the main result  in Section \ref{sec:GAR_Bootstrap}. Here we state a simplified version of the approximation results for a {uniformly bounded} function class $\cH$. Note that the bound on $\cH$ is allowed to vary with $n$.



\begin{definition}[VC type function class \cite{chen2017jackknife,chernozhukov2014gaussian}] \label{def:VC_type}
A collection,  $\cH$, of functions on $S^r$ with a measurable envelope function $H$ (i.e. $H \geq \sup_{h \in \cH} |h|$ pointwise) is said to be {VC type} with characteristics $(A,\nu)$ if $\sup_{Q} N(\cH, \|\cdot\|_{Q,2}, \epsilon \|H\|_{Q,2}) \leq (A/\epsilon)^{\nu}$ for any $\epsilon \in (0,1)$, where $\sup_{Q}$ is taken over all finitely discrete probability measures on $S^r$.
\end{definition}

We  work with the following assumptions.\\

\noindent \mylabel{cond:PM}{(PM)}.  $\cH$ is pointwise measurable in the sense that for any $n \in \bN$, there exists a countable subset $\cH'_n \subset \cH$ such that, almost surely, for every $h \in \cH$, there exists a sequence $\{h_{m}\} \subset \cH'_n$ with $\lim_{m} h_m(X_i)= h(X_i)$ for $i \in [n]$.  \\


\noindent \mylabel{cond:VC}{(VC)}. $\cH$ is VC type with envelope $H$ and characteristics $A \geq e \vee (e^{2(r-1)}/16)$ and $\nu \geq 1$. \\

\noindent \mylabel{cond:MB}{(MB)}. For some absolute constant $C_0 > 0$,
$\log(M) \leq C_0 \log(n)$. \\

\noindent \mylabel{cond:MT_inf}{(MT-$\infty$)}. There exist  absolute constants $\ubar{\sigma} > 0$, $c_0 \in (0,1)$,  and a sequence of reals $D_n \geq 1$ such that for each  $0 \leq \ell \leq r$ and $1 \leq s \leq 4$,
\begin{align}
&\Var\left(P^{r-1}h(X_1) \right) \geq \ubar{\sigma}^2,\qquad
\Var\left(h(X_1^r) \right) \geq c_0 D_n^{2r-2}, \nonumber \\
&\|P^{r-\ell} |h|^{s}\|_{P^{\ell}, \bar{q}} \; \leq  \; 
 D_n^{2r(s-1) + 2\ell - s -2\ell/\bar{q}}, \quad \text{ for }\; \bar{q} \in \{2,3,4\},\; h \in \cH, \nonumber\\
&\|P^{r-\ell} H^{s}\|_{P^{\ell}, \bar{q}} \; \leq  \; 
D_n^{2r(s-1) + 2\ell - s - (2\ell-2)/\bar{q}}\; \text{ for }\; \bar{q} \in \{2,\infty\},\quad \|(P^{r-2}H)^{\bigodot 2}\|_{P^2,\infty} \leq
D_n^{4}, \nonumber
\end{align}
where  for a function $f:S^2 \to \bR$, define 
$f^{\bigodot 2}(x_1,x_2) := \int f(x_1,x)f(x_2,x) dP(x)$.

\begin{theorem}
\label{thm:calibration_inf}
Assume the conditions \ref{cond:PM}, \ref{cond:VC}, \ref{cond:MB}  and \ref{cond:MT_inf}. 
Then there exists a constant $C$, depending only on constants $r, \ubar{\sigma}, c_0, C_0$, such that with probability at least $1 - C \varrho_n'$,
$$
\sup_{t \in \bR}\left|  \Pro({\bMn} \leq t) - 
\Pro_{\vert \cD_n'}(\bMn^{\#} \leq t)
\right| \leq C \varrho_n', 
$$
where $\varrho_n' :=  \left( \frac{D_n^{2r} K_n^{7}}{N \wedge N_2} \right)^{1/8}  +  \left( \frac{D_n^2 K_n^{7}}{n} \right)^{1/8} + 
\left( \frac{D_n^{3} K_n^{4}}{n} \right)^{2/7}$ and $K_n := \nu\log(A \vee n)$.
\end{theorem}
\begin{proof}
It follows  from Theorem \ref{thm:GAR_sup_incomplete_Uproc} and Theorem \ref{thm:uproc_bootstrap}. Specifically,  \ref{cond:MT_inf} verifies \ref{cond:MT} with $q = \infty$ and $B_n = D_n^r$. Further, we may without loss of generality assume that  $\eta_n^{(1)}$, $\eta_n^{(2)}$  and $\rho_n$, in Theorem \ref{thm:GAR_sup_incomplete_Uproc} and   \ref{thm:uproc_bootstrap}, are bounded by 1, and then it is clear that $\eta_n^{(1)}+\eta_n^{(2)} + \rho_n \leq C \varrho'_n$.
\end{proof}

\begin{remark}\label{general_MB_discussion}
The condition \ref{cond:MB} requires  $\log(M) \leq C_0 \log(n)$, and the impact of $M$ has been  absorbed into $K_n$, since $K_n \geq \log(n)$.


The condition \ref{cond:MT_inf} are motivated by  the application of testing the concavity of a regression function in Section \ref{sec:local_simplex_stat}. It holds if we use (i). the sign kernel $\{h_{v}^{sg}: v \in \cV\}$ in \eqref{def:sg_kernel}
 or  (ii). the identity kernel $\{h_{v}^{sg}: v \in \cV\}$ in \eqref{def:h_x_concavity} under the additional assumption that the observation noise $\varepsilon$ in \eqref{eqn:nonparametric_regression} is {bounded}; the more general results in Section \ref{sec:GAR_Bootstrap} are required to remove this assumption.
\end{remark}

\section{Stratified incomplete local simplex tests: statistical guarantees}
\label{sec:local_simplex_stat}

In this section, we apply the general theory in Section \ref{sec:prelim} to   the concavity test of a regression function, i.e., $H_0$ in \eqref{def:H_0}, formally introduce  
stratified incomplete local simplex tests, and establish the size validity and power consistency. Finally, we propose tests that combine  multiple bandwidths.

We first recall the {\it simplex statistics}   proposed in \cite{abrevaya2005nonparametric}. 
For $v_{1},\dots,v_{d+1} \in \R^{d}$, denote by 
\[
\Delta^{\circ} (v_{1},\dots,v_{d+1}) := \Big\{ \sum_{i=1}^{d+1} a_{i} v_{i} :  \sum_{i=1}^{d+1}a_{i} = 1,\;\; a_i > 0 \text{ for } i \in [d+1]\Big\}
\]
the interior of the simplex spanned by $v_{1},\dots,v_{d+1}$, and define $\calS := \bigcup_{j=1}^{r} \calS_{j}$, where $r := d+ 2$ and
\begin{align*}
\calS_{j} =  \Big \{ (v_{1},\dots,v_{r}) \in \R^{d \times r}:  \;
\begin{split} 
&v_{1},\dots,v_{j-1},v_{j+1},\dots,v_{r} \ \text{are affinely independent} \\ 
&\text{ and }  v_{j} \in \Delta^{\circ}(v_{1},\dots,v_{j-1},v_{j+1},\dots,v_{r})
\end{split}
\Big \}.
\end{align*}
Clearly, $\calS_{1},\dots,\calS_{r}$ are disjoint. To illustrate, in Figure \ref{fig:main_idea}, $(v_1,v_2,v_3,v_4) \in \cS_4$, but $(v_1,v_2,v_3,v_5) \not\in \cS$.

For $j \in [r]$, there exists a unique collection of functions $\{\tau^{(j)}_i: \calS_{j} \to (0,1): i \in [r]\setminus\{j\}\}$ such that for any $v_1^r := (v_{1},\dots,v_{r}) \in \calS_j$, 
\begin{align}
\label{tau_functions}
v_{j} = \sum_{ i \in [r]\setminus\{j\}}  \tau^{(j)}_i(v_1^r)\, v_i, \qquad
\sum_{ i \in [r]\setminus\{j\}}  \tau^{(j)}_i(v_1^r)\,  = 1.
\end{align}
Now define $w: \bR^{(d+1)\times r} \to \bR$ as follows: for $x_i := (v_i,y_i) \in \bR^{d+1}, i \in [r]$,
\begin{align}\label{def:id_kernel}
w(x_{1},\dots,x_{r}) := \sum_{j=1}^{r} 
\left(\sum_{ i \in [r]\setminus\{j\}}  \tau^{(j)}_i(v_1^r)\, y_i - y_j \right) \mathbbm{1}\left\{v_1^r \in \calS_{j} \right\}.
\end{align}
It is clear that $\calS$ is permutation invariant for $v_{1},\dots,v_{r}$, and that $w(\cdot)$ is symmetric in its arguments. Key observations are that if 
the regression function $f$ is concave (i.e.~$H_0$ holds), then $P^r w \leq 0$, 
and that if $f$ is an affine function, $P^r w = 0$, where recall that $P$ is the distribution of $X := (V,Y)$.

Let $L(\cdot)$ be a kernel function and $b_n > 0$  a bandwidth parameter.  Recall that $L_{b}(\cdot) := b^{-d} L(\cdot/b)$ for $b > 0$, and define $\cH^{\text{id}} := \{h_v^{\text{id}}: v \in \cV\}$, where
 for each $x_i = (v_i,y_i) \in \bR^{d+1}, i \in [r]$,
\begin{equation}
\label{def:h_x_concavity}
h_{v}^{\text{id}}(x_{1},\dots,x_{r}) := w(x_1, \dots, x_{r}) b_n^{d/2}\prod_{k=1}^{r} L_{b_n}(v-v_{k}), \quad v \in \cV.
\end{equation}
Now consider a partition of $\cV$,  $\{\cV_m: m \in [M]\}$, which induces a partition of 
$\cH^{\text{id}}$, i.e., $\cH_m^{\text{id}} := \{h_v^{\text{id}}: v \in \cV_m\}, m \in [M]$; see Figure \ref{fig:main_idea}.


Finally, recall the definitions of $U_{n,N}'(\cdot)$ in \eqref{def:incomplete_U_proc}, $\cD_n'$ in \eqref{def:Dn_prime}, and $\bMn^{\#}$ in 
\eqref{def:bootstrap_combined}. Given a nominal level $\alpha \in (0,1)$,
 we propose to reject the null  in \eqref{def:H_0} if and only if
\begin{equation}
\label{def:test_proc}
\sup_{v \in \cV} \sqrt{n} U_{n,N}'(h_{v}^{\text{id}}) \;\; \geq \;\; q^{\#}_\alpha,
\end{equation}
where $q^{\#}_\alpha$ is the $(1-\alpha)$-th quantile of $\bMn^\#$ conditional on $\cD_n'$.

\noindent \underline{\textbf{Sign function.}}
We also consider the  function class $\cH^{\text{sg}} := \{h_{v}^{\text{sg}}: v \in \cV\}$, where $h_{v}^{\text{sg}}$ is defined in \eqref{def:sg_kernel}. As we shall see, $\{h_{v}^{\text{sg}}: v \in \cV\}$ has the advantage of being bounded, but it requires that the conditional distribution of $\varepsilon$ given $V$ is symmetric about   zero. On the other hand,   $\{h_{v}^{\text{id}}: v \in \cV\}$ imposes no assumption on the shape of the conditional distribution, but requires $\varepsilon$ to have a light tail; in this Section, we  assume $\varepsilon$ to be bounded for $\{h_{v}^{\text{id}}: v \in \cV\}$ so that we can apply Theorem \ref{thm:calibration_inf}, and   relax this assumption in Appendix.

\subsection{Assumptions for concavity tests}\label{subsec:SILS_assumptions}
We assume the distribution of $(V,\varepsilon)$ in \eqref{eqn:nonparametric_regression} to be fixed, but allow $f$ to depend on the sample size $n$, which permits the study of local alternatives. We  make the following assumptions: for some absolute constant  $C_0 > 1$,\\ 

\noindent \mylabel{C:kernel}{(C1)}. The kernel $L: \bR^d \to \bR$ is continuous, of bounded variation, and
has support $[-1/2,1/2]^d$. {Or} $L(\cdot)$ is the uniform kernel on $[-1/2,1/2]^d$, i.e., $L(v) = \mathbbm{1}\{v \in (-1/2,1/2)^d\}$ for $v \in \bR^d$. \\


 


\noindent \mylabel{C:M}{(C2)}. The number of partitions, $M$, grows at most polynomially in  $n$, i.e., $\log(M) \leq C_0 \log(n)$.\\

\noindent \mylabel{C:bn}{(C3)}. The bandwidth $b_n$ does not vanish too fast in $n$, i.e.,
$1 \leq b_n^{-3d/2} \leq C_0 n^{1-1/C_0}$.\\

\noindent \mylabel{C:X}{(C4)}. $V$ has a Lebesgue density $p$ such that $C_0^{-1} \leq p(v) \leq C_0$ for $v \in \cV^{2b_n}$, where  $\cV^{b}:= \{ v' \in \bR^d: \inf_{v'' \in \cV}\|v'  - v'' \|_{\infty} \leq b \}$ is the $b$-enlargement of $\cV$. \\

\noindent \mylabel{C:non_degenerate}{(C5)}. Assume for 
$* = \text{id or sg}$ and $n \geq C_0$, 
$
\inf_{v \in \cV} \Var\left(
P^{r-1} h_v^{*}(X_1)
\right) \geq C_0^{-1}
$.\\

\noindent \mylabel{C:id_noise_prime}{(C6-id')} Assume that 
$\sup_{v \in \cV^{b_n}} |f(v)| \leq C_0$ and that 
$\varepsilon$ is bounded by $C_0$ almost surely.\\

\noindent \mylabel{C:sg_noise_prime}{(C6-sg')}
Assume that 
$\sup_{v \in \cV^{b_n}} |f(v)| \leq C_0$, and that $\varepsilon$ is independent of $V$, symmetric about zero, i.e., $\Pro\left( \varepsilon > t  \right) = \Pro\left( \varepsilon < -t \right)$ for any $t > 0$, and $\Pro\left(\varepsilon > 2 C_0\right) > 0$.\\

Some comments  are in order. 
\ref{C:kernel} is a standard assumption on the kernel $L$, which is satisfied by many commonly used kernels.  
Recall that $\cV$ is compact, so \ref{C:M} is satisfied if we partition each coordinate into segments of length $\eta b_n$ for some small $\eta \in (0,1)$. \ref{C:bn} imposes the same condition on the bandwidth $b_n$ as for the procedure using the complete $U$-process \cite[(T5) in Section 4]{chen2017jackknife}, which holds as long as $n^{-2/(3d) + \eta} \lesssim b_n$ for arbitrarily small $\eta \in (0,1)$; in comparison, \cite{abrevaya2005nonparametric} has a (slightly) milder condition on the bandwidth, $n^{-1/d  + \eta} \lesssim b_n$ for the discretized $U$-statistics. \ref{C:X} is necessary that for each $v \in \cV$, there are enough data points in the $b_n$-neighbourhood of $v$. The condition \ref{C:id_noise_prime} is assumed for the class $\cH^{\text{id}}$, while \ref{C:sg_noise_prime} for $\cH^{\text{sg}}$. 


Now we focus on \ref{C:non_degenerate} with the function class $\cH^{\text{id}}$, as the discussion for $\cH^{\text{sg}}$ is similar. For simplicity, assume $\varepsilon$ and $V$ in \eqref{eqn:nonparametric_regression} are independent. 
By a change-of-variable and due to the fact that 
$\tau_i^{(j)}$ in \eqref{tau_functions} is invariant under affine transformations (in particular $\tau_i^{(j)}(v-b_n u_1,\ldots, v-b_n u_r) = \tau_i^{(j)}( u_1,\ldots, u_r)$),
\begin{align*}
&\Exp\left[ \Var\left(
P^{r-1} h_v^{\text{id}}(V_1,Y_1) \vert V_1 \right)
\right] = \Var(\varepsilon_1) \int p(v-b_n u_1) L^2(u_1) \mathcal{T}_v^2(u_1) du_1, \quad \text{ where }\\
&\mathcal{T}_v(u_1) =
\int \left(\mathbbm{1}\{u_1^r \in \calS_1\} - \sum_{j=2}^{r} \tau^{(j)}_1(u_1^r) \mathbbm{1}\{u_1^r \in \calS_j\}\right) \prod_{i=2}^{r} L(u_i) p(v - b_n u_i) du_i.
\end{align*}
The key observation is that $\Var(
P^{r-1} h_v^{*}(X_1))$ does not vanish as $b_n \to 0$. Then we can find more primitive conditions for \ref{C:non_degenerate}. For example, \ref{C:non_degenerate} holds if $L(\cdot) = \mathbbm{1}\{\cdot \in (-1/2,1/2)^d\}$,  $p$ is continuous on $\cV$, and $\lim_{n \to \infty} b_n = 0$.

\begin{remark}
In Appendix \ref{proof:thm:id_kernel}, we relax the condition 
\ref{C:id_noise_prime}, requiring $\varepsilon$ to have a light tail, instead of being bounded. Further, in Appendix  \ref{proof:thm:sg_kernel}, we relax the condition \ref{C:sg_noise_prime}, allowing $\varepsilon$ and $V$ to be dependent.
\end{remark}
\subsection{Size validity and power consistency}\label{subsec:sz_power}
The following is the master theorem for the statistical guarantees for the stratified incomplete local simplex tests.

\begin{theorem}\label{thm:concavity_master}
Consider the function class  $\cH^{\text{id}}$  or $\cH^{\text{sg}}$.
Assume that \ref{C:kernel}-\ref{C:non_degenerate} hold and that
\ref{C:id_noise_prime} (resp. \ref{C:sg_noise_prime}) holds for  $\cH^{\text{id}}$ (resp.  $\cH^{\text{sg}}$). Further,  for some $\kappa, \kappa' > 0$, 
\begin{align}\label{cond:N_N_2_concavity}
N  := n^{\kappa} b_n^{-dr}, \quad
N_2  := n^{\kappa'} b_n^{-dr}.
\end{align}
Then there exists a constant $C$, depending only on $C_0, d, \kappa, \kappa'$, such that with probability at least $1 - C n^{-1/C}$,
\begin{align*}
\sup_{t \in \bR}\left|  \Pro({\bMn} \leq t) - 
\Pro_{\vert \cD_n'}(\bMn^{\#} \leq t)
\right| \leq C n^{-1/C}.
\end{align*}
\end{theorem}
\begin{proof}
In Appendix \ref{proof:thm:id_kernel} and \ref{proof:thm:sg_kernel}, we show that  \ref{cond:PM} and \ref{cond:VC} is implied by \ref{C:kernel}, where the latter is due to \cite[Proposition 3.6.12]{gine2016mathematical}. \ref{cond:MB}  is the same as  \ref{C:M}. Further, we verify that \ref{cond:MT_inf} holds  with $D_n = C b_n^{-d/2}$. Then the proof is complete by Theorem \ref{thm:calibration_inf} and due to the requirement on the bandwidth, i.e., \ref{C:bn}.
\end{proof}

\begin{remark}\label{rk:main_challenge}
The main challenge in working with $\{h_v^{\text{sg}}: v \in \cV\}$ (and also with $\cH^{\text{id}}$) is that the size of the projections of the kernels, $\{P^{r-\ell} |h_v^{\text{sg}}|: \ell = 0,1,\ldots, r\}$  has different orders of magnitude due to localization.
The same is true for the absolute moments of $\{|h_v^{\text{sg}}|^{s}\}$  for $s \geq 1$. 
Specifically, in Appendix \ref{proof:thm:sg_kernel}, we verify \ref{cond:MT_inf} holds for any $\bar{q} \geq 1$.  
Thus for a fixed $s$, projections onto consecutive levels differ by a factor of $b_n^{-d(1-1/\bar{q})}$. On the other hand, for a fixed $\ell$, the second moment ($s=2$) is greater than the first moment $(s=1)$ by a factor 
$b_n^{-d(r-1/2)}$.
\end{remark}


The next Corollary establishes the size validity of the proposed procedure. Among all concave functions, affine functions have the (asymptotically) largest rejection probabilities, which {attain} the nominal levels uniformly over $(0,1)$ for large $n$.

\begin{corollary}[Size validity]\label{cor:size_valid}
Consider the procedure \eqref{def:test_proc} for testing the hypothesis \eqref{def:H_0} with $\cH^{*}$ for 
$* = \text{id or sg}$. Assume  the conditions in  Theorem \ref{thm:concavity_master} hold. If the regression function $f$ is concave, i.e., $H_0$ holds, then for 
some constant $C$, depending only on $C_0, d, \kappa, \kappa'$, 
\begin{align*}
\Pro\left(
\sup_{v \in \cV} \sqrt{n} U_{n,N}'(h^{*}_v) \;\; \geq \;\; q^{\#}_\alpha\right)
\leq \alpha + C n^{-1/C}, \text{ for any  } \alpha \in (0,1).
\end{align*}
Further, if $f$ is an affine function, then 
\begin{align*}
\sup_{\alpha \in (0,1)}\left\vert \Pro\left(
\sup_{v \in \cV} \sqrt{n} U_{n,N}'(h^{*}_v) \;\; \geq \;\; q^{\#}_\alpha\right)
-\alpha  \right\vert \leq C n^{-1/C}.
\end{align*}
\end{corollary}
\begin{proof}
If $f$ is concave, then $P^r h_v^* \leq 0$ for $v \in \cV$. Further, if $f$ is affine, then $P^r h_v^* = 0$ for $v \in \cV$. Then the results follow   from Theorem \ref{thm:concavity_master}.
\end{proof}

The next Corollaries concern the power of the proposed procedure. The proofs can be found in Appendix \ref{proof:cor:power}.


\begin{corollary}[Power]\label{cor:power}
Consider the setup as in Corollary \ref{cor:size_valid}. If in addition
\begin{align}\label{equ:power_cond}
\sqrt{n} P^r h^{*}_{v_n} \geq (C_0)^{-1} n^{\kappa''}, \text{ for some }
v_n \in \cV, \;\; \kappa'' > \max\{(1 - \kappa)/2, 0\},
\end{align}
then  for  some constant $C$, depending only on $ C_0,   d, \kappa, \kappa', \kappa''$, 
\begin{align*}
\Pro\left(
\sup_{v \in \cV} \sqrt{n} U_{n,N}'(h^{*}_v) \;\; \geq \;\; q^{\#}_\alpha\right)
\geq 1 - Cn^{-1/C}, \text{ for any } \alpha \in (0,1).
\end{align*}
\end{corollary}


\begin{remark}
The condition \eqref{equ:power_cond} ensures that the bias $\sqrt{n} P^r h^{*}_{v_n}$ is significantly larger than the standard deviation of $\bMn$.  
If $\kappa \geq 1$, then $\kappa''$ in \eqref{equ:power_cond} can be arbitrarily small. 
Note that due to \ref{C:M}, the impact of $M$ is absorbed into the constant $C$.
\end{remark}

Next we provide examples for which \eqref{equ:power_cond} holds, and focus on the class $\cH^{\text{id}}$. 
The discussion for $\cH^{\text{sg}}$ is similar.

%

\begin{corollary}[Power - smooth $f$]\label{cor:power_smooth}
Consider the setup as in Corollary \ref{cor:size_valid}. 
Assume that $f$ is fixed and  twice continuously differentiable at some $v_0 \in \cV$ with a positive definite Hessian matrix at $v_0$, and that $\lim_{n \to \infty} b_n = 0$. Then $\liminf_{n \to \infty} {P^r h^{\text{id}}_{v_0}}/{\left(b_n^{2+d/2} \right)} > 0$. Thus if  $b_n^{-(d+4)} \leq C_0 n^{1-2\kappa''}$ for some $\kappa'' > \max\{(1-\kappa)/2,0\}$, then for any $\alpha \in (0,1)$,
the power converges to one as $n \to \infty$. 
\end{corollary}


\begin{remark}
Note that Theorem 7 in \cite{abrevaya2005nonparametric} establishes the consistency of their test using $\cH^{\text{sg}}$ (discrete, complete version)
under the condition that $nb_n^{d+4}/\log(n) \to \infty$, which is in the same spirit as the requirement on $b_n$ in Corollary \ref{cor:power_smooth}. 
\end{remark}

\begin{corollary}[Power - piecewise affine $f$]\label{cor:power_pieceAffine}
Consider the setup as in Corollary \ref{cor:size_valid}. For $j \in \{1,2\}$, let $\theta_{n,j} \in \bR^d$ and $\omega_{n,j} \in \bR$ such that $\theta_{n,1}\neq \theta_{n,2}$. Let
$$
f(v) = f_{n}(v) := \max \left\{f_{n,1}(v), f_{n,2}(v)\right\}, \;\;\text{ where }\;\; f_{n,j}(v) := \theta_{n,j}^T v + \omega_{n,j}\;\; \text{ for } j = 1,2.
$$
If there exists $v_n \in \cV$ such that $f_{n,1}(v_n) = f_{n,2}(v_n)$ for each $n$, then 
$$\liminf_{n \to \infty} {P^r h^{\text{id}}_{v_n}}/{\left(b_n^{1+d/2} \|\theta_{n,1}-\theta_{n,2}\|_2\right)} > 0.
$$ 
Thus if  $b_n^{-(d+2)} \leq C_0 n^{1-2\kappa''}\|\theta_{n,1}-\theta_{n,2}\|_2^2$  for some $\kappa'' > \max\{(1-\kappa)/2,0\}$, then for any $\alpha \in (0,1)$,
the power converges to one as $n \to \infty$.

\end{corollary}

\begin{remark}
If $f$ does not depend on $n$, in particular $\theta_{n,j} = \theta_{j}$ for each $n$, then the requirement on $b_n$ becomes 
$b_n^{-(d+2)} \leq C_0 n^{1-2\kappa''}$, which is weaker than that for smooth functions $f$ in Corollary \ref{cor:power_smooth}.

On the other, if we choose $b_n = b > 0$ 
for each $n$, then to achieve power consistency, we require $\|\theta_{n,1}-\theta_{n,2}\|_2 \geq C_{0}^{-1} n^{-1/2+\kappa''}$. Observe that $f_n$ is convex if $\theta_{n,1} \neq \theta_{n,2}$, and affine if $\theta_{n,1} = \theta_{n,2}$. Thus this allows ``local alternatives" that approach the null at the rate of  $n^{-1/2+\kappa''}$.
\end{remark}

\subsubsection{Discussions}\label{subsubsec:test_discussion}
The stratified incomplete local simplex test  (SILS) 
is a least favorable configuration test, with affine functions being (asymptotically) least favorable. 
This type of test was first proposed for testing the monotonicity of a (univariate) regression function by \cite{ghosal2000}, and then extended to test the (multivariate, coordinate-wise) {stochastic monotonicity} by \cite{lee2009testing}, and to test the (multivariate) convexity by \cite{abrevaya2005nonparametric}. See also \cite{chen2017jackknife} for the distribution  approximation of these test statistics. It is not clear how to extend this idea to test other shape constraints, such as quasi-convexity \cite{komarova2019testing}, because it seems difficult to identify the least favourable configuration, or to compute the expectation of test statistics under it.

From Corollary \ref{cor:size_valid}, the SILS test is asymptotically {non-conservative}; however, it is {non-similar} \cite{lehmann2006testing}, in the sense that for strictly concave functions, the probability of rejection is {strictly} less than the nominal level $\alpha$. Being non-similar alone is not evidence again the SILS test (e.g., Z-test for normal means is optimal despite being non-similar), but  a least favorable configuration test may be less powerful than alternative tests. The condition \eqref{equ:power_cond} requires ``local convex curvature" of $f$ for the test to be power consistent; see Corollary \ref{cor:power_smooth} and \ref{cor:power_pieceAffine} for examples. 
The question of how \eqref{equ:power_cond} is related to the global $L_2$ separation rate (see Appendix \ref{subsec:minimax_FS}) is left for future research.

In Appendix \ref{subsec:minimax_FS}, we discuss the $L_2$ minimax separation rate  for concavity test, and an alternative test (``FS" test) \cite{fang2020projection},  which  (almost) achieves the minimax rate for {smooth} functions for $d=1$ and may do so for  $d\geq 2$;
thus the FS test is expected to have decent power. We note that the validity of our SILS test does not require $f$ being smooth, and that in simulation studies (Section \ref{sec:simulation}) it achieves comparable power to the FS test. In contrast,  the FS test fails to control the size properly when $f$ is not smooth (e.g., piecewise affine); this is observed  in Section \ref{sec:simulation}, and we also provide a detailed  explanation in Appendix \ref{subsec:minimax_FS} (e.g., if $d=2$, it requires $f$ to be H\"{o}lder continuous with smoothness parameter $s >4$).

\subsection{Combining multiple bandwidths} \label{subsec:unif_bandwidth}
The theory in Subsection \ref{subsec:sz_power}  does not suggest a particular choice for the bandwidth $b_n$. 
Since the size validity holds for a wide range of $b_n$, its selection  depends on the targeted alternatives. If the targets are ``globally" convex, then $b_n$ should be large in order for the bias, $\{\sqrt{n} P^r h^{*}_{v}: v \in \cV\}$, to be large. On the other hand, if the targets are only convex in a small region, then $b_n$ should be able to localize those convex regions. See Subsection \ref{sim:multi_b} for  concrete examples.

One possible remedy is to use multiple bandwidths. Let $\cB_n \subset (0,\infty)$ be a \textit{finite} collection of bandwidths. For each $b \in \cB_n$, we denote the function $h_{v}^{\text{id}}$ in \eqref{def:h_x_concavity} (resp.~$h_{v}^{\text{sg}}$ in \eqref{def:sg_kernel}) by $h_{v,b}^{\text{id}}$  (resp.~$h_{v,b}^{\text{sg}}$) to emphasize the dependence on the bandwidth, and
$\cH^*_{b} = \{h_{v,b}^{*}: v \in \cV\}$ for $* = \text{id or sg}$.  
Further, for each $b \in \cB_n$, let $N_{b}$ and $N_{2,b}$ be two computational parameters, and consider two independent collections of Bernoulli random variables
\begin{align*}
\begin{split}
&\mathcal{S}_b := \left\{Z^{(m,b)}_{\iota}:  \; m \in [M], \;\iota \in  I_{n,r} \right\}  \;  \overset {i.i.d.}{\mathlarger{\sim}} \; \text{Bernoulli}(p_{n,b}),  \\
&\mathcal{S}_b' := \left\{ Z^{(k, m, b)}_{\iota}: \;k \in [n],\; m \in [M], \; \iota \in  \kjack   \right\}  \;  \overset {i.i.d.}{\mathlarger{\sim}} \; \text{Bernoulli}(q_{n,b}),
\end{split}
\end{align*}
where $p_{n,b} := N_b/|I_{n,r}|$, $q_{n,b} := N_{2,b}/|I_{n-1,r-1}|$, and they are independent of $X_1^n$. In other words, the sampling plan is independent for each $b \in \cB_n$.

Then for each $b \in \cB_n$, we denote $U_{n,N}'(h)$ in \eqref{def:incomplete_U_proc} by $U_{n,N,b}'(h)$ with the sampling plan given by $\mathcal{S}_b$. Similarly, we denote $\bG^{(k)}(h)$ and $\bar{\bG}(h)$ in \eqref{def:G_k_h} by 
$\bG^{(k,b)}(h)$ and $\overline{\bG}^{(b)}(h)$ respectively with  the sampling plan given by $\mathcal{S}_b'$.

Now let $\cD_n' := X_1^n \cup \{\mathcal{S}_b, \mathcal{S}_b': b \in \cB_n\}$, and denote Gaussian multipliers by
\[
\{\xi_k: k \in [n]\},\;\; \left\{\xi^{(m,b)}_{\iota}: m \in [M],\;\iota \in I_{n,r},  b \in \cB_n \right\} \;  \;\;\;  \overset {i.i.d.}{\mathlarger{\sim}} \;\;\; N(0,1),
\]
independent of $\cD_n'$. Define for $b \in \cB_n$ and $v \in \cV_m$,
\begin{align*}
\bU_{n,*,b}^{\#}(h_{v,b}^{*})
:= r  \bU_{n,A,b}^{\#}(h_{v,b}^{*}) + (n/N_b)^{1/2} \bU_{n,B,b}^{\#}(h_{v,b}^{*}),
\end{align*}
where  $\bU_{n,A,b}^{\#}(h_{v,b}^{*})
:= n^{-1/2} \sum_{k=1}^{n} \xi_k (
\bG^{(k,b)}(h_{v,b}^{*}) - \overline{\bG}^{(b)}(h_{v,b}^{*}))$, and 
\begin{align*}
\bU_{n,B,b}^{\#}(h_{v,b}^{*})
:= \left(\sum_{\iota \in I_{n,r}} Z_{\iota}^{(m,b)}\right)^{-1/2} \sum_{\iota \in I_{n,r}} \xi^{(m,b)}_{\iota} \sqrt{Z_{\iota}^{(m,b)}} (
h_{v,b}^{*}(X_{\iota}) - U'_{n,N,b}(h_{v,b}^{*})
).
\end{align*}


Finally, for each $\alpha \in (0,1)$, denote by $q^{\#}_{\alpha}$ the $(1-\alpha)^{th}$ quantile of 
$\sup_{b \in \cB_n, v \in \cV}  \bU_{n,*,b}^{\#}(h_{v,b}^{*})$, conditional on $\cD_n'$. Then we propose to reject the null in \eqref{def:H_0} if and only if
\begin{equation}
\label{def:test_proc_with_multi_b}
\sup_{b \in \cB_n, v \in \cV} \sqrt{n} U_{n,N,b}'(h_{v,b}^{\text{*}}) \;\; \geq \;\; q^{\#}_\alpha.
\end{equation}

\begin{remark}
It is possible to allow $\cB_n$ to be uncountable, for example, $\cB_n := [\ell_n, u_n]$, which corresponds to the uniform in bandwidth results \cite{einmahl2005uniform,chen2017jackknife}. However, we choose to present the results for a finite $\cB_n$ for simplicity, since otherwise we need to also stratify $\cB_n$. This approach has a similar spirit to the multi-scale testing of qualitative hypotheses \cite{dumbgen2001multiscale}. 

To establish the size validity and analyze the power of the test \eqref{def:test_proc_with_multi_b}, we need a more general theory than those in Section \ref{sec:prelim} for a function class $\{h_{v,b}: v \in \cV, b \in \cB_n\}$, where $\cV$ is an index set. The key difference is that for each $b \in \cB_n$, the computational parameters $N_b$ and $N_{2,b}$ may be of a different order (see, e.g., \eqref{cond:N_N_2_concavity}).  The rigorous statements for $\{h_{v,b}: v \in \cV, b \in \cB_n\}$, which follow  from similar arguments as those in Section \ref{sec:GAR_Bootstrap}, are not included  for simplicity of the presentation. In Subsection \ref{sim:multi_b}, we conduct a simulation study to investigate the empirical performance of the SILS test with multiple bandwidths.
\end{remark}

\section{Stratified incomplete local simplex tests: computation} \label{sub:comp_complex}
In this section, we discuss the computational complexity and implementation  for the stratified incomplete local simplex tests. 
We focus on $\cH^{\text{id}}$ in our discussion and omit the superscript for simplicity.  Assume that \ref{C:kernel}-\ref{C:id_noise} hold, and that the computational parameters $N,N_2$ are given in \eqref{cond:N_N_2_concavity}. Further, as $\cV$ is compact, we assume $\cV \subset [0,1]^d$ without loss of generality.

For some small $\eta \in (0,1/2)$, let $t := \lfloor 1/(\eta b_n) \rfloor$, and $\tau_i = i \eta b_n$ for $i = 0,1,\ldots,t$ and $\tau_{t+1} = 1$. Now we partition each coordinate into segments of length $\eta b_n$ (except for the rightmost one), i.e., each $\cV_m$ is determined by an ordered list $(j_1,\ldots,j_d)$ such that $0 \leq j_k < t$ for $k \in [d]$ and $\cV_m = \{v \in \cV: \tau_{j_k} \leq v_k < \tau_{j_k+1} \text{ for } k \in [d]\}$.
%
Then the number of partitions $M \leq (1 + \eta^{-1} b_n^{-1})^d$. 

For any $v \in \bR^d$ and $A \subset \bR^d$, we denote the $b_n$-neighbourhood by
\begin{align*}
\cN(v, b_n) := \{v' \in \bR^d: \|v - v'\|_{\infty} \leq b_n/2 \}, \quad
\cN(A, b_n) := \bigcup_{v \in A} \cN(v, b_n). 
\end{align*}
Denote by $\text{ND}(v, b_n) := \{ i \in [n]: V_i \in \cN(v, b_n)\}$ 
and $\text{ND}(A, b_n) := \bigcup_{v \in A} \text{ND}(v, b_n)$
the indices for data points within $b_n$-neighbourhood of $v$ and $A$ respectively.

As an illustration, in Figure \ref{fig:main_idea} (where $b=8, \eta=1/8$), $\cV$ is partitioned into small squares of size $1$. For the dotted region $\cV_m$,
$\cN(\cV_m, b_n)$ is area encompassed by the big dotted square, so $v_4 \in \cN(\cV_m, b_n)$, but $v_5 \not \in \cN(\cV_m, b_n)$. Further, $\text{ND}(\cV_m, b_n)$ are indices for  data points within the dotted square.

\subsection{Stratified sampling}  
For $m \in [M]$, let
$
\cA(\cV_m) := \left\{\iota =(i_1,\ldots,i_r) \in I_{n,r} \; : \; i_j \in \text{ND}(\cV_m, b_n)  \text{ for } j \in [r]\right\}
$
be the collection of $r$-tuples whose members are all within  $b_n$-neighbourhood of $\cV_m$. For example, in Figure \ref{fig:main_idea}, $(v_1,v_2,v_3,v_4) \in \cA(\cV_m)$, but $(v_1,v_2,v_3,v_5) \not \in \cA(\cV_m)$. 
Due to the localization by $L(\cdot)$ 
(cf. \ref{C:kernel}), 
\begin{align*}
h_v(x_\iota) = 0, \text{ for any } v \in \cV_m \text{ and } \iota \in |I_{n,r}| \setminus \cA(\cV_m).
\end{align*}
As a result, the individual values of $\{Z^{(m)}_{\iota}: \iota \in |I_{n,r}| \setminus \cA(\cV_m) \}$ are irrelevant, except for their sum, which is a part of $\widehat{N}^{(m)}$. Thus, we generate a Binomial$(|I_{n,r}\setminus \cA(\cV_m)|, p_n)$ random variable, that accounts for $ \sum_{\iota \in |I_{n,r}| \setminus \cA(\cV_m)} Z^{(m)}_\iota$.

 On the other hand, the number of selected $r$-tuples in $\cA(\cV_m)$ is on average
\begin{align*}
\Exp\left[ \sum_{\iota \in \cA(\cV_m)} Z^{(m)}_\iota \right] \lesssim 
 \binom{n (1+ \eta)^{d} b_n^{d} }{r} \frac{n^{\kappa} b_n^{-dr}}{|I_{n,r}|} \lesssim (1+ \eta)^{dr} n^{\kappa},
\end{align*}
since the $\|\cdot\|_{\infty}$-diameter of $\cV_m$ is $\eta b_n$, and the density of $V$ is bounded (see \ref{C:X}). Thus to compute $\sup_{v \in \cV_m} \sqrt{n} U_{n,N}'(h_v)$, the number of evaluations of $w(\cdot)$ is on average $\lesssim n^{\kappa}$, and the computational complexity can be made independent of the dimension $d$  (as $\eta$ can be chosen to be small).

\begin{remark}
Above calculation of complexity does not include the cost of maximizing over $\cV_m$.  
In practice, we  select a finite number of query points 
as in  Subsection \ref{subsec:implementation}.  
The discussion for the bootstrap part is similar, and we analyze below the complexity of its actual implementation.
\end{remark}

\noindent \underline{\bf Why stratification?} Without stratifying $\cV$, each $v \in \cV$ share the same sampling plan $\{Z_\iota : \iota \in I_{n,r}\}$. 
However, we cannot afford to generate  all $\{Z_\iota : \iota \in I_{n,r}\}$, as on average there are $N = n^{\kappa} b_n^{dr}$ non-zero terms. 
We may attempt to use the above short-cut. For $v_1, v_2 \in \cV$, to compute
$U_{n,N}'(h_{v_i})$ (for $i=1,2$),
 we only generate 
 $\{Z_\iota : \iota \in  \cA(\{v_i\})\}$, and the individual values of
 $\{Z_\iota : \iota \in I_{n,r} \setminus \cA(\{v_i\})\}$ are not explicitly generated.
 
However, the issue is to ensure consistency. (i)~In computing  $U_{n,N}'(h_{v_1})$, although the individual values of  $\{Z_\iota : \iota \in I_{n,r} \setminus \cA(\{v_1\})\}$ are irrelevant, we still need to  generate a Binomial random variable to account for their sum. However, $\left( I_{n,r} \setminus \cA(\{v_1\}) \right) \cap \cA(\{v_2\})$ in many cases is non-empty, and thus $ \sum_{\iota \in |I_{n,r}| \setminus \cA(\{v_1\})} Z^{(m)}_\iota$ and $\{Z_{\iota}: \iota \in \cA(\{v_2\})\}$ are not independent. (ii)~In many cases, $\cA(\{v_1\}) \cap \cA(\{v_2\})$ is non-empty, so we cannot independently generate  $\{Z_\iota : \iota \in  \cA(\{v_1\})\}$ and $\{Z_\iota : \iota \in  \cA(\{v_2\})\}$. Note also that the calculation is needed for multiple $v \in \cV$  instead of only $v_1,v_2$.

\begin{remark}
In Appendix \ref{subsec:without_partitioning}, we present an algorithm without stratification that addresses the above consistency issue. Its computational complexity is $ \lesssim 2^{dr} n^{\kappa} b_n^{-d}$ evaluations of $w(\cdot)$. If $d$ is fixed, it only loses a $b_n^{-d}$ factor in theory, but $2^{dr}$ can be very large in practice, and thus it is not computationally feasible (e.g.,~$2^{dr} = 32768$ if $d=3$).
\end{remark}

\subsection{Implementation of SILS}\label{subsec:implementation}
In practice, instead of taking the supremum over $\cV$, we choose a (finite) collection of query points, $\sV_n$, {one from each partition} $\{\cV_m: m \in [M]\}$, 
and approximate the supremum over $\cV$ by that over $\sV_n$. 
As a result,  each $v \in \sV_n$ has its individual sampling plan ($\{Z_{\iota}^{(m)}: \iota \in I_{n,r}\}$ if $v \in \cV_m$), which can be generated independently for different query points. Further, the test still takes the form of \eqref{def:test_proc} with a finite function class $\cH = \left\{h_v: v \in \sV_n\right\}$ .

\begin{remark}
It is without loss of generality to pick one query point from each region, since we could always decrease $\eta$, i.e., making each region smaller. Further, unlike \cite{abrevaya2005nonparametric}, we do not require query points to be well separately, that is, for small $\eta$, there are pairs fo queries points $v,v' \in \sV_n$, such that $\|v-v'\|_{\infty} \ll b_n$. Finally, since only one element is picked, if $v \in \cV_m$, instead of considering $\text{ND}(\cV_m, b_n)$,  we can focus on $\text{ND}(\{v\}, b_n)$. 
\end{remark}

\begin{remark}
In establishing the bootstrap validity for stratified, incomplete $U$-processes, we first consider the corresponding results for  high-dimensional $U$-statistics, and then approximate the supremum of a $U$-process by that of its  discretized version. Thus the above procedure, which can be viewed as a practical implementation of approximating the supremum of a process,  can also be directly justified by  Theorems in Appendix  \ref{sec:GAR_incomp_U_stat}.
\end{remark}

\noindent  \textbf{\underline{Computing the test statistic.}}  
In Algorithm \ref{alg:test_stat_gamma_B}, we show the pseudo-code to compute, for each $v \in \sV_n$, the statistic $U_{n,N}'(h_v)$, and at the same time the conditional (given $\cD_n'$) variance $\widehat{\gamma}_B(h_v)$ of $\bU_{n,B}^{\#}(h_v)$ in \eqref{def:bootstrap_A_and_B}; note that we write $\widehat{\gamma}_B(h_v)$ for $\widehat{\gamma}_B(h_v, h_v)$, which is defined following \eqref{def:bootstrap_A_and_B}. It is well known that
sampling $T$ items without replacement from $S$ elements ($S \gg T$) can be done in $O(T\log(T))$ time \cite{gupta1984efficient}. Then based on the discussions in the previous subsection,  the computational complexity for Algorithm \ref{alg:test_stat_gamma_B} is $O( M n^{\kappa} \log(n))$.\\

\noindent  \textbf{\underline{Bootstrap.}}  
For a fixed $k \in [n]$,  notice that $\{\bG^{(k)}(h_v): v \in \sV_n\} $ in \eqref{def:G_k_h} takes the same form of stratified, incomplete $U$-processes as the test statistics, and thus we can apply Algorithm \ref{alg:test_stat_gamma_B}, with appropriate inputs, to compute it. Since we need to compute $\bG^{(k)}$ for each $k \in [n]$,  the  complexity is
\begin{align*}
n \times M \times \binom{(n-1) b_n^{d} }{r-1} \frac{n^{\kappa'} b_n^{-dr}}{|I_{n-1,r-1}|}\log(n) \lesssim M n^{1+\kappa'} b_n^{-d}\log(n).
\end{align*}
Further, as we pick one element from each $\cV_m$, 
given $\cD_n'$, $\left\{\bU_{n,B}^{\#}(h_v): v \in \sV_n\right\}$ are {conditionally independent},
with  variances
$\{\widehat{\gamma}_B(h_v): v \in \sV_n \}$ already 
computed in Algorithm \ref{alg:test_stat_gamma_B}, and we no longer need to generate Gaussian multipliers $\{\xi^{(m)}_{\iota}\}$ for each summand indexed by $\iota$ in \eqref{def:bootstrap_A_and_B}.

Finally, for independent standard Gaussian multipliers $\{\xi_k, \xi^{(m)}: k \in [n], m \in [M]\}$, we compute for each $v\in \sV_n$,
\[
\frac{1}{\sqrt{n}} \sum_{k =1}^{n} \xi_{k}\left(\bG^{(k)}(h_v) - \overline{\bG}(h_v) \right) + \xi^{(\sigma(h_v))} \sqrt{\widehat{\gamma}_B(h_v)}.
\]
Since $\{\bG^{(k)}: k \in [n]\}$ and $\{\widehat{\gamma}_B(h_v): v \in \sV_n \}$ have already been computed,  the  complexity  is $O(BMn)$, where $B$ is the number of bootstrap iterations. 
Hence the overall computational cost is $O( M n^{\kappa} \log(n) + M n^{1+\kappa'} b_n^{-d}\log(n) + BMn)$.

\begin{remark}
The computational bottleneck is in computing $\{\bG^{(k)}: k \in [n]\}$, which, however, is outside the bootstrap iterations. Thus we can afford large $B$ in the bootstrap calibration. The above algorithms can be implemented in a parallel manner using clusters; in particular, $\bG^{(k)}$ can be computed separately for each $k \in [n]$. As a result, the efficiency scales linearly in the number of computing cores.
\end{remark} 

\begin{algorithm}[htbp!]
  \KwIn{Observations $\{X_i = (V_i, Y_i) \in \bR^{d+1}: i \in [n] \}$, budget $N$, kernel $L(\cdot)$, bandwidth $b_n$,  query points $\sV_n$ (size $M$).}
  
  \KwOut{ ${U}_{n,N}',\; {\widehat{\gamma}}_{B}$: two vectors of length $M$}
  
   \KwSty{Initialization:} $p_n = N/\binom{n}{r}$, ${U}_{n,N}', \widehat{\gamma}_B$ both set   zero \;

\For{$m \leftarrow 1$ \KwTo $M$ }{
 $v = \sV_n[m]$\;
 Generate $T_{1}  \sim  \text{Binom}(\binom{|\text{ND}(v, b_n)|}{r}, \;p_n), \quad  T_{2}  \sim  \text{Binom}(\binom{n}{r}-\binom{|\text{ND}(v, b_n)|}{r}, \;p_n)$\;
$\widehat{N} \leftarrow  T_{1} + T_{2}$\;
Sample without replacement $T_{1}$ terms, $\{\iota_\ell: 1 \leq \ell \leq T_{1}\}$,  from $\cA(\{v\})$\;
	\For{$\ell \leftarrow 1$ \KwTo $T_1$}{	
			 ${U}_{n,N}'[m] \leftarrow {U}_{n,N}'[m] +h_v(X_{\iota_{\ell}})$\;
			 ${\widehat{\gamma}}_{B}[m] \leftarrow{\widehat{\gamma}}_{B}[m] + (h_v(X_{\iota_{\ell}}))^2$\;
	}
	${U}_{n,N}'[m] \leftarrow {U}_{n,N}'[m]/\widehat{N}, \;\;
{\widehat{\gamma}}_{B}[m] \leftarrow {\widehat{\gamma}}_{B}[m]/\widehat{N} - ({U}_{n,N}'[m])^2$ 
}

   \caption{ compute $U_{n,N}'$ and $\widehat{\gamma}_B$ over $\sV_n$ for the concavity test.}
   \label{alg:test_stat_gamma_B}
\end{algorithm}

\section{Simulation results} \label{sec:simulation}


In the simulation studies,  we consider  setups  where  the regression function  $f$ in \eqref{eqn:nonparametric_regression} is defined on $(0,1)^d$,
and   the covariates $V = (V_1,\ldots, V_d)$ have a uniform distribution on $(0,1)^d$, for $d = 2,3, 4$. 
In this section, the error term $\varepsilon$ in \eqref{eqn:nonparametric_regression} has a {Gaussian distribution} with zero mean and variance $\sigma^2$.

\begin{remark}
The results for $d = 3 \text{ and } 4$ are qualitatively similar, and  presented mostly in Appendix \ref{app:simulations}, where  we also study {asymmetric} or {heavy tailed}  distributions for the noise $\varepsilon$.
\end{remark}

We compare our proposed procedure with the method in \cite{fang2020projection}, denoted by ``FS".

\noindent \underline{\it Proposed procedure.} 
We use the uniform localization kernel $L(\cdot) = \mathbbm{1}\{\cdot \in (-1/2,1/2)^d \}$. The query points are $\sV_n :=\{0.3,0.4,0.5,0.6,0.7\}^2$ for $d=2$, and $\sV_n :=\{0.3,0.5,0.7\}^d$ for $d = 3,4$. For parameters related to the computational budget, we set $N = 10\times 25 \times n \times b_n^{-d\times r}$ for $d=2,3,4$, $N_2 = 10^4\times b_n^{-d\times r}$ for $d=2,3$ and $N_2 = 2\times 10^4 \times b_n^{-d\times r}$ for $d=4$, and the Bootstrap iterations $B=1500$. The $N$ is selected so that $\alpha_n := n/N$ is very small, and further increasing it will not improve the power of the test. The estimation of $\{\bG^{(k)}(h_v^{*}): k \in [n], v \in \sV_n\}$ is the computational bottleneck, and empirically we find that further increasing the selected value for $N_2$ does not improve the accuracy in terms of the size of the proposed procedure. We   consider two types of kernels, $\cH^{\text{id}}$ and $\cH^{\text{sg}}$, and
use below ``ID" for the former and  ``SG" for later.  For each parameter configuration below, we independently generate (at least) 1,000 datasets, apply our procedure, and estimate the rejection probability.

\noindent \underline{\it FS method \cite{fang2020projection}.} We use the implementation provided by the authors\footnote{code for \cite{fang2020projection}: \url{https://www.dropbox.com/s/jmjshxznu31tnn2/ShapeCode.zip?dl=0}.  
}, where either quadratic or cube splines with $j$ knots {in each coordinate} are used in constructing an initial estimator for the regression function; we denote the former by FS-Q$j$ and later by FS-C$j$. We set the tuning parameter $\gamma_n = 0.01/ \log(n)$ and the Bootstrap iteration $B=200$ as recommended by \cite{fang2020projection}.  Below, the rejection probabilities are estimated based on 1,500 independently generated datasets.

\subsection{Running times}

The computational savings compared to using the complete $U$-process, 
$p_n := N/\binom{n}{r}$ and $q_n = N_2/\binom{n-1}{r-1}$,
are listed in Table \ref{tab:comp_eff} for several typical configurations. It is clear that for a moderate size dataset (say $n \sim\, 1000$), using the complete $U$-process has a very high, if not prohibitive, computational cost (see Table \ref{tab:running_time} for the running time using the stratified, incomplete $U$-process). For example, for $d =3, n = 1000,  b_n = 0.6$, it takes on average $5.26$ minutes to run our procedure with 40 cores, which implies that with the complete version it would take at least $7.2$ days ($=5.26 \text{ mins}/q_n)$.

In contrast, the FS method \cite{fang2020projection} has a much shorter running time. For example, with $d=2, n=1000$, it takes less than $20$ seconds with $4$ cores (see Table E4 in \cite{fang2020projection}). For $d \geq 3$, it could be challenging to apply the FS method due to the accuracy of estimating the regression function, the projection onto a function space, and the numerical integration needed to compute the distance etc. 



\begin{table}[htbp!]
\begin{tabular}{ccccc}
\hline
$d=2, b_n = 0.5$ &\multicolumn{3}{c}{$d=3, b_n = 0.6, |\sV_n| = 27 $} & $d=4, b_n = 0.7$ \\ \cline{2-4}
$n=1000, |\sV_n| = 25$ & $n = 500$   & $n = 1000$  & $n = 1500$  & $n = 2000, |\sV_n| = 81$       \\ \hline
5.06 mins &1.31 mins   & 5.26 mins   & 9.12 mins   & 33.4 mins        \\ \hline
\end{tabular}
\caption{ Running time of the proposed procedure in minutes  using 40 computer cores, where $N$ and $N_2$ are described in the introduction of Section \ref{sec:simulation}.}
\label{tab:running_time}
\begin{tabular}{cccccccc}
\hline
                & \multicolumn{3}{c}{$p_n$}  &  & \multicolumn{3}{c}{$q_n$}  \\ 
                                   \cline{2-4} \cline{6-8}
                & $n=500$  & $n=1000$ & $n = 1500$ &  & $n=500$  & $n=1000$ & $n = 1500$ \\ 
                                   \hline
$d=2, b_n = 0.5$ & 1.2E-2 & 1.5E-3 & 4.5E-4   &  & 1.2E-1 & 1.5E-2 & 4.6E-3   \\ 
$d=3, b_n = 0.6$ & 1.0e-3 & 6.4E-5 & 1.3E-5   &  & 8.3E-3 & 5.1E-4 & 1.0E-4     \\ \hline
\end{tabular}
\caption{Computational efficiency for typical configurations, where $N$ and $N_2$ are described in the introduction of Section \ref{sec:simulation}. For $d =4, b_n = 0.7, n =2000$, we have $p_n =$5.9E-8, $q_n =$ 3.9E-7. Here, sE-t = s$\times 10^{-t}$.}
\label{tab:comp_eff}
\end{table}

\subsection{Size validity}
We start with our proposed procedure, and consider {concave} functions given by
\begin{align}
\label{alt:poly}
f(v) = v_1^{\kappa_0} + \ldots + v_d^{\kappa_0}, \;\; \text{ for } v := (v_1,\ldots,v_d) \in (0,1)^d,
\end{align}
for  $0<\kappa_0 \leq 1$. Here, we consider the rejection probabilities for $\kappa_0=1$; that is, $f$ is affine and thus the  (asymptotically) least favourable configuration in the null. In Appendix \ref{app:d2_strictly_concave}, we present results for strictly concave function for $0 < \kappa_0 < 1$.




For each query point, the average number of data points within its $b_n$ neighbourhood is $n\times b_n^{-d}$. Since a decent size of local points is necessary for the validity of Gaussian approximation, we select $b_n$ so that locally there are at least $150$ data points. 
As we shall see in Subsection \ref{sim:d2_power},
smaller $b_n$ has a better localization power, while larger $b_n$ is suitable if the targeted alternatives are globally convex.


In Table \ref{tab:sv_d2_n5_n10}, we list the size for different bandwidth $b_n$ and error variance $\sigma^2$ at levels $5\%$ and $10\%$ for $f$ in \eqref{alt:poly} with $\kappa_0 = 1$. From the Table 
\ref{tab:sv_d2_n5_n10},  it is clear that the proposed procedure is consistently on the conservative side. We note that the conservativeness is not due to the stratified sampling.
For $d=1$, we were able to implement the complete version, and observed a similar phenomenon. Further,  \cite{chen2017jackknife} uses  complete $U$-processes  to test regression monotonicity, which are also conservative
(see Table 1 therein). 

\noindent \underline{\it FS method \cite{fang2020projection}.} In Table E2 of \cite{fang2020projection}, we can observe the slight inflation of the empirical size of the FS method when the function is linear. Here,
we consider the following \textit{concave}, piecewise affine regression function:
\begin{align}\label{equ:FS_d2_example}
f(v_1,v_2) = -|v_1 -0.8| - |v_2 - 0.8|, \;\; \text{ for } v_1,v_2 \in (0,1).
\end{align}


The rejection probabilities  of the FS method \cite{fang2020projection} at the nominal level $5\%$ are listed in Table \ref{tab:FS_d2_SIZE}. Recall that FS-Q$j$ (resp. FS-C$j$) is for using quadratic (resp. cubic) splines with $j$ knots in each coordinate as the initial estimator for the regression function. Except for the global test FS-Q0, which places no interior knots, these probabilities far exceed the nominal level. We provide explanations for the significant size inflation of the FS method \cite{fang2020projection} in Appendix \ref{subsec:minimax_FS}.

%

\begin{table}[htbp!]
\begin{tabular}{cccccccc}
\hline
                   & \multicolumn{3}{c}{$n=500$}       &  & \multicolumn{3}{c}{$n=1000$}      \\ 
                   \cline{2-4} \cline{6-8}
$d =2$, Level = 5\%  & $b_n=0.6$ & $b_n=0.55$ & $b_n=0.5$ &  & $b_n=0.5$ & $b_n=0.45$ & $b_n=0.4$ \\ 
                   \hline
ID, $\sigma= 0.1$  & 3.1       & 2.9        & 2.8       &  & 4.1       & 2.6        & 3.4       \\                
SG, $\sigma=0.1$   & 3.1       & 4.1        & 2.8       &  & 3.0       & 3.8        & 2.7       \\ 
ID, $\sigma= 0.2$  & 3.2       & 3.7        & 3.0       &  & 2.9       & 3.1        & 2.4       \\ 
SG, $\sigma=0.2$   & 3.6       & 3.3        & 2.9       &  & 3.6       & 3.1        & 2.5       \\ \hline
                   &           &            &           &  &           &            &           \\ 
$d =2$, Level = 10\% & $b_n=0.6$ & $b_n=0.55$ & $b_n=0.5$ &  & $b_n=0.5$ & $b_n=0.45$ & $b_n=0.4$ \\ \hline
ID, $\sigma= 0.1$  & 8.3       & 6.8        & 6.7       &  & 8.1       & 7.5        & 7.7       \\ 
SG, $\sigma=0.1$   & 7.4       & 8.0        & 6.6       &  & 8.1       & 7.7        & 6.1       \\ 
                  
ID, $\sigma= 0.2$  & 7.8       & 8.0        & 6.2       &  & 7.6       & 7.8        & 6.0       \\ 
SG, $\sigma=0.2$   & 8.7       & 7.1        & 7.0       &  & 8.6       & 7.0        & 6.7       \\ \hline
\end{tabular}
%
\caption{Size validity of the proposed procedure for $d=2$. The probabilities of rejection under {affine} regression functions, are in the unit of percentage.}
\label{tab:sv_d2_n5_n10}
\medskip
\begin{tabular}{ccccccccccl}
\hline
Knots $j$    & 0    & 1    & 2    & 3    & 4    & 5 & 6    & 7  &8 & 9  \\ \hline
FS-Q$j$ & 0    & 99.5 & 46.5 & 9.2  & 23.6 & 33.9  & 16.0 & 16.0 & 23.1 & 20.3\\ 
FS-C$j$ & 97.1 & 50.7 & 15.3 & 32.8 & 23.9 &  13.6 & 17.7 & 21.0 & 19.9 & 22.2\\ \hline
\end{tabular}
\caption{ The rejection probabilities (in percentage) for FS-Q$j$ and FS-C$j$ \cite{fang2020projection} at  level $5\%$ for the concave (i.e. $H_0$ holds) function in \eqref{equ:FS_d2_example}, where $n = 1000$ and $\sigma = 0.1$.}
\label{tab:FS_d2_SIZE}
\end{table}

\subsection{Power comparison} \label{sim:d2_power}
We study two types of alternatives for the regression function.


\noindent \underline{\it Polynomial functions.} In the first, we consider $f$ in \eqref{alt:poly} for  $\kappa_0 \in \{1.2,1.5\}$.


\noindent \underline{\it Locally convex functions.} 
For the second, we consider regression functions that are mostly concave over $(0,1)^d$, but convex in a small region. Specifically, let $\varphi(v) := \exp\left(-\|v\|^2/2\right)$ for $v \in \bR^d$, which is concave on the region $\{v \in \bR^d: \|v\|_{\infty} < 1\}$. Then for $c_1, c_2, \omega_1, \omega_2 > 0$ and $\mu_1, \mu_2 \in \bR^d$,
we consider
\begin{equation}
\label{alt:local_conv}
f(v) = c_1 \varphi\left( (v-\mu_1)/\omega_1\right) - c_2 \varphi \left((v-\mu_2)/\omega_2 \right), \text{ for } v \in (0,1)^d.
\end{equation}
We let $c_1 = 1$, $\omega_1 = 1.5$, and $\mu_1 = (0.75,\ldots, 0.75)$ so that without the second term, $f$ would be concave in the entire region $(0,1)^d$. We let $\mu_2 = (0.25,\ldots,0.25)$, and set $c_2$ and $\omega_2$ to be small so that $f$ is mostly concave and locally convex in a small neighbourhood of $\mu_2$. 
(In  Appendix \ref{app:figure}, we plot the regression function $f$ together with one realization of dataset.)


In Table \ref{tab:d2_Power_Comparison} (a) and (b), for the two types of alternatives, we list the power of $\cH^{id}$ with different bandwidth parameters $b_n$, and the FS method \cite{fang2020projection} using either quadratic (Q) or cubic (C) splines with $j=0,1,2,5$ knots in each coordinate\footnote{$j=0,1$ is used in \cite{fang2020projection}}.  
For our proposed method, if $f$ is a polynomial function \eqref{alt:poly}, the power   increases as $b_n$ increases, as $f$ is globally convex. However, for the locally convex function $f$ \eqref{alt:local_conv}, the power initially increases as $b_n$ increases, but later drops significantly, as $f$ is only locally convex, but  ``globally concave". Thus the choice of bandwidth depends on the targeted alternatives. Similar statements can be made about the FS method \cite{fang2020projection}. Adding knots decreases its power for \eqref{alt:poly}, while
a ``global" test such as FS-Q0  has little power against \eqref{alt:local_conv}.

In summary, the proposed procedure has a comparable power to the FS method \cite{fang2020projection}, which however fails to control the size in general. Further, we show next that the issue with selecting $b_n$  can be partly solved by combining multiple bandwidths.

\begin{table}[htbp!]
\subfloat[Polynomial $f$ \eqref{alt:poly} with $\kappa_0=1.5$]{
\begin{tabular}{ccccccccccccc}
\hline
           & \multicolumn{3}{c}{$\cH^{\text{id}}$ with single $b_n$} &  & \multicolumn{8}{c}{FS method \cite{fang2020projection}} \\ \cline{2-4} \cline{6-13} 
Level 5\%  & $b_n = 0.6$         & $b_n=0.8$         & $b_n=1$        &  & Q0         & Q1         & Q2        &Q5 & C0         & C1        & C2 & C5       \\ \hline
Rej. Prob. & 25.6                & 69.1              & 96.6           &  & 93.8       & 81.6       & 57.7 &37.7       & 85.6       & 60.5      & 50.7 & 36.3      \\ \hline
\end{tabular}
}

\subfloat[Locally convex  $f$ \eqref{alt:local_conv} with $\omega_2 = 0.15, c_2 = 0.3$]{
\begin{tabular}{ccccccccccccc}
\hline
           & \multicolumn{3}{c}{$\cH^{\text{id}}$ with single $b_n$} &  & \multicolumn{8}{c}{FS method \cite{fang2020projection}} \\ \cline{2-4} \cline{6-13} 
Level 10\% & $b_n = 0.5$        & $b_n=0.6$        & $b_n=0.7$        &  & Q0        & Q1         & Q2        &Q5 & C0         & C1         & C2     & C5   \\ \hline
Rej. Prob. & 20.3               & 40.3             & 15.8             &  & 7.1       & 49.7       & 28.5 & 30.4       & 44.3       & 26.5       & 27.7    & 29.4  \\ \hline
\end{tabular}
}


\subfloat[Multiple bandwidths $\{\cH^{\text{id}}_b: b \in \{0.6,0.8,1\}\}$]{
\begin{tabular}{ccccc}
\hline
 & & Poly $f$ \eqref{alt:poly} with $\kappa_0=1.5$ at $5\%$ &  &                            Locally convex  $f$ \eqref{alt:local_conv}  with $\omega_2 = 0.15, c_2 = 0.3$ at $10\%$ \\ \hline 
Rej. Prob. & & 71.7              &  & 39.8 \\ \hline
\end{tabular}
}
\caption{The rejection probabilities (in percentage) of  the proposed method $\cH^{\text{id}}$, the FS method \cite{fang2020projection}, and the proposed method with multiple bandwidth
$\{\cH^{\text{id}}_b: b \in \{0.6,0.8,1\}\}$ 
 for $d = 2$, $n =1000$,  $\sigma = 0.5$.}
\label{tab:d2_Power_Comparison}
\end{table}

\subsection{Combining multiple bandwidths}\label{sim:multi_b}
We consider the procedure \eqref{def:test_proc_with_multi_b} in Subsection \ref{subsec:unif_bandwidth}  that combines multiple bandwidths, $\{\cH^{\text{id}}_b: b \in \cB_n\}$  with $\cB_n = \{0.6,0.8,1\}$. 
For $f$ in \eqref{alt:poly} with $\kappa_0 = 1$ (i.e. affine functions), the probability of rejection is $8.4\%$ when the nominal level is $10\%$; for strictly concave functions with $\kappa_0 < 1$, the probabilities of rejection
are listed in Appendix \ref{app:d2_strictly_concave}.
In Table \ref{tab:d2_Power_Comparison} (c), we present its  power against the two alternatives. 

With a range of bandwidths,  $\{\cH^{\text{id}}_b: b \in \cB_n\}$ achieves a reasonable power, and is adaptive to the properties of the regression function $f$, with its computational cost {linear} in $|\cB_n|$. As expected, it is not as powerful as the best performance achievable by $\cH^{\text{id}}_{b_n}$ with a single bandwidth $b_n$, which,  however, is unknown in practice.

Thus, we would recommend  the procedure \eqref{def:test_proc_with_multi_b} with multiple bandwidths $\cB_n$. In choosing $\cB_n$, one approach  is to first decide reasonable lower and upper bounds, $b_{\text{min}}$ and $b_{\text{max}}$, for the bandwidth, and then based on  available computational resource, select a few bandwidths in $[b_{\text{min}},b_{\text{max}}]$ (say equally spaced) to form $\cB_n$. As a rule of thumb, one may  choose $b_{\text{min}}$ so that there are enough data points in the $b_n$ neighbourhood
of each query point (say $\geq 120$). On the other, $b_{\text{max}}$ could be decided based on the diameter of the region of interest, $\cV$.

\section{Gaussian approximation and bootstrap  for stratified, incomplete U-processes}
\label{sec:GAR_Bootstrap}
In this section, we   consider a general function class $\cH$, and establish Gaussian approximation and bootstrap results for its associated stratified, incomplete $U$-processes  in Section \ref{sec:prelim}, under the more general moment assumptions \ref{cond:MT} instead of \ref{cond:MT_inf}. In particular, the condition \ref{cond:MT} does not require the envelope function $H$ in \ref{cond:VC} to be bounded.\\


%




%

\noindent \mylabel{cond:MT}{(MT)}. There exist  absolute constants $\ubar{\sigma} > 0$, $c_0 \in (0,1)$, $q \in [4,\infty]$, and $B_n\geq  D_n \geq 1$ such that
\begin{align}
\label{MT0}
&\Var\left(P^{r-1}h(X_1) \right) \geq \ubar{\sigma}^2,\; \text{ for } h\in \cH, \tag{MT-0}\\
\label{MT1} 
&\sup_{h \in \cH} \Exp\left|P^{r-1}h(X_1) -P^rh \right|^{2+k} \;\leq\; D_n^{k} \text{ for }\;\;  k = 1,2, \quad
\|P^{r-1}H\|_{P,q} \leq D_n,
 \tag{MT-1}\\ 
\label{MT2}
\begin{split}
&\|P^{r-\ell} {H}^s\|_{P^{\ell}, 2} \leq B_{n}^{2s-2} D_n^{\ell + 1-s}, \text{ for } \ell \geq 2, s = 1,2,3,4, \\
&\|P^{r-\ell} {H}^s\|_{P^{\ell}, q} \leq B_{n}^{2s-2} D_n^{\ell(2-2/q)+2/q-s}, \text{ for } \ell =1,2,\; s = 2,3,4, \;\;
\|P^{r-2} H\|_{P^{2}, q}  \leq D_n^{3-2/q},
\\
&\|P^{r-\ell} {|h|}^s\|_{P^{\ell}, 2} \leq 
B_{n}^{2s-2} D_n^{\ell-s}, \text{ for } \ell =0,1,2,\; s\in [4]\text{ with } \ell + s >2, \; h \in \cH,
\end{split} \tag{MT-2}\\
\label{MT3}
&\|{H}\|_{P^{r},q} \leq B_n^{2-2/q} D_n^{2/q-1},  \quad \|H\|_{P^r,2} \leq B_n, \tag{MT-3}\\
\label{MT4}
&c_0 B_n^2 D_n^{-2} \leq \Var(h(X_1^r)) \leq \min\{D_n^{2(r-1)},\;\; B_n^2  D_n^{-2}\}, \text{ for } h \in \cH, \tag{MT-4}\\
\label{MT5}
& \sup_{h \in \cH } \|P^{r-2} h\|_{P^2,4} \leq D_n^2,\quad
\|(P^{r-2}H)^{\bigodot 2}\|_{P^2,q/2} \leq
D_n^{4-4/q}, \tag{MT-5}
\end{align}
where $1/q = 0$ if $q = \infty$, and recall that for a measurable function $f:S^2 \to \bR$, define $f^{\bigodot 2}$ to be a function on $S^2$ such that
$f^{\bigodot 2}(x_1,x_2) := \int f(x_1,x)f(x_2,x) dP(x)$.


\subsection{Gaussian approximation}\label{sec:GaussApprox_Proce}
We first approximate the supremum of the stratified, incomplete $U$-process \eqref{def:incomplete_U_proc} by that of an appropriate Gaussian process. Specifically, denote by $P^{r-1} h$ the function on $S$ such that $P^{r-1} f(x_1) := \Exp[h(x_1, X_{2},\ldots,X_r)]$, and $(\ell^{\infty}(\cH),\|\cdot\|_{\infty})$ the space of bounded functions indexed by $\cH$ equipped with the supremum norm. Assume there exists a tight Gaussian random variable 
$W_P$ in $\ell^{\infty}(\cH)$ with zero mean and covariance function $\gamma_{*}(h,h') := \Cov\left(W_P(h), W_P(h')\right) = r^2 \gamma_A(h,h') + \alpha_n \gamma_B(h,h')$ for $h,h' \in \cH$ where 
$\alpha_n:= n/N$, $\gamma_A(h,h'):= \Cov\left(P^{r-1}h(X_1),\; P^{r-1}h'(X_1)\right)$, and 
$\gamma_B(h,h') := \Cov\left(h(X_1^r),\; h'(X_1^r)\right) \mathbbm{1}\{\sigma(h) = \sigma(h')\}$. 
Note that $\gamma_B(h,h') = 0$ if $\sigma(h) \neq \sigma(h')$, which is due to the stratification. The existence of $W_P$ is implied by \ref{cond:VC} and \ref{cond:MT} (see \cite{chernozhukov2014gaussian}[Lemma 2.1]). 
Further, denote $\widetilde{\bM}_n := \sup_{h \in \cH} W_P(h)$.
We  bound the Kolmogorov distance between the two suprema.

\begin{theorem}\label{thm:GAR_sup_incomplete_Uproc}
Assume the conditions \ref{cond:PM}, \ref{cond:VC}, \ref{cond:MB}, and \eqref{MT0}-\eqref{MT4}. Then there exists a constant $C$, depending only on $r, q, \ubar{\sigma},c_0, C_0$, such that
\begin{align*}
&\sup_{t \in \bR}\left|
\Pro(\bMn \leq t) - \Pro(\widetilde{\bM}_n \leq t)\right| \; \leq \;
C\eta_n^{(1)} + C \eta_n^{(2)}, \;\; \text{ with }\\
& \eta_n^{(1)}:=
\left( \frac{D_n^2 K_n^7}{n}\right)^{1/8} + 
\left(\frac{D_n^2 K_n^4}{n^{1-2/q}}  \right)^{1/4} +
\left(\frac{D_n^{3-2/q} K_n^{5/2}}{n^{1-1/q}}\right)^{1/2},\\
& \eta_n^{(2)}:=
\left( \frac{B_n^2 K_n^7}N \right)^{1/8} +  \left( \frac{n^{4r/q} K_n^5 B_n^{2-8/q} D_n^{8/q}}{N}
\right)^{1/4} + \left( \frac{M^{2/q}B_n^{2-4/q}D_n^{4/q} K_n^{5}}{N^{1-2/q}}\right)^{1/4}, 
\end{align*}
where  $1/q = 0$ if $q = \infty$.
\end{theorem}
\begin{proof}
The strategy is to first establish Gaussian approximation results for a finite, yet ``dense", subset $\cH'$ of $\cH$ (Appendix \ref{sec:GAR_incomp_U_stat}), and then approximate the supremum over $\cH$ by that over $\cH'$, which requires the local maximal inequalities developed in Appendix \ref{app:local_max_ineq_incom}.
See details in Appendix \ref{proof:thm:GAR_sup_incomplete_Uproc}.
\end{proof}



\subsection{Bootstrap validity}
The next Theorem shows that conditional on $\cD'_n$, the maximum of the bootstrap process, $\bMn^{\#}$ in \eqref{def:bootstrap_combined}, is well approximated by the maximum of $W_P$, $\widetilde{\bM}_n$, in distribution.

\begin{theorem}\label{thm:uproc_bootstrap}
Assume the conditions \ref{cond:PM}, \ref{cond:VC}, \ref{cond:MB}, and \eqref{MT0}-\eqref{MT5}. Let
\begin{align*}
&\varrho_n :=  \left(\frac{M^{2/q}B_n^{2-4/q} D_n^{4/q} K_n^{5}}{(N \wedge N_2)^{1-2/q}} \right)^{1/4} + \left( \frac{B_n^2 K_n^{7}}{N \wedge N_2} \right)^{1/8}  +\\
& \left( \frac{D_n^2 K_n^{7}}{n} \right)^{1/8} +
\left( \frac{D_n^2 K_n^{4}}{n^{1-2/q}} \right)^{1/4}+ \left( \frac{D_n^{8-8/q} K_n^{15}}{n^{3-4/q}} \right)^{1/14} +
\left( \frac{D_n^{3-2/q} K_n^{4}}{n^{1-1/q}} \right)^{2/7}.
\end{align*}
There exists a constant $C$, depending only on  $r,q, \ubar{\sigma}, c_0,C_0$, such that with probability at least $1 - C \varrho_n$,
$$
\sup_{t \in \bR}\left|
\Pro_{\vert \cD_n'}(\bMn^{\#} \leq t) - \Pro(\widetilde{\bM}_n \leq t)
\right| \leq C \varrho_n.
$$
\end{theorem}
\begin{proof}
The proof can be found in Appendix \ref{proof:thm:uproc_bootstrap}.
The key steps are to show that 
given $\cD'_n$, the conditional covariance functions $\widehat{\gamma}_A(\cdot,\cdot)$ and $\widehat{\gamma}_B(\cdot,\cdot)$, for $\bU^{\#}_{n,A}$ and $\bU^{\#}_{n,B}$ in \eqref{def:bootstrap_A_and_B},
 are good estimators for ${\gamma}_A(\cdot,\cdot)$ and ${\gamma}_B(\cdot,\cdot)$.
\end{proof}


\subsection{Related work} \label{subsec:related_U_proc}  
$U$-processes offer a general framework for many statistical applications such as testing for qualitative features (e.g., monotonicity, curvature) of regression functions \cite{ghosal2000,bowman1998testing, abrevaya2005nonparametric}, testing for stochastic monotonicity \cite{lee2009testing}, nonparametric density
estimation \cite{nolan1987u,frees1994,gine2007local}, and establishing limiting distributions of $M$-estimators \cite{arcones1993limit,sherman1994maximal}. When indexing function classes are fixed, it is known that the Uniform Central Limit Theorems (UCLTs) \cite{nolan1988functional,arcones1993limit,de2012decoupling,borovskikh1996}, as well as limit theorems for bootstrap \cite{arcones1994u,zhang2001bayesian}, hold for $U$-processes under metric (or bracketing) entropy conditions. These references~\cite{nolan1988functional,arcones1993limit,de2012decoupling,borovskikh1996,arcones1994u,zhang2001bayesian} cover both non-degenerate and degenerate $U$-processes where limiting processes of the latter are Gaussian chaoses rather than Gaussian processes. When the UCLTs do not hold for a possibly changing (in $n$) indexing function class (i.e., the function class cannot be embedded in any fixed Donsker class), \cite{chen2017jackknife} develops a general non-asymptotic theory for approximating the suprema of $U$-processes, extending the earlier work by \cite{chernozhukov2014gaussian} on empirical processes. Incomplete $U$-statistics for a fixed dimension are first considered in \cite{blom1976}, and the asymptotic distributions are studied in \cite{brownkildea1978,janson1984_PTRF}. In high dimensions, non-asymptotic Gaussian approximation and bootstrap results for randomized incomplete $U$-statistics are established in \cite{chen2017randomized} for a fixed order and in \cite{song2019approximating} for diverging orders. The current work considers randomized incomplete (local) $U$-processes with stratification.

\section*{Acknowledgements}
Y. Song is supported by
the Natural Sciences and Engineering Research Council of Canada (NSERC). This research is enabled in part by support provided by Compute Canada (\url{www.computecanada.ca}). X. Chen is supported in part by NSF CAREER Award DMS-1752614, UIUC Research Board Award RB18099, and a Simons Fellowship. X. Chen acknowledges that part of this work is carried out at the MIT Institute for Data, System, and Society (IDSS). K. Kato is partially supported by NSF grants DMS-1952306 and DMS-2014636.
%
%



\bibliographystyle{imsart-number} 
\bibliography{SILS_combined}       

\begin{appendix}

\section{Maximal inequalities}

\subsection{Local maximal inequalities for multiple incomplete U-processes}\label{app:local_max_ineq_incom}
In this section, we establish a local maximal inequality in Theorem \ref{thm:lmax_sampling_proc} for multiple incomplete $U$-processes. 
Thus, let $\cF$ be a collection of symmetric, measurable functions $f: (S^r,\cS^r) \to (\bR, \cB(\bR))$ with   a measurable envelope function $F:S^r \to [0,\infty)$ such that $0 < P^r F^2 < \infty$. 
For $\tau > 0$, define the uniform entropy integral
\begin{equation}
\label{def:uniform_entropy_integral_r}
\begin{split}
\bar{J}(\tau) &:= \bar{J}(\tau, \cF, F) := \int_{0}^{\tau} \sqrt{
1 + \sup_{Q} \log N(\cF, \|\cdot\|_{Q,2}, \epsilon \|F\|_{Q,2}) }\;\; d\epsilon, 
\end{split}
\end{equation}
where  
the $\sup_{Q}$ is taken over  all finitely supported {\it probability measures} on $\cS^r$. An important observation is that it is equivalent to take the supremum over all finitely support \textit{measures} $Q$ with $Q(\cS^r) < \infty$. 

Let $M \geq 1$ be an integer, and 
$\{Z_{\iota}^{(m)}: \iota \in I_{n,r}, m \in [M]\}$  be a collection of i.i.d. Bernoulli random variables with success probability $p_n = N/|I_{n,r}|$ for some integer $0 < N \leq |I_{n,r}|$, which are independent of the data $X_1^n$.
For $m \in [M]$, let
\begin{align}\label{def:m_incomp_U}
\bD_n^{(m)}(f) := \frac{1}{\sqrt{N}} \sum_{\iota \in I_{n,r}} 
(Z_{\iota}^{(m)} - p_n) f(X_{\iota}), \;\; \text{ for } f \in \cF.
\end{align}


\begin{theorem}
\label{thm:lmax_sampling_proc}
Let $\sigma_r > 0$ be such that
$\sup_{f \in \cF} \|f\|_{P^r,2} \leq \sigma_r \leq 
\|F\|_{P^r,2}
$.
Then for some absolute constant $C > 0$,
$C^{-1} \Exp\left[\max_{m \in [M]}\|{\bD}_{n}^{(m)}(f)\|_{\cF}\right]$ is upper bounded by
\begin{align*}
\sqrt{M \log(2M) }\bar{J}(\delta_r)\|F\|_{P^r,2} 
\; + \; \frac{ \log(2M)  \|\widetilde{\cT}\|_{\Pro,2} }{\sqrt{N}} \frac{\bar{J}^2(\delta_r)}{\delta_r^2} \; + \;
\sqrt{\Delta\log(2M) } \frac{  \bar{J}(\delta_r)}{\delta_r},
\end{align*}
where we define
\[\delta_r := \frac{\sigma_r}{\sqrt{M} \|F\|_{P^{r},2}},\quad
\Delta := 
\Exp\left[  \sup_{f \in \cF}\left|I_{n,r}\right|^{-1}\sum_{\iota \in I_{n,r}} f^2(X_{\iota})\right], \quad \widetilde{\mathcal{T}} := \max_{\iota \in I_{n,r}, m \in [M]} Z_{\iota}^{(m)} F(X_{\iota}).
\]
Further,   $\|\widetilde{\mathcal{T}}\|_{\Pro,2} \leq M^{1/q} N^{1/q} \|F\|_{P^r,q}$, for any $q \in [2,\infty]$, with the understanding that $1/q = 0$ if $q = \infty$. In addition, if $\cF$ is VC type class with characteristics $(A, \nu)$,
$\bar{J}(\tau) \leq C \tau \sqrt{\nu \log(A/\tau)}$ for $\tau \in (0,1]$. 
\end{theorem}

\begin{remark}
For a VC-type class,  the upper bound depends  on $M$ only via $\log^2(2M)$ if we use $\|F\|_{P^r,\infty}$ to bound $\|\widetilde{\mathcal{T}}\|_{\Pro,2}$. 
To further bound $\Delta$, which is the expectation of the supremum of a \textit{complete} $U$-process, we use the multi-level local maximal inequalities developed in \cite{chen2017jackknife}, which are summarized in Appendix \ref{subsec:multi-level_local} for convenience.
\end{remark}


\begin{remark}
There are two difficulties in the proof. First,  $p_n := N/|I_{n,r}|$ is very small (in most cases, asymptotically vanishing), which prevents us from directly applying sub-Gaussian inequality to $\bD_n^{(m)}(\cdot)$. One possible approach is to use Bernstein's inequality, which however involves the infinity norm, $\max_{\iota \in I_{n,r}} F(X_{\iota})$ (in comparison to $\widetilde{\mathcal{T}}$ in the above Lemma), and leads to sub-optimal rate. Second, we hope to achieve logarithmic dependence on $M$, so that the number of partitions will only impose a mild requirement on the sample size $n$ and computational budget $N$. 
\end{remark}

The idea of the current proof is in similar spirit as the proof for \cite[Theorem 5.2]{chernozhukov2014gaussian} and \cite[Theorem 5.1]{chen2017jackknife}. That is, we try to bound second moment by a concave function of first moment (see Step 3 in the proof of Theorem \ref{thm:lmax_sampling_proc} below), which, in \cite[Theorem 5.2]{chernozhukov2014gaussian} and \cite[Theorem 5.1]{chen2017jackknife}, was achieved by the contraction principle (\cite[Lemma A.5]{chernozhukov2014gaussian}) and Hoffman-J{\o}rgensen inequality(\cite[Theorem A.1]{chernozhukov2014gaussian}). However, in the presence of sampling ($\{Z_{\iota}^{(m)}\}$), these tools cannot be applied easily, and we need the following generalization of the contraction principle.

\begin{lemma}[Contraction principle]\label{lemma:max_contraction}
Let $M, L \geq 1$ be integers, $\Theta \subset \bR^{L}$ and $\{\epsilon_\ell^{(m)}: \ell \in [L],\; m \in [M]\}$ be a collection of independent Rademacher variables. Then
\begin{align*}
\Exp\left[\sup_{m \in [M], \theta \in \Theta} \left|\sum_{\ell \in [L]} \epsilon_\ell^{(m)} \theta_\ell^2 \right| \right]
\; \leq \; 4 \mathcal{T} 
\Exp\left[\sup_{m \in [M], \theta \in \Theta} \left|\sum_{\ell \in [L]} \epsilon_\ell^{(m)} \theta_\ell \right| \right],
\end{align*} 
where $\mathcal{T} := \sup_{\theta \in \Theta, \ell \in [L]} |\theta_{\ell}|$.
\end{lemma}
\begin{proof}
Define $\mathcal{A} := \cup_{m \in [M]} \mathcal{A}_m$, where
\[
\mathcal{A}_m := \bigcup_{\theta \in \Theta}\left\{ \alpha \in \bR^{L\times M}: \alpha_{\ell, j} = \theta_\ell \mathbbm{1}\{j = m\}\;\;\; \text{ for }  \ell \in [L],\; j \in [M]
\right\}.
\]
It is clear that $\Exp\left[\sup_{m \in [M], \theta \in \Theta} \left|\sum_{\ell \in [L]} \epsilon_\ell^{(m)} \theta_\ell \right| \right]
= \Exp\left[\sup_{\alpha \in \mathcal{A}} \left|\sum_{\ell \in [L], j \in [M]} \epsilon_\ell^{(j)} \alpha_{\ell, j} \right| \right]$ and \\
$\Exp\left[\sup_{m \in [M], \theta \in \Theta} \left|\sum_{\ell \in [L]} \epsilon_\ell^{(m)} \theta_\ell^2 \right| \right]
= \Exp\left[\sup_{\alpha \in \mathcal{A}} \left|\sum_{\ell \in [L], j \in [M]} \epsilon_\ell^{(j)} (\alpha_{\ell, j})^2 \right| \right]$.
Then the proof is complete by the usual contraction principle \cite[Theorem 4.12]{ledoux2013probability}.
\end{proof}

\begin{proof}[Proof of Theorem \ref{thm:lmax_sampling_proc}]
First,  define a collection of random measures (not necessarily probability measures) on $(S^r, \cS^r)$ that play important roles:
\begin{align}\label{def:two_random_measures}
\widehat{Q}_m := {N^{-1}}  \sum_{\iota \in I_{n,r}} Z_{\iota}^{(m)} \delta_{X_\iota}  \;\text{ for } m \in [M], \quad
\widehat{Q} := \sum_{m=1}^{M} \widehat{Q}_m,
\end{align}
where $\delta_{(x_1,\ldots,x_r)}$ is the Dirac measure such that
$
\delta_{(x_1,\ldots,x_r)}(A) = \ind\{(x_1,\ldots,x_r) \in A\}$
for any $A \in \cS^r$. Further, define
\begin{align}\label{def:V_hat_tilde}
\begin{split}
&\widehat{V}_n := \max_{m \in [M]}\sup_{f \in \cF} \|f\|_{\widehat{Q}_m,2} = \max_{m \in [M]}
\sup_{f \in \cF} \sqrt{N^{-1} \sum_{\iota \in I_{n,r}, m \in [M]} Z_{\iota}^{(m)} f^2(X_{\iota})},\\
&z^* := \sqrt{M^{-1}\Exp[\widehat{V}_n^2]/\|F\|_{P^r,2}^2} .
\end{split}
\end{align}
Finally, let $\{\epsilon_{\iota}^{(m)}: \iota \in I_{n,r}, m \in [M]\}$ be a collection of independent Rademacher random variables that are independent of 
$\cD_n := \{Z_{\iota}^{(m)}: \iota \in I_{n,r}, m \in [M]\} \cup X_1^n$. Define
for $f \in \cF$,
\begin{align*}
\widehat{\bD}_{n}^{(m)}(f) &:= N^{-1/2} \sum_{\iota \in I_{n,r}} \epsilon_{\iota}^{(m)} Z_{\iota}^{(m)} f(X_{\iota}) \text{ for } m \in [M], \quad \widehat{\bD}_{n}(f) := \max_{m \in [M]} |\widehat{\bD}_n^{(m)}(f)|, \\
\widetilde{\bD}_{n}(f) &:= \max_{m \in [M]} \left| N^{-1/2}  \sum_{\iota \in I_{n,r}} \epsilon_{\iota}^{(m)} Z_{\iota}^{(m)} f^2(X_{\iota}) \right|.
\end{align*}


\noindent\underline{{Step 1. bound $\Exp\left[\max_{m \in [M]}\|{\bD}_{n}^{(m)}\|_{\cF}\right]
$}}. 
Since $\{Z^{(m)}_{\iota}: \iota \in I_{n,r},\; m \in [M]\}$ is independent of $X_1^n$, by first conditioning on $X_1^n$ and then applying symmetrization inequality, 
\begin{align*}
\Exp\left[\max_{m \in [M]}  \left\|\bD_{n}^{(m)}(f)\right\|_{\cF}\right]
&= \Exp\left[ \Exp\left[
\max_{m \in [M]}\left\|\bD_{n}^{(m)}(f)\right\|_{\cF}\, |\, X_1^n
\right]\right]
 \lesssim \Exp\left[ \Exp\left[ \max_{m \in [M]}
\left\|\widehat{\bD}_{n}^{(m)}(f)\right\|_{\cF}\, |\, X_1^n
\right]\right] \\
\; &= \;
\Exp\left[ \max_{m \in [M]}
\left\|\widehat{\bD}_{n}^{(m)}(f)\right\|_{\cF}
\right] = \Exp\left[ \|\widehat{\bD}_n(f)\|_{\cF} \right].
\end{align*}

Then we apply the standard entropy integral bound conditioning on $\cD_n$. Specifically, 
for $m \in [M]$ and $f_1,f_2 \in \cF$, by Hoeffding's inequality \cite[Lemma 2.2.7]{van1996weak},
\begin{align*}
\|\widehat{\bD}_{n}^{(m)}(f_1) - \widehat{\bD}_{n}^{(m)}(f_2)\|_{\psi_2 \vert \cD_n} \; \lesssim \;
\sqrt{\frac{1}{N}\sum_{\iota \in I_{n,r}} (Z_{\iota}^{(m)})^2 (f_1(X_{\iota}) - f_2(X_{\iota}))^2} \; = \;
\|f_1 - f_2\|_{\widehat{Q}_m,2},
\end{align*}
where note that $Z^2 = Z$ for any  Bernoulli random variable $Z$. 
Now we use $\widehat{V}_n$ as the diameter for $\cF$ under $\|\cdot\|_{\widehat{Q}_m,2}$, and  by the entropy integral bound \cite[Corollary 2.2.5]{van1996weak}, 
\begin{align*}
\begin{split}
\left\| \|\widehat{\bD}_{n}^{(m)}(f)\|_{\cF}\right\|_{\psi_2  \vert \,
\cD_n } \; &\lesssim \;
\int_{0}^{\widehat{V}_n} \sqrt{1 + \log N(\cF, \|\cdot\|_{\widehat{Q}_m,2}, \epsilon)} \,d\epsilon  \\
& \;\leq\; \int_{0}^{\widehat{V}_n} \sqrt{1 + \log N(\cF, \|\cdot\|_{\widehat{Q},2}, \epsilon)} \,d\epsilon 
\;\leq\; \|F\|_{\widehat{Q},2} \bar{J}(\widehat{V}_n/\|F\|_{\widehat{Q},2}),
\end{split}
\end{align*}
where in the second inequality we used the fact that $\widehat{Q}$ dominates $\widehat{Q}_m$, and
in the third we used change-of variable technique and the definition of $\bar{J}(\cdot)$ (recall again $\widehat{Q}$ may not be a probability measure). Then by maximal inequality \cite[Lemma 2.2.2]{van1996weak}, 
\begin{align}\label{aux_inequ1}
\Exp\left[ \max_{m \in [M]} \|\widehat{\bD}_{n}^{(m)}(f)\|_{\cF} \, \vert \, \cD_n  \right]     \; &\lesssim \;
\sqrt{\log(2M)} \|F\|_{\widehat{Q},2} \bar{J}(\widehat{V}_n/\|F\|_{\widehat{Q},2}).
\end{align}
Since $(x,y) \mapsto \bar{J}(\sqrt{x/y})\sqrt{y}$ is jointly concave (see \cite[Lemma A.2]{chernozhukov2014gaussian}), by Jensen's inequality,
\begin{align}\label{aux_inequ2}
\begin{split}
\Exp\left[\max_{m \in [M]} \|\widehat{\bD}_{n}^{(m)}(f)\|_{\cF} \right] \; &\lesssim \; \sqrt{\log(2M)} 
\sqrt{\Exp\left[ \|F\|_{\widehat{Q},2}^2\right]} 
\bar{J}(\sqrt{\Exp[\widehat{V}_n^2]/{\Exp\left[ \|F\|_{\widehat{Q},2}^2\right]}}) \\
&= \sqrt{M \log(2M)} \|F\|_{P^r,2} \bar{J}(z^*),
\end{split}
\end{align}
where 
for the equality, we use the fact that
$$
\Exp\left[ \|F\|_{\widehat{Q},2}^2\right] = 
\Exp\left[ {N^{-1}} \sum_{m \in [M]}\sum_{\iota \in I_{n,r}} Z_{\iota}^{(m)}F^2(X_{\iota})\right] = M \|F\|_{P^r,2}^2.
$$

\noindent\underline{{Step 2. bound $\Exp[\widehat{V}_n^2]$}}. 
Observe that
\begin{align}
\begin{split}
\Exp[\widehat{V}_n^2] 
&\leq \frac{1}{{N}}\Exp\left[ \max_{m \in [M]}\sup_{f \in \cF} \sum_{\iota \in I_{n,r}} (Z_{\iota}^{(m)} - p_n) f^2(X_{\iota}) \right] +  \Exp\left[
\sup_{f \in \cF}\left|I_{n,r}\right|^{-1}\sum_{\iota \in I_{n,r}} f^2(X_{\iota})
\right]\\
\;\leq\; & N^{-1/2}\Exp\left[ \left\| \widetilde{\bD}_n(f)\right\|_{\cF}\right] 
+  \Delta,
\end{split}
\label{bound_exp_Vn_2}
\end{align}
where 
in the last inequality we again first condition on $X_1^n$ and then apply symmetrization inequality. For the first term, conditional on $\cD_n$, by Lemma \ref{lemma:max_contraction} (contraction principle) and \eqref{aux_inequ1},
\begin{align*}
\Exp\left[ \left. \left\| \widetilde{\bD}_n(f) \right\|_{\cF} \;\right\vert \cD_n \right]
\; \lesssim \;& \widetilde{\mathcal{T}}\; \Exp\left[ \|\widehat{\bD}_n(f)\|_{\cF}\;\vert \cD_n 
\right]
\; \lesssim \;
\sqrt{\log(2M)} \widetilde{\mathcal{T}} \|F\|_{\widehat{Q},2} \bar{J}(\widehat{V}_n/\|F\|_{\widehat{Q},2}).
\end{align*}
Denote $\mvxi := {\widetilde{\mathcal{T}}} \|F\|_{\widehat{Q},2}$. Since $(x,y) \mapsto \bar{J}(x/y){y}$ is jointly concave (see \cite[Lemma A.3]{chernozhukov2014gaussian}), by Jensen's inequality and Cauchy-Schwarz inequality,
\begin{align*}
\Exp\left[ \left\| \widetilde{\bD}_n(f) \right\|_{\cF} \right] 
&\; \lesssim \;
\sqrt{\log(2M)} \Exp\left[ \mvxi  \bar{J} \left(\frac{\widehat{V}_n}{ \|F\|_{\widehat{Q},2}} \right)\right] 
\; \lesssim \;
\sqrt{\log(2M)}  \Exp\left[ \mvxi  \right] \bar{J}\left(\frac{\Exp\left[ {\widetilde{\mathcal{T}}} \widehat{V}_n \right]}{\Exp\left[ \mvxi  \right]} \right) \\
&\; \lesssim \;
\sqrt{\log(2M)}  \Exp\left[ \mvxi \right] \bar{J} \left(\frac{\|\widetilde{\mathcal{T}}\|_{\Pro,2} \sqrt{\Exp[\widehat{V}_n^2]} }{\Exp\left[ \mvxi \right]} \right).
\end{align*}
Since $\bar{J}(\cdot)$ is non-decreasing in $[0,\infty)$ and for $c \geq 1$, $\bar{J}(c\tau) \leq c\bar{J}(\tau)$ (see \cite[Lemma A.2]{chernozhukov2014gaussian}) and by definition of $z^*$, we have
\begin{align}
\begin{split}
 \Exp\left[\|\widetilde{\bD}_{n}(f)\|_{\cF} \right] \; \lesssim \;&
\sqrt{\log(2M)}  \Exp\left[ \mvxi \right]
\bar{J}\left(\left(\Exp\left[ \mvxi \right]^{-1}\|\widetilde{\mathcal{T}}\|_{\Pro,2} \sqrt{M} \|F\|_{P^r,2} \vee 1\right) z^*  \right)  \\
\; \leq \;&
\sqrt{\log(2M)} \left(\sqrt{M}\|\widetilde{\mathcal{T}}\|_{\Pro,2} \|F\|_{P^r,2}
+ \Exp\left[ \mvxi \right]
\right) \bar{J}(z^*) \\
\leq \; &
2 \sqrt{M\log(2M)}\|\widetilde{\mathcal{T}}\|_{\Pro,2} \|F\|_{P^r,2} \bar{J}(z^*), 
\end{split} \label{aux_inequ3}
\end{align}
where in the last inequality we applied Cauchy-Schwarz inequality.\\

\noindent\underline{{Step 3. bound $(z^*)^2$ by $\bar{J}(z^*)$}}. 
Applying \eqref{aux_inequ3} to the first term in \eqref{bound_exp_Vn_2}, 
 we have
\begin{align*}
(z^*)^2  = M^{-1}\Exp\left[\widehat{V}_n^2\right]/\|F\|_{P^r,2}^2
\;\leq\; 
& \frac{\Delta}{M \|F\|_{P^r,2}^2} 
+ \frac{\sqrt{\log(2M)} \|\widetilde{\mathcal{T}}\|_{\Pro,2} }{\|F\|_{P^r,2}\sqrt{M N}}\bar{J}(z^*).
\end{align*}
Now define 
$$\Delta' := \max\left\{
{\delta}_r,\; \sqrt{M^{-1}\Delta}/\|F\|_{P^r,2}
\right\} \; \geq \;
{\delta}_r.
 $$
Applying \cite[Lemma 2.1]{van2011local}, we have
\begin{align*}
\bar{J}(z^*) 
\;\lesssim\; \bar{J}(\Delta')  + 
\frac{\sqrt{\log(2M)} \|\widetilde{\mathcal{T}}\|_{\Pro,2} \bar{J}^2(\Delta')}{\|F\|_{P^r,2} \sqrt{M N}(\Delta')^2}.
\end{align*}
Now we apply the above result to \eqref{aux_inequ2}:
\begin{align*}
\Exp\left[\max_{m \in [M]}  \left\|\bD_{n}^{(m)}(f)\right\|_{\cF}\right]
\; \lesssim \;
\sqrt{M \log(2M)} \left(
\bar{J}(\Delta')\|F\|_{P^r,2}  + 
\frac{\sqrt{\log(2M)} \|\widetilde{\mathcal{T}}\|_{\Pro,2} \bar{J}^2(\Delta')}{ \sqrt{M N}(\Delta')^2}
\right).
\end{align*}
Since $J(\tau)/\tau$ is non-increasing(see \cite[Lemma 5.2]{chernozhukov2014gaussian}), we have
\begin{align*}
\bar{J}(\Delta') \leq \Delta' \frac{\bar{J}(\delta_r)}{\delta_r}
= \max\left\{
\bar{J}(\delta_r),\;
\frac{\sqrt{\Delta} \bar{J}(\delta_r)}{\sqrt{M}\|F\|_{P^r,2} \delta_r}
\right\}, \quad
\frac{\bar{J}^2(\Delta')}{(\Delta')^2} \leq 
\frac{\bar{J}^2(\delta_r)}{\delta_r^2},
\end{align*}
which completes the proof of the first inequality. 
\end{proof}

\begin{remark}
In Appendix \ref{sec:non_local_maximal}, we also establish a non-local maximal inequality for the supremum of multiple $U$-processes, which is simpler, and will be used to bound the second moment of the supremum.
\end{remark}

The following Corollary is needed in establishing the validity of bootstrap, which has essentially been established in the proof of Theorem \ref{thm:lmax_sampling_proc}.

\begin{corollary}\label{cor:aux_lemmas}
Assume the conditions and notations in Theorem \ref{thm:lmax_sampling_proc}. Recall the definitions of $\widehat{Q}$ in \eqref{def:two_random_measures} and $\widehat{V}_n$ in \eqref{def:V_hat_tilde}. Then for some absolute constant $C > 0$,
\begin{align*}
&C^{-1} \sqrt{\log(2M)}\Exp\left[\|F\|_{\widehat{Q},2} \bar{J}(\widehat{V}_n/\|F\|_{\widehat{Q},2} ) \right] 
\leq  \;  \; \\
& \sqrt{M \log(2M) }\bar{J}(\delta_r)\|F\|_{P^r,2} + \frac{ \log(2M)  \|\widetilde{\cT}\|_{\Pro,2} }{\sqrt{N}} \frac{\bar{J}^2(\delta_r)}{\delta_r^2} \; + \;
\sqrt{\Delta\log(2M) } \frac{  \bar{J}(\delta_r)}{\delta_r},
\\
& \Exp\left[ \widehat{V}_n^2\right] \lesssim 
\Delta + \sigma_r^2 + \frac{\log(2M)\|\widetilde{\mathcal{T}}\|_{\Pro,2}^2 \bar{J}^2(\delta_r)}{N\delta_r^2},
\end{align*}
where recall that $\Delta := 
\Exp\left[ \sup_{f \in \cF} \left|I_{n,r}\right|^{-1}\sum_{\iota \in I_{n,r}} f^2(X_{\iota})\right]$.
\end{corollary}

\begin{proof}
In the Step 1 of proving Theorem \ref{thm:lmax_sampling_proc}, we bound $\Exp\left[\max_{m \in [M]}\|{\bD}_{n}^{(m)}(f)\|_{\cF}\right]$ through an upper bound on $\sqrt{\log(2M)}\Exp\left[\|F\|_{\widehat{Q},2} \bar{J}(\widehat{V}_n/\|F\|_{\widehat{Q},2} ) \right]$ (see \eqref{aux_inequ1}). Thus the first inequality has been established.

For the second inequality, in the Step 3 of proving Theorem \ref{thm:lmax_sampling_proc}, we showed that
\begin{align*}
&\Exp\left[ \widehat{V}_n^{2} \right] \lesssim
\Delta + \frac{ \|\widetilde{\mathcal{T}}\|_{\Pro,2}}{\sqrt{N}} \sqrt{M\log(2M)}\|F\|_{P^r,2} \bar{J}(z^*), \\
&\sqrt{M\log(2M)} \|F\|_{P^r,2} \bar{J}(z^*) \lesssim  \\
&\sqrt{M \log(2M) }\bar{J}(\delta_r)\|F\|_{P^r,2} + \frac{ \log(2M)  \|\widetilde{T}\|_{\Pro,2} }{\sqrt{N}} \frac{\bar{J}^2(\delta_r)}{\delta_r^2} \; + \;
\sqrt{\Delta\log(2M) } \frac{  \bar{J}(\delta_r)}{\delta_r}.
\end{align*}
As a result,
\begin{align*}
\Exp\left[ \widehat{V}_n^{2} \right] \lesssim &
\Delta + \frac{\log(2M)\|\widetilde{\mathcal{T}}\|_{\Pro,2}^2 \bar{J}^2(\delta_r)}{N\delta_r^2}
+ \frac{\sqrt{M \log(2M) } \|\widetilde{\mathcal{T}}\|_{\Pro,2} \|F\|_{P^r,2} \bar{J}(\delta_r)}{\sqrt{N}} \\
&+ \frac{\sqrt{\log(2M)\Delta} \|\widetilde{\mathcal{T}}\|_{\Pro,2}  \bar{J}(\delta_r)}{\sqrt{N} \delta_r}.
\end{align*}
Then the results follow 
by Cauchy-Schwarz  inequality and due to  the definition that $\sigma_r := \sqrt{M}\delta_r \|F\|_{P^r,2}$.
\end{proof}

\subsection{Non-local maximal inequalities for multiple incomplete U-processes}
\label{sec:non_local_maximal}
In this subsection, we establish an upper bound for the second moment of the supremum over multiple incomplete $U$-processes, as defined in Section \ref{app:local_max_ineq_incom}, which however is non-local. Recall the notations in Section \ref{app:local_max_ineq_incom}.


\begin{lemma}\label{lemma:incom_F}
Denote $z := \Exp\left[\max_{m \in [M]} N^{-1} \sum_{\iota \in I_{n,r}} Z_{\iota}^{(m)} F^2(X_{\iota}) \right]$, and recall the definition of $\widetilde{\mathcal{T}}$ in Theorem \ref{thm:lmax_sampling_proc}. Then 
for some absolute constant $C > 0$,
\begin{align*}
z \leq C \left( \|F\|_{P^r,2}^2 + \frac{\log(2M) \|\widetilde{\mathcal{T}}\|_{\Pro,2}^2}{N} \right).
\end{align*}
\end{lemma}

\begin{proof}
Observe that
$z \leq  \Exp\left[\max_{m \in [M]} N^{-1} \sum_{\iota \in I_{n,r}} (Z_{\iota}^{(m)} - p_n) F^2(X_{\iota}) \right] + \|F\|_{P^r,2}^2$. 
Let $\{\epsilon_{\iota}^{(m)}: \iota \in I_{n,r}, m \in [M]\}$ be a collection of independent Rademacher random variables that are independent of 
$\cD_n  := \{Z_{\iota}^{(m)}: \iota \in I_{n,r}, m \in [M]\} \cup X_1^n$. 
By the symmetrization inequality (conditional on $X_1^n$), the contraction principle (Lemma \ref{lemma:max_contraction}, conditional on $\cD_n$), and Cauchy-Schwarz inequality,
\begin{align*}
z \;&\lesssim\; \Exp\left[\max_{m \in [M]} \left| N^{-1} \sum_{\iota \in I_{n,r}} \epsilon_{\iota}^{(m)}  Z_{\iota}^{(m)}  F^2(X_{\iota}) \right| \right] + \|F\|_{P^r,2}^2 \\
\;&\lesssim\;
\Exp\left[\widetilde{\mathcal{T}} \max_{m \in [M]} \left| N^{-1} \sum_{\iota \in I_{n,r}} \epsilon_{\iota}^{(m)}  Z_{\iota}^{(m)}  F(X_{\iota}) \right| \right] + \|F\|_{P^r,2}^2\\
\;&\lesssim\;
\|\widetilde{\mathcal{T}}\|_{\Pro,2} \left( \Exp\left[ \max_{m \in [M]} \left| N^{-1} \sum_{\iota \in I_{n,r}} \epsilon_{\iota}^{(m)}  Z_{\iota}^{(m)}  F(X_{\iota}) \right|^2 \right] \right)^{1/2} + \|F\|_{P^r,2}^2.
\end{align*}
By  Hoeffding's inequality \cite[Lemma 2.2.7]{van1996weak}, for $m \in [M]$,
\[
\left\|  N^{-1} \sum_{\iota \in I_{n,r}} \epsilon_{\iota}^{(m)}  Z_{\iota}^{(m)}  F(X_{\iota}) \right\|_{\psi_2 \vert \cD_n}
\lesssim \sqrt{N^{-2} \sum_{\iota \in I_{n,r}} Z_{\iota}^{(m)}  F^2(X_{\iota})}.
\]
Thus by maximal inequality \cite[Lemma 2.2.2]{van1996weak}, 
\begin{align*}
\Exp\left[ \max_{m \in [M]} \left| N^{-1} \sum_{\iota \in I_{n,r}} \epsilon_{\iota}^{(m)}  Z_{\iota}^{(m)}  F(X_{\iota}) \right|^2 \right] 
\; \lesssim \; & 
\Exp\left[ \log(2M) \max_{m \in [M]} N^{-2} \sum_{\iota \in I_{n,r}} Z_{\iota}^{(m)}  F^2(X_{\iota}) \right] \\
\; \leq \; & \log(2M) z/N.
\end{align*}
Thus we have 
$z \lesssim \frac{\|\widetilde{\mathcal{T}}\|_{\Pro,2} \sqrt{\log(2M)} }{\sqrt{N}} \sqrt{z} + \|F\|_{P^r,2}^2$.
which implies the conclusion.
\end{proof}

Recall the definition of ${\bD}_{n}^{(m)}$ in \eqref{def:m_incomp_U}.
\begin{lemma}\label{lemma:sampling_higher_order}
Recall the definition of $\widetilde{\mathcal{T}}$ in Theorem \ref{thm:lmax_sampling_proc}. Then for some absolute constant $C > 0$, 
\begin{align*}
&\Exp\left[ \max_{m \in [M]} \|{\bD}_{n}^{(m)}(f)\|_{\cF}^2 \right]
\;\leq\;   C \bar{J}^2(1) \left( \log(2M) \|F\|_{P^r,2}^2 + \frac{\log^2(2M) \|\widetilde{\mathcal{T}}\|_{\Pro,2}^2}{N}
\right).
\end{align*}
\end{lemma}
\begin{proof}
Recall the definitions of $\widehat{\bD}_{n}^{(m)}$, $\widehat{Q}_m$, and $\cD_n$ in the proof for Theorem \ref{thm:lmax_sampling_proc}.
Define for $m \in [M]$,
$
\widehat{V}_n^{(m)} := \sup_{f \in \cF} \|f\|_{\widehat{Q}_m,2}
$.
By the same argument leading to  \eqref{aux_inequ1} as in the proof for Theorem \ref{thm:lmax_sampling_proc}, we have  
\begin{align*}
&\Exp\left[ \max_{m \in [M]} \|{\bD}_{n}^{(m)}(f)\|_{\cF}^2 \, \vert \,
X_{1}^n \right]
\lesssim
\Exp\left[\max_{m \in [M]}\|\widehat{\bD}_{n}^{(m)}(f)\|_{\cF}^2\, \vert \,X_{1}^n\right] \\
&\Exp\left[\max_{m \in [M]}\|\widehat{\bD}_{n}^{(m)}(f)\|_{\cF}^2\, \vert \,
\cD_n \right]
\; \lesssim \; \log(2M)
\max_{m \in [M]}\left(
\int_{0}^{\widehat{V}_n^{(m)}} \sqrt{1 + \log N(\cF, \|\cdot\|_{\widehat{Q}_m,2}, \epsilon)} \,d\epsilon  \right)^2 \\
&\;\leq\; \log(2M) \max_{m \in [M]} \|F\|_{\widehat{Q}_m,2}^2 \bar{J}^2(1),
\end{align*}
where we used change-of variable, and the fact that $\widehat{V}_n^{(m)} \leq \|F\|_{\widehat{Q}_m,2}$. 
Then the proof is complete  by taking expectation on both sides, and
 applying the Lemma \ref{lemma:incom_F}.
\end{proof}

\subsection{A discretization lemma} \label{proof:discretization_lemma}
Theorem \ref{thm:lmax_sampling_proc} will usually be applied 
after discretization. 
The following lemma is useful for discretizating a VC-type class, as defined in \eqref{def:VC_type}, so that the difference class has the desired property.

\begin{lemma}\label{discretization_lemma}
Assume $(\cF,F)$ is a VC-type class  with characteristics $(A,\nu)$  and $P^r F^4 < \infty$. For any $\epsilon \in (0,1)$, there exists a finite collection $\{f_j : 1 \leq j \leq d\} \subset \cF$ such that
the following two conditions hold: (i). $d \leq (4A/\epsilon)^{\nu}$; (ii). for any $f \in \cF$, there exists $1 \leq j^* \leq d$ such that 
\begin{align*}
\max\left\{\|f - f_{j^*} \|_{P^r,2}, \;\; \|f - f_{j^*} \|_{P^r,4}^2 \right\} \;\leq \; \epsilon \|1 + F^2\|_{P^r,2}.
\end{align*} 
Further, define 
\[
\cF_{\epsilon} := \left\{ f - f': \max\left\{\|f - f' \|_{P^r,2}, \;\; \|f - f' \|_{P^r,4}^2 \right\} \;\leq \; \epsilon \|1 + F^2\|_{P^r,2}
\right\}.
\]
Then $(\cF_{\epsilon}, 2F)$ is a VC-type class with characteristics $(A,2\nu)$.
\end{lemma}

\begin{proof}
We start with the first claim. Define a measure $Q^*$ (not probability) on $(S^r, \cS^r)$ as follows:
\[
Q^*(A) := \int_A (1 + F^2) dP^r, \text{ for any } A \in \cS^r.
\]
Since $P^r F^4 < \infty$, $Q^*$ is a finite measure and $Q^* F^2 < \infty$. By definition \ref{def:VC_type} of VC-type class and \cite[Problem 2.5.1, Page 133]{van1996weak},
\begin{align*}
N(\cF, \|\cdot\|_{Q^*,2}, 2^{-1}\epsilon\|F\|_{Q^*,2})
\leq \sup_{Q} N(\cF, \|\cdot\|_{Q,2}, 4^{-1}\epsilon\|F\|_{Q,2})
\leq (4A/\epsilon)^{\nu},
\end{align*}
where the $\sup_{Q}$ is taken over  all finitely supported probability measures on $\cS^r$. Thus there exists an integer $d \leq (4A/\epsilon)^{\nu}$ and a subset $\{f_j: 1\leq j \leq d\} \subset \cF$ such that for any $f \in \cF$, there exists $1 \leq j^* \leq d$ such that
\begin{align*}
\|f - f_{j^*}\|_{Q^*,2}^2 \leq  2^{-2}\epsilon^2\|F\|_{Q^*,2}^2 
= 4^{-1}\epsilon^2 \int (F^2 + F^4) dP^r \leq 
4^{-1}\epsilon^2\|1 + F^2 \|_{P^r,2}^2.
\end{align*}
On the other hand,
\begin{align*}
&\|f - f_{j^*}\|_{Q^*,2}^2 = \int (f-f_{j^*})^2 (1 + F^2) dP^r 
\geq \|f - f_{j^*}\|_{P^r,2}^2, \\
&\|f - f_{j^*}\|_{Q^*,2}^2 = \int (f-f_{j^*})^2 (1 + F^2) dP^r 
\geq 4^{-1} \int (f-f_{j^*})^4 dP^r 
=  4^{-1} \|f-f_{j^*}\|_{P^r,4}^4,
\end{align*}
which completes the proof of the first claim.
\vspace{0.2cm}

For the second claim, for any $\tau > 0$ and finite probability measure $Q$, there exists $\{f_j': 1\leq j \leq d'\} \subset \cF$ such that $d' = N(\cF, \|\cdot\|_{Q,2}, \tau \|F\|_{Q,2}) \leq (A/\tau)^\nu$, and for any $f \in \cF$, there exists $1 \leq j' \leq d'$ such that
\begin{align*}
\|f - f_j'\|_{Q,2} \leq \tau \|F\|_{Q,2}.
\end{align*}
By triangle inequality, $\{f_j' - f_k': 1 \leq j,k \leq d' \}$ is a $\tau\|2F\|_{Q,2}$ cover for $\cF- \cF:= \{f-f': f,f' \in \cF\}$. As a result,
\begin{align*}
N(\cF_{\epsilon}, \|\cdot\|_{Q,2}, \tau\|2F\|_{Q,2}) \leq 
N(\cF - \cF, \|\cdot\|_{Q,2}, \tau\|2F\|_{Q,2}) \leq 
(A/\tau)^{2\nu},
\end{align*}
which completes the proof.
\end{proof}

\subsection{Multi-level local maximal inequality for complete U-processes} 
\label{subsec:multi-level_local}
 In this subsection, let $\cF$ be a collection of symmetric, measurable functions $f: (S^r,\cS^r) \to (\bR, \cB(\bR))$ with   a measurable envelope functrion $F:S^r \to [0,\infty)$ such that $0 < P^r F^2 < \infty$. For $f \in \cF$, denote by
 $U_n^{(r)}(f) := \left|I_{n,r}\right|^{-1}\sum_{\iota \in I_{n,r}} f(X_{\iota}) $ its associated $U$-process.

For each $1 \leq \ell \leq r$, the Hoeffding projection (with respect to P) is defined by
\[
( \pi_\ell f) (x_1,\ldots,x_\ell) :=
(\delta_{x_1} - P)\cdots (\delta_{x_{\ell}}  - P) P^{r-\ell} f.
\]
Then the Hoeffding decomposition \cite{hoeffding1948class} of  ${U}^{(r)}_n$ is given by
\begin{align}
\label{Hoeffding decomp}
U_n^{(r)}(f) - P^r f = \sum_{i=1}^{\ell}  \binom{r}{\ell} U_n^{(\ell)}(\pi_\ell f),\;\; \text{ for } f \in \cF.
\end{align}
For $1 \leq \ell \leq r$, let $F_\ell$ be an envelope function for
$P^{r-\ell} \cF := \{P^{r-\ell}f : f \in \cF\}$, 
i.e., $|P^{r-\ell} f (x)| \leq F_{\ell}(x)$ for any $f \in \cF$ and $x \in S^{\ell}$.
Further for $1 \leq \ell \leq r$, let $\sigma_\ell$ be such that
$\sup_{f \in \cF} \|P^{r-\ell} f\|_{P^\ell,2} \leq \sigma_\ell \leq 
\|F_{\ell}\|_{P^\ell,2}
$ and define 
\[
\mathcal{T}_\ell := \max_{1 \leq i \leq \lfloor n/\ell \rfloor} F_{\ell} (X_{(i-1)\ell+1}^{i\ell}).
\]
For $\tau > 0$ and $1 \leq \ell \leq r$, define the uniform entropy integral
\begin{equation}
\label{def:uniform_entropy_integral}
\begin{split}
J_{\ell}(\tau) &:= J_{\ell}(\tau, P^{r-\ell}\cF, F_{\ell}) := \int_{0}^{\tau} \left[
1 + \sup_{Q} \log N(P^{r-\ell} \cF, \|\cdot\|_{Q,2}, \epsilon \|F_{\ell}\|_{Q,2}) \right]^{\ell/2}\;\; d\epsilon, 
\end{split}
\end{equation}
where  
the $\sup_{Q}$ is taken over  all finitely supported probability measures on $\cS^\ell$. 

The following Theorem is due to \cite[Theorem 5.1 and Corollary 5.6]{chen2017jackknife}, and included here due to its repeated use in this paper. Together with \eqref{Hoeffding decomp}, it provides multi-scale local inequalities for the complete $U$-process. 

\begin{theorem}\label{thm:multi-level}
For $1 \leq \ell \leq r$, let $\delta_\ell := \sigma_{\ell}/\|F_{\ell}\|_{P^{\ell},2}$. Then 
\begin{align*}
n^{\ell/2} \Exp\left[
\|U_n^{(r)}(\pi_{\ell} f)\|_{\cF}
\right] \;&\lesssim\;  \min\left\{J_{\ell}(1) \|F_{\ell}\|_{P^{\ell},2}, \;\;J_{\ell}(\delta_{\ell})\|F_{\ell}\|_{P^{\ell},2} 
+ \frac{J^2_{\ell}(\delta_{\ell}) \|\mathcal{T}_\ell\|_{\Pro,2}}{\delta_{\ell}^2 \sqrt{n}}\right\},
\end{align*}
where $\|\cdot\|_{\cF} = \sup_{f\in \cF} |\cdot|$, and $\lesssim$ means up to a multiplicative constant only depending on $r$. 
Further, if $\cF$ has a finite cardinality $d < \infty$, then
\begin{align*}
n^{\ell/2} \Exp\left[
\|U_n^{(r)}(\pi_{\ell} f)\|_{\cF}
\right] \;&\lesssim\;  \min\left\{ \|F_{\ell}\|_{P^{\ell},2} \log^{\ell/2}(d), \;\;\sigma_{\ell} \log^{\ell/2}(d)
+ n^{-1/2} \|\mathcal{T}_{\ell}\|_{\Pro,2} \log^{\ell/2+1/2}(d) \right\}.
\end{align*}
If $(\cF,F)$ is VC type class with characteristics $(A,\nu)$ with $A \geq e \vee (e^{2(r-1)}/16)$ and $\nu \geq 1$, and
$F_{\ell} = P^{r-\ell} F$ for $1 \leq \ell \leq r$,
 then there exists a constant $C$, depending only on $r$, such that for any $\tau \in (0,1]$,
\begin{align*}
J_{\ell}(\tau) &\leq C \tau \left( \nu \log(A/\tau) \right)^{\ell/2} \; \text{ for } 1 \leq \ell \leq r.
\end{align*}
Finally, for any $q \in [2,\infty]$,  $\|\mathcal{T}_\ell\|_{\Pro,2} \leq n^{1/q} \|F_{\ell}\|_{P^\ell,q}$, where $1/q = 0$ if $q = \infty$. 
\end{theorem}
\begin{proof}
The first  inequality is established in \cite[Theorem 5.1 and Corollary 5.6]{chen2017jackknife}. The second inequality slightly improves the dependence of the second term inside $\min\{\cdot,\cdot\}$ on $\log(d)$ for a finite class. Its proof is essentially the same, 
except that we use the fact that $J_{\ell}(\tau) \lesssim \tau\log^{\ell/2}(d)$ for any $\tau > 0$, and thus is omitted. 
The third claim is established in \cite[Corollary 5.3]{chen2017jackknife},
while the last claim is obvious
\end{proof}

\section{Stratified, incomplete high-dimensional U-statistics}\label{sec:GAR_incomp_U_stat}

In this subsection, we establish Gaussian approximation and bootstrap results for stratified, incomplete high-dimensional $U$-statistics, which is a key step in establishing the distribution approximation for the supremum of the $U$-process. 
 Thus let $\{h_1,\ldots, h_d\}$ be a collection of $d$ elements in $\cH$, and define $\underline{h}:= (h_1,\ldots, h_d): S^r \to \bR^d$.
Consider the following complete and stratified, incomplete $d$-dimensional $U$-statistics:
\begin{align}
\label{eqn:d-dimension U stat}
\begin{split}
&\underline{U}_{n,j} := U_n(h_j) = |I_{n,r}|^{-1} \sum_{\iota \in I_{n,r}} {h}_j(X_{i_1},\ldots, X_{i_r}) :={|I_{n,r}|}^{-1} \sum_{\iota \in I_{n,r}} {h}_j(X_{\iota}),\\
&\underline{U}'_{n,N,j} :=U_{n,N}'(h_j) = (\widehat{N}^{\sigma(h_j)})^{-1} \sum_{\iota \in I_{n,r}} Z_{\iota}^{\sigma(h_j)} {h}_j(X_{\iota}),
\end{split}
\end{align}
where we recall that $\widehat{N}^{(m)} := \sum_{\iota \in I_{n,r}} Z_{\iota}^{(m)}$ for $m \in [M]$, and $\sigma(h_j) = m$ if and only if $h_j \in \cH_m$. Further, define $\underline{\theta} := \Exp[\underline{U}_n]$, and  $d$-by-$d$ matrices $\Gamma_A$, $\Gamma_B$, and $\Gamma_{*}$ such that for $1 \leq i,j \leq d$,
\begin{align}\label{def:gamma_AB_mat}
\Gamma_{A,ij} := \gamma_A(h_i, h_j),\quad
\Gamma_{B,ij} := \gamma_B(h_i, h_j), \quad
\Gamma_{*,ij} := \gamma_{*}(h_i, h_j),
\end{align}
where we recall that the covariance functions $\gamma_A$, $\gamma_B$, and $\gamma_{*}$ are defined as follows: for $h,h' \in \cH$,
\begin{align}\label{def:cov_G}
\begin{split}
&\gamma_{*}(h,h') := r^2 \gamma_A(h,h') + \alpha_n \gamma_B(h,h'),\quad 
\text{ where }\;\; \alpha_n := n/N,\\
&\gamma_A(h,h') :=  \Cov\left(P^{r-1}h(X_1),\; P^{r-1}h'(X_1)\right),  \\
&\gamma_B(h,h') := \Cov\left(h(X_1^r),\; h'(X_1^r)\right) \mathbbm{1}\{\sigma(h) = \sigma(h')\}.
\end{split}
\end{align}
If $h = h'$, we write ${\gamma}_{*}(h)$ for ${\gamma}_{*}(h,h)$; the same convention applies to ${\gamma}_{A}$ and ${\gamma}_{B}$.

Denote by $\cR := \{\prod_{j=1}^d [a_j,b_j]: -\infty \leq a_j \leq b_j \leq \infty\}$  the collection of hyperrectangles in $\bR^d$. 
We start with the Gaussian approximation results.
\begin{theorem}\label{thm:GAR_complete_hd_U}
Assume \eqref{MT0}-\eqref{MT2} hold. Then there exists a constant $C$, depending only on
$r, q, \ubar{\sigma}$, such that
\begin{align*}
&C^{-1} \sup_{R \in \cR}\left\vert
\Pro\left(\sqrt{n}(\underline{U}_{n}-\underline{\theta}) \in R \right) - \Pro(r Y_A \in R)
\right\vert \\
\; \leq \; &
\left( \frac{D_n^2 \log^7(dn)}{n}\right)^{1/6} + 
\left(\frac{D_n^2 \log^3(dn)}{n^{1-2/q}}  \right)^{1/3} +
\left(\frac{D_n^{3-2/q} \log^{2}(d)}{n^{1-1/q}}\right)^{1/2},
\end{align*}
where  $Y_A \sim N(0, \Gamma_{A} )$   and $1/q = 0$ if $q = \infty$.
\end{theorem}
\begin{proof}
See Section \ref{sec:proof_complete_Ustat}.
\end{proof}

\begin{theorem}\label{thm:GAR_hd_U}
Assume \eqref{MT0}-\eqref{MT4} and \ref{cond:MB} hold.
Then there exists a constant $C$, depending only on
$r, q, \ubar{\sigma}, c_0, C_0$, such that
\begin{align*}
&\sup_{R \in \cR}\left\vert
\Pro\left(\sqrt{n}(\underline{U}_{n,N}'-\underline{\theta}) \in R \right) - \Pro(Y \in R)
\right\vert  \; \leq \; C \varpi_n^{(1)} + C \varpi_n^{(2)}, \;\; \text{ with }\\
& \varpi_n^{(1)}:=
\left( \frac{D_n^2 \log^7(dn)}{n}\right)^{1/8} + 
\left(\frac{D_n^2 \log^3(dn)}{n^{1-2/q}}  \right)^{1/4} +
\left(\frac{D_n^{3-2/q} \log^{2}(d)}{n^{1-1/q}}\right)^{1/2},\\
& \varpi_n^{(2)}:=
\left( \frac{B_n^2 \log^7(dn)}N \right)^{1/8} +  \left( \frac{n^{4r/q} \log^5(dn) B_n^{2-8/q} D_n^{8/q}}{N}
\right)^{1/4}, 
\end{align*}
where  $Y \sim N(0,r^2 \Gamma_{A} + \alpha_{n} \Gamma_B)$,  $\alpha_n := n/N$, and $1/q = 0$ if $q = \infty$.
\end{theorem}
\begin{proof}
See Section \ref{proof:GAR_hd_U}.
\end{proof}

Recall the definitions of  $\bU_{n,A}^{\#}(\cdot)$,  $\bU_{n,B}^{\#}(\cdot)$, $\bU_{n,*}^{\#}(\cdot)$, and $\cD_{n}'$ in Section \ref{sec:bootstrap_uproc}. 
Note that conditional on $\cD'_n$, $\bU^{\#}_{n,*}$, $\bU^{\#}_{n,A}$, and $\bU^{\#}_{n,B}$ are centered Gaussian processes with the following covariance functions :  
\begin{align}\label{def:est_Covs}
\begin{split}
\widehat{\gamma}_{*}(h,h') &:= r^2 \widehat{\gamma}_A(h,h') + \alpha_n \widehat{\gamma}_B(h,h'),\\
\widehat{\gamma}_A(h,h') &:= {n}^{-1}
\sum_{k = 1}^{n}  \left(\bG^{(k)}(h) - \overline{\bG}(h) \right)\left(\bG^{(k)}(h') - \overline{\bG}(h') \right),\\
\widehat{\gamma}_B(h,h') &:= \frac{\mathbbm{1}\{\sigma(h) = \sigma(h')\}}{\widehat{N}^{(\sigma(h))}} 
\sum_{\iota \in I_{n,r}} Z_{\iota}^{(\sigma(h))} \left(
h(X_{\iota}) - U'_{n,N}(h)
\right)\left(
h'(X_{\iota}) - U'_{n,N}(h')
\right).
\end{split}
\end{align}
If $h = h'$, we write $\widehat{\gamma}_{*}(h)$ for $\widehat{\gamma}_{*}(h,h)$; the same convention applies to $\widehat{\gamma}_{A}$ and $\widehat{\gamma}_{B}$.

Define three $d$-dimensional random vectors as follows: for $1 \leq j \leq d$,
\begin{align*}
\underline{U}_{n,A,j}^{\#} := \bU_{n,A}^{\#}(h_j), \quad
\underline{U}_{n,B,j}^{\#} := \bU_{n,B}^{\#}(h_j), \quad
\underline{U}_{n,*,j}^{\#} := \bU_{n,*}^{\#}(h_j).
\end{align*}
The next Theorem establishes the validity of multiplier bootstrap for high-dimensional incomplete $U$-statistics.

\begin{theorem}\label{thrm:appr_U_full}
Assume the conditions \eqref{MT0}-\eqref{MT5}, and \ref{cond:MB} hold. There exists a constant $C$, depending only on $\ubar{\sigma}, r, q, c_0, C_0$, such that with probability at least $1 -C \log^{1/4}(2M)\cX_n$,
\begin{equation*}
\sup_{R \in \cR}
\left| 
\Pro_{\vert \cD'_{n}} \left( \underline{U}^{\#}_{n,*} \in R \right)
- \Pro(r Y_A + \alpha_n^{1/2} Y_B \in R)
\right| \; \leq \; C \log^{1/4}(2M) \cX_n,
\end{equation*}
where  $Y_A \sim N(0,\Gamma_A)$,
 $Y_B \sim N(0,\Gamma_B)$ and $Y_A,Y_B$ are independent, and $\cX_n$ is defined as follows:
 \begin{align*}
& \left(\frac{M^{2/q}B_n^{2-4/q} D_n^{4/q} \log^{3}(d)}{(N \wedge N_2)^{1-2/q}} \right)^{1/4} + \left( \frac{B_n^2 \log^{5}(dn)}{N\wedge N_2} \right)^{1/8} + \\
& \left( \frac{D_n^2 \log^{5}(d)}{n} \right)^{1/8} +
\left( \frac{D_n^2 \log^{3}(d)}{n^{1-2/q}} \right)^{1/4}+ \left( \frac{D_n^{8-8/q} \log^{11}(d)}{n^{3-4/q}} \right)^{1/14} +
\left( \frac{D_n^{3-2/q} \log^{3}(d)}{n^{1-1/q}} \right)^{2/7}.
\end{align*}
\end{theorem}
\begin{proof}
See Subsection \ref{proof:appr_U_full}.
\end{proof}

In the following proofs in this section, we assume without loss of generality that
\begin{equation}
\label{wlog_theta_zero}
\underline{\theta} = 0,
\end{equation}
since otherwise we can always center $\underline{h}$ first.

\subsection{Supporting calculations}
\label{sec:supp_lemmas_u_stat}

\begin{lemma}\label{lemma:Bernstein}
Let $L \geq N \geq 3$ be positive integers and $p_n := N/L$. Let $M \geq 1$ be an integer and $\log(M) \leq C_0 \log(N)$ for some absolute constant $C_0$. Let
$\{Z_{\ell}^{(m)} : \ell \in [L], m \in [M]\}$ be a collection of Bernoulli random variables  with success probability $p_n$ and $\widehat{N}^{(m)} := \sum_{\ell \in [L]} Z_{\ell}^{(m)}$ for $m \in [M]$. Then  there exists a constant $C$, depending only on $C_0$, such that
\[
\Pro\left( \bigcup_{m \in [M]}\left\{\left|N/\widehat{N}^{(m)} - 1\right| > C\sqrt{\frac{\log(N)}{N}} \right\}\right) \leq C N^{-1}.
\]
\end{lemma}
\begin{proof}
By Bernstein's inequality \cite[Lemma 2.2.9]{van1996weak} with $x = \sqrt{ C N \log(N)}$ for $C > 0$ and union bound,
\begin{align*}
&\Pro\left( \bigcup_{m \in [M]} \left\{ \left| \widehat{N}^{(m)}/N - 1\right| > \sqrt{\frac{C\log(N)}{N}}  \right\} \right) \leq 2 M \exp\left(
\frac{-C\log(N)/2}{1+3^{-1}\sqrt{C\log(N)/N})} \right)\\
 \leq & 2\exp\left(
\frac{-C\log(N)/2}{1+3^{-1}\sqrt{C\log(N)/N})}  + C_0 \log(N)\right).
\end{align*}
For $N \geq 3$, $\log(N)/N \leq 0.4$. Thus if we let $C_1$ to be large enough such that
$\frac{C_1}{2(1+ 3^{-1}\sqrt{0.4 C_1})} - C_0 \geq 1$, 
we have 
\[
\Pro\left( \bigcup_{m \in [M]} \left\{ \left| \widehat{N}^{(m)}/N - 1\right| > \sqrt{\frac{C_1\log(N)}{N}}  \right\} \right) \leq 2 N^{-1}.
\]
Since $x \mapsto \log(x)/x$ is a deceasing function on $[e,\infty)$, 
there exists some $C_2 > 0$, such that if $N \geq C_2$,
$\sqrt{C_1 \log(N)/N} \leq 0.5$. 
Further, since $|z^{-1} - 1| \leq 2|z-1|$ for $|z - 1|\leq 0.5$, we have
for $N \geq C_2$, 
\[
\Pro\left( \bigcup_{m \in [M]}\left\{\left|N/\widehat{N}^{(m)} - 1\right| > 2\sqrt{C_1\frac{\log(N)}{N}} \right\}\right) \leq 2 N^{-1}.
\]
Then the conclusion holds with $C = \max\{C_2, 2\sqrt{C_1}\}$.

\end{proof}

\begin{lemma}\label{lemma:Bernstein_second_moment}
Let $L \geq N \geq 3, M \geq 1$ be positive integers.
Let $p_n := N/L$, and 
$\{Z_{\ell}^{(m)} : \ell \in [L], m \in [M]\}$ be a collection of Bernoulli random variables  with success probability $p_n$. Then  there exists an absolute constant $C$,
\[
\Exp \left[ \max_{m \in [M]}\left|\sum_{\ell \in [L]}\left(
Z_{\ell}^{(m)} - p_n
\right)
\right|^2 \right] \leq C  \left(N \log(2M) + \log^2(2M)\right).
\]
\end{lemma}

\begin{proof}
By Hoffman-J{\o}rgensen inequality \cite[Proposition A.1.6]{van1996weak}  and
\cite[Lemma 8]{chernozhukov2015comparison},
\begin{align*}
\Exp\left[\max_{m \in [M]}
\left\vert  \sum_{\ell \in [L]}
(Z^{(m)}_{\ell} - p_n) \right\vert^2 \right]
&\lesssim \left(\Exp\left[\max_{m \in [M]}
\left\vert  \sum_{\ell \in [L]}
(Z^{(m)}_{\ell} - p_n) \right\vert \right] \right)^2 + 1 \\
&\lesssim N \log(2M) + \log^2(2M).
\end{align*}
\end{proof}

\begin{lemma} \label{lemma:first_order_without_hajek}
Denote $S_{n,j} := r n^{-1} \sum_{i=1}^{n} P^{r-1} h_j(X_i)$ for $1 \leq j \leq d$, and recall \eqref{wlog_theta_zero}. Assume that  \eqref{MT1} and \eqref{MT2} hold, and that 
\begin{equation}\label{assumption_for_geometric_series}
n^{-1} D_n^2 \log(d) \leq 1/6.
\end{equation}
Then there exists a constant $C$, depending only on $q, r$, such that
\begin{align*}
&C^{-1}\Exp\left[ 
\max_{1 \leq j \leq d} \left \vert \underline{U}_{n,j}  \right\vert \right] 
\leq    n^{-1/2}  D_n \log^{1/2}(d) + 
n^{-1+1/q} D_n \log(d)
+n^{-3/2+1/q}  D_n^{3-2/q} \log^{3/2}(d). \\
&C^{-1}\Exp\left[  \sqrt{n}
\max_{1 \leq j \leq d} \left \vert \underline{U}_{n,j} - S_{n,j} \right\vert \right] 
\leq 
n^{-1/2}  D_n \log(d) + n^{-1+1/q}  D_n^{3-2/q} \log^{3/2}(d).
\end{align*}
\end{lemma}

\begin{proof}
We only prove the second inequality, as the proof for the first is almost identical.
We apply Theorem \ref{thm:multi-level} and Hoeffding decomposition \eqref{Hoeffding decomp}
to the finite collection $\{h_j: 1 \leq j \leq d\} \subset \cH$ with envelopes $\{P^{r-\ell} H: 2 \leq \ell \leq r\}$:
\begin{align*}
\Exp\left[ 
\max_{1 \leq j \leq d} \left \vert \underline{U}_{n,j} - S_{n,j} \right\vert \right] \;
\lesssim \;&  
 n^{-1} \left( 
 \sup_{1 \leq j \leq d} \|P^{r-2} h_j\|_{P^{2}, 2} \log(d)
+  n^{-1/2 + 1/q} \|P^{r-2} H\|_{P^{2},q} \log^{3/2}(d)
\right) \\
&+ \sum_{\ell = 3}^{r}  n^{-\ell/2}
 \|P^{r-\ell} H \|_{P^\ell,2} \log^{\ell/2}(d).
\end{align*}
Then due to condition \eqref{MT2}, we have
\begin{align*}
&\Exp\left[ 
\max_{1 \leq j \leq d} \left \vert \underline{U}_{n,j} - S_{n,j} \right\vert \right] \;\\ 
\lesssim \;&  
n^{-1} D_n  \log(d)
+ n^{-3/2+1/q}  D_n^{3-2/q} \log^{3/2}(d)
+ \sum_{\ell = 3}^{r}  \left(n^{-1/2} D_n \log^{1/2}(d)\right)^{\ell}.
\end{align*}
Due to \eqref{assumption_for_geometric_series},  we have
\begin{align*}
\Exp\left[  \sqrt{n}
\max_{1 \leq j \leq d} \left \vert \underline{U}_{n,j} - S_{n,j} \right\vert \right] 
\lesssim &  
n^{-1/2}  D_n \log(d)  
+ n^{-1+1/q}  D_n^{3-2/q} \log^{3/2}(d)
+
n^{-1} D_n^3 \log^{3/2}(d), \\
\lesssim \;&  
n^{-1/2}  D_n \log(d) + n^{-1+1/q}  D_n^{3-2/q} \log^{3/2}(d).
\end{align*}
\end{proof}

\begin{lemma}\label{lemma:higher_order}
Assume that \eqref{assumption_for_geometric_series}, \eqref{MT2}, and \eqref{MT4} holds. Denote 
\begin{equation}
\label{varphi_n}
\varphi_n := n^{-1/2} D_n \log^{1/2}(d) 
+ n^{-1+1/q} D_n^2 \log(d) + n^{-3/2+1/q} D_n^{4-2/q} \log^{3/2}(d).
\end{equation}
Then there exists a constant $C$, depending only on $q, r, c_0$, such that for $s = 2,3,4$,
\begin{align*}
&\Exp\left[\max_{1 \leq j,k \leq d}\left\vert |I_{n,r}|^{-1} \sum_{\iota \in I_{n,r}} \widetilde{h}_j(X_{\iota}) \widetilde{h}_k(X_{\iota}) - P^r (\widetilde{h}_j \widetilde{h}_k ) \right\vert \right] \; \leq \; C \varphi_n, \\
&\Exp\left[\max_{1 \leq j \leq d}|I_{n,r}|^{-1} \sum_{\iota \in I_{n,r}}\vert \widetilde{h}_j(X_{\iota}) \vert^s \right] \; \leq \;
C B_n^{s-2}\left(1  +  \varphi_n \right),
\end{align*}
where $\widetilde{h}_j(\cdot) = h_j(\cdot)/\|h_j\|_{P^r,2}$ for $1 \leq j \leq d$.
\end{lemma}
\begin{proof}
We only prove the second claim, as the first can be established by the same argument. 
We apply the Theorem \ref{thm:multi-level} and Hoeffding decomposition \eqref{Hoeffding decomp}
to the finite collection $\{|\widetilde{h}_j|^s: 1 \leq j \leq d\}$ with envelopes $\{P^{r-\ell}\widetilde{H}^s: 1 \leq \ell \leq r\}$, where $\widetilde{H}(\cdot):= H(\cdot)/\inf_{1 \leq j \leq d} \|h_j\|_{P^r,2}$:
\begin{align*}
&\Exp\left[\max_{1 \leq j \leq d}|I_{n,r}|^{-1} \sum_{\iota \in I_{n,r}}\vert \widetilde{h}_j(X_{\iota}) \vert^s \right] \; \lesssim\; 
\sup_{1 \leq j \leq d} P^{r}|\widetilde{h}_j|^s +
\\
 & \sum_{\ell = 1}^{2}  n^{-\ell/2}
\left(
\sup_{1 \leq j \leq d} \|P^{r-\ell} |\widetilde{h}_j|^s\|_{P^{\ell},2}
\log^{\ell/2}(d)
+ n^{-1/2+1/q} \|P^{r-\ell} \widetilde{H}^s \|_{P^\ell,q} \log^{\ell/2+1/2}(d)
\right) +
 \\
& \sum_{\ell = 3}^{r}  n^{-\ell/2}
 \|P^{r-\ell} \widetilde{H}^s \|_{P^\ell,2} \log^{\ell/2}(d).
\end{align*}
Due to \eqref{MT2}, \eqref{MT4}, and \eqref{assumption_for_geometric_series}, we have
\begin{align*}
&\Exp\left[\max_{1 \leq j \leq d}|I_{n,r}|^{-1} \sum_{\iota \in I_{n,r}}\vert \widetilde{h}_j(X_{\iota}) \vert^s \right] \; \lesssim\; 
B_n^{s-2} + \sum_{\ell = 3}^{r}  n^{-\ell/2}
B_n^{s-2} D_n^{\ell+1} \log^{\ell/2}(d) 
\\
&+n^{-1/2} B_n^{s-2} D_{n} \log^{1/2}(d) + n^{-1+1/q} B_n^{s-2} D_n^{2} \log(d)\\
&+n^{-1} B_n^{s-2} D_{n}^{2} \log(d) + n^{-3/2+1/q} B_n^{s-2} D_n^{4-2/q} \log^{3/2}(d) \\
\; \lesssim\; & 
B_n^{s-2} + B_n^{s-2} n^{-1/2} D_{n} \log^{1/2}(d) + B_n^{s-2} n^{-1+1/q} D_n^{2} \log(d) \\
& + B_n^{s-2} n^{-3/2+1/q} D_n^{4-2/q} \log^{3/2}(d),
\end{align*}
which completes the proof.
\end{proof}

\subsection{Proof of Theorem \ref{thm:GAR_complete_hd_U}}
\label{sec:proof_complete_Ustat}
%

\begin{proof}[Proof of Theorem \ref{thm:GAR_complete_hd_U}]
Without loss of generality, we may assume the following hold: 
\begin{equation}
\label{convention:GAR_comp}
\frac{D_n^2 \log(d)}{n} \leq 1/6,
\quad \frac{D_n^2 \log^7(dn)}{n} \leq 1.
\end{equation}
since otherwise we can increase the constant $C$.

\noindent \underline{\textit{Step 1. }}
Denote $S_{n,j} := r n^{-1} \sum_{i=1}^{n} P^{r-1} h_j(X_i)$.
Due to condition \eqref{MT0}, \eqref{MT1} and \cite[Proposition 2.1]{chernozhukov2017}, we have
\begin{align*}
\sup_{R \in \cR}\left\vert
\Pro\left(\sqrt{n}(S_{n}- \underline{\theta}) \in R \right) - \Pro(rY_A \in R)
\right\vert \; \lesssim \; 
\left( \frac{D_n^2 \log^7(dn)}{n}\right)^{1/6} + 
\left(\frac{D_n^2 \log^3(dn)}{n^{1-2/q}}  \right)^{1/3}.
\end{align*}
\vspace{0.2cm}

\noindent \underline{\textit{Step 2. }} Fix any rectangle $R = [a,b] \in \cR$, where $a,b \in \bR^d$ and $a \leq b$. For any $ t >0$,
\begin{align*}
&\Pro\left(\sqrt{n} (\underline{U}_n - \underline{\theta}) \in R \right) \\
\leq \;& \Pro\left( \sqrt{n}\sup_{1 \leq j \leq d}\left\vert \underline{U}_{n,j} - S_{n,j} \right\vert \geq t
 \right) 
  + \Pro\left(-\sqrt{n} (S_{n} - \underline{\theta}) \leq  -a+t \cap
\sqrt{n} (S_{n} - \underline{\theta}) \leq b+t 
 \right) \\ 
 \leq_{(1)} \;& t^{-1} \left(n^{-1/2}  D_n \log(d) + n^{-1+1/q}  D_n^{3-2/q} \log^{3/2}(d)\right) + \Pro(-rY_A \leq -a +t \cap rY_A \leq b+t) \\
& + 
\left( \frac{D_n^2 \log^7(dn)}{n}\right)^{1/6} + 
\left(\frac{D_n^2 \log^3(dn)}{n^{1-2/q}}  \right)^{1/3},\\
 \leq_{(2)} \;& t^{-1} \left(n^{-1/2}  D_n \log(d) + n^{-1+1/q}  D_n^{3-2/q} \log^{3/2}(d)\right) + Ct\sqrt{\log(d)} + \Pro(rY_A \in R) \\
& + 
C\left( \frac{D_n^2 \log^7(dn)}{n}\right)^{1/6} + 
C\left(\frac{D_n^2 \log^3(dn)}{n^{1-2/q}}  \right)^{1/3},
\end{align*}
where we used Lemma \ref{lemma:first_order_without_hajek}, the Markov inequality  and results from Step 1 in (1),and the anti-concentration inequality \cite[Lemma A.1]{chernozhukov2017} in (2).

Now let $t = n^{-1/4} D_n^{1/2} \log^{1/4}(d) +
n^{-1/2+1/(2q)} D_n^{3/2-1/q} \log^{1/2}(d)$, and we have
\begin{align*}
&\Pro\left(\sqrt{n} (\underline{U}_n - \underline{\theta}) \in R \right) - \Pro\left(rY_A \in R \right) \\
\lesssim \;& 
\left(\frac{D_n^{2} \log^{3}(d)}{n}\right)^{1/4}
+
\left(\frac{D_n^{3-2/q} \log^{2}(d)}{n^{1-1/q}}\right)^{1/2}
 +  
\left( \frac{D_n^2 \log^7(dn)}{n}\right)^{1/6} + 
\left(\frac{D_n^2 \log^3(dn)}{n^{1-2/q}}  \right)^{1/3}\\
\lesssim \;& 
\left(\frac{D_n^{3-2/q} \log^{2}(d)}{n^{1-1/q}}\right)^{1/2}
 +  
\left( \frac{D_n^2 \log^7(dn)}{n}\right)^{1/6} + 
\left(\frac{D_n^2 \log^3(dn)}{n^{1-2/q}}  \right)^{1/3} := 
{\epsilon}_n,
\end{align*}
where we used \eqref{convention:GAR_comp} in the last inequality. By a similar argument, we can show that 
\[ \Pro\left(rY_A \in R \right) - 
\Pro\left(\sqrt{n} (\underline{U}_n - \underline{\theta}) \in R \right) \lesssim {\epsilon}_n,\]
which completes the proof.
\end{proof}

\subsection{Bounding the effect due to sampling}
Let $m \in [M]^d$ such that $m_j = \sigma(h_j)$ for $1 \leq j \leq d$.
For $1 \leq j \leq d$, let $\widetilde{h}_j(\cdot) = h_j(\cdot)/\|h_j\|_{P^r,2}$, and define
\begin{align} \label{zeta_n}
\zeta_{n,j} :=  \frac{1}{\sqrt{|I_{n,r}|}}\sum_{\iota \in I_{n,r}} \mathcal{T}_{\iota,j}, \quad
\text{ where }
\mathcal{T}_{\iota,j} := \frac{Z_{\iota}^{(m_j)} - p_n}{\sqrt{p_n(1-p_n)}} \widetilde{h}_j(X_{\iota}).
\end{align}

%
%
%

\begin{lemma}\label{lemma:effect_of_sampling}
Assume \eqref{MT2}, \eqref{MT3}, and \eqref{MT4} hold. Define
\begin{align*}
\eta_n &:= \left( \frac{B_n^2 \log^7(dn)}N \right)^{1/8}, \quad \omega_n := \left( \frac{n^{4r/q} \log^5(dn) B_n^{2-8/q} D_n^{8/q}}{N}
\right)^{1/4},
\end{align*}
and recall the definition of $\varphi_n$ in Lemma \ref{lemma:higher_order}. 
There exists a constant $C$, depending only on
$r, q, c_0$, such that
with probability at least $1 - C \varphi_n^{1/4} \log^{1/2}(d) - C\eta_n$,
\begin{align*}
\rho^{\cR}_{\vert X_1^n}( \zeta_n,  \Lambda_B^{-1/2} {Y}_{B})
&:= \sup_{R \in \cR}
\left\vert
\Pro_{\vert X_1^n}( \zeta_n \in R)
-
\Pro\left( \Lambda_B^{-1/2} {Y} \in R \right)
\right\vert \\
& \leq C \eta_n + C\omega_n + C \varphi_n^{1/4} \log^{1/2}(d),
\end{align*}
where $Y_B \; \sim \; N(0, \Gamma_B)$ and $\Lambda_B$ is a $d*d$ diagonal matrix such that $\Lambda_{B,jj} := P^r h_j^2$.
\end{lemma}

\begin{proof} 
The constants in this proof may depend on 
$r, q, c_0$. 
Consider conditionally independent (conditioned on $X_1^n$) $\bR^d$-valued random vectors $\{\widehat{Y}_{\iota}: \iota \in I_{n,r}\}$ such that
$$
\widehat{Y}_{\iota} \vert X_1^n
\;\sim \; N(0, \widehat{\Sigma}_{\iota}), \;\;
\text{ where }
\widehat{\Sigma}_{\iota, ij} =  \widetilde{h}_i(X_{\iota})\widetilde{h}_j(X_{\iota}) \mathbbm{1}\{m_i = m_j\}.
$$
Let $\widehat{Y} := |I_{n,r}|^{-1/2} \sum_{\iota \in I_{n,r}} \widehat{Y}_{\iota}$. 
Further, define
\begin{align*}
\rho^{\cR}_{\vert X_1^n}(\zeta_n, \widehat{Y})
&:= \sup_{R \in \cR}
\left\vert
\Pro_{\vert X_1^n}\left(\zeta_n \in R \right)
- \Pro_{\vert X_1^n}(\widehat{Y}\in R)
\right\vert, \\
\rho^{\cR}_{\vert X_1^n}(\widehat{Y},  \Lambda_B^{-1/2}Y_B)
&:= \sup_{R \in \cR}
\left\vert
\Pro_{\vert X_1^n}\left(\widehat{Y} \in R \right)
- \Pro( \Lambda_B^{-1/2}{Y}_B\in R)
\right\vert.
\end{align*}
By triangle inequality, it then suffices to show that each of the following events happens with probability at least $1-C \varphi_n^{1/4} \log^{1/2}(d) - C\eta_n$,
\begin{align}\label{aux_to_show}
\rho^{\cR}_{\vert X_1^n}(\zeta_n, \widehat{Y}) \leq C\eta_n +  C\omega_n,\quad
\rho^{\cR}_{\vert X_1^n}(\widehat{Y}, \Lambda_B^{-1/2}Y_B) \leq C\varphi_n^{1/4} \log^{1/2}(d),
\end{align}
on which we now focus. Without loss of generality, we assume
\begin{equation}
\label{aux_wlog}
\varphi_n^{1/4} \log^{1/2}(d) \leq 1/6, \quad \eta_n \leq c_1, \quad \omega_n \leq c_1 ,
\end{equation}
for some sufficiently small constant $c_1 \in (0,1/2)$ that is to be determined, since otherwise we could always increase $C$.

\vspace{0.3cm}

\noindent{\underline{\textit{Step 0}}}. 
Due to \eqref{aux_wlog}, Lemma \ref{lemma:higher_order} and Markov inequality, we have
\begin{equation}
\label{equ:Delta_B2}
\Pro\left(
\left \|
\Cov_{\vert X_1^n}(\widehat{Y})  - \Cov(\Lambda_B^{-1/2} Y_B)\right\|_{\infty} \geq 1/2
\right) \lesssim \varphi_n \leq \varphi_n^{1/4} \log^{1/2}(d).
\end{equation}
where for $1 \leq i,j \leq d$,
\[\Cov_{\vert X_1^n}(\widehat{Y})_{ij} :=  |I_{n,r}|^{-1} \sum_{\iota \in I_{n,r}}  \widetilde{h}_i(X_{\iota})
\widetilde{h}_j(X_{\iota}) \mathbbm{1}\{m_i = m_j\}.\]
By definition, for any $\iota \in I_{n,r}$ and $1 \leq j \leq d$, $\Exp_{\vert X_1^n}\left[ \widehat{Y}_{\iota,j}^2\right] 
= \widetilde{h}_j^2(X_{\iota})$, which implies  that $\|\widehat{Y}_{\iota,j}\|_{\psi_2\vert X_1^n} \lesssim |\widetilde{h}_j(X_{\iota})|$. Thus by 
  maximal inequality \cite[Lemma 2.2.2]{van1996weak}, there exists an absolute constant $C_0 \geq 1$ such that
\begin{align}
\label{aux_subG_max}
\|\max_{1 \leq j \leq d} \widehat{Y}_{\iota,j}\|_{\psi_1 \vert X_1^n} \leq
\|\max_{1 \leq j \leq d} \widehat{Y}_{\iota,j}\|_{\psi_2 \vert X_1^n} \leq C_0 \max_{1 \leq j \leq d}|\widetilde{h}_j(X_{\iota})|  \log^{1/2}(d).
\end{align}


\vspace{0.3cm}
\noindent{\underline{\textit{Step 1}}}.  The goal is to show that 
the first event in \eqref{aux_to_show}, $\rho^{\cR}_{\vert X_1^n}(\zeta_n, \widehat{Y}) \leq C\eta_n + C \omega_n$ , holds with probability at least $1-C\varphi_n^{1/4} \log^{1/2}(d) - C\eta_n$.

\vspace{0.2cm}
\underline{\textit{Step 1.1.}} Define 
\begin{align}\label{def:Lhat}
\widehat{L}_n :=
\max_{1 \leq j \leq d} |I_{n,r}|^{-1}
\sum_{\iota \in I_{n,r}} \Exp_{\vert X_1^n}\left[
\left\vert
\mathcal{T}_{\iota,j} \right\vert^3 \right].
\end{align}
Further,  $\widehat{M}_n(\phi) := \widehat{M}_{n,X}(\phi) + \widehat{M}_{n,Y}(\phi)$, where
\begin{equation}
\label{def:M_XY}
\begin{split}
&\widehat{M}_{n,X}(\phi) := |I_{n,r}|^{-1}\sum_{\iota \in I_{n,r}}
\Exp_{\vert X_1^n}\left[\max_{1 \leq j \leq d} \left\vert
\mathcal{T}_{\iota,j} \right\vert^3 ; \max_{1 \leq j \leq d}\left\vert
\mathcal{T}_{\iota,j}   \right\vert > 
\frac{\sqrt{|I_{n,r}|}}{4\phi \log d}
\right],\\
&\widehat{M}_{n,Y}(\phi) := |I_{n,r}|^{-1}\sum_{\iota \in I_{n,r}}
\Exp_{\vert X_1^n}\left[\max_{1 \leq j \leq d} |\widehat{Y}_{\iota,j}|^3; \max_{1 \leq j \leq d} |\widehat{Y}_{\iota,j}| > 
\frac{\sqrt{|I_{n,r}|}}{4\phi \log d}
\right].
\end{split}
\end{equation}
By Theorem 2.1 in \cite{chernozhukov2017},  there exist absolute constants $K_1$ and $K_2$  such that for
any real numbers $\overline{L}_n$ and $\overline{M}_n$, we have
\begin{align*}
\rho^{\cR}_{\vert X_1^n}(\zeta_n, \widehat{Y})
\leq K_1 \left( \left(
\frac{\overline{L}_n^2 \log^7(d)}{|I_{n,r}|}
\right)^{1/6} + \frac{\overline{M}_n}{\overline{L_n}} \right) \quad \text{ with } \quad
\phi_n := K_2 \left( \frac{\overline{L}_n^2 \log^4(d)}{|I_{n,r}|}\right)^{-1/6},
\end{align*}
on the event 
\begin{equation} \label{def:E_n}
\mathcal{E}_n := \left\{\widehat{L}_n \leq \overline{L}_n \right\} \bigcap \left\{\widehat{M}_n(\phi_n) \leq \overline{M}_n \right\} \bigcap \left\{\min_{1 \leq j \leq d} \Cov_{\vert X_1^n}(\widehat{Y})_{jj} \geq 2^{-1} \right\} \bigcap \left\{\phi_n \geq 1 \right\}.
\end{equation}

In Step 0, we have shown $\Pro\left(\min_{1 \leq j \leq d} \Cov_{\vert X_1^n}(\widehat{Y})_{jj} \geq 2^{-1}\right) \geq 1- C \varphi_n^{1/4} \log^{1/2}(d)$, since $  \Cov(\Lambda_B^{-1/2} Y_B)_{jj} = 1$ for $1 \leq j \leq d$.
In Step 1.2-1.4, we select proper $\overline{L}_n$ and $\overline{M}_n$ such that the first two events happen with probability at least $1 - C \eta_n$, and $\phi_n \geq 1$ with small enough $c_1$. In Step 1.5, we plug in these values.

\vspace{0.3cm}
\underline{\textit{Step 1.2: Select $\overline{L}_n$.}}\hspace{0.1cm}  
Since $p_n \leq 1/2$, $\Exp|Z_{\iota}^{(\ell)} - p_n|^3 \leq C p_n$ for $\ell \in [M]$, and thus
$$\widehat{L}_n \leq C p_n^{-1/2} \mathcal{T}_1, \;\text{ where } \mathcal{T}_1 := \max_{1 \leq j \leq d}\frac{1}{|I_{n,r}|}\sum_{\iota \in I_{n,r}} \left\vert
\widetilde{h}_j(X_{\iota})\right\vert^3
.$$
Due to \eqref{aux_wlog}, Lemma \ref{lemma:higher_order} and Markov inequality, we have
\begin{align*}
\Pro\left(\mathcal{T}_1 \leq  C B_n\, (\eta_n)^{-1}\right)) \geq 1- \eta_n.
\end{align*}
Thus there exists a constant $C_1$, depending on $q,r, c_0$, such that the following two conditions hold: (i).
\begin{align}\label{C_1_sel}
4^{-1}C_0^{-1}  K_2^{-1} C_1^{1/3} \geq 3r, 
\end{align}
where $C_0$ appears in \eqref{aux_subG_max}, and (ii). if  we let
\begin{equation}
\label{Ln_sel}
\overline{L}_n := 
C_1 p_n^{-1/2} B_n\, (\eta_n)^{-1}
\left(1 + \frac{n^{3r/q} \log^{11/2}(dn) B_n^{2-6/q}D_n^{6/q} \eta_n^{1-3/q}}{N}
\right),
\end{equation}
then $\Pro\left(L_n \leq \overline{L}_n \right) \geq 1 - \eta_n$.

With this selection of $\overline{L}_n$, and due to the definition of $\eta_n$ and \eqref{aux_wlog}, we have
\begin{align*}
\phi_n^{-1} &\leq 2 K_2^{-1} C_1^{1/3} \left(
\left( \frac{B_n^2 \eta_n^{-2}\log^4(dn)}{N}\right)^{1/6}
+
\left(\frac{n^{6r/q} \log^{15}(dn) B_n^{6-12/q} D_n^{12/q} \eta_n^{-6/q}}{N^3 }\right)^{1/6} \right) \\
& = 2 K_2^{-1} C_1^{1/3} \left(
\left( \frac{B_n^2 \log^3(dn)}{N}\right)^{1/8}
+
\left( \frac{n^{2r/q} \log^{5-7/(4q)}(dn) B_n^{2-9/(2q)} D_n^{4/q} }{N^{1-1/(4q)} }\right)^{1/2} \right) \\
& \leq 2 K_2^{-1} C_1^{1/3} \left(
\left( \frac{B_n^2 \log^3(dn)}{N}\right)^{1/8}
+
\left( \frac{n^{2r/q}   \log^{3/2}(dn) B_n^{1-4/q} D_n^{4/q} }{N^{1/2} }\right)^{1/2} \right) .
\end{align*}
Thus,  if we select  $c_1$ in \eqref{aux_wlog} to be small enough, 
i.e.,
\begin{align*}
4 c_1   K_2^{-1} C_1^{1/3} \leq 1,
\end{align*}
we have $\phi_n^{-1} \leq 1$.\\

\vspace{0.2cm}

\underline{\textit{Step 1.3: bounding $\widehat{M}_{n,X}(\phi_n)$.}}\hspace{0.1cm}  
Since $p_n \leq 1/2$, by its definition in \eqref{def:M_XY}, we have
$\widehat{M}_{n,X}(\phi_n) = 0$ on the event 
\begin{align*}
\left\{ \Upsilon_n :=
 \max_{\iota \in I_{n,r}} \max_{1 \leq j \leq d}
\left\vert \widetilde{h}_j(X_{\iota}) \right\vert
\; \leq \; \frac{\sqrt{N}}{4\phi_n \log(d)} \right\}.
\end{align*}

Observe that by the definition of $\bar{L}_n$ in \eqref{Ln_sel},
\begin{align*}
\phi_n^{-1} &\geq K_2^{-1} C_1^{1/3} \left(\frac{p_n^{-1} n^{6r/q} \log^{11}(dn)\log^4(d) B_n^{6-12/q} D_n^{12/q} \eta_n^{-6/q}}{N^2 |I_{n,r}|}\right)^{1/6}\\
&\geq
K_2^{-1} C_1^{1/3} \left(\frac{ n^{2r/q} \log^3(dn) \log^2(d) B_n^{2-4/q} D_n^{4/q} \eta_n^{-2/q}}{N}\right)^{1/2},
\end{align*}
which  implies that
\begin{align}\label{aux_cal_incomp_Ustat}
\frac{\sqrt{N}}{4\phi_n \log(d)}  \geq 
4^{-1} K_2^{-1} C_1^{1/3} n^{r/q} B_n^{1-2/q} D_n^{2/q} \eta_n^{-1/q} \log^{3/2}(dn)
\geq n^{r/q} B_n^{1-2/q} D_n^{2/q} \eta_n^{-1/q},
\end{align}
where in the last inequality, we used \eqref{C_1_sel}.
Due to \eqref{MT3} and \eqref{MT4}, and Markov inequality (the case for $q = \infty$ is obvious),
\begin{align}\label{Upsilon_bound}
\Pro\left(\Upsilon_n \geq  n^{r/q} B_n^{1-2/q} D_n^{2/q} \eta_n^{-1/q}\right) \leq \eta_n.
\end{align}
Thus  $\Pro\left( \widehat{M}_{n,X}(\phi_n) = 0\right) \geq 1 - \eta_n$.

\vspace{0.2cm}

\underline{\textit{Step 1.4: bounding $\widehat{M}_{n,Y}(\phi_n)$  and selecting $\overline{M}_n$.}} 
Due to the calculation in Step 0, 
\begin{align*}
\Pro_{\vert X_1^n}\left(
\max_{1 \leq j \leq d} \widehat{Y}_{\iota,j} \geq t
\right) \leq 2 \exp\left( \frac{t}{C_0\max_{1 \leq j \leq d}|\widetilde{h}_j(X_{\iota})| \log^{1/2}(d)}
\right),
\end{align*}
where $C_0$ is the absolute constant  in \eqref{aux_subG_max}.  
In Step 1.3 and 1.2, we have shown that 
\begin{align*}
&\Pro(\cE'_n)\geq 1-\eta_n,\;\; \text{ where }  \cE'_n := \left\{
\Upsilon_n \leq  n^{r/q} B_n^{1-2/q} D_n^{2/q} \eta_n^{-1/q}
\leq \frac{\sqrt{N}}{4\phi_n \log(d)}\right\}, \\
& \phi_n^{-1} \leq 1,\quad
\phi_n^{-1} \geq K_2^{-1} C_1^{1/3}  \left(\frac{ n^{2r/q} \log^3(dn)\log^2(d) B_n^{2-4/q} D_n^{4/q} \eta_n^{-2/q}}{N}\right)^{1/2}.
\end{align*}
Thus on the event $\cE'_n$, we have
\[
\Upsilon_n \phi_n \leq K_2 C_1^{-1/3}  \log^{-3/2}(dn) \log^{-1}(d)N^{1/2}.
\]
By \cite[Lemma C.1]{chernozhukov2017},
on the event $\cE'_n$,  for each $\iota \in I_{n,r}$,
\begin{align*}
&\Exp_{\vert X_1^n}\left[\max_{1 \leq j \leq d} |\widehat{Y}_{\iota,j}|^3; \max_{1 \leq j \leq d} |\widehat{Y}_{\iota,j}| > 
\frac{\sqrt{|I_{n,r}|}}{4\phi_n \log d}
\right] \\
\leq \;\;&
12 C_0^3 \left(\frac{\sqrt{|I_{n,r}|}}{4\phi_n \log d} +  \Upsilon_n \log^{1/2}(d) \right)^3 \exp\left(-
\frac{\sqrt{|I_{n,r}|}}{ 4C_0 \Upsilon_n \phi_n \log^{3/2} d} 
\right) \\
\leq \;\;& 12 C_0^3 n^{3r/2} \exp\left(-
\frac{{|I_{n,r}|^{1/2}}}{ 4 C_0  K_2 C_1^{-1/3}  \log^{-1}(dn) N^{1/2}} 
\right) \\
\leq \;\;& 12 C_0^3 n^{3r/2} \exp\left(- 4^{-1}C_0^{-1}  K_2^{-1} C_1^{1/3}\log(dn)
\right).
\end{align*}
Due to \eqref{C_1_sel}, we have that on the event $\cE'_n$,  for each $\iota \in I_{n,r}$,
\begin{align*}
\Exp_{\vert X_1^n}\left[\max_{1 \leq j \leq d} |\widehat{Y}_{\iota,j}|^3; \max_{1 \leq j \leq d} |\widehat{Y}_{\iota,j}| > 
\frac{\sqrt{|I_{n,r}|}}{4\phi_n \log d}
\right] 
\leq \;\;& 12 C_0^3 n^{-3r/2}.
\end{align*}
Thus if we select
\begin{align}
\label{Mn_sel}
\overline{M}_n :=  12C_0^3 n^{-3r/2},
\end{align}
then $\Pro(\widehat{M}_{n,Y}(\phi_n) \leq \overline{M}_n ) \geq  1 - C \eta_n$.
\vspace{0.2cm}

\underline{\textit{Step 1.5: plug in $\overline{L}_n$ and $\overline{M}_n$.}}\hspace{0.1cm}  
Recall the definition $\overline{L}_n$ and $\overline{M}_n$ in \eqref{Ln_sel} and \eqref{Mn_sel}. With these selections, we have shown  that $\Pro(\mathcal{E}_n) \geq 1 - C \varphi_n^{1/4} \log^{1/2}(d) - C\eta_n$, where
$\mathcal{E}_n$ is defined in \eqref{def:E_n}.
 Further, on the event $\mathcal{E}_n$, due to \eqref{aux_wlog}, we have
\begin{align*}
&\rho^{\cR}_{\vert X_1^n}( \zeta_n, \widehat{Y})
\lesssim \left(
\frac{\overline{L}_n^2 \log^7(dn)}{|I_{n,r}|}
\right)^{1/6} +
\frac{\overline{M}_n}{\overline{L_n}} \\
\lesssim &
\left(\frac{B_n^2 \log^7(dn)}{N}
\right)^{1/8} +
\left( \frac{n^{2r/q} 
\log^{6-7/(4q)}(dn) B_n^{2-9/(2q)} D_n^{4/q}}{N^{1-1/(4q)}}
\right)^{1/2} 
+ n^{-3r/2} p_n^{1/2} B_n^{-1} \eta_n,
\\
\lesssim &
\left(\frac{B_n^2 \log^7(dn)}{N}
\right)^{1/8} +
\left( \frac{n^{4r/q} \log^5(dn)B_n^{2-8/q} D_n^{8/q}}{N}
\right)^{1/4} = \eta_n + \omega_n,
\end{align*}
which completes the proof of Step 1.

\vspace{0.4cm}
\underline{\textit{Step 2.}}\hspace{0.1cm}  
The goal is to show that 
the second event in \eqref{aux_to_show}, $\rho^{\cR}_{\vert X_1^n}(\widehat{Y}, \Lambda_B^{-1/2}Y_B) \leq C   \varphi_n^{1/4} \log^{1/2}(d)$, holds with probability at least $1 - C \varphi_n^{1/4} \log^{1/2}(d)$.

By the Gaussian comparison inequality \citep[Lemma C.5] {chen2017randomized},
$$
\rho^{\cR}_{\vert X_1^n}(\widehat{Y},  \Lambda_B^{-1/2}Y_B) \lesssim \overline{\Delta}^{1/3} \log^{2/3}(d),
$$
on the event that $\{\|\Cov_{\vert X_1^n}(\widehat{Y})  - \Cov(\Lambda_B^{-1/2} Y_B) \|_{\infty} \leq \overline{\Delta}\}$.
By Lemma \ref{lemma:higher_order}, due to \eqref{aux_wlog}, and by Markov inequality,
\begin{align*}
\Pro\left(
\left \|\Cov_{\vert X_1^n}(\widehat{Y})  - \Cov(\Lambda_B^{-1/2} Y_B)\right\|_{\infty} \leq C \varphi_n^{3/4} \log^{-1/2}(d)
\right) \geq 1 - C \varphi_n^{1/4} \log^{1/2}(d).
\end{align*}
Thus if we set $\overline{\Delta} :=C \varphi_n^{3/4} \log^{-1/2}(d)$, then with probability at least $1- C \varphi_n^{1/4} \log^{1/2}(d)$,
\begin{align*}
\rho^{\cR}_{\vert X_1^n}(\widehat{Y}, \Lambda_B^{-1/2} Y_B) \lesssim 
 C \varphi_n^{1/4} \log^{1/2}(d).
\end{align*}
\end{proof}

\subsection{Proof of Theorem \ref{thm:GAR_hd_U}}
\label{proof:GAR_hd_U}

For $1 \leq j \leq d$, let $\widetilde{h}_j(\cdot) := h_j(\cdot)/\| h_j\|_{P^r,2}$, and let $m \in [M]^d$ such that $m_j = \sigma(h_j)$. Recall
 the definition of $\zeta_n$ in \eqref{zeta_n}. Let $Y_A$ and $Y_B$ be two \textit{independent} $d$-dimensional Gaussian vectors such that
\[
Y_A \; \sim \; N(0, \Gamma_A), \quad
Y_B \; \sim \; N(0, \Gamma_B).
\]
Define 
\[
Y:= r Y_A + \sqrt{\alpha_n} Y_B, \qquad
\widetilde{Y} :=  \Lambda_{*}^{-1/2}(r Y_A + \sqrt{\alpha_n} Y_B),
\] 
where $ \Lambda_{*}$ is a $d\times d$ \textit{diagonal} matrix such that
$\Lambda_{*,jj} = \Exp[Y_j^2]$. For two $d$-dimensional random vectors $Z$ and $Z'$, define $\rho(Z,Z') := \sup_{R \in \cR}|\bP(Z \in R) - \bP(Z' \in R)|$.

\begin{proof}
Without loss of generality, we assume \eqref{wlog_theta_zero}, and
\begin{equation}
\label{wlog_c_incomplete}
\begin{split}
& \varpi_n^{(1)}:=
\left( \frac{D_n^2 \log^7(dn)}{n}\right)^{1/8} + 
\left(\frac{D_n^2 \log^3(dn)}{n^{1-2/q}}  \right)^{1/4} +
\left(\frac{D_n^{3-2/q} \log^{2}(d)}{n^{1-1/q}}\right)^{1/2} \leq 1/2, \\
& \varpi_n^{(2)}:=
\left( \frac{B_n^2 \log^7(dn)}N \right)^{1/8} +  \left( \frac{n^{4r/q} \log^5(dn) B_n^{2-8/q} D_n^{8/q}}{N}
\right)^{1/4} \leq 1/2.
\end{split}
\end{equation}

We recall the definitions of $\varphi_n$, $\eta_n$, and $\omega_n$ in Lemma \ref{lemma:effect_of_sampling}. Due to \eqref{wlog_c_incomplete}, 
\begin{align} \label{aux_order_cal}
\varphi_n^{1/4} \log^{1/2}(d) \lesssim \varpi_n^{(1)},  \qquad
\eta_n + \omega_n = \varpi_n^{(2)}.
\end{align}


Observe that for $1 \leq j \leq d$,
\begin{align*}
\underline{U}_{n,N,j}' &= \frac{N}{\widehat{N}^{(m_j)}}\left( 
\frac{1}{N}\sum_{\iota \in I_{n,r}} (Z_{\iota}^{(m_j)} - p_n) \Lambda_{B,jj}^{1/2}\widetilde{h}_j(X_{\iota})
+
\frac{1}{|I_{n,r}|}\sum_{\iota \in I_{n,r}} h_j(X_{\iota})
\right) \\
&= \frac{N}{\widehat{N}^{(m_j)}}\left( \sqrt{\frac{1-p_n}{N}} \left(\Lambda_B^{1/2}\zeta_n \right)_j + \underline{U}_{n,j}\right)
:= \frac{N}{\widehat{N}^{(m_j)}} \Phi_{n,j}.
\end{align*}
where  $ \Lambda_{B}$ is a $d\times d$ \textit{diagonal} matrix such that
$\Lambda_{B,jj} = P^r h_j^2$.

\vspace{0.2cm}
\underline{Step 1:} the goal is to show that
$$
\rho\left(\sqrt{n} \Phi_n, \;\; rY_A +\alpha_n^{1/2}Y_B \right)
\lesssim \varpi_n^{(1)} + \varpi_n^{(2)}.
$$
For any rectangle $R \in \cR$, observe that
\begin{align*}
&\Pro(\sqrt{n}\left(\underline{U}_n + \sqrt{1-p_n} N^{-1/2} \Lambda_B^{1/2} \zeta_n \right) \in R)\\
= \;\; &\Exp\left[
\Pro_{\vert X_1^n}\left(
 \zeta_n \in \left( \frac{1}{\sqrt{\alpha_n (1-p_n)}} \Lambda_B^{-1/2} R-
\sqrt{\frac{N}{1-p_n}}\Lambda_B^{-1/2} \underline{U}_n
\right)
\right)
\right].
\end{align*}
By Lemma \ref{lemma:effect_of_sampling} and \eqref{aux_order_cal}, we have
\begin{align*}
&\Pro(\sqrt{n}\left( \underline{U}_n + \sqrt{1-p_n} N^{-1/2}\Lambda_B^{1/2}\zeta_n \right) \in R)\\
\leq  \;\; &\Exp\left[
\Pro_{\vert X_1^n}\left(
\Lambda_B^{-1/2} Y_B \in \left( \frac{1}{\sqrt{\alpha_n (1-p_n)}} \Lambda_B^{-1/2} R-
\sqrt{\frac{N}{1-p_n}}\Lambda_B^{-1/2} \underline{U}_n
\right)
\right)
\right] + C \varpi_n^{(1)} +  C \varpi_n^{(2)}\\
= \;\; &
\Pro\left(\sqrt{n} \underline{U}_n \in 
\left[   R- \sqrt{\alpha_n(1-p_n)}Y_B
\right] \right)+ C \varpi_n^{(1)} +  C \varpi_n^{(2)},
\end{align*}
where we recall that $Y_B$ is independent of all other random variables.
Further, by Theorem \ref{thm:GAR_complete_hd_U},
\begin{align*}
&\Pro(\sqrt{n}\left( \underline{U}_n + \sqrt{1-p_n} N^{-1/2}\Lambda_B^{1/2}\zeta_n \right) \in R)\\
\leq \;\; &
\Exp\left[
\Pro_{\vert Y_B}\left(\sqrt{n} \underline{U}_n \in 
\left[  R- \sqrt{\alpha_n(1-p_n)} Y_B
\right] \right)
\right]+  C \varpi_n^{(1)} +  C \varpi_n^{(2)}, \\
\leq \;\; &
\Exp\left[
\Pro_{\vert Y_B}\left(r Y_A \in 
\left[   R- \sqrt{\alpha_n(1-p_n)} Y_B
\right] \right)
\right]+ C \varpi_n^{(1)} +  C \varpi_n^{(2)}, \\
= \;\; &
\Pro\left( \Lambda_{*}^{-1/2}(r Y_A + \sqrt{\alpha_n(1-p_n)} Y_B) \in \Lambda_{*}^{-1/2} R
\right) + C \varpi_n^{(1)} +  C \varpi_n^{(2)},
\end{align*}
By definition, 
$\Exp[\widetilde{Y}_j^2] = 1$ for each $1 \leq j \leq d$. Then by the Gaussian comparison inequality \citep[Lemma C.5] {chen2017randomized}, and due to \eqref{MT4} and $\alpha_n p_n = n/|I_{n,r}| \lesssim n^{-(r-1)}$, 
\begin{align*}
&\sup_{R' \in \cR} \left| 
\Pro\left(\Lambda_{*}^{-1/2}(r Y_A + \sqrt{\alpha_n(1-p_n)} Y_B) \in R' \right) -
\Pro\left(\widetilde{Y} \in R' \right) 
\right|  \\
& \lesssim \left(\frac{\alpha_n p_n D_n^{2(r-1)}}{r^2 \ubar{\sigma}^2} \right)^{1/3} \log^{2/3}(d)
\lesssim \left( \frac{D_n^2 \log^{2/(r-1)}(d)}{n} \right)^{(r-1)/3} \lesssim \varpi_n^{(1)}.
\end{align*}
As a result,
\begin{align*}
\Pro(\sqrt{n} \Phi_n \in R) &\; \leq \; 
\Pro\left( r Y_A + \sqrt{\alpha_n} Y_B \in R
\right) + C \varpi_n^{(1)} + C\varpi_n^{(2)}.
\end{align*}
Similarly, we can show $\Pro(\sqrt{n}\Phi_n \in R)
\;\geq \; \Pro\left( r Y_A + \sqrt{\alpha_n} Y_B \in R
\right) - C \varpi_n^{(1)} - C\varpi_n^{(2)}$. Thus the proof of Step 1 is complete.

\vspace{0.3cm}
\underline{Step 2:} we show that with probability at least $1 -  C \varpi_n^{(1)} - C \varpi_n^{(2)}$,
 $$\max_{1 \leq j \leq d} \left| \left(\frac{N}{\widehat{N}^{(m_j)}}-1 \right)\sqrt{N} \Phi_{n,j} \right| \leq C \nu_n, \;\;\; \text{ where }
 \nu_n := \left(\frac{D_n^2 \log^{3}(dn)}{n} \right)^{1/2}
+ \left( \frac{B_n^{2} \log^{3}(dn)}{N}\right)^{1/2}.
 $$
Due to \eqref{MT1} and \eqref{MT4}, $\Exp[Y_j^2] = r^2 \gamma_A(h_j) + \alpha_n \gamma_B(h_j) \lesssim D_n^2 + \alpha_n B_n^2$. 
Since $Y$ is a multivariate Gaussian, 
$\max_{1\leq j \leq d} \|Y_j\|_{\psi_2} \leq \sqrt{D_n^2 + \alpha_n B_n^2}$. Then by the maximal inequality \cite[Lemma 2.2.2]{van1996weak}, $\|\max_{1\leq j \leq d} |Y_j|\|_{\psi_2} \leq C \sqrt{(D_n^2+\alpha_n B_n^2) \log(d)}$,
which further implies that 
$$\Pro\left(\max_{1\leq j \leq d} |Y_j| \geq C \sqrt{(D_n^2+\alpha_n B_n^2)  \log(d)\log(n)} \right) \leq 2n^{-1}.$$
Since $n^{-1} \lesssim \varpi_n^{(1)}$,  and from the result in Step 1, we have
$$\Pro\left(\|\sqrt{n} \Phi_n\|_{\infty} \geq C \sqrt{(D_n^2 + \alpha_n B_n^2)  \log(d)\log(n)} \right) \leq C  \varpi_n^{(1)} + C  \varpi_n^{(2)}.$$
Finally, due to Lemma \ref{lemma:Bernstein} and \ref{cond:MB}, we have with probability at least $1 - C \varpi_n^{(1)} - C \varpi_n^{(2)}$,
\begin{align*}
 \max_{1 \leq j \leq d} \left|\left(\frac{N}{\widehat{N}^{(m_j)}}-1 \right)\sqrt{N} \Phi_{n,j} \right| &\leq  C
\sqrt{(D_n^2 + \alpha_n B_n^2) \log(d)\log^2(n) N^{-1} \alpha_n^{-1}} \\
&\leq C \left(\frac{D_n^2 \log^{3}(dn)}{n} \right)^{1/2}
+ C \left( \frac{B_n^{2} \log^{3}(dn)}{N}\right)^{1/2}.
 \end{align*} 
\vspace{0.3cm}

\underline{Step 3: final step.}  Recall that  $\sqrt{N} U_{n,N,j}' = \sqrt{N} \Phi_{n,j} + (N/\widehat{N}^{(m_j)}-1)\sqrt{N} \Phi_{n,j}$
for $1 \leq j \leq d$, 
and $\nu_n$ is defined in Step 2. For any rectangle $R = [a,b]$ with $a \leq b$, by Step 2,
\begin{align*}
&\Pro\left( \sqrt{N} U_{n,N}' \in R\right)\\
\leq &\Pro\left( \sqrt{N} U_{n,N}' \in R \;\bigcap\; \left\{ \max_{1 \leq j \leq d} \left| (N/\widehat{N}^{(m_j)}-1)\sqrt{N} \Phi_{n,j} \right| \leq C \nu_n \right\}\right)
+ C \varpi_n^{(1)} + C \varpi_n^{(2)} \\
\leq &\Pro\left( -\sqrt{N} \Phi_n \leq - a + C\nu_n \;\cap\; \sqrt{N} \Phi_n \leq b + C\nu_n\right)
+ C \varpi_n^{(1)} + C \varpi_n^{(2)}.
\end{align*}
Then by the result in Step 1, we have
\begin{align*}
\Pro\left( \sqrt{N} U_{n,N}' \in R\right)
&\leq \Pro\left( -\alpha_n^{-1/2} Y \leq - a + C\nu_n \;\cap\; \alpha_n^{-1/2}Y  \leq b + C\nu_n\right)
+ C \varpi_n^{(1)} + C \varpi_n^{(2)}.
\end{align*}
Observe that $\Exp[ (\alpha_n^{-1/2} {Y}_j)^2] \geq \Exp[\gamma_B(h_j)]
\geq  \ubar{\sigma}^2$ for $1 \leq j \leq d$, and thus by anti-concentration inequality \cite[Lemma A.1]{chernozhukov2017},
\begin{align*}
&\Pro\left( \sqrt{N} U_{n,N}' \in R\right)
\leq \Pro\left( -\alpha_n^{-1/2} {Y} \leq - {a}  \;\cap\; \alpha_n^{-1/2}Y  \leq {b} \right) + 
C\nu_n \log^{1/2}(d) 
+ C \varpi_n^{(1)} + C \varpi_n^{(2)}
\\
&=  \Pro\left( \alpha_n^{-1/2} Y \in R\right) + 
\sqrt{\frac{\log^{4}(dn) D_n^2}{n}} +
\sqrt{\frac{\log^{4}(dn) B_n^2}{N}} +
C \varpi_n^{(1)} + C \varpi_n^{(2)}\\
&\leq   \Pro\left( \alpha_n^{-1/2} Y \in R\right)  + C \varpi_n^{(1)} + C \varpi_n^{(2)},
\end{align*}
where the last inequality is due to \eqref{wlog_c_incomplete}. 
Similarly, we can show
$$\Pro\left( \sqrt{N} U_{n,N}' \in R\right) \geq  \Pro\left( \alpha_n^{-1/2} Y \in R\right)  - C \varpi_n^{(1)} - C \varpi_n^{(2)},
$$ 
and thus
$
\rho(\sqrt{N} U_{n,N}', \;\;\alpha_n^{-1/2} Y) \lesssim \varpi_n^{(1)} +  \varpi_n^{(2)},
$
which completes the proof.
\end{proof}

\subsection{Proof of Theorem \ref{thrm:appr_U_full}}
We will deal with the bootstrap for $\Gamma_A$ and $\Gamma_B$ separately.

\subsubsection{Bootstrap for $\Gamma_B$}
Recall the definition of $\widehat{\gamma}_B$ in \eqref{def:est_Covs}.

\begin{lemma} \label{lemma:Delta_B_bound}
Assume that \eqref{assumption_for_geometric_series}, \eqref{MT1}- \eqref{MT4}, and \ref{cond:MB} holds. Define
\begin{align*}
\widehat{\Delta}_B &:= \max_{1 \leq j,k \leq d} \left| \widehat{\gamma}_B(h_j,h_k) - \gamma_B(h_j, h_k) \right|, \\
\chi_{n, B} := &\log^{1/4}(2M) \left( \left(\frac{M^{2/q} B_n^{2-4/q} D_n^{4/q} \log^{3}(d)}{N^{1-2/q}} \right)^{1/4} + \left( \frac{B_n^2 \log^{5}(dn)}{N} \right)^{1/8} \right.\\
& \left. + \left(\frac{D_n^2 \log^5(d)}{n} \right)^{1/8}
+ \left(\frac{D_n^2 \log^3(d)}{n^{1-1/q}} \right)^{1/4}
+ \left(\frac{D_n^{8-4/q} \log^7(d)}{n^{3-2/q}} \right)^{1/8} \right).
\end{align*}
Then there exists a constant $C$, that only depends on $q,r,c_0, C_0$, such that
\[
\Pro\left( \widehat{\Delta}_B \leq C B_n^2 D_n^{-2} \log^{-2}(d) \chi_{n,B}^3 \right) \geq 1 - C \chi_{n,B}.
\]
\end{lemma}

\begin{proof}
Without loss of generality, we assume $\chi_{n,B} \leq 1/2$, 
since we can always let $C \geq 2$. 
For $1 \leq j,k \leq d$, and $\sigma(h_j) = \sigma(h_k) = m \in [M]$,
\begin{align*}
&\left|
\widehat{\gamma}_B(h_j,h_k) - \gamma_B(h_j, h_k)
\right| \\
= & \left| 
\frac{1}{\widehat{N}^{(m)}} 
\sum_{\iota \in I_{n,r}} Z_{\iota}^{(m)} \left(
h_j(X_{\iota}) - U'_{n,N}(h_j)
\right)\left(
h_k(X_{\iota}) - U'_{n,N}(h_k)
\right) - P^r(h_j h_k)
\right| \\
\leq & \left(\max_{\ell \in [M]} N/\widehat{N}^{(\ell)} \right) \left( 
\widehat{\Delta}_{B,1} + 
\widehat{\Delta}_{B,2}\right)+\widehat{\Delta}_{B,3} + \left(\max_{\ell \in [M]} N/\widehat{N}^{(\ell)} \right)^2 \widehat{\Delta}_{B,4}^2,  \\
\leq & \left(\max_{\ell \in [M]} N/\widehat{N}^{(\ell)} \right) \left( 
\widehat{\Delta}_{B,1} + 
\widehat{\Delta}_{B,2}\right)
+\widehat{\Delta}_{B,3} + 2\left(\max_{\ell \in [M]} N/\widehat{N}^{(\ell)} \right)^2 \left( \widehat{\Delta}_{B,5}^2 + \widehat{\Delta}_{B,6}^2 \right).
\end{align*}
where we define
\begin{align} \label{def:hat_Delta_B}
\begin{split}
\widehat{\Delta}_{B,1} &:= \max_{1 \leq j,k \leq d}\max_{\ell \in [M]} \left|{N}^{-1} \sum_{\iota \in I_{n,r}} (Z_{\iota}^{(\ell)}-p_n){h}_j(X_{\iota}) 
{h}_k(X_{\iota})  \right|, \\
\widehat{\Delta}_{B,2} &:= \max_{1 \leq j,k \leq d} \left\vert|I_{n,r}|^{-1} \sum_{\iota \in I_{n,r}} {h}_j(X_{\iota}) {h}_k(X_{\iota}) - P^r ({h}_j {h}_k )\right\vert,\\
\widehat{\Delta}_{B,3} &:= \left(\max_{\ell \in [M]} \left|N/\widehat{N}^{(\ell)} -1\right| \right) \max_{1 \leq j \leq d}\left|P^r h_j^2\right|, \;\;
\widehat{\Delta}_{B,4} := \max_{1 \leq j \leq d} \max_{\ell \in [M]} \left|N^{-1}\sum_{\iota \in I_{n,r} } Z_{\iota}^{(\ell)}{h}_j(X_{\iota}) \right|\\
\widehat{\Delta}_{B,5} &:= \max_{1 \leq j \leq d} \max_{\ell \in [M]}
\left|N^{-1}\sum_{\iota \in I_{n,r} } (Z_{\iota}^{(\ell)}-p_n) {h}_j(X_{\iota}) \right|,\\
\widehat{\Delta}_{B,6} &:= 
\max_{1 \leq j \leq d} \left ||I_{n,r}|^{-1}\sum_{\iota \in I_{n,r} }  {h}_j(X_{\iota}) \right |. 
\end{split}
\end{align}

%
%
%

Then by Markov inequality, \eqref{MT4}, and due to Lemma \ref{lemma: Delta_B1_B5_bound} (ahead), Lemma \ref{lemma:higher_order}, and Lemma \ref{lemma:Bernstein}, for $i = 1,2,3$,
\begin{align*}
\Pro\left( \widehat{\Delta}_{B,i} \geq C B_n^2 D_n^{-2} \log^{-2}(d) \chi_{B,n}^3\right) \leq C \chi_{B,n}.
\end{align*}
Further, by Lemma \ref{lemma: Delta_B1_B5_bound} (ahead), and Lemma \ref{lemma:first_order_without_hajek}, since $B_n \geq D_n$, for $i = 5,6$,
\begin{align*}
\Pro\left( \widehat{\Delta}_{B,i}^2 \geq C B_n^2 D_n^{-2} \log^{-2}(d) \chi_{B,n}^3\right) \leq C \chi_{B,n}^{5/2} \leq C \chi_{B,n}.
\end{align*}
Then the proof is complete due to \ref{cond:MB} and Lemma \ref{lemma:Bernstein}.
\end{proof}
\begin{lemma} \label{lemma: Delta_B1_B5_bound}
Recall the definition of $\widehat{\Delta}_{B,1}$ and $\widehat{\Delta}_{B,5}$ in \eqref{def:hat_Delta_B}. Assume that \eqref{assumption_for_geometric_series}, \eqref{MT2}, \eqref{MT3}, and \eqref{MT4} hold. Then there exists a constant $C$, depending only on $q, r, c_0$, such that 
\begin{align*}
&C^{-1} \log^{-1}(2M)\Exp\left[ \widehat{\Delta}_{B,1} \right] 
 \; \leq \; \\
 & B_n^2 D_n^{-2} \left(
N^{-1+2/q} M^{2/q} B_n^{2-4/q} D_n^{4/q} \log(d)  +
N^{-1/2} B_n \log^{1/2}(d) + 
N^{-1/2} B_n \varphi_n^{1/2} \log^{1/2}(d)
\right), \\ 
&C^{-1}\log^{-1}(2M)\Exp\left[ \widehat{\Delta}_{B,5} \right] 
 \; \leq \; \\
 &B_n D_n^{-1} \left(
N^{-1+1/q} M^{1/q} B_n^{1-2/q} D_n^{2/q}\log(d)  +
N^{-1/2}  \log^{1/2}(d) + 
N^{-1/2} \varphi_n^{1/2} \log^{1/2}(d)
\right), 
\end{align*}
where $\varphi_n$ is defined in \eqref{varphi_n}. 
\end{lemma}
\begin{proof}
We first focus on $\widehat{\Delta}_{B,1}$, and 
 apply Theorem \ref{thm:lmax_sampling_proc} to the finite collection 
$\{h_j h_k: 1 \leq j,k \leq d\}$ with envelope $H^2$. Since this is a finite collection of functions with cardinality $d^2$, we have
\[
\bar{J}(\tau) \leq C\tau \log^{1/2}(d), \text{ for } \tau > 0.
\]
Thus by Theorem \ref{thm:lmax_sampling_proc}, and Lemma \ref{lemma:higher_order} with $s = 4$, 
\begin{align*}
\Exp\left[ \sqrt{N}\widehat{\Delta}_{B,1} \right]
\lesssim &\sup_{1 \leq j \leq d} \|h_j^2\|_{P^r,2} \log^{1/2}(d) \log^{1/2}(2M)+ 
N^{-1/2+2/q} M^{2/q}\|H^2\|_{P^r,q/2} \log(d)  \log(2M)\\
& + B_n^{3} D_n^{-2} (1+\varphi_n^{1/2}) \log^{1/2}(d)\log^{1/2}(2M).
\end{align*}
Then due to \eqref{MT2} and \eqref{MT3},
\begin{align*}
\Exp\left[ \sqrt{N}\widehat{\Delta}_{B,1} \right]
\lesssim  \log(2M) &\left( B_n^3 D_n^{-2} \log^{1/2}(d) + 
N^{-1/2+2/q} M^{2/q} B_n^{4-4/q} D_n^{4/q-2} \log(d) \right. \\
& \;\; \left. + B_n^3 D_n^{-2} (1+\varphi_n^{1/2}) \log^{1/2}(d) \right),
\end{align*}
which completes the proof of the first result.

Now for $\widehat{\Delta}_{B,5}$, we apply Theorem \ref{thm:lmax_sampling_proc} to the finite collection 
$\{h_j : 1 \leq j \leq d\}$ with envelope $H$. By  Theorem \ref{thm:lmax_sampling_proc} and Lemma \ref{lemma:higher_order} with $s = 2$, 
\begin{align*}
\Exp\left[ \sqrt{N}\widehat{\Delta}_{B,5} \right]
\lesssim \log(2M) &\left( \sup_{1 \leq j \leq d} \|h_j\|_{P^r,2} \log^{1/2}(d) + 
N^{-1/2+1/q} M^{1/q}\|H\|_{P^r,q} \log(d)  \right.\\
&\;\; \left. + B_n D_n^{-1} (1+\varphi_n^{1/2}) \log^{1/2}(d)\right).
\end{align*}
Then due to \eqref{MT3}, and \eqref{MT4},
\begin{align*}
\Exp\left[ \sqrt{N}\widehat{\Delta}_{B,5} \right]
\lesssim \log(2M) &\left( B_n D_n^{-1} \log^{1/2}(d) + 
N^{-1/2+1/q} M^{1/q} B_n^{2-2/q}D_n^{2/q-1} \log(d)  \right.\\
&\;\; \left. + B_n D_n^{-1} (1+\varphi_n^{1/2}) \log^{1/2}(d)\right),
\end{align*}
which completes the proof of the second result.
\end{proof}

\subsubsection{Bootstrap for $\Gamma_A$} Recall the definition of $\widehat{\gamma}_A$ in \eqref{def:est_Covs}.

\begin{lemma} \label{lemma:Delta_A_bound}
Assume that \eqref{assumption_for_geometric_series}, \eqref{MT0}- \eqref{MT3}, \eqref{MT5}, and \ref{cond:MB} holds. Define
\begin{align*}
\widehat{\Delta}_A &:= \max_{1 \leq i,j \leq d} \left| \widehat{\gamma}_A(h_i,h_j) - \gamma_A(h_i, h_j) \right|/\sqrt{\gamma_A(h_i)\gamma_A(h_j)}, \\
\chi_{n, A} &:= \left(\frac{\log(2M) B_n^{2} \log^{5}(d)}{N_2} \right)^{1/7} 
+ \left(\frac{\log(2M)M^{1/q}B_n^{2-2/q}D_n^{2/q-1} \log^{5/2}(d)}{N_2^{1-1/q}} \right)^{2/7}
+ \\
&\left( \frac{D_n^2 \log^{5}(d)}{n} \right)^{1/8} +
\left( \frac{D_n^2 \log^{3}(d)}{n^{1-2/q}} \right)^{1/4}
+ \left( \frac{D_n^{8-8/q} \log^{11}(d)}{n^{3-4/q}} \right)^{1/14} +
\left( \frac{D_n^{3-2/q} \log^{3}(d)}{n^{1-1/q}} \right)^{2/7}.
\end{align*}
Then there exists a constant $C$, that only depends on $\ubar{\sigma}^2, q,r,c_0, C_0$, such that
\[
\Pro\left( \widehat{\Delta}_A \leq C \log^{-2}(d) \chi_{n,A}^3 \right) \geq 1 - C \chi_{n,A}.
\]
\end{lemma}
\begin{proof}
Without loss of generality, we assume $\chi_{n,A} \leq 1/2$, 
By the same argument as in the proof of \cite[Theorem 4.2] {chen2017randomized},
\begin{align*}
\widehat{\Delta}_A \lesssim \widehat{\Delta}_{A,1}^{1/2} + \widehat{\Delta}_{A,1} + \widehat{\Delta}_{A,2} + \widehat{\Delta}_{A,3}^2,
\end{align*}
where we define 
\begin{align}\label{def:hat_Delta_A}
\begin{split}
&\widehat{\Delta}_{A,1} :=   \max_{1 \leq j \leq d} \frac{1}{n}\sum_{k = 1}^{n} \left(
\bG^{(k)}(h_j) - P^{r-1}h_j(X_{k})
\right)^2,\\
&\widehat{\Delta}_{A,2} := \max_{1 \leq i,j \leq d} \left\vert
\frac{1}{\sqrt{\Gamma_{A,ii} \Gamma_{A,jj}} n} \sum_{k = 1}^{n} \left(P^{r-1}h_i(X_{k}) P^{r-1}h_j(X_{k})- \Gamma_{A,ij}\right)
\right\vert, \;\;\\
&\widehat{\Delta}_{A,3} := \max_{1 \leq j \leq d} \left\vert
\frac{1}{\sqrt{\Gamma_{A,jj}}  n} \sum_{k = 1}^{n} P^{r-1}h_j(X_{k})
\right\vert.
\end{split}
\end{align}

By Lemma \ref{lemma:Delta_A2_A3_bound} (ahead), 
\begin{align*}
\Pro\left( \widehat{\Delta}_{A,2} \geq C \log^{-2}(d) \chi_{n,A}^3 \right) \leq C \chi_{n,A}, \;\;
\Pro\left( \widehat{\Delta}_{A,3}^2 \geq C \log^{-2}(d) \chi_{n,A}^3 \right) \leq C \chi_{n,A}^{5/2} \leq C  \chi_{n,A}.
\end{align*}

Now we focus on $\widehat{\Delta}_{A,1}$. By Lemma \ref{lemma:Bernstein}, and due to \ref{cond:MB} and the union bound, since $N_2 \geq n$,
\begin{align*}
\Pro\left(\mathcal{E}'  \right) \geq 1 - C_1 N_2^{-1}, \text{ where } 
\mathcal{E}' := \bigcap_{k  \in [n], \ell \in [M]}\left\{ |N_2/\widehat{N}_2^{(k, \ell)} - 1| \leq C_1 \sqrt{\log(N_2)/N_2}\right\}.
\end{align*}
Since $N_2 \geq n \geq 4$, on the event $\mathcal{E}'$, $|N_2/\widehat{N}_2^{(k,\ell)} - 1| \leq C_1$ for each $k \in [n], \ell \in [M]$. Then by definition, on the event $\mathcal{E}'$, for fixed $k\in [n]$, we have
for $1 \leq j \leq d$,
\begin{align*}
&\left(
\bG^{(k)}(h_j) - P^{r-1}h_j(X_{k})
\right)^2
\lesssim \left( T_2^{(k)} \right)^2 + \left( T_3^{(k)} \right)^2 + \\
&\left( |I_{n-1,r-1}^{(k)}|^{-1} \sum_{\iota \in \kjack} \left( h_j(X_{\iota^{(k)}}) - P^{r-1}h_j(X_k) \right) \right)^2
\end{align*}
where we define
\begin{align}\label{def:T2_T3}
\begin{split} 
T_2^{(k)} &:= \max_{\ell \in [M]}\max_{1 \leq j \leq [d]} \left\vert N_2^{-1} \sum_{\iota \in \kjack}
(Z^{(k,\ell)}_{\iota} - q_n) h_j(X_{\iota^{(k)}}) \right\vert, \\
T_3^{(k)} &:=  \max_{1 \leq j \leq d} \left\vert P^{r-1}h_j(X_k) \right\vert  \max_{\ell \in [M]}\left\vert N_2^{-1} \sum_{\iota \in \kjack}
(Z^{(k, \ell)}_{\iota} - q_n) \right\vert.
\end{split}
\end{align} 

As a result, on the event $\mathcal{E}'$,
\begin{align*}
\widehat{\Delta}_{A,1} \lesssim T_1 +n^{-1} \sum_{k = 1}^{n} \left( \left( T_2^{(k)} \right)^2 + \left( T_3^{(k)} \right)^2 \right).
\end{align*}
where we define
\begin{align}\label{def:T1}
T_1 &:=  \max_{1 \leq j \leq d}  n^{-1} \sum_{k = 1}^{n} \left( |I_{n-1,r-1}^{(k)}|^{-1} \sum_{\iota \in \kjack} \left( h_j(X_{\iota^{(k)}}) - P^{r-1}h_j(X_k) \right) \right)^2.
\end{align}

Now by Markov inequality, Lemma \ref{lemma:jacknife_complete_part} and Lemma \ref{lemma:T2_T3_bound} (both ahead), 
\begin{align*}
&\Pro\left(T_1 \geq C \log^{-4}(d) \cX_{n,A}^6 \right) \leq C \cX_{n,A}, \\
&\Pro\left( n^{-1} \sum_{k = 1}^{n} \left( \left( T_2^{(k)} \right)^2 + \left( T_3^{(k)} \right)^2\right) \geq C \log^{-4}(d) \cX_{n,A}^6 \right) \leq C \cX_{n,A}.
\end{align*}
As a result, 
\begin{align*}
\Pro\left(\widehat{\Delta}_{A,1} \geq C \log^{-4}(d) \cX_{n,A}^6 \right) \leq C \cX_{n,A}.
\end{align*}
Since $\cX_{n,A} \leq 1/2$, $\log^{-4}(d) \cX_{n,A}^6 \leq \log^{-2}(d) \cX_{n,A}^3$. thus,
\begin{align*}
\Pro\left(\widehat{\Delta}_{A,1} \geq C \log^{-2}(d) \cX_{n,A}^3 \right) \leq C \cX_{n,A}, \quad
\Pro\left(\widehat{\Delta}_{A,1}^{1/2} \geq C \log^{-2}(d) \cX_{n,A}^3 \right) \leq C \cX_{n,A},
\end{align*}
which completes the proof.
\end{proof}

\begin{lemma} \label{lemma:Delta_A2_A3_bound}
Recall the definition of $\widehat{\Delta}_{A,2}$ and $\widehat{\Delta}_{A,3}$ in \eqref{def:hat_Delta_A}. Assume that 
\eqref{assumption_for_geometric_series},
\eqref{MT0}, and \eqref{MT1} hold. Then there exists a constant $C$, depending only on $q, \ubar{\sigma}^2$, such that 
\begin{align*}
&C^{-1}\Exp\left[ \widehat{\Delta}_{A,2} \right] 
 \; \leq \; n^{-1/2} D_n
\log^{1/2}(d) + 
n^{-1+2/q} D_n^2 \log(d), \\ 
&C^{-1}\Exp\left[ \widehat{\Delta}_{A,3} \right] 
 \; \leq \; n^{-1/2}
\log^{1/2}(d) +
n^{-1+1/q} D_n\log(d).
\end{align*}
\end{lemma}
\begin{proof}
For the first result, we apply \cite[Lemma E.1]{chernozhukov2017} to the finite collection $\{P^{r-1}h_i  P^{r-1}h_j: 1 \leq i, j \leq d\}$ with envelope $(P^{r-1}H)^2$:
\begin{align*}
\Exp\left[ \widehat{\Delta}_{A,2} \right] 
\leq n^{-1/2}\max_{1 \leq j \leq d} \|(P^{r-1}h_j)^2\|_{P,2}
\log^{1/2}(d) +
n^{-1+2/q}\|(P^{r-1}H)^2\|_{P,q/2} \log(d).
\end{align*}
Due to \eqref{MT1}, 
\begin{align*}
\Exp\left[ \widehat{\Delta}_{A,2} \right] 
\leq n^{-1/2} D_n
\log^{1/2}(d) +
n^{-1+2/q} D_n^2 \log(d).
\end{align*}

For the second result, we apply \cite[Lemma E.1]{chernozhukov2017} to the finite collection $\{ P^{r-1}h_j: 1 \leq  j \leq d\}$ with envelope $P^{r-1}H$: due to \eqref{MT1},
\begin{align*}
\Exp\left[ \widehat{\Delta}_{A,3} \right] 
\leq n^{-1/2}
\log^{1/2}(d) +
n^{-1+1/q} D_n\log(d),
\end{align*}
which completes the proof.
\end{proof}

\begin{lemma} \label{lemma:jacknife_complete_part}
Recall the definition of $T_1$ in \eqref{def:T1}.
Assume that \eqref{assumption_for_geometric_series},   \eqref{MT2},  and \eqref{MT5} hold. Then
there exists a constant $C$, depending only on $q,r$, such that 
\begin{align*}
T_1 
\leq C\left( n^{-1} D_n^2 \log(d) + n^{-3/2+2/q} D_n^{4-4/q} \log^{3/2}(d)   + n^{-2+2/q} D_n^{6-4/q} \log^{2}(d) \right).
\end{align*}
\end{lemma}
\begin{proof}
From the proof of \cite[Theorem 3.1, equation (32)]{chen2017jackknife}, we have
\begin{align*}
&\Exp\left[T_{1} \right] \; \lesssim \; \sum_{\ell=2}^{r-1} n^{-\ell} \|P^{r-\ell-1}H\|_{P^{\ell+1},2}^2 \log^{\ell}(d) + n^{-1}\max_{1 \leq j \leq d} \|P^{r-2}h_j\|_{P^2,2}^2  \\
&  + n^{-3/2} \max_{1 \leq j \leq d} \|(P^{r-2}h_j)^2\|_{P^2,2} \log^{1/2}(d) + n^{-2+2/q} \|(P^{r-2}H)^2\|_{P^2,q/2}\log(d) \\
& + n^{-2}\|(P^{r-2}H)^2\|_{P^2,2} \log(d)  \\
&+ n^{-1}  \max_{1 \leq j \leq d}\| (P^{r-2}h_j)^{\bigodot 2} \|_{P^2,2} \log(d)  + n^{-3/2+2/q} \|(P^{r-2} H)^{\bigodot 2}\|_{P^2,q/2} \log^{3/2}(d)  \\
&+ n^{-3/2}\max_{1 \leq j \leq d}\|(P^{r-2}h_j)^2\|_{P^2,2} \log^{3/2}(d)  + n^{-2+2/q}\|(P^{r-2}H)^2\|_{P^2,q/2} \log^{2}(d),
\end{align*}
and  for $1 \leq j \leq d$,
$\|(P^{r-2}h_j)^{\bigodot 2}\|_{P^2,2} \leq
\|P^{r-2}h_j\|_{P^2,2}^2$.

Then since $q \geq 4$, due to \eqref{assumption_for_geometric_series}, \eqref{MT5}, \eqref{MT2}
\begin{align*}
&\Exp\left[T_1 \right] \; \lesssim \; \sum_{\ell=2}^{r-1} n^{-\ell} D_n^{2(\ell+1)} \log^{\ell}(d)  \\
&+ n^{-1}  \max_{1 \leq j \leq d}\|P^{r-2}h_j\|_{P^2,2}^2 \log(d)  + n^{-3/2+2/q} \|(P^{r-2} H)^{\bigodot 2}\|_{P^2,q/2} \log^{3/2}(d)  \\
&+ n^{-3/2}\max_{1 \leq j \leq d}\|(P^{r-2}h_j)^2\|_{P^2,2} \log^{3/2}(d)  + n^{-2+2/q}\|(P^{r-2}H)^2\|_{P^2,q/2} \log^{2}(d),\\
\lesssim & n^{-2} D_n^{6} \log^2(d) +
n^{-1} D_n^2 \log(d) + n^{-3/2+2/q} D_n^{4-4/q} \log^{3/2}(d) + \\
\ & n^{-3/2} D_n^4 \log^{3/2}(d)  + n^{-2+2/q} D_n^{6-4/q} \log^{2}(d),
\end{align*}
which completes the proof due to \eqref{assumption_for_geometric_series}.
\end{proof}

\begin{lemma}\label{lemma:T2_T3_bound}
Fix $k \in [n]$, and recall the definitions of  $T_2^{(k)}$ and $T_3^{(k)}$ in \eqref{def:T2_T3}. Assume \eqref{MT3} holds. Then
there exists an absolute constant $C$  such that for each $k \in [n]$, 
\begin{align*}
&C^{-1} \Exp [ ( T_2^{(k)} )^2 ]  \; \leq \; 
\frac{\log(2M) B_n^2 \log(d)}{N_2}
+\frac{\log^2(2M)M^{2/q}B_n^{4-4/q}D_n^{4/q-2} \log(d)}{N_2^{2-2/q}},\\
&C^{-1} \Exp [ ( T_3^{(k)} )^2 ]  \; \leq \; 
\frac{\log(2M) B_n^2 }{N_2}
+\frac{\log^2(2M)B_n^{2} }{N_2^{2}}.
\end{align*} 
\end{lemma}

\begin{proof}
We apply Lemma \ref{lemma:sampling_higher_order} to $\{h_j: 1 \leq j \leq d\}$ conditional on $X_k$:
\begin{align*}
\Exp_{\vert X_k} [ ( T_2^{(k)} )^2 ] \; \lesssim \;
& N_2^{-1}\log(d)\left(
 \|H(X_k,\cdot)\|_{P^{r-1},2}^2 \log(2M)  +
 \frac{\log^2(2M) M^{2/q} N_2^{2/q}	\|H(X_k,\cdot)\|_{P^{r-1},q}^2}{N_2}
\right).
\end{align*}
Due to \eqref{MT3}, $q \geq 4$ and Jensen's inequality, we have
\begin{align*}
\Exp [ ( T_2^{(k)} )^2 ]  \; \leq \; C  N_2^{-1}  \log(d)
\left(\log(2M) B_n^2 + \frac{\log^2(2M)M^{2/q}B_n^{4-4/q}D_n^{4/q-2}}{N_2^{1-2/q}} \right).
\end{align*}

Now we focus on the second inequality. Due to  \eqref{MT3},
$\left\| \max_{1 \leq j \leq d} \left\vert P^{r-1}h_j(X_k) \right\vert \right\|_{\Pro,2}^2  \leq  B_n^2$. 
By Lemma \ref{lemma:Bernstein_second_moment},
\begin{align*}
\Exp\left[\max_{\ell \in [M]}
\left\vert  \sum_{\iota \in \kjack}
(Z^{(k, \ell)}_{\iota} - q_n) \right\vert^2 \right]
\lesssim N_2 \log(2M) + \log^2(2M).
\end{align*}

As a result, due to independence,
\begin{align*}
\Exp[(T_3^{(k)})^2] \lesssim B_n^2 N_2^{-1} \log(2M) + B_n^2 N_2^{-2}\log^2(2M).
\end{align*}
\end{proof}

\subsubsection{Proof of Theorem \ref{thrm:appr_U_full}}
\label{proof:appr_U_full}
The constants in the following proof may depend on $r, \ubar{\sigma},q, c_0, C_0$, and may vary from line to line.

\begin{proof}[Proof of Theorem \ref{thrm:appr_U_full}]
Without loss of generality, we can assume $Y_A$ and $Y_B$ are independent of all other random variables. 
Recall the definition of $\cX_{n,A}$ and $\cX_{n,B}$ in Lemma \ref{lemma:Delta_A_bound} and \ref{lemma:Delta_B_bound}. Define two $d\times d$ \textit{diagonal} matrices $\Lambda_A$ and $\Lambda_B$ such that
\[
\Lambda_{A,jj} = \gamma_A(h_j), \quad
\Lambda_{B,jj} = \gamma_B(h_j), \;\;\text{ for } 1 \leq j \leq d.
\]

\noindent\underline{Step 1.}  Due to Lemma \ref{lemma:Delta_B_bound} and \eqref{MT4},
\begin{align*}
\Pro\left( \mathcal{E}_1 \right) \geq 1 -  C \cX_{nB}, \;\;\text{ where }
\mathcal{E}_1 :=
\left\{  \max_{1 \leq j,k \leq d} \frac{\left| \widehat{\gamma}_B(h_j,h_k) - \gamma_B(h_j, h_k) \right|}{
\sqrt{\gamma_B(h_j)\gamma_B(h_k)}
} \leq C \log^{-2}(d)\cX_{n,B}^3
\right\}.
\end{align*}
Then by the Gaussian comparison inequality \citep[Lemma C.5]{chen2017randomized}, on the event $\mathcal{E}_1$,
\begin{align*}
&\sup_{R \in \cR} \left\vert \Pro_{\vert \cD_n'}(\underline{U}_{n,B}^{\#} \in R) - \Pro(Y_B \in R)
\right\vert \\
= & \sup_{R \in \cR} \left\vert \Pro_{\vert \cD_n'}(\Lambda_B^{-1/2}\underline{U}_{n,B}^{\#} \in R) - \Pro(\Lambda_B^{-1/2}Y_B \in R) 
\right\vert \\
\lesssim & \left(\log^{-2}(d)\cX_{n,B}^3\right)^{1/3} \log^{2/3}(d) =   \cX_{n,B}.
\end{align*}

\noindent\underline{Step 2.}  Due to Lemma \ref{lemma:Delta_A_bound},
\begin{align*}
\Pro\left( \mathcal{E}_2 \right) \geq 1 - C \cX_{n,A}, \;\;\text{ where }
\mathcal{E}_2 :=
\left\{  \widehat{\Delta}_A \leq C \log^{-2}(d)\cX_{n,A}^3
\right\},
\end{align*}
where $\widehat{\Delta}_A$ is defined in Lemma \ref{lemma:Delta_A_bound}. 
Then by the Gaussian comparison inequality \citep[Lemma C.5]{chen2017randomized}, on the event $\mathcal{E}_2$,
\begin{align*}
&\sup_{R \in \cR} \left\vert \Pro_{\vert \cD_n'}(\underline{U}_{n,A}^{\#} \in R) - \Pro(Y_A \in R)
\right\vert \\
= & \sup_{R \in \cR} \left\vert \Pro_{\vert \cD_n'}(\Lambda_A^{-1/2}\underline{U}_{n,A}^{\#} \in R) - \Pro(\Lambda_A^{-1/2}Y_A \in R) 
\right\vert \\
\lesssim & \left(\log^{-2}(d)\cX_{n,A}^3\right)^{1/3} \log^{2/3}(d) =   \cX_{n,A}.
\end{align*}

\noindent\underline{Step 3.} Now we focus on the event $\mathcal{E}_1 \cup \mathcal{E}_2$, which occurs with probability at least $1 - C\chi_{n,A} - C \chi_{n,B}$. Fix $R \in \cR$.

On the event $\mathcal{E}_1 \cup \mathcal{E}_2$, by the results in Step 1 and 2,
\begin{align*}
\Pro_{\vert \cD_n}\left(\underline{U}_{n,*}^{\#} \in R\right)
&= \Pro_{\vert \cD_n'}\left(\alpha_n^{1/2} \underline{U}_{n,B}^{\#} + r \underline{U}_{n,A}^{\#} \in R
\right) 
= \Pro_{\vert \cD_n'}\left( \underline{U}_{n,B}^{\#} \in \alpha_n^{-1/2}(R - r  \underline{U}_{n,A}^{\#}) \right) \\
&\leq \Pro_{\vert \cD_n'}\left( Y_B \in \alpha_n^{-1/2}(R - r  \underline{U}_{n,A}^{\#} ) 
\right) + C \cX_{n,B} \\
&= \Pro_{\vert \cD_n'}\left( \underline{U}_{n,A}^{\#} \in r^{-1}(R - \alpha_n^{1/2}Y_B  ) 
\right) +  C \cX_{n,B} \\
& \leq \Pro\left( Y_A \in r^{-1}(R - \alpha_n^{1/2}Y_B  ) 
\right) +  C \cX_{n,B} + C\cX_{n,A} \\
&= \Pro\left( r Y_A + \alpha_n^{1/2}Y_B \in R  ) 
\right) +  C \cX_{n,B} + C\cX_{n,A} .
\end{align*}
The reverse inequality is similar. Then the proof is complete by noticing that $\cX_{n,A} + \cX_{n,B} \leq \log^{1/4}(M)\cX_{n}$.
\end{proof}

\section{Proofs regarding stratified, incomplete U-processes}
First, define 
\begin{equation}
\label{def:H_all_epsilon}
\partial \cH_{\epsilon} := \{h - h':\; h,h' \in \cH, \;\;\max\{ \|h-h'\|_{P^r,2}, \|h-h'\|_{P^r,4}^2\} \leq \epsilon\|1 + H^2\|_{P^r,2}\}.
\end{equation}
By  Lemma \ref{discretization_lemma}, $\{\partial\cH_{\epsilon}, 2H\}$ is a VC-type class with characteristics $(A,2\nu)$. Then due to \cite[Corollary A.1]{chernozhukov2014gaussian},
$\{(\partial\cH_{\epsilon})^2, (2H)^2\}$ is also VC-type with characteristics $(\sqrt{2}A,4\nu)$.

%
Due to \eqref{MT2}, \eqref{MT4}, if \eqref{wlog_h_epsilon} holds, we have
\begin{align}\label{log_epsilon}
\log(N \|H\|_{P^r,2}) \leq C \log(n), \quad
\log(N\|H^2\|_{P^r,2}) \leq C \log(n), \quad
n^{-1} D_n^2 K_n \leq 1.
\end{align}
where the constant $C$ depends on $r$. 

\subsection{Supporting calculations}

\begin{lemma} \label{lemma:discr_first_order}
Assume \eqref{log_epsilon}, \ref{cond:VC}, \eqref{MT1}, \eqref{MT2} and \eqref{MT4} hold. Let $\epsilon^{-1} := N \|1 + H^2\|_{P^r,2}$. Then there exists a constant $C$, depending only on $q,r$, such that
\begin{align*}
C^{-1}\Exp\left[ \|\sqrt{n}\left(U_n(h) - P^r h \right)\|_{\partial \cH_{\epsilon}}\right] 
\;\leq\;
n^{-1/2+1/q} D_n K_n + n^{-1+1/q} D_n^{3-2/q} K_n^2.
\end{align*}
\end{lemma}

\begin{proof}
By definition, for $1 \leq \ell \leq r$, $\sup_{h \in \cH_{\partial \epsilon}} \|P^{r-\ell}h\|_{P^{\ell},2} 
\leq \sup_{h \in \partial \cH_{\epsilon}} \|h\|_{P^{r},2} 
\leq N^{-1}$. Thus let
\begin{align*}
\sigma_{\ell} := N^{-1}, \quad 
\delta_\ell := \sigma_{\ell}/\|2 P^{r-\ell} H \|_{P^r,2} \geq (2N \|H\|_{P^r,2})^{-1}.
\end{align*}
Now we apply Theorem \ref{thm:multi-level} to the class $\{\partial\cH_{\epsilon}, 2H\}$ with envelopes $\{2P^{r-\ell} H:1 \leq \ell \leq r\}$. Note that by Theorem \ref{thm:multi-level}, due to \ref{cond:VC} and \eqref{log_epsilon},
\[
J_{\ell}(\delta_\ell) \lesssim \delta_{\ell} \left(
\nu \log(A/\delta_{\ell})
\right)^{\ell/2} \lesssim
\delta_{\ell} K_n^{\ell/2}.
\]
Thus by Theorem \ref{thm:multi-level}, due to \eqref{MT1}, \eqref{MT2}, \eqref{log_epsilon}, and $N \geq n/r$,
\begin{align*}
&\Exp\left[ \|U_n(h) - P^r h \|_{\partial\cH_{\epsilon}}\right] 
\lesssim  \;n^{-1/2}  N^{-1} K_n^{1/2} + n^{-1+1/q} \|P^{r-1} H\|_{P,q} K_n  \\
&+ n^{-1} N^{-1} K_n^{1} + n^{-3/2+1/q} \|P^{r-2} H\|_{P^2,q}  K_n^{2}
+\sum_{\ell=3}^{r} n^{-\ell/2} \|P^{r-\ell} H\|_{P^{\ell},2}  K_n^{\ell/2} \\
\lesssim & \;
n^{-1+1/q} D_n K_n + n^{-3/2+1/q} D_n^{3-2/q} K_n^{2}
+ n^{-3/2} D_n^3 K_n^{3/2} \\
\lesssim & \;
n^{-1+1/q} D_n K_n + n^{-3/2+1/q} D_n^{3-2/q} K_n^{2},
\end{align*}
which completes the proof.
\end{proof}

\begin{lemma} \label{lemma:discr_second_order}
Assume \eqref{log_epsilon}, \ref{cond:VC}, \eqref{MT2}, and \eqref{MT4}  hold. Let $\epsilon^{-1} := N \|1 + H^2\|_{P^r,2}$. Then there exists a constant $C$, depending only on $q,r,c_0$, such that
\begin{align*}
C^{-1}\Exp\left[ \left\| |I_{n,r}|^{-1} \sum_{\iota \in I_{n,r}} h^2(X_{\iota}) \right\|_{\partial \cH_{\epsilon}}\right] 
\;\leq\; 
\left( \inf_{h \in \cH} \gamma_B(h) \right) \left(
n^{-1+1/q} D_n^{2} K_n +
n^{-3/2+1/q}  D_n^{4-2/q} K_n^2  \right).
\end{align*}
\end{lemma}
\begin{proof}
By definition, for $h \in \partial\cH_{\epsilon}$,
\[
\|h^2\|_{P^{r},2} = \|h\|_{P^{r},4}^{2} \leq N^{-1}.
\]
As a result, for $ 1 \leq \ell \leq r$,
\begin{align*}
&\sup_{h \in \partial\cH_{\epsilon}} \|P^{r-\ell}h^2\|_{P^{\ell},2} 
\leq \sup_{h \in \partial\cH_{\epsilon}} \|h^2\|_{P^{r},2} 
\leq N^{-1} := \sigma_{\ell}, \quad \\
&\delta_{\ell} := \sigma_{\ell}/\|4P^{r-\ell} H^2\|_{P^{\ell},2} 
\geq \left(4N\|H^2\|_{P^r,2}\right)^{-1}.
\end{align*}
Now we apply Theorem \ref{thm:multi-level} to the class $\{(\partial\cH_{\epsilon})^2, 4H^2\}$ with envelopes $\{4P^{r-\ell} H^2 :1 \leq \ell \leq r\}$. Note that by Theorem \ref{thm:multi-level}, due to \ref{cond:VC} and \eqref{log_epsilon}, $J_{\ell}(\delta_\ell) \lesssim  \delta_{\ell} K_n^{\ell/2}$. Thus by Theorem \ref{thm:multi-level}, due to \eqref{MT2}, \eqref{MT4}, and \eqref{log_epsilon}, and $N \geq n/r$,
\begin{align*}
&\Exp\left[ \left\| |I_{n,r}|^{-1} \sum_{\iota \in I_{n,r}} h^2(X_{\iota}) \right\|_{\partial \cH_{\epsilon}}\right] 
\leq N^{-1} + n^{-1/2} N^{-1} K_n^{1/2} + n^{-1+1/q} \|P^{r-1} H^2 \|_{P,q} K_n \\
& + n^{-1} N^{-1} K_n + n^{-3/2+1/q} \|P^{r-2} H^2\|_{P^2,q} K_n^2 + \sum_{\ell=3}^{r} n^{-\ell/2} \|P^{r-\ell} H^2\|_{P^{\ell},2} K_n^{\ell/2} \\
\lesssim & \;
\left( \inf_{h \in \cH} \gamma_B(h) \right) \left(
n^{-1+1/q} D_n^{2} K_n +
n^{-3/2+1/q}  D_n^{4-2/q} K_n^2 
+ n^{-3/2}  D_n^{4} K_n^{3/2}  \right)\\
\lesssim & \;
\left( \inf_{h \in \cH} \gamma_B(h) \right) \left(
n^{-1+1/q} D_n^{2} K_n +
n^{-3/2+1/q}  D_n^{4-2/q} K_n^2  \right),
\end{align*}
which completes the proof.
\end{proof}

Recall the definition of $W_P$ and $\gamma_{*}(h)$ in Section \ref{sec:GaussApprox_Proce}.

\begin{lemma}\label{lemma:WP_size}
Assume the conditions \ref{cond:PM}, \ref{cond:VC}, \ref{cond:MB}, \eqref{MT0}, and \eqref{MT3} hold, and let $\widetilde{\cM} := \sup_{h \in \cH} W_P(h/\sqrt{\gamma_{*}(h)})$. 
Then there exist a constant $C$ only depending on $r,\ubar{\sigma}, C_0, C_1$ such that
\begin{align*}
\Pro\left( \widetilde{\cM} \geq C K_n^{1/2} \right) \leq n^{-1}.
\end{align*}
\end{lemma}

\begin{proof}
For $m \in [M]$,
define 
\[
\widetilde{\cH}_m := \left\{h/\sqrt{\gamma_{*}(h)} : h \in \cH_m\right\}, \quad
d^2_m(h_1, h_2) := \Exp\left[ \left(W_P(h_1) - W_P(h_2) \right)^2 \right],
\]
for $h_1, h_2 \in \widetilde{\cH}_m$. 
Then for $m \in [M]$ and $h_1, h_2 \in \widetilde{\cH}_m$, by definition  (recall $W_P$ is prelinear \cite[Theorem 3.1.1]{dudley1999uniform}) and Jensen's inequality,
\begin{align*}
d^2_m(h_1, h_2)& = \Exp\left[ W_P^2(h_1 - h_2)\right]
\leq  r^2 \Exp\left[ \left( P^{r-1}(h_1 - h_2)(X_1) \right)^2\right]
+ \alpha_n \Exp\left[ \left((h_1-h_2)(X_1^r) \right)^2\right] \\ 
&\leq (r^2 + r) \|h_1 - h_2\|_{P^r,2}^2.
\end{align*}

Due to \ref{cond:VC} and \eqref{MT0},  $\left\{ \widetilde{\cH}_m, \ubar{\sigma}^{-1} H \right\}$ is a VC-type class
with characteristics $(A,\nu+1)$ for $m \in[M]$. As a result, for any $\tau > 0$ and $m \in [M]$,
\begin{align*}
&N(\widetilde{\cH}_m, d_m, \tau\sqrt{ r^2 + r }\|\ubar{\sigma}^{-1}H\|_{P^r,2})  \\
\leq & 
N(\widetilde{\cH}_m, \sqrt{ r^2 + r } \|\cdot\|_{P^r,2}, \tau\sqrt{ r^2 + r }\| \ubar{\sigma}^{-1}H\|_{P^r,2}) 
\leq (A/\tau)^{\nu+1}.
\end{align*}

Since $\Exp\left[W_P^2(h) \right] = 1$ for any $h \in \widetilde{\cH}_m$, by the entropy integral bound \cite[Corollary 2.2.5]{van1996weak} and due to \eqref{MT3}, for $m \in [M]$,
\begin{align*}
\Exp\left[ \|W_P\|_{\widetilde{\cH}_m} \right]
\lesssim \int_0^{2} \sqrt{1 + (\nu+1)\log\left(2Ar \|\ubar{\sigma}^{-1} H\|_{P^r,2}/\tau \right)}\;\; d\,\tau 
\lesssim  K_n^{1/2}.
\end{align*}

By  the Borell-Sudakov-Tsirelson concentration inequality \cite[Theorem 2.2.7]{gine2016mathematical}, for $m \in [M]$,
 \begin{align*}
 \Pro\left(
 \|W_P\|_{\widetilde{H}_m} \geq \Exp\left[ \|W_P\|_{\widetilde{H}_m}\right] + \sqrt{2\log(Mn)}
 \right) \leq M^{-1} n^{-1}.
 \end{align*}
 Then the proof is complete due to \ref{cond:MB} and the union bound.
\end{proof}

\subsection{Proof of the Gaussian approximation for the suprema}\label{proof:thm:GAR_sup_incomplete_Uproc}
Now we prove Theorem \ref{thm:GAR_sup_incomplete_Uproc}. For $\epsilon \in (0,1)$ and $m \in [M]$, denote
\begin{equation}
\label{def:H_m_epsilon}
\partial \cH_{m,\epsilon} := \{h - h':\; h,h' \in \cH_m, \;\;\max\{ \|h-h'\|_{P^r,2}, \|h-h'\|_{P^r,4}^2\} \leq \epsilon\|1 + H^2\|_{P^r,2}\}.
\end{equation}
By Lemma \ref{discretization_lemma}, due to condition \ref{cond:VC}, $\{\partial \cH_{m,\epsilon}, 2H\}$ is a VC-type class with characteristics $(A,2\nu)$. 
For each $m \in [M]$ and $h_1, h_2 \in \cH_m$ such that
$\max\{ \|h_1-h_2\|_{P^r,2}, \|h_1-h_2\|_{P^r,4}^2\} \leq \epsilon\|1 + H^2\|_{P^r,2}$, we define for $h := h_1 - h_2$,
\begin{align*}
\bar{U}_{n,N}'(h) &:= U_{n,N}'(h_1) - U_{n,N}'(h_2), \;\;
\bar{W}_P(h) := W_P(h_1) - W_P(h_2).
\end{align*}
$\bar{U}_{n,N}'(\cdot)$ is well defined (i.e. independent of the choice of $h_1, h_2$) since $U_{n,N}'(\cdot)$ is linear in its argument.
Further, by \cite[Theorem 3.1.1]{dudley1999uniform}, $W_P$ can be extended to the linear hull of $\cH_m$. With above discussions, we will simply write $U_{n,N}'$ and $W_P$ for their extensions $\bar{U}_{n,N}'$ and $\bar{W}_P$ to $\partial \cH_{m,\epsilon}$. Similar convention applies to $\bU_{n,N}'$.

Further, define
\begin{align}\label{min_variance}
\underline{\gamma}_{*} := \inf_{h \in \cH} \gamma_*(h,h),
\end{align}
where we recall that $\gamma_A(\cdot, \cdot)$, $\gamma_B(\cdot, \cdot)$, and $\gamma_{*}(\cdot, \cdot)$ are defined in \eqref{def:cov_G}, and $\alpha_n = n/N$.

In this section, $C$ denotes a constant that depends only on $r,q,\ubar{\sigma}, c_0, C_0$, and that may vary from line to line.  The notation $\lesssim$ means that the left hand side is bounded by the right hand side up to a constant that depends only on $r,q,\ubar{\sigma}, c_0, C_0$. Clearly, in proving Theorem \ref{thm:GAR_sup_incomplete_Uproc}, without loss of generality, we may assume 
\begin{align}
\label{wlog_h_epsilon}
\eta_n^{(1)} \leq 1/2, \quad
\eta_n^{(2)} \leq 1/2.
\end{align}

\begin{proof}[Proof of Theorem \ref{thm:GAR_sup_incomplete_Uproc}] Fix some $t \in \bR$. 
Let $\epsilon^{-1} := N \|1 + H^2\|_{P^r,2}$.
By Lemma \ref{discretization_lemma} and \ref{cond:MB},  there exists a finite collection $\{h_j : 1 \leq j \leq d\} \subset \cH$ such that
the following two conditions hold: (i).
$\log(d) \leq \log \left( M (4A/\epsilon)^{\nu} \right) \lesssim K_n$; 
(ii). for any $h \in \cH$, there exists $1 \leq j^* \leq d$ such that
$\sigma(h) = \sigma(h_{j^*})$, and
\begin{align*}
\max\left\{\|h - h_{j^*} \|_{P^r,2}, \;\; \|h - h_{j^*} \|_{P^r,4}^2 \right\} \;\leq \; \epsilon \|1 + H^2\|_{P^r,2}.
\end{align*} 
Define $
\bMn^{\epsilon} := \max_{1 \leq j \leq d} \bU_{n,N}'(h_j), \;\;
\widetilde{\bM}_n^{\epsilon} := \max_{1 \leq j \leq d} W_P(h_j)$.
Then by definition, 
\begin{align*}
\bMn^{\epsilon} \leq \bMn \leq  \bMn^{\epsilon} + \max_{m \in [M]}\|\bU_{n,N}'(h)\|_{\partial \cH_{m,\epsilon}}, \quad
\widetilde{\bM}_n^{\epsilon} \leq \widetilde{\bM}_n \leq \widetilde{\bM}_n^{\epsilon} + \max_{m \in [M]} \|W_P(h)\|_{\partial \cH_{m,\epsilon}}.
\end{align*}
Observe that $\Exp[\gamma_{*}^{-1/2}(h_j)W_P(h_j)] = 1$ for $1 \leq j \leq d$. Then
for any  $\Delta >0$, by Gaussian anti-concentration inequality \cite[Lemma A.1]{chernozhukov2017},
\begin{align*}
&\Pro\left(\widetilde{\bM}_n^{\epsilon}  \leq t -  \Delta  \underline{\gamma}_{*}^{1/2} \right)
= \Pro\left(\bigcap_{j=1}^{d} \left\{ W_P(h_j) \leq t - \Delta  \underline{\gamma}_{*}^{1/2}\right\} \right) \\
=& \Pro\left(\bigcap_{j=1}^{d} \left\{ \gamma_{*}^{-1/2}(h_j)W_P(h_j)  \leq \gamma_{*}^{-1/2}(h_j) t - \Delta  \gamma_{*}^{-1/2}(h_j )\underline{\gamma}_{*}^{1/2}\right\} \right) \\
\geq & \Pro\left(\bigcap_{j=1}^{d} \left\{ \gamma_{*}^{-1/2}(h_j)W_P(h_j)  \leq \gamma_{*}^{-1/2}(h_j) t - \Delta \right\} \right) \\
\geq & \Pro\left(\bigcap_{j=1}^{d} \left\{ \gamma_{*}^{-1/2}(h_j)W_P(h_j)  \leq \gamma_{*}^{-1/2}(h_j) t \right\} \right) - C \Delta \log^{1/2}(d)
\geq  \Pro\left(\widetilde{\bM}_n^{\epsilon}  \leq t \right) - C \Delta K_n^{1/2}.
\end{align*}
Thus for any $\Delta > 0$, we have
\begin{align*}
&\Pro(\widetilde{\bM}_n \leq t) \; \leq \; \Pro(\widetilde{\bM}_n^{\epsilon}  \leq t)   \;\leq\;
\Pro(\widetilde{\bM}_n^{\epsilon}  \leq t -\Delta \underline{\gamma}_{*}^{1/2}) + C \Delta K_n^{1/2} \\
&\;\leq_{(1)} \;
\Pro(\bMn^{\epsilon} \leq t -\Delta \underline{\gamma}_{*}^{1/2}) +
C \Delta K_n^{1/2}+ C \eta_{n}^{(1)} + C \eta_n^{(2)}\\
&\; \leq \;
\Pro({\bMn} \leq t) + \Pro\left(\max_{m \in [M]}\|\bU_{n,N}'(h)\|_{\partial \cH_{m,\epsilon}} \geq \Delta \underline{\gamma}_{*}^{1/2}\right) + C \Delta K_n^{1/2}+ C \eta_{n}^{(1)} + C \eta_n^{(2)}\\
&\; \leq_{(2)} \;
\Pro({\bMn} \leq t) 
+ C \Delta K_n^{1/2} 
+ C \eta_n^{(1)} + C \eta_n^{(2)}\\
&\quad + \frac{1}{\Delta} \left( \frac{ M^{1/q} B_n^{1-2/q} D_n^{2/q} K_n^2}{ N^{1/2-1/q}} 
 +\frac{D_n K_n^{3/2}}{ n^{1/2-1/q}}
+ \frac{D_n^{2-1/q} K_n^{2}}{ n^{3/4-1/(2q)}}  +
+ \frac{D_n^{3-2/q} K_n^2}{ n^{1-1/q}} \right)+ \frac{1}{N},
\end{align*}
where $(1)$ is due to Theorem \ref{thm:GAR_hd_U}, and  $(2)$ is due to Lemma \ref{lemma:size_H_discretize} (ahead). Now let $\Delta$ to be the following
\[
\left( \frac{M^{1/q}  B_n^{1-2/q} D_n^{2/q} K_n^{3/2}}{ N^{1/2-1/q}} \right)^{1/2}
+\left(\frac{D_n K_n}{ n^{1/2-1/q}}\right)^{1/2}
+ \left( \frac{D_n^{2-1/q} K_n^{3/2}}{ n^{3/4-1/(2q)}} \right)^{1/2} 
+ \left(\frac{D_n^{3-2/q} K_n^{3/2}}{ n^{1-1/q}}\right)^{1/2}.
\]
Then due to \eqref{wlog_h_epsilon}, we have
\begin{align*}
\Pro(\widetilde{\bM}_n  \leq t)  \;\leq \;
\Pro(\bMn \leq t) + C \eta_{n}^{(1)} + C \eta_n^{(2)}.
\end{align*} 
By a similar argument (using the bound on $\Pro(\max_{m \in [M]} \|W_P(h)\|_{\partial \cH_{m,\epsilon}} \geq \Delta)$ instead of the one on 
$\Pro(\max_{m \in [M]}\|\bU_{n,N}'(h)\|_{\partial \cH_{m,\epsilon}} \geq \Delta \underline{\gamma}_{*}^{1/2} )$ in Lemma \ref{lemma:size_H_discretize}), we have
\begin{align*}
\Pro(\bMn \leq t) \; \leq \; \Pro(\widetilde{\bM}_n  \leq t)  +C \eta_{n}^{(1)} + C \eta_n^{(2)},
\end{align*}
which completes the proof.
\end{proof}

\subsection{A lemma for establishing validity of Gaussian approximation}
Recall the definition of $\partial \cH_{m,\epsilon}$ in \eqref{def:H_m_epsilon}.  
The next Lemma controls the size of $\bU_{n,N}'$ and $W_P$ over $\partial \cH_{m,\epsilon}$.

In this section, the notation $\lesssim$ means that the left hand side is bounded by the right hand side up to a multiplicative constant that only depends on $r, C_0, c_0, \ubar{\sigma},q$. It is clear  that
for $m \in [M]$, $\partial \cH_{m,\epsilon} \subset \partial \cH_{\epsilon}$,
where $ \cH_{\epsilon}$ is defined  in \eqref{def:H_all_epsilon}.

\begin{lemma}\label{lemma:size_H_discretize}
Assume \eqref{wlog_h_epsilon},  \ref{cond:PM}, \ref{cond:VC}, \ref{cond:MB}, \eqref{MT0}-\eqref{MT4} hold. Let $\epsilon^{-1} := N \|1 + H^2\|_{P^r,2}$. 
Recall the definition of $\underline{\gamma}_{*}$ in \eqref{min_variance}. 
For any $\Delta > 0$,
\begin{align*}
&\Pro\left( \max_{m \in [M]}\|\bU_{n,N}'(h)\|_{\partial \cH_{m,\epsilon}} \;\geq\; \Delta  \underline{\gamma}_{*}^{1/2} \right)  \; \lesssim \;\\
& \frac{1}{\Delta} \left( \frac{ M^{1/q} B_n^{1-2/q} D_n^{2/q} K_n^2}{ N^{1/2-1/q}} 
 +\frac{D_n K_n^{3/2}}{ n^{1/2-1/q}}
+ \frac{D_n^{2-1/q} K_n^{2}}{ n^{3/4-1/(2q)}}  +
+ \frac{D_n^{3-2/q} K_n^2}{ n^{1-1/q}} \right)+ \frac{1}{N}, \\
&\Pro\left(\max_{m \in [M]} \|W_P(h)\|_{\partial \cH_{m,\epsilon}} \;\geq\; \Delta \right) 
\leq \Delta^{-1} N^{-1} K_n.
\end{align*}
\end{lemma}

\begin{proof}
\underline{We start with the first claim.} Observe that for $m \in [M]$ and $h \in \partial \cH_{m,\epsilon}$,
\begin{align*}
&\left({\widehat{N}^{(m)}}/{N} \right) \sqrt{n}( U_{n,N}'(h) - P^{r} h)
=\;\;\alpha_n^{1/2} \frac{1}{\sqrt{N}}\sum_{\iota \in I_{n,r}} (Z_{\iota}^{(m)} -p_n) h(X_{\iota}) \\
&\;\;- 
\alpha_n^{1/2}  P^rh \frac{1}{\sqrt{N}}\sum_{\iota \in I_{n,r}} (Z_{\iota}^{(m)} -p_n) 
+ \sqrt{n}\frac{1}{|I_{n,r}|} \sum_{\iota \in I_{n,r}} \left(h(X_{\iota}) - P^r h \right).
\end{align*}
As a result, 
$\max_{m \in [M]}\left( \widehat{N}^{(m)}/{N} \|\bU_{n,N}'(h)\|_{\partial \cH_{m,\epsilon}} \right) \leq \alpha_n^{1/2} I + 
\alpha_n^{1/2} II
+ III$, where
\begin{align*}
I &:= \max_{m \in [M]} \sup_{h \in \partial \cH_{\epsilon}}\left|
\frac{1}{\sqrt{N}}\sum_{\iota \in I_{n,r}} (Z_{\iota}^{(m)} -p_n) h(X_{\iota})
\right|, \\
II &:=  \sup_{h \in \partial \cH_{\epsilon}} \left| P^rh \right|
\max_{m \in [M]} \left|\frac{1}{\sqrt{N}}\sum_{\iota \in I_{n,r}} (Z_{\iota}^{(m)} -p_n)  \right|, \\
III &:= \sqrt{n}\sup_{h \in \partial \cH_{\epsilon}} \left|
\frac{1}{|I_{n,r}|} \sum_{\iota \in I_{n,r}} \left(h(X_{\iota}) - P^r h \right)
\right|.
\end{align*}

For the first term, we apply Theorem \ref{thm:lmax_sampling_proc} with $\{\partial \cH_{\epsilon},2H\}$. 
By definition of $\partial  \cH_{\epsilon}$,
\[
\sup_{h \in \partial  \cH_{\epsilon}} \|h\|_{P^r,2} \leq N^{-1} := \sigma_r,\quad
\delta_r := \sigma_r/(M^{1/2}\|2H\|_{P^r,2}) \geq (2N M^{1/2}\|H\|_{P^r,2})^{-1}.
\]
By Theorem \ref{thm:lmax_sampling_proc} and due to \ref{cond:MB} and \eqref{log_epsilon}, $\bar{J}(\delta_r) \lesssim \delta_r K_n^{1/2}$ and $\log(2M) \lesssim K_n$. Then due to Theorem \ref{thm:lmax_sampling_proc} and Lemma \ref{lemma:discr_second_order},
\begin{align*}
\Exp\left[I \right]\;
\lesssim \;   &N^{-1}K_n 
+ M^{1/q} N^{-1/2+1/q} \|H\|_{P^r,q} K_n^2 \\
&+ \left(\inf_{h \in \cH} \gamma_B(h)\right)^{1/2} \left(n^{-1/2+1/(2q)} D_n K_n^{3/2} +
n^{-3/4+1/(2q)} D_n^{2-1/q} K_n^{2}
\right).
\end{align*}
Due to \eqref{MT3} and \eqref{MT4}, and since $N \geq n/r$,
\begin{align*}
\alpha_n^{1/2} \Exp\left[ I \right] 
\lesssim \; 
\left(\alpha_n \inf_{h \in \cH} \gamma_B(h)\right)^{1/2} \left(
 \frac{ M^{1/q} B_n^{1-2/q} D_n^{2/q} K_n^2}{N^{1/2-1/q}} + \frac{D_n K_n^{3/2}}{n^{1/2-1/(2q)}}
+ \frac{D_n^{2-1/q} K_n^{2}}{n^{3/4-1/(2q)}} \right).
\end{align*}

For the second term, by Lemma \ref{lemma:Bernstein_second_moment},
\begin{align*}
\Exp\left[ II \right]  \lesssim  N^{-1} 
\left( K_n^{1/2} + N^{-1/2}K_n \right) =
 N^{-1} K_n^{1/2} +N^{-3/2}K_n.
\end{align*}

Finally, due to Lemma \ref{lemma:discr_first_order}, we have
\begin{align*}
&\Exp\left[
\max_{m \in [M]}\left( \widehat{N}^{(m)}/{N} \|\bU_{n,N}'(h)\|_{\partial \cH_{m,\epsilon}} \right)
\right]  \lesssim \frac{D_n K_n}{n^{1/2-1/q}}
+ \frac{D_n^{3-2/q} K_n^2}{n^{1-1/q}} + \\
&\qquad \left(\alpha_n \inf_{h \in \cH} \gamma_B(h)\right)^{1/2} \left(
 \frac{ M^{1/q} B_n^{1-2/q} D_n^{2/q} K_n^2}{N^{1/2-1/q}} + \frac{D_n K_n^{3/2}}{n^{1/2-1/(2q)}}
+ \frac{D_n^{2-1/q} K_n^{2}}{n^{3/4-1/(2q)}} \right).
\end{align*}
Recall  $\underline{\gamma}_{*} := \inf_{h \in \cH} \gamma_*(h)$ (in \eqref{min_variance}). Clearly, by definition and due to \eqref{MT0},
\[
\underline{\gamma}_{*} \geq \max\left\{
r^2 \underline{\sigma}^2,  \alpha_n \inf_{h \in \cH} \gamma_B(h)
\right\}.
\]
Then  by Lemma \ref{lemma:Bernstein} and by the Markov inequality, 
\begin{align*}
&\Pro\left(\max_{m \in [M]}\|\bU_{n,N}'(h)\|_{\partial \cH_{m,\epsilon}} \;\geq\; \Delta \underline{\gamma}_{*}^{1/2}\right) \\
\leq  &
\Pro\left( \max_{m \in [M]}\|\bU_{n,N}'(h)\|_{\partial \cH_{m,\epsilon}} \;\geq\; \Delta\underline{\gamma}_{*}^{1/2}, \;\; \max_{m \in [M]} N/\widehat{N}^{(m)} \leq C \right) 
 + C N^{-1}\\
\leq \;&\Pro\left(\max_{m \in [M]}\left( \left( \widehat{N}^{(m)}/{N} \right) \|\bU_{n,N}'(h)\|_{\partial \cH_{m,\epsilon}} \right) \;\geq\; \Delta \underline{\gamma}_{*}^{1/2}/C \right) 
 + C N^{-1}, \\
 \lesssim \;&
 \frac{ M^{1/q} B_n^{1-2/q} D_n^{2/q} K_n^2}{\Delta N^{1/2-1/q}} + \frac{D_n K_n^{3/2}}{\Delta n^{1/2-1/(2q)}}
+ \frac{D_n^{2-1/q} K_n^{2}}{\Delta n^{3/4-1/(2q)}}  +
\frac{D_n K_n}{\Delta n^{1/2-1/q}}
+ \frac{D_n^{3-2/q} K_n^2}{\Delta n^{1-1/q}} + \frac{1}{N}, \\
 \lesssim \;&
 \frac{ M^{1/q} B_n^{1-2/q} D_n^{2/q} K_n^2}{\Delta N^{1/2-1/q}} 
 +\frac{D_n K_n^{3/2}}{\Delta n^{1/2-1/q}}
+ \frac{D_n^{2-1/q} K_n^{2}}{\Delta n^{3/4-1/(2q)}}  +
+ \frac{D_n^{3-2/q} K_n^2}{\Delta n^{1-1/q}} + \frac{1}{N}, 
\end{align*}
which completes the proof of the first claim.
\vspace{0.2cm}

\underline{Now we focus on the second claim.} Observe that for $m \in [M]$ and 
$h_1, h_2 \in \partial \cH_{m,\epsilon}$, by Jensen's inequality,
\begin{align*}
\Exp\left[\left(W_P(h_1) - W_P(h_2) \right)^2\right] 
&= \Exp \left[\left(W_P(h_1 - h_2) \right)^2\right] \\
&=r^2 \Var(P^{r-1}(h_1 - h_2)(X_1)) 
+ \alpha_n \Var((h_1 - h_2)(X_1^r)) \\
&\lesssim \|h_1 - h_2\|_{P^r,2}^2.
\end{align*}
Since $\{\partial \cH_{m,\epsilon}, 2H\}$ is VC type with characteristics $(A,2\nu)$, by entropy integral bound \cite[Theorem 2.3.7]{van1996weak}, if we denote  $D := \sup_{h \in \partial \cH_{m,\epsilon}} \|h\|_{P^r,2} \leq N^{-1}$,   we have for $m \in [M]$,
\begin{align*}
&\left\| \|W_P(h)\|_{ \partial \cH_{m,\epsilon}} \right\|_{\psi_2} \lesssim 
\int_0^{D} \sqrt{1 + \log N(\partial \cH_{m,\epsilon}, \|\cdot\|_{P^r,2}, \tau)} d\tau \\
&\qquad \lesssim \|2H\|_{P^r,2} \int_0^{D/\|2H\|_{P^r,2}} \sqrt{1 + \log N(\partial \cH_{m,\epsilon}, \|\cdot\|_{P^r,2}, \tau \|2H\|_{P^r,2})} d\tau,
\end{align*}
By Lemma \ref{thm:multi-level} and \eqref{log_epsilon}, for $m \in [M]$,
\begin{align*}
&\left\| \|W_P(h)\|_{ \partial \cH_{m,\epsilon}} \right\|_{\psi_2} \lesssim 
N^{-1} K_n^{1/2}.
\end{align*}
Then the  proof is complete by maximal inequality \cite[Lemma 2.2.2]{van1996weak},  and Markov inequality.
%
%
%
\end{proof}

\subsection{Proof of the validity of bootstrap for stratified, incomplete U-processes}\label{proof:thm:uproc_bootstrap}
In this section, $C$ denotes a constant that depends only on $r,q, \ubar{\sigma},c_0, C_0$, and that may vary from line to line. 
Recall the definition of $\partial \cH_{m,\epsilon}$ for $m \in [M]$ in \eqref{def:H_m_epsilon}. 
The next Lemma establishes a high probability  bound on the size of $W_P$ and $\bU_{n,*}^{\#}$ over $\partial \cH_{m,\epsilon}$ (recall the discussions in Section \ref{proof:thm:GAR_sup_incomplete_Uproc} on extending the domain from $\cH_{m}$ to $\partial \cH_{m,\epsilon}$; similar convention applies to 
$\bU_{n,*}^{\#}$).

\begin{lemma} \label{lemma:boostrap_size_He}
Assume the conditions \ref{cond:PM}, \ref{cond:VC}, \ref{cond:MB}, and \eqref{MT0}-\eqref{MT5}, and for some $C_1 > 0$,
\[
\log(B_n) + \log(D_n) \leq C_1 \log(n), \quad
n^{-1} D_n^2 K_n \leq C_1, \quad N_2^{-1/2} B_n \leq C_1.
\]
Let $\epsilon^{-1} := N \|1 + H^2\|_{P^r,2}$.  Then there exists a constant $C$, depending only on $r, C_0$, such that
$$\Pro\left( \max_{m \in [M]} \|W_P(h)\|_{\partial \cH_{m,\epsilon}}  \; \geq \; C N^{-1} K_n^{1/2} \right) \leq n^{-1}.$$
Further, there exists a constant $C'$, depending only on $r,q,c_0, \ubar{\sigma}, C_0, C_1$, such that for any $\Delta >0$, with probability at least $1 - C' (N \wedge N_2)^{-1} - C' \Delta - C' \Delta^2$,
\begin{align*}
\Pro_{\vert \cD_n'} &\left( 
\max_{m \in [M]}\|\bU_{n,*}^{\#}\|_{\partial \cH_{m,\epsilon}}
\geq C' \underline{\gamma}_{*}^{1/2} N^{-1/2} B_n K_n^{1/2} + 
C' \underline{\gamma}_{*}^{1/2}  \Delta^{-1}\chi_{n,*}^{\#}  
\right) \leq n^{-1}, \;\text{ where }\\
\chi_{n,*}^{\#} := &(N \wedge N_2)^{-1/2+1/q} M^{1/q} B_n^{1-2/q} D_n^{2/q} K_n^2+ 
n^{-1/2+1/q} D_n  K_n^{3/2} + \\
& n^{-3/4+1/q} D_n^{2-2/q} K_n^{2} + 
n^{-1+1/q} D_n^{3-2/q} K_n^2.
\end{align*}
\end{lemma}
\begin{proof}
See Section \ref{proof:lemma:boostrap_size_He}.
\end{proof}

\begin{proof}[Proof of Theorem \ref{thm:uproc_bootstrap}]
Without loss of generality, we may assume $\varrho_n \leq 1$.

Let $\epsilon^{-1} := N \|1 + H^2\|_{P^r,2}$.
By Lemma \ref{discretization_lemma} and \ref{cond:MB},  there exists a finite collection $\{h_j : 1 \leq j \leq d\} \subset \cH$ such that
the following two conditions hold: (i).
$\log(d) \leq \log \left( M (4A/\epsilon)^{\nu} \right) \lesssim K_n$; 
(ii). for any $h \in \cH$, there exists $1 \leq j^* \leq d$ such that 
$\sigma(h) = \sigma(h_{j^*})$, and
\begin{align*}
\max\left\{\|h - h_{j^*} \|_{P^r,2}, \;\; \|h - h_{j^*} \|_{P^r,4}^2 \right\} \;\leq \; \epsilon \|1 + H^2\|_{P^r,2}.
\end{align*} 
Define 
 $
\bMn^{\#,\epsilon} := \max_{1 \leq j \leq d} \bU_{n,*}^{\#}(h_j), \;\;
\widetilde{\bM}_n^{\epsilon} := \max_{1 \leq j \leq d} W_P(h_j)$.
Then
\[
\bMn^{\#,\epsilon} \leq \bMn^{\#} \leq  \bMn^{\#,\epsilon} + \max_{m \in [M]}\|\bU_{n,*}^{\#}\|_{\partial \cH_{m,\epsilon}}
\text{ and } \widetilde{\bM}_n^{\epsilon} \leq \widetilde{\bM}_n \leq \widetilde{\bM}_n^{\epsilon} + \max_{m \in [M]} \|W_P\|_{\partial \cH_{m,\epsilon}}.\]

Fix $t \in \bR$. For some $\Delta > 0$ to be determined, denote
\[
\Delta' := C' N^{-1/2}B_n K_n^{1/2} + C' \Delta^{-1} \chi_{n,*}^{\#}.
\]
where $C', \chi_{n,*}^{\#}$ are defined in Lemma \ref{lemma:boostrap_size_He}.
We have shown in the proof of Theorem \ref{thm:GAR_sup_incomplete_Uproc} that
\begin{align*}
\Pro(\widetilde{\bM}_n \leq t) \; \leq \; 
\Pro(\widetilde{\bM}_n^{\epsilon}  \leq t -\Delta' \underline{\gamma}_{*}^{1/2}) + C \Delta' K_n^{1/2}.
\end{align*}
Then,  by Theorem \ref{thrm:appr_U_full}, with probability at least $1 - C \varrho_n$,
\begin{align*}
\Pro(\widetilde{\bM}_n \leq t)  
\;&\leq \;
\Pro_{\vert \cD_n}\left(\bMn^{\#,\epsilon}  \leq t -\Delta' \underline{\gamma}_{*}^{1/2} \right) + C \varrho_n +  C \Delta' K_n^{1/2}.
\end{align*}
Finally, due to Lemma \ref{lemma:boostrap_size_He} and union bound, with probability at least $1 - C \varrho_n - C\Delta 
-C\Delta^2$,
\begin{align*}
\Pro(\widetilde{\bM}_n \leq t)  \;&\leq \;
\Pro_{\vert \cD_n}\left(\bMn^{\#}  \leq t \right) 
+\Pro_{\vert \cD_n}\left( \max_{m \in [M]}\|\bU_{n,*}^{\#}\|_{\partial \cH_{m,\epsilon}} \geq \Delta' \underline{\gamma}_{*}^{1/2} \right) 
+ C \varrho_n + C \Delta' K_n^{1/2} \\
\;&\leq \;
\Pro_{\vert \cD_n}\left(\bMn^{\#}  \leq t \right) 
+n^{-1}+ C \varrho_n + C \Delta' K_n^{1/2} \\
\;&\leq \;
\Pro_{\vert \cD_n}\left(\bMn^{\#}  \leq t \right) 
 + C \varrho_n
 + C \Delta^{-1} \chi_{n,*}^{\#} K_n^{1/2}. 
\end{align*}
Now let $\Delta = (\chi_{n,*}^{\#})^{1/2} K_n^{1/4}$. Since $\varrho_n \leq 1$, we have that with probability at least $1 - C \varrho_n$
\begin{align*}
\Pro(\widetilde{\bM}_n \leq t)  
\;&\leq \;
\Pro_{\vert \cD_n}\left(\bMn^{\#}  \leq t \right)+ C\varrho_n.
\end{align*}
By a similar argument (using the bound in Lemma \ref{lemma:boostrap_size_He} for $\max_{m \in [M]} \|W_P\|_{\partial \cH_{m,\epsilon}}$), 
$$
\Pro_{\vert \cD_n}\left(\bMn^{\#}  \leq t \right) \; \leq \; \Pro(\widetilde{\bM}_n \leq t)  
+ C \varrho_n, 
$$ 
with probability at least $1 - C\varrho_n$. Thus the proof is complete.
\end{proof}

\subsection{Proof of Lemma \ref{lemma:boostrap_size_He}}\label{proof:lemma:boostrap_size_He}
In this section, the constants may depend on $r,q,c_0, C_0, \ubar{\sigma}$, and  vary from line to line. Recall the definitions of $\bU_{n,A}^{\#}, \bU_{n,B}^{\#}, \bU_{n,*}^{\#}$ in \eqref{def:bootstrap_A_and_B} and \eqref{def:bootstrap_combined}.

\begin{proof}
\noindent{\underline{First claim}}. In the proof of Lemma \ref{lemma:size_H_discretize}, we have shown that for $m \in [M]$,
\begin{align*}
\Exp\left[\|W_P(h)\|_{\partial \cH_{m,\epsilon}} \right] \; \lesssim \; N^{-1} K_n^{1/2}.
\end{align*}
Further, by definition of $\partial \cH_{m,\epsilon} $ \eqref{def:H_m_epsilon}, we have
\begin{align*}
\sup_{h \in \partial \cH_{m,\epsilon} }\Exp\left[ (W_P(h))^2\right] & =
\sup_{h \in \partial \cH_{m,\epsilon} }\left(
r^2 \Var(P^{r-1}h(X_1)) + \alpha_n \Var(h(X_1^r))
\right)  \\
 & \lesssim
\sup_{h \in \partial \cH_{\epsilon}}\|h\|_{P^r,2}^2 \leq N^{-2}.
\end{align*}
By the Borell-Sudakov-Tsirelson concentration inequality \cite[Theorem 2.2.7]{gine2016mathematical}, for $m \in [M]$,
\begin{align*}
\Pro\left(\|W_P\|_{\partial \cH_{m,\epsilon}} \geq \Exp\left[\|W_P\|_{\partial \cH_{m,\epsilon}} \right]  +
\sqrt{2\sup_{h \in \partial \cH_{m,\epsilon}}\Exp\left[ (W_P(h))^2\right]  \log(Mn)}
\right) \leq M^{-1} n^{-1}.
\end{align*}
Thus by union bound and due to \ref{cond:MB},
\[
\Pro\left(\max_{m \in [M]}\|W_P\|_{\partial \cH_{m,\epsilon}} \geq C N^{-1} K_n^{1/2}\right) \leq   n^{-1},
\]
which  completes the proof of the first claim.

\vspace{0.5cm}
\noindent{\underline{Second claim}}.
By the Borell-Sudakov-Tsirelson concentration inequality \cite[Theorem 2.2.7]{gine2016mathematical}, for $m \in [M]$,
\begin{align*}
\Pro_{\vert \cD_n'}\left(
\|\bU_{n,*}^{\#}(h)\|_{\partial \cH_{m,\epsilon}} \; \geq \;
\Exp_{\vert \cD_n'} \left[ \|\bU_{n}^{\#}(h)\|_{\partial \cH_{m,\epsilon}} \right] +
\sqrt{2\Sigma_{n,*}^{(m)} \log(Mn)}
\right) \leq M^{-1}n^{-1}, 
\end{align*}
where $\Sigma_{n,*}^{(m)} := \sup_{h \in \partial \cH_{m,\epsilon}} \Exp_{\vert \cD_n'}\left[ \left( \bU_{n,*}^{\#}(h)\right)^2 \right]$.
Conditioned on $\cD_n'$, for $ m \in [M]$,
\begin{align*}
&\Exp_{\vert \cD_n'} \left[ \|\bU_{n,*}^{\#}(h)\|_{\partial \cH_{m,\epsilon}} \right] \leq r \Exp_{\vert \cD_n'} \left[ \|\bU_{n,A}^{\#}(h)\|_{\partial \cH_{m,\epsilon}} \right]
+ \alpha_n^{1/2} \Exp_{\vert \cD_n'} \left[ \|\bU_{n,B}^{\#}(h)\|_{\partial \cH_{m,\epsilon}} \right] \\
& \Sigma_{n,*}^{(m)} \leq r^2 \Sigma_{n,A}^{(m)} + \alpha_n \Sigma_{n,B}^{(m)}.
\end{align*}
where $\Sigma_{n,A}^{(m)}$ and $\Sigma_{n,B}^{(m)}$ are defined in Lemma \ref{lemma:aux_UB_bootstrap_size} and \ref{lemma:Sigma_nA_bound} (both ahead) respectively.

Recall that $\underline{\gamma}_{*} := \inf_{h \in \cH} \gamma_*(h)$ in \eqref{min_variance}. In particular, due to \eqref{MT0} and \eqref{MT4},
\[
\underline{\gamma}_{*} \geq \max\left\{ r^2 \ubar{\sigma}^2, \;\; \alpha_n c_0 B_n^{2} D_n^{-2}
\right\}
\]
Then by Lemma \ref{lemma:aux_UB_bootstrap_size} (ahead), for any $\Delta > 0$, with probability at least $1 - C N^{-1} - C\Delta - C\Delta^2$,  for each $m \in [M]$,
\begin{align*}
&C^{-1} \Delta \underline{\gamma}_{*}^{-1/2} \alpha_n^{1/2} \left( \Exp_{\vert \cD_n'}\left[\|\bU_{n,B}^{\#}\|_{\partial \cH_{m,\epsilon}} \right] +
(\Sigma_{n,B}^{(m)} \log(Mn))^{1/2} 
\right) \leq \\ 
&M^{1/q}N^{-1/2+1/q} B_n^{1-2/q} D_n^{2/q} K_n^2+ 
n^{-1/2+1/(2q)} D_n  K_n^{3/2} + n^{-3/4+1/(2q)} D_n^{2-1/q} K_n^{2}. 
\end{align*}
Further, by Lemma \ref{lemma:Sigma_nA_bound}
and \ref{lemma:U_nA_bootstrp_size_expect} (both ahead), for any $\Delta > 0$, with probability at least $1 - C \Delta^2 - C N_2^{-1}$, for each $m \in [M]$,
\begin{align*}
&C^{-1} \Delta \underline{\gamma}_{*}^{-1/2} \left( \Exp\left[\|\bU_{n,A}^{\#}\|_{\partial \cH_{m,\epsilon}} \right] +
(\Sigma_{n,A}^{(m)} \log(Mn))^{1/2} 
\right) \\
&\leq   (\Delta N^{-1/2} B_n K_n^{1/2}) + 
N_2^{-1/2} B_n K_n^{3/2}  
 + N_2^{-1+1/q}{M^{1/q} B_n^{2-2/q} D_n^{2/q-1} K_n^2}\\
& \quad + n^{-1/2+1/q} D_n  K_n + n^{-3/4+1/q} D_n^{2-2/q} K_n^{3/2} + n^{-1+1/q} D_n^{3-2/q} K_n^2.
\end{align*}
Combining above results and by union bound, for any $\Delta >0$, with probability at least $1 - C' (N \wedge N_2)^{-1} - C' \Delta - C' \Delta^2$,
\begin{align*}
\Pro_{\vert \cD_n'}\left( \max_{m \in [M]}
\|\bU_{n,*}^{\#}(h)\|_{\partial \cH_{m,\epsilon}} \; \geq \; C' \underline{\gamma}_{*}^{1/2} N^{-1/2} B_n K_n^{1/2} + 
C' \underline{\gamma}_{*}^{1/2}  \Delta^{-1}\chi_{n,*}^{\#}
\right) \leq n^{-1}, 
\end{align*}
which completes the proof.
\end{proof}

Recall the definition of $\bU_{n,B}^{\#}$ in \eqref{def:bootstrap_A_and_B}.

\begin{lemma}\label{lemma:aux_UB_bootstrap_size}
Let $\epsilon^{-1} := N \|1 + H^2\|_{P^r,2}$, and
denote 
$\Sigma_{n,B}^{(m)} := \sup_{h \in \partial \cH_{m,\epsilon}} \Exp_{\vert \cD_n'}\left[ \left( \bU_{n,B}^{\#}(h)\right)^2 \right]$ for $m \in [M]$. Assume \eqref{log_epsilon}, \ref{cond:PM}, \ref{cond:VC}, \ref{cond:MB}, \eqref{MT2}, \eqref{MT3}, and \eqref{MT4} hold. 
Then there exists a constant $C$, depending only on $q,r, c_0$, such that for any $\Delta > 0$, with probability at least $1 - C\Delta - C\Delta^2 -CN^{-1}$, the following two events hold:
\begin{align*}
C^{-1}\; \Delta\; \max_{m \in [M]} \Exp_{\vert \cD_n'} \left[ \|\bU_{n,B}^{\#}\|_{\partial \cH_{m,\epsilon}} \right] \leq    & M^{1/q} N^{-1/2+1/q} B_n^{2-2/q} D_n^{2/q-1} K_n^2+ \\
&n^{-1/2+1/(2q)} B_n  K_n^{3/2} +
n^{-3/4+1/(2q)}  B_n D_n^{1-1/q} K_n^2. \\
C^{-1} \; \Delta^2\;  \max_{m \in [M]} \Sigma_{n,B}^{(m)}  \leq &   N^{-1+2/q}  M^{2/q} B_n^{4-4/q} D_n^{4/q-2} K_n^2 +\\
& n^{-1+1/q} B_n^{2} K_n +
n^{-3/2+1/q}  B_n^2D_n^{2-2/q} K_n^2.
\end{align*}
\end{lemma}
\begin{proof}
\noindent{\underline{First event}}. Note that for   
$m \in [M]$, and $h_1,h_2 \in \partial \cH_{m,\epsilon}$,
\begin{align*}
&\|\bU_{n,B}^{\#}(h_1) - \bU_{n,B}^{\#}(h_2)\|_{\psi_2\vert \cD_n'}^2 \; 
\lesssim \;
\frac{1}{\widehat{N}^{(m)}}\sum_{\iota \in I_{n,r}} Z_{\iota}^{(m)}(h_1(X_{\iota})- h_2(X_{\iota}) - U_{n,N}'(h_1) + U_{n,N}'(h_2) )^2\\
\; = \;&
\frac{1}{\widehat{N}^{(m)}}\sum_{\iota \in I_{n,r}} Z_{\iota}^{(m)}(h_1(X_{\iota})- h_2(X_{\iota}))^2 - (U_{n,N}'(h_1) - U_{n,N}'(h_2) )^2\\
\leq \;&
\left({N}/{\widehat{N}^{(m)}} \right) \|h_1 - h_2\|_{\widehat{Q}_m,2}^2,
\end{align*}
where recall that $\widehat{Q}_m$ is defined in \eqref{def:two_random_measures}. Recall that 
$\widehat{V}_n := \max_{m \in [M]}\sup_{h \in \partial \cH_{\epsilon}} \|h\|_{\widehat{Q}_m,2}$ (see \eqref{def:V_hat_tilde} with $\cF$ replaced by $\partial \cH_{\epsilon}$), that
$\widehat{Q}$, which dominates $\widehat{Q}_m$ for $m \in [M]$, is defined in \eqref{def:two_random_measures}, and that $\partial \cH_{\epsilon}$ is defined in \eqref{def:H_all_epsilon}. Then by the entropy integral bound \cite[Corollary 2.2.5]{van1996weak},
\begin{align*}
&\left\| 
\sqrt{\widehat{N}^{(m)}/{N}}\, \|\bU_{n,B}^{\#}(h)\|_{ \partial \cH_{m,\epsilon}}\right\|_{\psi_2 \vert \cD_n'}\,
 \; \lesssim \;
\int_0^{\widehat{V}_n} \sqrt{1 + \log N(\partial\cH_{\epsilon}, \|\cdot\|_{\widehat{Q}_m,2}, \tau)} d\,\tau \\
\; &\lesssim \; 
\int_0^{\widehat{V}_n} \sqrt{1 + \log N(\partial\cH_{\epsilon}, \|\cdot\|_{\widehat{Q},2}, \tau)} d\,\tau 
\;\;\lesssim \; \|H\|_{\widehat{Q},2} \bar{J}(\widehat{V}_n/\|H\|_{\widehat{Q},2}).
\end{align*}
By maximal inequality \cite[Lemma 2.2.2]{van1996weak} and Corollary \ref{cor:aux_lemmas}, 
\begin{align*}
&\Exp\left[\max_{m \in [M]}
\sqrt{\widehat{N}^{(m)}/{N}}\, \|\bU_{n,B}^{\#}\|_{\partial \cH_{m,\epsilon}} 
\right] 
\;\lesssim\; \sqrt{\log(2M)} \Exp \left[ \|H\|_{\widehat{Q},2} \bar{J}(\widehat{V}_n/\|H\|_{\widehat{Q},2}) \right] \lesssim \\
&\sqrt{M \log(2M) }\bar{J}(\delta_r)\|H\|_{P^r,2} + \frac{ \log(2M) M^{1/q} \|H\|_{P^r,q} }{N^{1/2-1/q}} \frac{\bar{J}^2(\delta_r)}{\delta_r^2} \; + \;
\sqrt{\Delta' \log(2M) } \frac{  \bar{J}(\delta_r)}{\delta_r},
\end{align*}
where due to definition of $\partial \cH_{\epsilon}$, we may take
\begin{align*}
&\sigma_{r} := N^{-1}, \quad 
\delta_r := \sigma_{r}/\left( \sqrt{M} \|2  H \|_{P^r,2}\right) \geq (2N M^{1/2}\|H\|_{P^r,2})^{-1} ,\\ 
&\Delta' := \Exp\left[ \sup_{h \in \partial \cH_{\epsilon}}|I_{n,r}|^{-1} \sum_{\iota \in I_{n,r}} h^2(X_{\iota})
\right].
\end{align*}
Due to Theorem \ref{thm:lmax_sampling_proc}, \ref{cond:MB}, and \eqref{log_epsilon},
$\bar{J}(\delta_r)/\delta_r \lesssim K_n^{1/2}$. Thus due to Lemma \ref{lemma:discr_second_order}, \ref{cond:MB}, \eqref{MT3} and \eqref{MT4}, 
\begin{align*}
&\Exp\left[\max_{m \in [M]}
\sqrt{\widehat{N}^{(m)}/{N}}\, \|\bU_{n,B}^{\#}\|_{\partial \cH_{m,\epsilon}} 
\right]  
\lesssim N^{-1} K_n  + M^{1/q} N^{-1/2+1/q} B_n^{2-2/q} D_n^{2/q-1} K_n^2   \\
&  + B_n D_n^{-1} \left(
n^{-1/2+1/(2q)} D_n K_n^{1/2} +
n^{-3/4+1/(2q)}  D_n^{2-1/q} K_n  \right) K_n.
\end{align*}
Then the proof for the first inequality is complete due to Markov inequality and Lemma \ref{lemma:Bernstein}.

\vspace{0.2cm}
\noindent{\underline{Second event}}. Observe that
\begin{align*}
\max_{m \in [M]}\Sigma_{n,B}^{(m)} \;&=\; \max_{m \in [M]}\sup_{h \in \partial \cH_{m,\epsilon}} (\widehat{N}^{(m)})^{-1} \sum_{\iota \in I_{n,r}} Z_{\iota}^{(m)}\left(
h(X_{\iota}) - U_{n,N}'(h))\right)^2 \\
\;&=\;
\max_{m \in [M]}\sup_{h \in \partial \cH_{m,\epsilon}} \left((\widehat{N}^{(m)})^{-1} \sum_{\iota \in I_{n,r}} Z_{\iota} h^2(X_{\iota}) - (U_{n,N}'(h))^2\right) \\
\;&\leq\; \left(N/\widehat{N}^{(m)}\right) \max_{m \in [M]}\sup_{h \in \partial \cH_{m,\epsilon}}\|h\|_{\widehat{Q}_m,2}^2 =\left(N/\widehat{N}^{(m)}\right) (\widehat{V}_n)^2.
\end{align*}
Due to Corollary \ref{cor:aux_lemmas}, 
Lemma \ref{lemma:discr_second_order}, \eqref{MT3}, \eqref{MT4}, \ref{cond:MB}, and above calculations,
\begin{align*}
\Exp\left[\widehat{V}_n^2\right]
\lesssim & B_n^2 D_n^{-2} \left(n^{-1+1/q} D_n^{2} K_n +
n^{-3/2+1/q}  D_n^{4-2/q} K_n^2 
\right) \\
&+N^{-2} + N^{-1+2/q} M^{2/q} B_n^{4-4/q} D_n^{4/q-2} K_n^2 .
\end{align*}
Then the proof is complete by Markov inequality and Lemma \ref{lemma:Bernstein}.
\end{proof}

Recall the definition of $U_{n,A}^{\#}$ in \eqref{def:bootstrap_A_and_B}.

\begin{lemma} \label{lemma:Sigma_nA_bound}
Let $\epsilon^{-1} := N \|1 + H^2\|_{P^r,2}$, and 
denote 
$\Sigma_{n,A}^{(m)} := \sup_{h \in \partial \cH_{m,\epsilon}} \Exp_{\vert \cD_n'}\left[ \left( \bU_{n,A}^{\#}(h)\right)^2 \right]$ for $m \in [m]$. Assume \eqref{log_epsilon}, \ref{cond:PM},  \ref{cond:VC},\ref{cond:MB}, \eqref{MT1}, \eqref{MT2}, \eqref{MT3}, and \eqref{MT5} hold. 
Then there exists a constant $C$, depending only on $q,r$, such that for any $\Delta > 0$, with probability at least $1 - C\Delta^2 -CN_2^{-1}$, 
\begin{align*}
C^{-1} \; \Delta^2\; \max_{m \in [M]} \Sigma_{n,A}^{(m)} \leq & N_2^{-1}{B_n^2 K_n^2}
+ N_2^{-2+2/q}{M^{2/q} B_n^{4-4/q} D_n^{4/q-2} K_n^3}\\
&+  n^{-1+2/q} D_n^2 K_n+ n^{-3/2+2/q} D_n^{4-4/q} K_n^2 + n^{-2+2/q} D_n^{6-4/q} K_n^3.
\end{align*}
\end{lemma}
\begin{proof}
Observe first that by definition of $\partial \cH_{\epsilon}$ in \eqref{def:H_all_epsilon},
$\sup_{h \in \cH_{\epsilon}} \left\{ \|h\|_{P^r,2} \vee \|h\|_{P^r,4}^2 \right\} \leq N^{-1}$.
By definition of $U_{n,A}^{\#}$, for $m \in [M]$ and $h \in \partial \cH_{m,\epsilon}$,
\begin{align*}
&\Sigma_{n,A}^{(m)} = \sup_{h \in \partial \cH_{m,\epsilon}} n^{-1} \sum_{k=1}^{n} \left( \bG^{(k)}(h) - \bar{\bG}(h) \right)^2 
\leq \sup_{h \in \partial \cH_{m,\epsilon}} n^{-1} \sum_{k=1}^{n} \left( \bG^{(k)}(h)  \right)^2 \\
&\leq 4 \left(\max_{1 \leq k \leq n} N_2/\widehat{N}_2^{(k,m)} \right)^2 \left( I + II +  III \right), \;\; \text{ where } \\
&I := \max_{m \in [M]}\sup_{h \in \partial \cH_{m,\epsilon}} n^{-1} \sum_{k=1}^{n} \left( N_2^{-1} \sum_{\iota \in \kjack} (Z_{\iota}^{(k,m)} -q_n)h(X_{\iota^{(k)}})  \right)^2, \\
&II := \sup_{h \in \partial \cH_{\epsilon}} n^{-1} \sum_{k=1}^{n} \left( |\kjack|^{-1} \sum_{\iota \in \kjack} \left(h(X_{\iota^{(k)}}) -P^{r-1}h(X_k) \right)  \right)^2,\\
&III := \sup_{h \in \partial \cH_{\epsilon}}n^{-1} \sum_{k=1}^{n} \left(P^{r-1}h(X_k) \right)^2.
\end{align*}
As a result, $ \max_{ m \in [M]} \Sigma_{n,A}^{(m)}  \lesssim 
\left(\max_{k \in [n], m \in [M]} N_2/\widehat{N}_2^{(k,m)} \right)^2 (I + II +III)
$.

Due to Lemma \ref{lemma:Bernstein} and \ref{cond:MB}, 
\begin{align}\label{aux_Sigma_nA_bernstein}
\Pro\left(\max_{k \in [n], m \in [M]} N_2/\widehat{N}_2^{(k,m)} \leq C \right) \geq 1 - CN_2^{-1}.
\end{align}

By Lemma \ref{lemma:sampling_higher_order} and due to \ref{cond:VC} and \ref{cond:MB}, for each $k \in [n]$, 
\begin{align*}
&\Exp_{\vert X_k} \left[ \max_{m \in [M]}\sup_{h \in \partial \cH_{m,\epsilon}} \left( N_2^{-1} \sum_{\iota \in \kjack} (Z_{\iota}^{(k,m)} -q_n)h(X_{\iota^{(k)}})  \right)^2\right] \\
&\lesssim  N_2^{-1}K_n
\left(
K_n \|H(X_k,\cdot)\|_{P^{r-1},2}^2 
+ \frac{K_n^2 M^{2/q} \|H(X_k,\cdot)\|_{P^{r-1},q}^2}{N_2^{1-2/q}}
\right)
\end{align*}
Then due to \eqref{MT3}, we have
\begin{align}\label{aux_Sigma_nA_I}
\Exp\left[ I \right] \lesssim 
N_2^{-1}{B_n^2 K_n^2}
+ N_2^{-2+2/q}{M^{2/q} B_n^{4-4/q} D_n^{4/q-2} K_n^3}.
\end{align}

From the proof of \cite[Theorem 3.1]{chen2017jackknife} and similar to the proof for Lemma \ref{lemma:jacknife_complete_part},  
\begin{align*}
&\Exp\left[ II\right] \;\lesssim \;  \sum_{\ell=2}^{r-1} n^{-\ell} \|P^{r-\ell-1}H\|_{P^{\ell+1},2}^2 K_n^{\ell} + n^{-1} \sup_{h \in \partial \cH_{\epsilon}} \|P^{r-2}h\|_{P^2,2}^2  \\
&  + n^{-3/2} \sup_{h \in \partial \cH_{\epsilon}} \|(P^{r-2}h)^2\|_{P^2,2} K_n^{1/2} + n^{-2+2/q} \|(P^{r-2}H)^2\|_{P^2,q/2} K_n \\
& + n^{-2}\|(P^{r-2}H)^2\|_{P^2,2} K_n  \\
&+ n^{-1}  \sup_{h \in \partial \cH_{\epsilon}}  \| (P^{r-2}h)^{\bigodot 2} \|_{P^2,2} K_n  + n^{-3/2+2/q} \|(P^{r-2} H)^{\bigodot 2}\|_{P^2,q/2} K_n^2  \\
&+ n^{-3/2} \sup_{h \in \partial \cH_{\epsilon}} \|(P^{r-2}h)^2\|_{P^2,2} K_n^{3/2}  + n^{-2+2/q}\|(P^{r-2}H)^2\|_{P^2,q/2} K_n^3,
\end{align*}
Thus due to the definition of $\partial \cH_{\epsilon}$, \eqref{log_epsilon}, \eqref{MT5}, and \eqref{MT2},
\begin{align}\label{aux_Sigma_nA_II}
&\Exp\left[ II\right] \;\lesssim \;  n^{-3/2+2/q} D_n^{4-4/q} K_n^2 + n^{-2+2/q} D_n^{6-4/q} K_n^3.
\end{align}

By \cite[Theorem 5.2]{chernozhukov2014gaussian}, and due to \ref{cond:VC}, \eqref{log_epsilon},  and \eqref{MT1},
\begin{align}\label{aux_Sigma_nA_empirical_process}
\begin{split}
&\Exp[III] = \Exp\left[ \sup_{h \in \partial \cH_{\epsilon}}
 n^{-1} \sum_{k=1}^{n} \left(P^{r-1}h(X_k) \right)^2 \right] \\
 &\lesssim n^{-1/2} N^{-1} K_n^{1/2} + n^{-1+2/q} \| (P^{r-1}H)^2\|_{P,q/2} K_n \lesssim n^{-1+2/q} D_n^2 K_n.
 \end{split}
\end{align}

Then the proof is complete by  Markov inequality,  \eqref{aux_Sigma_nA_bernstein},
 \eqref{aux_Sigma_nA_I},  \eqref{aux_Sigma_nA_II}, and \eqref{aux_Sigma_nA_empirical_process}.
\end{proof}

\begin{lemma} \label{lemma:U_nA_bootstrp_size_expect}
Let $\epsilon^{-1} := N \|1 + H^2\|_{P^r,2}$, and recall the definition of $\Sigma_{n,A}^{(m)}$ in Lemma \ref{lemma:Sigma_nA_bound}. Assume the conditions in Lemma \ref{lemma:Sigma_nA_bound} hold. 
Then there exists a constant $C$, depending only on $r$, such that with probability at least $1-CN_2^{-1}$,
\begin{align*}
C^{-1}  \max_{m \in [M]}\Exp_{\vert \cD_n'} \left[ \|U_{n,A}^{\#}\|_{\partial \cH_{m,\epsilon}} \right] \leq    & N^{-1/2} B_n K_n^{1/2}  + 
\left(\max_{m \in [M]} \Sigma_{n,A}^{(m)} \right)^{1/2}K_n^{1/2}.
\end{align*}
\end{lemma}
\begin{proof}
Define $d_m(h_1, h_2) :=
\|\bU_{n,A}^{\#}(h_1) - \bU_{n,A}^{\#}(h_2)\|_{\psi_2 \vert \cD_n'}$ for $m \in [M]$ and $h_1, h_2 \in \partial \cH_{m,\epsilon}$.
By Cauchy-Schwarz inequality, for $m \in [M]$ and  $h_1,h_2 \in \partial \cH_{m,\epsilon}$,
\begin{align*}
&d^2_m(h_1, h_2) 
\;\lesssim\;n^{-1} \sum_{i = 1}^{n} \left(\bG^{(k)}(h_1) - \bG^{(k)}(h_2) - \overline{\bG}(h_1) + \overline{\bG}(h_2) \right)^2 \\
\; \leq \;&
n^{-1}\sum_{i = 1}^{n} \left(\bG^{(k)}(h_1) - \bG^{(k)}(h_2) \right)^2 \\
\; \leq \;&
n^{-1}\sum_{i = 1}^{n} (\widehat{N}_2^{(k,m)})^{-2}
\left(  \sum_{\iota \in \kjack} (Z_{\iota}^{(k,m)})^2\right) \left( \sum_{\iota \in \kjack} \left( h_1(X_{\iota^{(k)}}) - h_2(X_{\iota^{(k)}}) \right)^2 \right)
\\
\;\leq \;&
  n^{-1}  \sum_{k=1}^{n} 
\sum_{\iota \in \kjack} \left( h_1(X_{\iota^{(k)}}) - h_2(X_{\iota^{(k)}}) \right)^2 \\
\;= \;& \|h_1-h_2\|_{\overline{Q},2}^2,
\end{align*}
where $\overline{Q} := n^{-1}  \sum_{k=1}^{n} 
\sum_{\iota \in \kjack}  \delta_{X_{\iota^{(k)}}}$ is a random measure on $\cS^r$.
Since $(\partial \cH_{m,\epsilon}, 2H)$ is a VC type class with characteristics $(A,2\nu)$, we have for $\tau > 0$, 
\begin{align*}
N(\partial \cH_{m,\epsilon}, d_m(\cdot,\cdot), \tau \|2H\|_{\overline{Q},2} )
\leq N(\partial \cH_{m,\epsilon}, \|\cdot\|_{\overline{Q},2}, \tau \|2H\|_{\overline{Q},2} )
\leq (A/\tau)^{2\nu}.
\end{align*}
By Markov inequality and since $N_2 \leq |I_{n-1,r-1}|$,
\[
\Pro(\mathcal{E}') \geq 1 - N_2^{-1}, \; \text{ where } \; 
\mathcal{E}' := \left\{
\|H\|_{\overline{Q},2} \leq   n^{r-1} \|H\|_{P^r,2} \right\}.
\]
By entropy integral bound \cite[Theorem 2.3.7]{van1996weak}, if we use $2\left[N^{-1/2} \|H\|_{P^r,2} \vee (\Sigma_{n,A}^{(m)})^{1/2}\right]$ as a $d_m$-diameter for $\partial \cH_{m,\epsilon}$, we have on the event $\cE'$,
for $m \in [M]$,
\begin{align*}
&\Exp_{\vert \cD_{n}'} \left[ \| \bU_{n,A}^{\#}\|_{\partial \cH_{m,\epsilon}} \right]   \lesssim \int_0^{ N^{-1/2} \|H\|_{P^r,2} \vee (\Sigma_{n,A}^{(m)})^{1/2}} \sqrt{\nu \log(A \|H\|_{\bar{Q},2}/\tau)} d\tau  \\
& \lesssim  \left(  N^{-1/2} \|H\|_{P^r,2} \vee (\Sigma_{n,A}^{(m)})^{1/2} \right) \sqrt{\nu \log\left( A  n^{r-1} N^{1/2} \right)} \\
&\lesssim  \left(  N^{-1/2} \|H\|_{P^r,2} \vee (\Sigma_{n,A}^{(m)})^{1/2} \right) K_n^{1/2}.
\end{align*}
Then the proof is complete due to  \eqref{MT3}. 
\end{proof}

\section{More simulation results} \label{app:simulations}
We present more simulation results in the same setup as in Section \ref{sec:simulation}, except that in Subsection \ref{subsec:asym_ht}, the distribution of $\varepsilon$ in \eqref{eqn:nonparametric_regression}  has a different distribution than the centered Gaussian. 

\subsection{A figure for the locally convex regression function in \eqref{alt:local_conv}} \label{app:figure}
In Figure \ref{fig:FS_conv_dip}, we plot the regression function $f$ in \eqref{alt:local_conv} with $c_2 = 0.2, \omega_2 = 0.15$, together with one realization of dataset with the sample size $n = 1000$ and the variance $\sigma^2 = 0.2^2$. 

\begin{figure}[htbp!]
    \centering
\includegraphics[width=0.55\textwidth]{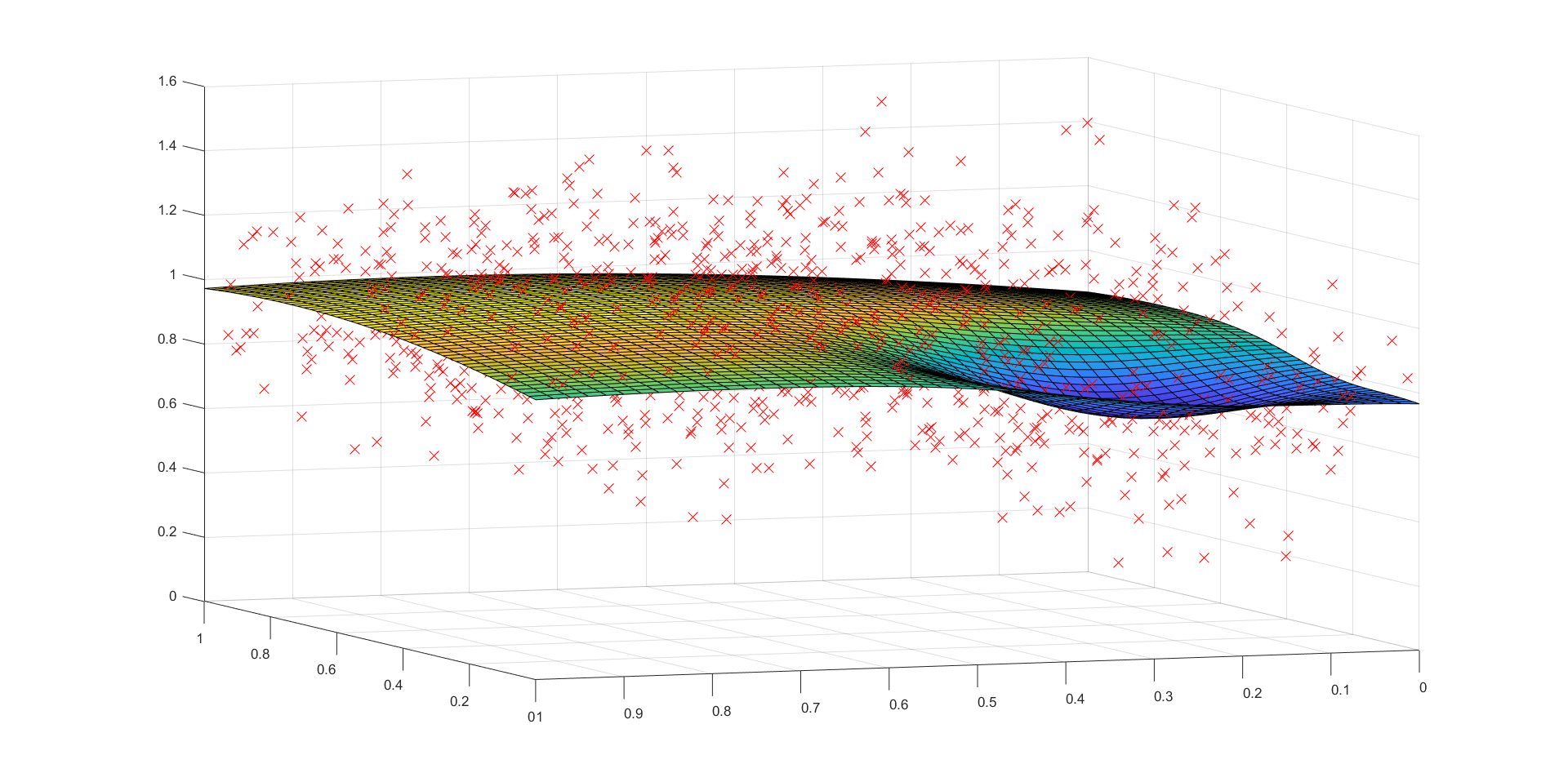} 
    \caption{The regression function $f$ in  \eqref{alt:local_conv} with  a realization of the data points.}%
    \label{fig:FS_conv_dip}%
\end{figure}

\subsection{Rejection probabilities for strictly concave functions}\label{app:d2_strictly_concave}
In this subsection, we present the rejection probabilities when the regression function is strictly concave, under the simulation setup in Section \ref{sec:simulation} with $d=2$.

In Table \ref{tab:strictlyConcave_size} (resp.~Table \ref{tab:multibn_size}), we list the rejection probabilities of our proposed procedure, with a single bandwidth $b_n = 0.5$ (resp. multiple bandwidths  $\cB_n = \{0.6,0.8,1\}$), for $f \in H_0$ in \eqref{alt:poly} with $0< \kappa_0 \leq 1$. 
As expected, the rejection probabilities becomes smaller as $\kappa_0$ decreases and $f$ becomes more ``concave".

\begin{table}[htbp!]
\begin{tabular}{cccccccc}
\hline
 $\kappa_0 =$    & $1$ & $0.95$   & $0.90$   & $0.85$    & $0.80$    & $0.75$    & $0.70$    \\ \hline
ID & 7.6 & 6.5 & 4.6 & 3.9 & 3.0 & 1.9 & 1.7  \\ \hline
SG & 8.6 & 6.0 & 5.0 & 3.4 & 2.9 & 2.1 & 1.5\\ \hline
\end{tabular}
\caption{The rejection probabilities (in percentage) of the proposed procedure  for concave functions $f \in H_0$ in \eqref{alt:poly} with varying $\kappa_0$ at level \textit{$10\%$} for $n = 1000$, $b_n = 0.5$, $\sigma = 0.2$.}
\label{tab:strictlyConcave_size}
\end{table}

\begin{table}[htbp!]
\begin{tabular}{cccccccc}
\hline
 $\kappa_0 =$    & $1$  & $0.95$   & $0.90$   & $0.85$    & $0.80$    & $0.75$    & $0.70$    \\ \hline
ID & 8.4 & 7.6 & 6.1 & 4.7 & 3.8 & 2.4 & 2.2  \\  \hline
\end{tabular}
\caption{The rejection probabilities (in percentage) of the proposed procedure with multiple bandwidth, $\{\cH_{b}^{\text{id}}: b \in \{0.6,0.8,1\}\}$, for concave functions $f \in H_0$ in \eqref{alt:poly} with varying $\kappa_0$ at level $10\%$ with $n = 1000$,  $\sigma = 0.5$.}
\label{tab:multibn_size}
\end{table}

\subsection{Asymmetric and heavy tailed noise.} \label{subsec:asym_ht}
In this subsection, we consider the simulation setup in Section \ref{sec:simulation} with $d=2$, but under asymmetric and heavy tailed noise.

The use of $\cH^{\text{sg}}$ requires the conditional distribution of $\varepsilon$ to be symmetric about zero, but otherwise allows $\varepsilon$ to have a heavy tail. On the other hand, to use $\cH^{\text{id}}$, our theory requires $\varepsilon$ to a light tail, but otherwise imposes no additional restriction. Specifically, we consider the following two types of error distributions. \\

\noindent \underline{\it Asymmetric distribution.} In the first, the distribution of $\varepsilon$ is a mixture of two Gaussian distributions: $G_{\varepsilon} := 0.5\times N(-0.1, 0.06) + 0.5\times N(0.1,0.24)$, under which $\Pro(\varepsilon < 0) > 0.5$. Further, if $(\Lambda_1,\ldots, \Lambda_{d+1})$ has a uniform distribution on the simplex $\{(v_1,\ldots,v_{d+1}): \sum_{i=1}^{d} v_i = 1, v_i >0\}$, and $\varepsilon_1,\ldots, \varepsilon_{d+2}$ are i.i.d.~with common distribution $G_{\varepsilon}$, then 
$\Pro(S < 0) < 0.5$, where $S := \sum_{i=1}^{d+1} \Lambda_i \varepsilon_i - \varepsilon_{d+2}$. The density of $G_{\varepsilon}$ and the histogram of $S$ are plotted in Figure \ref{fig:asym_noise}. In this case, $\cH^{\text{sg}}$  fails to achieve the prescribed level as  $P^r h_v^{\text{sg}} > 0$ for $v \in \sV_n$, while $\cH^{\text{id}}$ is  valid, which agrees with  results in Table \ref{tab:ht_asym}. \\

\noindent \underline{\it Heavy tail distribution.} In the second, the distribution of $\varepsilon$ is $0.2/\sqrt{3}\times t_3$, where $t_3$ means standard $t$ distribution with $3$ degree of freedom. From Table \ref{tab:ht_asym}, $\cH^{\text{id}}$ still achieves valid size, but its power is significantly smaller than that of  $\cH^{\text{sg}}$.  \\

 In Figure \ref{fig:asym_ht_bootstrap}, we plot the probability of rejection for $\cH^{\text{sg}}$ under the asymmetric noise ($\varepsilon \sim G_{\varepsilon}$), and 
 $\cH^{\text{id}}$ under the heavy tail, when the regression function is linear. As we can see, both tests are slightly undersized, but the approximation by bootstrap is reasonable over all levels, as predicted by our theory.
 
 \begin{figure}[htbp!]
 \centering
\includegraphics[width =0.4\textwidth]{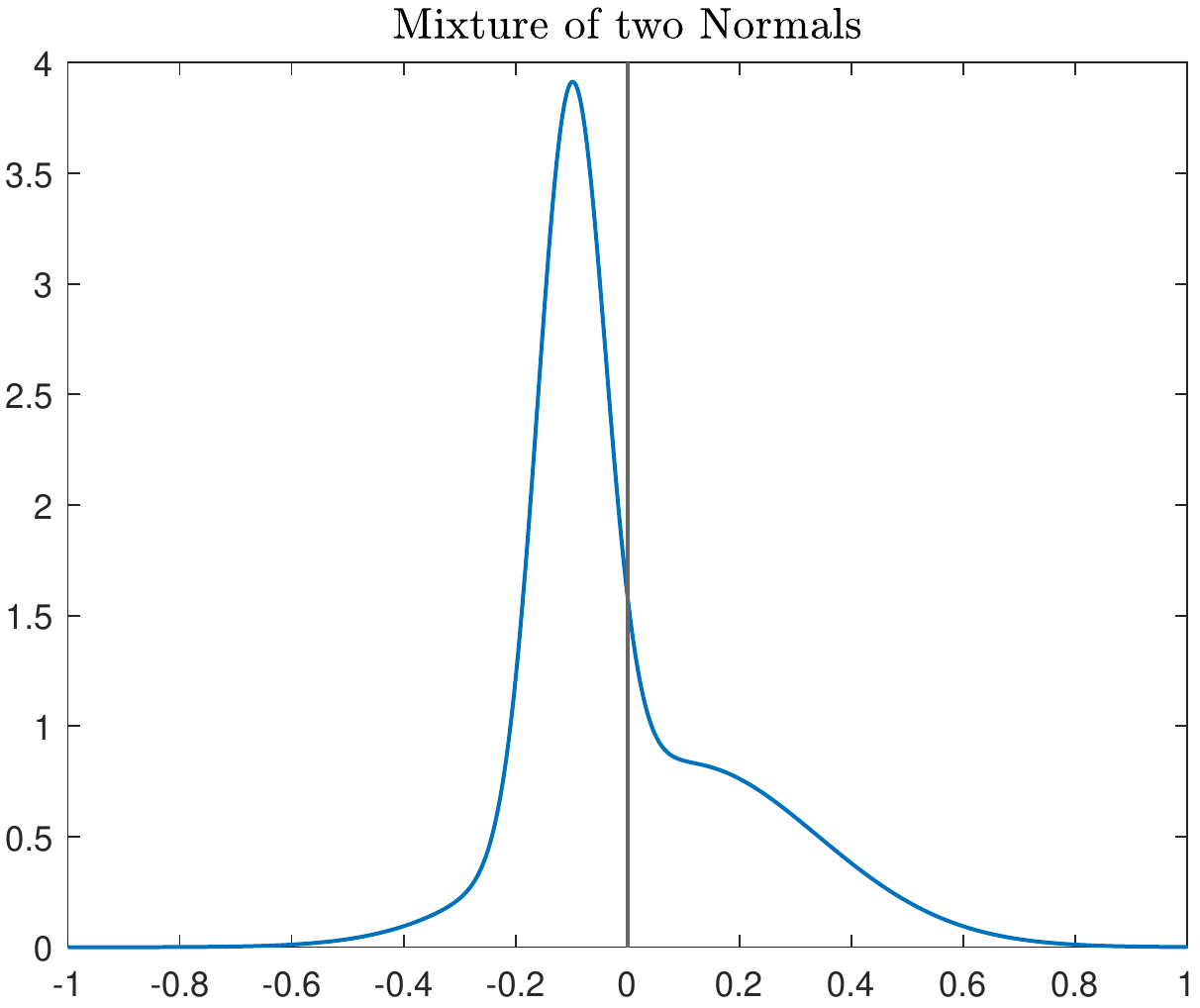}  \hspace{0.2cm}
\includegraphics[width =0.4\textwidth]{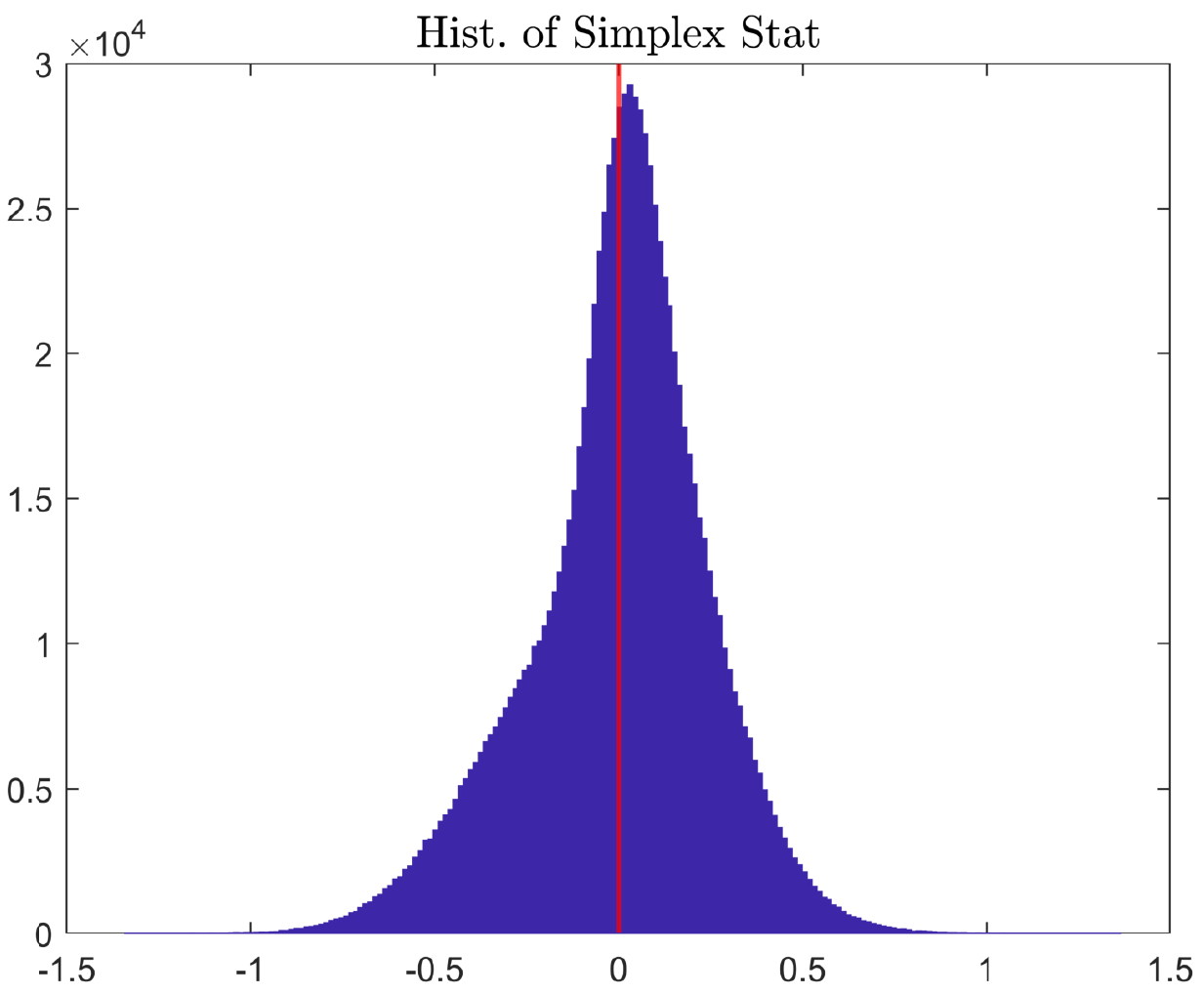} 
\caption{The left  is the density of $G_{\epsilon}$, while the right  the histogram for $S$.} 
\label{fig:asym_noise}
\end{figure}

\begin{table}[htbp!]
\begin{tabular}{ccccccc}
\hline
$\kappa_0 = 1$ (size) & HT  & Asymm  & \hspace{0.3cm} & $\kappa_0 = 1.5$ (alternative) & HT   & Asymm \\ \hline
ID           & 7.4 & 7.5    &  & ID             & 60.4 & 60.3  \\ \hline
SG           & 9.3 & 89.3**  &  & SG             & 80.9 & ---   \\ \hline
\end{tabular}
\caption{The rejection probability (in percentage) at level  $10\%$ for $d = 2$, $n = 1000$,  polynomial 
$f$ \eqref{alt:poly}, and two types of noise distributions. ``HT" is for heavy tail, while ``Asymm"  for asymmetric. '**' indicates the serious size inflation, and thus the corresponding power is irrelevant (`---').}
\label{tab:ht_asym}
\end{table}

 \begin{figure}[htbp!]
\includegraphics[width =0.48\textwidth]{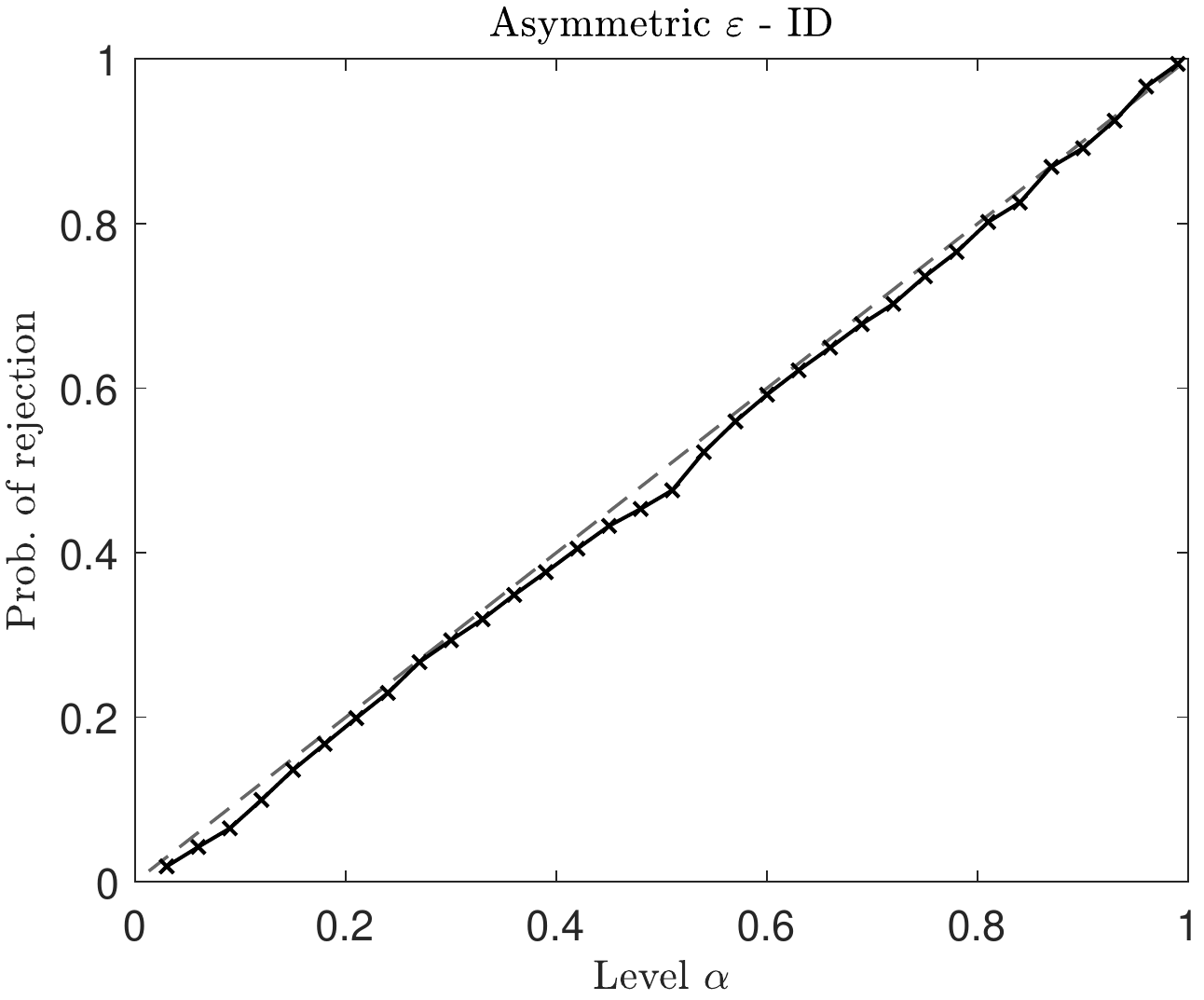} 
\includegraphics[width =0.48\textwidth]{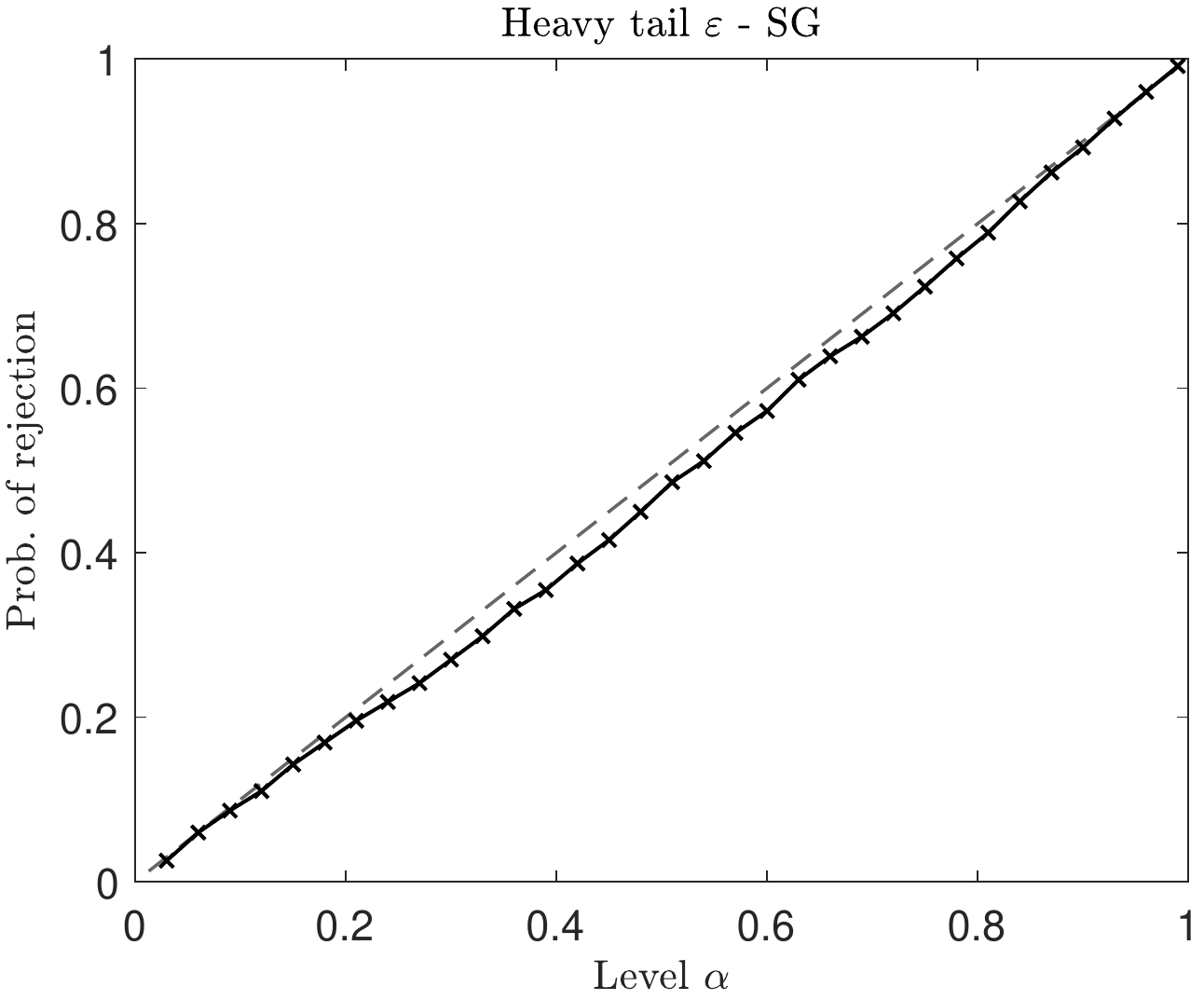} 
\caption{The x axis is prescribed level. The dashed line is $y=x$, and the solid line with crosses is the actual probability of rejection under a linear regression function.} 
\label{fig:asym_ht_bootstrap}
\end{figure}

\subsection{Size validity - Gaussian noise}

In Table \ref{tab:sv_d3}, \ref{tab:d4_poly} and \ref{tab:sv_d2_n1500}, we list the size for different bandwidth $b_n$ and error variance $\sigma^2$ at levels $5\%$ and $10\%$ for $d = 3$ and $d = 4$ (the column with $\kappa_0 = 1$), as well as $d = 2, n = 1500$. 
In Figure \ref{fig:d3_d4}, we also show the probability of rejection  over all levels ranging from $(0, 1)$.
Similar to $d = 2$,  the proposed procedure is consistently on the conservative side.

\begin{table}[htbp!]
\begin{tabular}{cccccccc}
\hline
$d=3$         & \multicolumn{3}{c}{$n=500$}       &  & \multicolumn{3}{c}{$n=1000$}       \\ \cline{2-4} \cline{6-8}
$\sigma =0.1$ & $b_n=0.7$ & $b_n=0.65$ & $b_n=0.6$ &  & $b_n=0.65$ & $b_n=0.6$ & $b_n=0.55$ \\ \hline
Level = 5\%     & 2.5       & 2.3        & 1.8       &  & 3.2        & 2.8       & 2.5        \\ 
Level = 10\%    & 6.3       & 4.8        & 3.3       &  & 7.1        & 7.1       & 5.2        \\ \hline
              &           &            &           &  &            &           &            \\ 
$\sigma =0.2$    & $b_n=0.7$ & $b_n=0.65$ & $b_n=0.6$ &  & $b_n=0.65$ & $b_n=0.6$ & $b_n=0.55$ \\ \hline
Level = 5\%     & 2.9       & 1.4        & 1.2       &  & 3.5        & 3.8       & 1.8        \\ 
Level = 10\%    & 6.8       & 4.3        & 4.1       &  & 7.4        & 8.1       & 5.5        \\ \hline
\end{tabular}

\vspace{0.3cm}
\begin{tabular}{cccc}
\hline
$d=3, n =1500, \sigma= 0.2$ & $b_n =0.6$ & $b_n =0.55$ & $b_n =0.5$ \\ \hline
Level = 5 \%                  & 4.0        & 3.4        & 2.1        \\ \hline
Level = 10\%                  & 8.2        & 7.2        & 7.1        \\ \hline
\end{tabular}

\caption{Size validity using $\cH^{\text{sg}}$ for $d=3$ under the Gaussian noise. The sizes are in the unit of percentage.}
\label{tab:sv_d3}
\end{table}

\begin{table}[htbp]
\begin{tabular}{cccc}
\hline
$\kappa_0$      & 1 (size)  & 1.2  & 1.5  \\ \hline
Level $5\%$  & 3.4 & 23.6 & 83.3 \\ \hline
Level $10\%$ & 6.8 & 37.7 & 95.1 \\ \hline
\end{tabular}
\caption{The rejection probability using $\cH^{\text{sg}}$ (in percentage) at level $5\%$ and $10\%$ for $d = 4$, $n =2000$, $b_n = 0.7$, Gaussian noise $N(0,0.2^2)$, and polynomial regression function \eqref{alt:poly}. Note that $\kappa_0 = 1$ corresponds to a linear function.}
\label{tab:d4_poly}
\end{table}

\begin{table}[htbp!]
\begin{tabular}{cccccccc}
\hline
$d =2 , n=1500$       & \multicolumn{3}{c}{Level 5\%}        &  & \multicolumn{3}{c}{Level 10\%}       \\ \cline{2-4} \cline{6-8}
$\sigma = 0.2$ & $b_n=0.45$ & $b_n=0.4$ & $b_n=0.35$ &  & $b_n=0.45$ & $b_n=0.4$ & $b_n=0.35$ \\ \hline
ID             & 3.9        & 3.6       & 3.1        &  & 8.4        & 7.4       & 6.9        \\ \hline
SG             & 4.0        & 3.9       & 2.6        &  & 7.4        & 8.6       & 7.3        \\ \hline
\end{tabular}
\caption{Size validity using $\cH^{\text{sg}}$ for $d=2, n= 1500$ under the Gaussian noise. The sizes are in the unit of percentage.}
\label{tab:sv_d2_n1500}
\end{table}

 \begin{figure}[htbp!]
\includegraphics[width =0.48\textwidth]{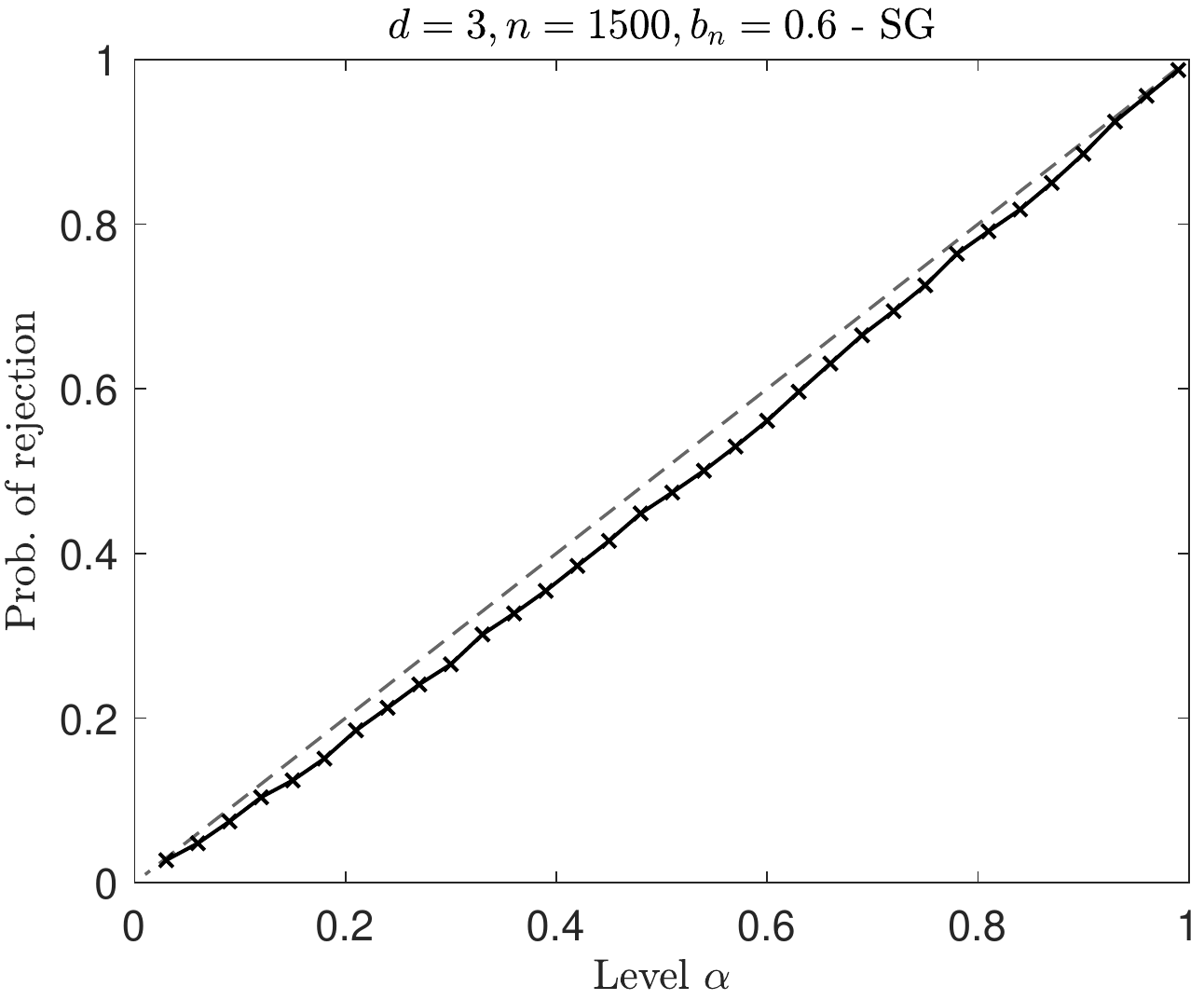} 
\includegraphics[width =0.48\textwidth]{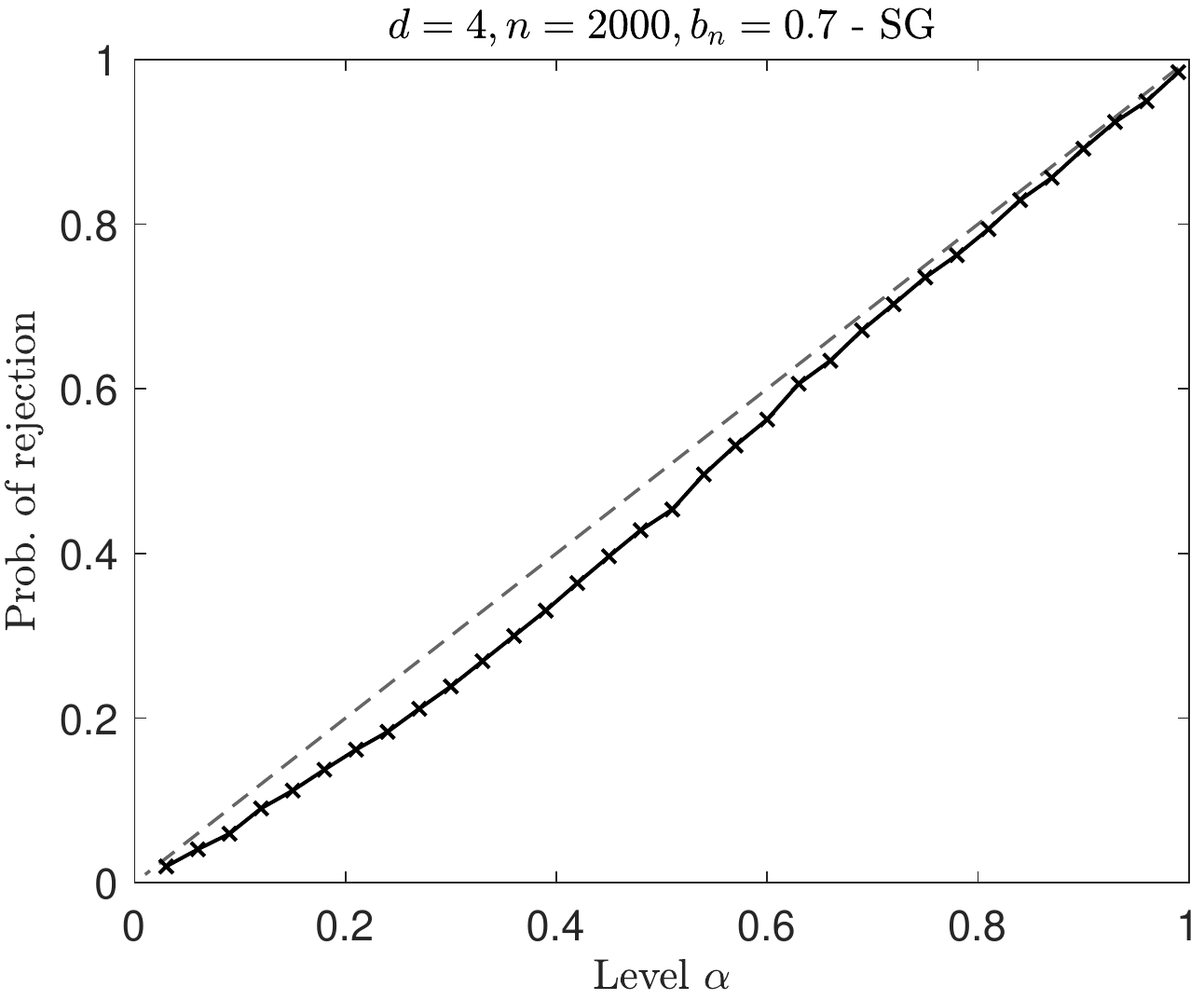} 
\caption{The x axis is prescribed level. The dashed line is $y=x$, and the solid line with crosses is the actual rejection probability    using $\cH^{\text{sg}}$ under a linear regression function and $\varepsilon \sim N(0,0.2^2)$.} 
\label{fig:d3_d4}
\end{figure}

\subsection{Power - Gaussian Noise}
For the polynomial regression functions \eqref{alt:poly}, in Table \ref{tab:d3_power} and \ref{tab:d4_poly}, we list the power using $\cH^{\text{sg}}$ for $d = 3$ and $d = 4$. In Table \ref{tab:d3_power}, the power for locally convex regression functions \eqref{alt:local_conv} and $d=3$ is also listed.


\begin{table}[htbp!]
\subfloat[Polynomial $f$ \eqref{alt:poly} for  $\cH^{\text{sg}}$, varying $\kappa_0$, and $\varepsilon \sim N(0,\sigma^2)$ ]{
\begin{tabular}{ccccccc}
\hline
$d=3$                   &  & \multicolumn{2}{c}{Level 5\%}           &  & \multicolumn{2}{c}{Level 10\%}          \\ 
\cline{2-4} \cline{5-7} 
$(\kappa_0, \sigma) \; /\; (n, b_n)$                      &  & $(500, 0.7)$ & $(1000, 0.65)$ &  & $(500, 0.7)$ & $(1000, 0.65)$ \\ \hline
$(1.2, 0.1)$ &  & 29.7              & 50.4                &  & 51.5              & 73.7                \\ 
$(1.2, 0.2)$  &  & 9.5               & 18.2                &  & 22.6              & 34.1                \\ 
$(1.5,  0.1)$ &  & 91.4              & 100                 &  & 98.5              & 100                 \\ 
$(1.5, 0.2)$  &  & 38.7              & 65.8                &  & 61.2              & 84.7                \\ \hline
\end{tabular}
}\\
\subfloat[
Locally convex $f$ \eqref{alt:local_conv} 
for  $\cH^{\text{sg}}$, varying $(c_2, \omega_2)$, and $\varepsilon \sim N(0,0.2^2)$.
]{
\begin{tabular}{ccccc}
\hline
$d = 3, n =1500, b_n = 0.55$            & \multicolumn{2}{c}{Level 5\%}       & \multicolumn{2}{c}{Level 10\%}      \\ \cline{2-5} 

            & $\omega_2 = 0.15$ & $\omega_2 = 0.2$ & $\omega_2 = 0.15$ & $\omega_2 = 0.2$ \\ \hline
$c_2 = 0.2$ & 12.0              & 12.6             & 21.0              & 17.9             \\
$c_2 = 0.3$ & 35.3              & 32.7             & 46.5              & 44.5             \\ 
$c_2 = 0.4$ & 64.4              & 60.0             & 73.9              & 69.2             \\ 
$c_2 = 0.5$ & 79.7              & 82.4             & 86.6              & 88.1             \\ 
$c_2 = 0.6$ & 89.9              & 91.1             & 95.0              & 94.2             \\ \hline
\end{tabular}
}

\caption{The rejection probability  using $\cH^{\text{sg}}$ (in percentage) at level $5\%$ and $10\%$ for $d = 3$.}
\label{tab:d3_power}
\end{table}

\section{Proofs and discussions for concavity test}

\subsection{Proof of Theorem \ref{thm:concavity_master} - Identity kernel}
\label{proof:thm:id_kernel}
We will prove Theorem \ref{thm:concavity_master} for $\cH^{\text{id}}$ under the following condition \ref{C:id_noise}, instead of \ref{C:id_noise_prime}, and apply Theorem \ref{thm:GAR_sup_incomplete_Uproc} and  \ref{thm:uproc_bootstrap}, instead of Theorem \ref{thm:calibration_inf}. \\

\noindent \mylabel{C:id_noise}{(C6-id)}. 
Assume that 
$\sup_{v \in \cV^{b_n}} |f(v)| \leq C_0$, and that for some $\beta > 0$,
\begin{equation}
\label{equ:id_noise}
\inf_{v \in \cV^{b_n}} \Var\left(\varepsilon \vert V = v
\right) \geq 1/C_0, \qquad
\sup_{v \in \cV^{b_n}} \|\epsilon\|_{\psi_{\beta} \vert V = v} \leq C_0.
\end{equation}

\begin{remark}
Clearly, \ref{C:id_noise_prime} implies \ref{C:id_noise}, since \ref{C:id_noise} only requires $\varepsilon$ to have a light tail, instead of being bounded. Further, under the condition \ref{C:id_noise}, 
in Theorem \ref{thm:concavity_master}, Corollary \ref{cor:size_valid} and Corollary \ref{cor:power}, 
the constant $C$  may thus depend on $\beta$.
\end{remark}

\begin{proof}[Proof of Theorem \ref{thm:concavity_master} under \ref{C:id_noise} - Identity kernel] 
First, let
\begin{align}\label{q:id_kernel}
q := \max\left\{ \frac{4r}{\kappa}, \;\;
2 + \frac{2C_0 + 2}{\kappa \wedge \kappa'}, \;\; \frac{2 + C_0}{3}, \;\; 6
\right\} + 1.
\end{align}
By \ref{C:M} and \ref{C:bn}, $M \leq n^{C_0}$ and 
$b_n^{-d/2} \leq C_0 n^{(1-1/C_0)/3} \leq C_0 n^{1/3}$.
Then due to Theorem \ref{thm:GAR_sup_incomplete_Uproc}, Theorem \ref{thm:uproc_bootstrap}, and the definition of $q$ in \eqref{q:id_kernel}, it suffices to verify that \eqref{MT1}-\eqref{MT5} holds with  above $q$, and 
\begin{align}\label{concav_B_D}
 D_n :=  C b_n^{-d/2}, \; B_n :=  C b_n^{-dr/2}, \;
 K_n \leq C \log(n),
\end{align}
in addition to  \ref{cond:PM}, \ref{cond:VC}, \ref{cond:MB}, and \eqref{MT0}.



Due to \ref{C:kernel}  and \ref{C:id_noise}, 
we consider the envelope function(s) $H_n: \bR^{(d+1)\times r} \to \bR$ for $\cH^{\text{id}}$:
\begin{align} \label{equ:H}
\begin{split}
H_n(x_1,\ldots,x_r) := &\left(2C_0 + \sum_{i=1}^{r}|y_i - f(v_i)|  \right) C_0^r \times\\ 
&b_n^{-d(r-1/2)} \prod_{1 \leq i < j \leq r}  \mathbbm{1} \left\{ \frac{|v_i - v_{j}|}{b_n}  \in (-1, 1)^d \right\}
\prod_{i=1}^{r} \mathbbm{1}\left\{ v_i \in \cV^{b_n}
\right\},
\end{split}
\end{align}
where  for $i \in [r]$, $x_i := (v_i,y_i)$ with $v_i \in \bR^d$ and $y_i \in \bR$.\\

\noindent \underline{Verify \ref{cond:PM}, \ref{cond:VC}, \ref{cond:MB}, and \eqref{MT0}.}
By \cite[Proposition 3.6.12]{gine2016mathematical}, if $L(\cdot)$ is of bounded variation (see \ref{C:kernel}),  $(\cH^{\text{id}}, H_n)$ is a VC type class for  some absolute constants $(A,\nu)$, which also implies that $K_n \leq C_0 \log(n)$. The conditions \ref{cond:PM}, \ref{cond:MB}, and \eqref{MT0} are satisfied due to \ref{C:kernel}, \ref{C:M}, and \ref{C:non_degenerate} respectively.\\

\noindent \underline{Verify the bounds involving $H_n$ in \eqref{MT1}-\eqref{MT5}.} 
Due to \ref{C:id_noise}, for $q$ in \eqref{q:id_kernel} and $1 \leq s \leq 4$,
there exists a constant $C$ depending on $\beta, C_0, r, \kappa, \kappa'$ such htat
\begin{align}\label{aux:bdd_moments}
\Exp\left[ |\epsilon|^t \vert V = v \right] \leq C, \text{ for any } v \in \cV^{b_n}, \text{ and } t \leq 4q.
\end{align}
Then for $1 \leq s \leq 4$  and  $1 \leq \ell \leq r$, due to \ref{C:X},
\begin{align*}
&P^{r-\ell}H_n^s  \lesssim 
(1 + \sum_{i=1}^{\ell}|\varepsilon_i|^s) b_n^{-sd(r-1/2)} \;\;\times \\
&\int  \prod_{1\leq i < j \leq r}  \mathbbm{1}\left\{ \frac{|v_i - v_{j}|}{b_n}  \in (-1, 1)^d \right\} \prod_{i=1}^{r} \mathbbm{1}\{v_i \in \cV^{b_n}\}p(v_i)
d{v_{\ell+1}} \ldots d{v_r} \\
& \lesssim (1 + \sum_{i=1}^{\ell}|\varepsilon_i|^s)  b_n^{-sd(r-1/2)} b_n^{d (r-\ell)} \prod_{1 \leq i < j \leq \ell} \mathbbm{1}\left\{ \frac{|v_i - v_{j}|}{b_n}  \in (-1, 1)^d \right\}\prod_{i=1}^{\ell} \mathbbm{1}\{v_i \in \cV^{b_n}\}.
\end{align*}
Now again due to \eqref{aux:bdd_moments} and \ref{C:X}, for $1 \leq s \leq 4$  and  $1 \leq \ell \leq r$,
\begin{align*}
&\| P^{r-\ell}H_n^s \|_{P^{\ell}, q}
 \lesssim b_n^{- sd(r-1/2)} b_n^{ d (r-\ell)} b_n^{d(\ell - 1)/q}
 =  b_n^{-d\left(
r(s-1) + \ell(1-1/q) -(s/2-1/q)
\right)}, \\
&\| P^{r-\ell}H_n^s \|_{P^{\ell}, 2}
 \lesssim b_n^{- sd(r-1/2)} b_n^{ d (r-\ell)} b_n^{d(\ell - 1)/2}
 =  b_n^{-d\left(
r(s-1) + \ell/2 -(s/2-1/2)
\right)}.
\end{align*}
Recall in \eqref{concav_B_D} that $D_n =  C b_n^{-d/2}$ and $B_n =  C b_n^{-dr/2}$.
Then the bounds involving $H_n$ in \eqref{MT1}-\eqref{MT5} are verified, except for $\|(P^{r-2}H)^{\bigodot 2}(x_1, x_2)\|_{P^2,q/2}$ in \eqref{MT5}, on which we now focus. 
With $\ell = 2, s = 1$, we have
\begin{align*}
&(P_n^{r-2}H)^{\bigodot 2}(X_1, X_2) \lesssim  \\
&(1 + |\varepsilon_1|  + |\varepsilon_2|) b_n^{-2d(r-1/2)} b_n^{2d(r-2)}  \mathbbm{1}\{V_1 \in \cV^{b_n}\} 
 \mathbbm{1}\{V_2 \in \cV^{b_n}\} 
 \;\times \\
&\int
\mathbbm{1}\left\{ \frac{|v_3 - V_1|}{b_n} \in (-1,1)^d\right\} \left\{ \frac{|v_3 - V_2|}{b_n} \in (-1,1)^d\right\} 
 \mathbbm{1}\{v_3 \in \cV^{b_n}\} 
p(v_3) dv_3 \\
&\lesssim 
(1 + |\varepsilon_1| + |\varepsilon_2|) b_n^{-2d}  \mathbbm{1}\left\{ \frac{|V_1 - V_2|}{b_n} \in (-2,2)^d\right\} 
\mathbbm{1}\{V_1 \in \cV^{b_n}\} 
 \mathbbm{1}\{V_2 \in \cV^{b_n}\},
\end{align*}
which implies that $\|(P^{r-2}H)^{\bigodot 2}(x_1, x_2)\|_{P^2,q/2} \leq D_n^{4-4/q}$.\\

\noindent \underline{Verify the \textbf{upper} bounds involving $\{h_v^{\text{id}}: v \in \cV \}$ in \eqref{MT1}-\eqref{MT5}.} We will write $h_v$ for $h_v^{\text{id}}$ to simplify notations. Due to \eqref{aux:bdd_moments}, \ref{C:X}, and \ref{C:id_noise}, for $1 \leq s \leq 4$ ,  $0 \leq \ell \leq r$, and 
$v \in \cV$,
\begin{align*}
&P^{r-\ell}|h_v|^s \lesssim (1 + \sum_{i=1}^{\ell}|\varepsilon_i|^s) b_n^{-sd(r-1/2)} b_n^{d(r-\ell)}
\prod_{i=1}^{\ell} \left|L\left(\frac{v - V_i}{b_n} \right) \right|^s,
\end{align*}
which  implies that for $q' \in \{2,3,4,q\}$,
\begin{align*}
&\|P^{r-\ell}|h_v|^s\|_{P^{\ell},q'} \lesssim   b_n^{-sd(r-1/2)} b_n^{d(r-\ell)} b_n^{d\ell/q'} = b_n^{-d\left(r(s-1) + \ell(1-1/q') - s/2 \right)}.
\end{align*}
Recall in \eqref{concav_B_D} that $D_n =  C b_n^{-d/2}$ and $B_n =  C b_n^{-dr/2}$. Then the upper bounds involving $\{h_v^{\text{id}}: v \in \cV \}$ in \eqref{MT1}-\eqref{MT5} are verified.\\

\noindent \underline{Verify the \textbf{lower} bound in \eqref{MT4}.} By definition of $w(\cdot)$ in \eqref{def:id_kernel} and due to \ref{C:id_noise},
\begin{align*}
& \Var\left( h_v(X_1^r) \,\vert\, V_1^r\right) 
= b_n^{-2d(r-1/2)} \prod_{i=1}^{r} L^2\left(\frac{v - V_i}{b_n}\right) \times \\ &\sum_{j=1}^{r} \mathbbm{1}\left\{V_1^r \in \calS_j \right\} 
\Exp\left[
\left( \sum_{ i \in [r]\setminus\{j\}}  \tau^{(j)}_i(v_1^r)\, \varepsilon_i - \varepsilon_j \right)^2 \;\vert \; V_1^r\right] \\
&\geq  \frac{1}{C_0} b_n^{-2d(r-1/2)} \prod_{i=1}^{r} L^2\left(\frac{v - V_i}{b_n}\right)  \mathbbm{1}\left\{V_1^r \in \calS \right\}.
\end{align*}
Then for $v \in \cV$, due to \ref{C:X} and $\{(v_1,\ldots,v_r) \in \calS\} = \{ (\frac{v-v_1}{b_n},\ldots,  \frac{v-v_r}{b_n}) \in \calS\}$, we have
\begin{align*}
&\Var\left( h_v(X_1^r) \right) \geq 
\Exp\left[ \Var\left( h_v(X_1^r) \,\vert\, V_1^r\right) \right] \\
& \geq 
\frac{1}{C_0} b_n^{-2d(r-1/2)} b_n^{dr} \int  \prod_{i=1}^{r} L^2\left( u_i\right)  \mathbbm{1}\left\{u_1^r \in \calS \right\}
\prod_{i=1}^{r} p\left( v - b_n u_i\right) du_1\ldots du_r \\
& \geq 
\frac{1}{C_0^{r+1}} b_n^{-dr + d} \int  \prod_{i=1}^{r} L^2\left( u_i\right)  \mathbbm{1}\left\{u_1^r \in \calS\right\}
 du_1\ldots du_r,
\end{align*}
which verifies the lower bound in \eqref{MT4} due to \ref{C:kernel} and the definitions of $B_n, D_n$ in \eqref{concav_B_D}.
\end{proof}

\subsection{Proof of Theorem \ref{thm:concavity_master} - Sign kernel}
\label{proof:thm:sg_kernel}
We will prove Theorem \ref{thm:concavity_master} for $\cH^{\text{sg}}$ under the following condition \ref{C:sg_noise}, instead of \ref{C:sg_noise_prime}.\\

\noindent \mylabel{C:sg_noise}{(C6-sg)}. Assume that 
$|f(v) - f(v')| \leq 1/C_0$ for $v,v' \in \cV^{b_n}$ such that $ \|v - v'\|_\infty \leq 1/C_0$; $|b_n| \leq 1/C_0$ for $n \geq C_0$; for each $v \in \cV^{b_n}$, 
conditional on $V = v$,   $\varepsilon$ has a symmetric distribution, i.e., $\Pro\left( \epsilon > t \vert V = v \right) = \Pro\left( \epsilon < -t \vert V = v \right)$ for any $t > 0$, and that
\begin{equation}
\label{equ:sg_noise}
\inf_{v \in \cV^{b_n}} \Pro\left( |\varepsilon| \geq C_0^{-1} \vert V = v
\right) \geq C_0^{-1}.
\end{equation}

\begin{remark}
Although \ref{C:sg_noise_prime} does not implies \ref{C:sg_noise}, the only difference is in the proof of Lemma \ref{lemma:sg_gamma_B_lower}, which follows from similar but simpler arguments if we work with the condition \ref{C:sg_noise_prime}.
\end{remark}

\begin{proof}[Proof of Theorem \ref{thm:concavity_master} under  \ref{C:sg_noise_prime} - Sign kernel] 
Now we focus on the class $\cH^{\text{sg}}$ and apply Theorem \ref{thm:calibration_inf}. Due to \ref{C:kernel}, 
we consider the envelope function(s) $H_n: \bR^{(d+1)\times r} \to \bR$ for $\cH^{\text{sg}}$:
\begin{align*} 
\begin{split}
H_n(x_1,\ldots,x_r) := C_0^r b_n^{-d(r-1/2)} \prod_{1 \leq i < j \leq r}  \mathbbm{1} \left\{ \frac{|v_i - v_{j}|}{b_n}  \in (-1, 1)^d \right\}
\prod_{i=1}^{r} \mathbbm{1}\left\{ v_i \in \cV^{b_n}
\right\},
\end{split}
\end{align*}
where  for $i \in [r]$, $x_i := (v_i,y_i)$ with $v_i \in \bR^d$ and $y_i \in \bR$. The other conditions in Theorem \ref{thm:calibration_inf}, except for the lower bound in \eqref{MT4} (about $\Var(h(X_1^r))$), can be verified in the same way as for $\cH^{\text{id}}$ (see Subsection \ref{proof:thm:id_kernel}) with $D_n = C b_n^{-d/2}$; since $H_n$ and $\{h_v^{\text{sg}} : v \in \cV\}$ are bounded, the arguments are simpler, and thus omitted. Now we focus on verifying the lower bound in \eqref{MT4}, and we will write $h_v$ for $h_v^{\text{sg}}$ to simplify notations.

By Lemma \ref{lemma:sg_gamma_B_lower} (ahead), there exists a constant $C$ only depending on $C_0$ such that for any $(u_1,\ldots,u_r) \in (-1/2,1/2)^{r} \cap \calS$ and $v \in \cV$, 
\begin{align*}
\textup{Var}\left( \left. w(X_1,\ldots, X_r) \; \right\vert \; V_1 = v - b_n u_1, \;\ldots,\; V_r = v - b_n u_r
\right) \geq C^{-1}.
\end{align*}

By definition of $w(\cdot)$ in \eqref{def:id_kernel},
\begin{align*}
& \Var\left( h_v(X_1^r) \vert V_1^r\right) 
= b_n^{-2d(r-1/2)} \prod_{i=1}^{r} L^2\left(\frac{v - V_i}{b_n}\right)
\textup{Var}\left( \left. w(X_1,\ldots, X_r) \; \right\vert \; V_1^r \right)
\end{align*}
Thus due to \ref{C:sg_noise}, we have
\begin{align*}
&\Var\left( h_v(X_1^r) \right) \geq 
\Exp\left[ \Var\left( h_v(X_1^r) \vert V_1^r\right) \right] = b_n^{-2d(r-1/2)} b_n^{dr} \times \\
&\int 
\left( \prod_{i=1}^{r} L^2\left(u_i\right) p(v - b_n u_i) \right)
\textup{Var}\left( \left. w(X_1,\ldots, X_r) \; \right\vert \; V_i = v - b_n u_i \text{ for } i \in [r] \right) du_1\ldots du_r \\
& \geq C_0^{-r} b_n^{-dr + d} C^{-1}  
\int  \left( \prod_{i=1}^{r} L^2\left(u_i\right) \right) \mathbbm{1}\{u_1^r \in \calS\} du_1\ldots du_r,
\end{align*}
which verifies the lower bound in \eqref{MT4} due to \ref{C:kernel}.
\end{proof}


\begin{lemma} \label{lemma:sg_gamma_B_lower}
Assume \ref{C:sg_noise} holds. There exists a constant $C$ only depending on $C_0$ such that for any $(u_1,\ldots,u_r) \in (-1/2,1/2)^{r} \cap \calS_r$ and $v \in \cV$, 
\begin{align*}
\textup{Var}\left( \left. \textup{sign}\left(\sum_{i=1}^{r-1} \tau_i^{(r)}(u_1^r) Y_i  - Y_r \right) \; \right\vert \; V_1 = v - b_n u_1, \;\ldots,\; V_r = v - b_n u_r
\right) \geq C^{-1}.
\end{align*}
\end{lemma}
\begin{proof}
Fix any $v \in \cV$ and  $(u_1,\ldots,u_r) \in (-1/2,1/2)^{r} \cap \calS_r$. Note that
\begin{align*}
&\textup{Var}\left( \left. \textup{sign}\left(\sum_{i=1}^{r-1} \tau_i^{(r)}(u_1^r) Y_i  - Y_r \right) \; \right\vert \; V_k = v - b_n u_k, \text{ for } k \in [r]
\right) \\
= & 1 - \left(2\Pro\left( \left. \sum_{i=1}^{r-1} \tau_i^{(r)}(u_1^r) Y_i  - Y_r > 0 \; \right\vert \; V_k = v - b_n u_k, \text{ for } k \in [r] \right)- 1 \right)^2.
\end{align*}
Since $Y_i = f(v - b_n u_i) + \varepsilon_i$ for $i \in [r]$, due to \ref{C:sg_noise}, for $n \geq C_0$,
\begin{align*}
&\Pro\left( \left. \sum_{i=1}^{r-1} \tau_i^{(r)}(u_1^r) Y_i  - Y_r > 0 \; \right\vert \; V_k = v - b_n u_k, \text{ for } k \in [r] \right) \\
\leq & 
\Pro\left( \left. \sum_{i=1}^{r-1} \tau_i^{(r)}(u_1^r) \varepsilon_i  - \varepsilon_r > -1/C_0 \; \right\vert \; V_k = v - b_n u_k, \text{ for } k \in [r]\right) \\
= & 1 - \Pro\left( \left. \sum_{i=1}^{r-1} \tau_i^{(r)}(u_1^r) \varepsilon_i  - \varepsilon_r \leq -1/C_0 \; \right\vert \;  V_k = v - b_n u_k, \text{ for } k \in [r] \right) \\
\leq & 1 - \left(\prod_{i=1}^{r-1} \Pro\left( \epsilon_i \leq 0 \vert V_i = v - b_n u_i\right)  \right) \Pro\left(\epsilon_r \geq 1/C_0 \vert V_r = v - b_n u_r
\right)  \leq  1 - \frac{1}{2^r C_0}.
\end{align*}
By a similar argument, $\Pro\left( \left. \sum_{i=1}^{r-1} \tau_i^{(r)}(u_1^r) Y_i  - Y_r > 0 \; \right\vert \;V_k = v - b_n u_k, \text{ for } k \in [r]\right) \geq \frac{1}{2^r C_0}$. Then the proof is complete.
\end{proof}


%

\subsection{Proofs related to the power of the proposed concavity test} \label{proof:cor:power}

\begin{proof}[Proof of Corollary \ref{cor:power}]  
By Theorem \ref{thm:uproc_bootstrap} (the conditions have been verified in the proof of Theorem \ref{thm:concavity_master}), with probability at least $1 - C n^{-1/C}$,
\[\Pro_{\vert \cD_n'}\left( \widetilde{\bM}_n \geq q_{\alpha}^{\#} \right) > \alpha - C n^{-1/C}.
\]
In the proof of Theorem \ref{thm:concavity_master}, we have shown that
\begin{align}\label{power:aux1}
\sup_{v \in \cV} \left\{ r^2 \gamma_A(h^{*}_{v}) + \alpha_n \gamma_B(h^{*}_{v}) \right\} 
\leq C + C \alpha_n b_n^{-dr + d} 
= C + C n^{1-\kappa} b_n^{d}. 
\end{align}
Then due to Lemma \ref{lemma:WP_size}, 
\begin{align} \label{power:aux2}
\Pro \left(q_{\alpha}^{\#} \geq C K_n^{1/2} (1 + n^{(1-\kappa)/2} b_n^{d/2})\right) \leq Cn^{-1/C}.
\end{align}
Observe that
\begin{align*}
\Pro\left(
\sup_{v \in \cV} \sqrt{n} U_{n,N}'(h^{*}_v) \;\; \geq \;\; q^{\#}_\alpha\right)
\geq \Pro\left(
\sqrt{n} \left(U_{n,N}'(h^{*}_{v_n}) - P^r h^{*}_{v_n} \right) \;\; \geq \;\; q^{\#}_\alpha -\sqrt{n} P^r h^{*}_{v_n} \right).
\end{align*}
By Theorem \ref{thm:GAR_complete_hd_U} with $d = 1$ (the conditions have been verified in the proof of Theorem \ref{thm:concavity_master}), we have
\begin{align*}
\Pro\left(
\sup_{v \in \cV} \sqrt{n} U_{n,N}'(h^{*}_v) \;\; \geq \;\; q^{\#}_\alpha\right)
&\geq \Pro\left(
Y \;\; \geq \;\; q^{\#}_\alpha -\sqrt{n} P^r h^{*}_{v_n} \right) -  Cn^{-1/C},
\end{align*}
where $Y \sim N(0, r^2 \gamma_A(h^{*}_{v_n}) + \alpha_n \gamma_B(h^{*}_{v_n}))$ and is independent of $\cD_n'$. 
Then due to \eqref{power:aux1} and \eqref{power:aux2}, and 
$\kappa'' > (1-\kappa)/2$, we have with probability at least $1 - Cn^{-1/C}$,
\[
\frac{\sqrt{n} P^r h^{*}_{v_n} - q_\alpha^{\#}}{\sqrt{\Var(Y)}} \geq C^{-1} n^{1/C},
\]
which completes the proof.
\end{proof}

\begin{proof}[Proof of Corollary \ref{cor:power_smooth}]  

By the definition of $w(\cdot)$ in \eqref{def:id_kernel} and a change-of-variable, we have
\begin{align}\label{power_exp_formula}
\begin{split}
P^r h_{v_0}^{\text{id}} = b_n^{d/2} \int &\left( \sum_{j=1}^{r} 
\left(\sum_{ i \in [r]\setminus\{j\}}  \tau^{(j)}_i(u_1^r)\, f(v_0 - b_n u_i) - f(v_0 - b_n u_j) \right) \mathbbm{1}\left\{u_1^r \in \calS_{j} \right\} \right) \\
&\prod_{i=1}^{r} L(u_i) p(v - b_n u_i) du_1 \ldots du_r
\end{split}
\end{align}
By Taylor's Theorem, if $\|u\|_{\infty} < 1/2$,
\begin{align*}
f(v_0 - b_n u) = f(v_0) - b_n \nabla^T f(v_0) u 
+ \frac{b_n^2}{2} u^T \nabla^2 f(v_0) u
+ R(u, b_n),  
\end{align*}
where $\nabla f(v_0)$  and $\nabla^2 f(v_0)$ are the gradient and the Hessian matrix of $f$ at $v_0$ respectively, and 
\[
|R(u, b_n)| \leq Cb_n^2 \mathcal{R}(b_n),  \;\; \text{ where }
\mathcal{R}(b_n) :=
\max_{\|\xi\|_{\infty} \leq b_n/2} \|\nabla^2 f(v_0) - \nabla^2 f(v_0 - \xi)\|_{\text{op}},
\]
with $\|\cdot\|_{\text{op}}$ being the operator norm of a matrix.
As a result, due to \ref{C:kernel} and \ref{C:X},
\begin{align*}
P^r h_{v_0}^{\text{id}} \geq & \left( C^{-1}b_n^{2+d/2} \int \left( \sum_{j=1}^{r} 
\left(\sum_{ i \in [r]\setminus\{j\}}  \tau^{(j)}_i(u_1^r)\, u_i^T \nabla^2 f(v_0) u_i - u_j^T \nabla^2 f(v_0) u_j \right) \mathbbm{1}\left\{u_1^r \in \calS_{j} \right\} \right) \right.\\
&\left.\prod_{i=1}^{r} L(u_i) du_1 \ldots du_r \right)
\;-\; C b_n^{2+d/2} \mathcal{R}(b_n).
\end{align*}
Since $f$ is twice continuously differentiable at $v_0$ and $\nabla^2 f(v_0)$ is positive definite, we have $\mathcal{R}(b_n) \to 0$ as $b_n \to 0$, and   
$\liminf_{n \to \infty} P^r h^{\text{id}}_{v_0}/\left(b_n^{2+d/2} \right) > 0$.
\end{proof}
\vspace{0.1cm}

\begin{proof}[Proof of Corollary \ref{cor:power_pieceAffine}] 
Since $f_n$ is a convex function, due to
the condition \ref{C:X} and in view of \eqref{power_exp_formula}, we have $P^r h_{v_n}^{\text{id}}$ is \textit{lower} bounded by
\begin{align*}
C_0^{-1} b_n^{d/2} \int &\left( \sum_{j=1}^{r} 
\left(\sum_{ i \in [r]\setminus\{j\}}  \tau^{(j)}_i(u_1^r)\, f_n(v_n - b_n u_i) - f_n(v_n - b_n u_j) \right) \mathbbm{1}\left\{u_1^r \in \calS_{j} \right\} \right) \prod_{i=1}^{r} L(u_i) d u_i.
\end{align*}
Now for $i \in [r]$, since $f_{n,1}(v_n) = f_{n,2}(v_n)$, we have
$$
f_n(v_n - b_n u_i) = f_n(v_n) - b_n g_n(u_i) + b_n \|\theta_{n,1}-\theta_{n,2}\|_2\; h_n(u_i), 
$$
where
$g_n(u) := \theta_{n,1}^T u$ and 
$h_n(u) := \max \{0,\; (\theta_{n,1}-\theta_{n,2})^T u/\|\theta_{n,1}-\theta_{n,2}\|_2 \}$.

Since $g_n$ is linear,  we have
\begin{align*}
P^r h_{v_n}^{\text{id}} \geq 
C_0^{-1}b_n^{1+d/2} \|\theta_{n,1}-\theta_{n,2}\|_2 \int &\left( \sum_{j=1}^{r} 
\left(\sum_{ i \in [r]\setminus\{j\}}  \tau^{(j)}_i(u_1^r)\, h_n(u_i) - h_n(u_j)\right) \mathbbm{1}\left\{u_1^r \in \calS_{j} \right\} \right) \\&\prod_{i=1}^{r} L(u_i) d u_i.
\end{align*}
Now pick an arbitrary $\theta_* \in \bR^d$ such that 
$\|\theta_*\|_2 = 1$, 
and define
$h_{*}(u) = \max\{0, \theta_*^T u\}$. Then due to rotation invariance,
\begin{align*}
P^r h_{v_n}^{\text{id}} \geq 
C_0^{-1}b_n^{1+d/2} \|\theta_{n,1}-\theta_{n,2}\|_2 \int &\left( \sum_{j=1}^{r} 
\left(\sum_{ i \in [r]\setminus\{j\}}  \tau^{(j)}_i(u_1^r)\, h_*(u_i) - h_*(u_j)\right) \mathbbm{1}\left\{u_1^r \in \calS_{j} \right\} \right) \\&\mathbbm{1}\left\{u_1^2 + \ldots + u_r^2 \leq 1/4 \right\}  \prod_{i=1}^{r} L(u_i) d u_i,
\end{align*}
which completes the proof as the integral does not depend on $n$.
\end{proof}

\subsection{$L_2$ minimax separation rate  and the FS test} \label{subsec:minimax_FS}
We follow the definition of $L_2$ minimax separation rate  in \cite{juditsky2002}. We will ignore logarithmic factors in describing  rates. 

First, consider the univariate case, i.e., $d=1$. Denote $\cC$ the collection of \textit{concave} functions on $S = [0,1]$, and $\cS_{s}$ the collection of H\"{o}lder continuous functions with smoothness parameter $s$ (see \cite{juditsky2002} for precise definitions). For any  $f \in \cS_{s}$, define 
\[
\Phi_2(f) = \inf\left\{\|f-g\|_2:  g \in \cC \cap \cS_s \right\},\;\; \text{ where }\;
\|h\|_2^{2}:=
\int_0^1 |h(x)|^2 dx.
\]
Fix some $p \in (0,1/8)$. Define the $L_2$ minimax separation rate $\cR_n^*$ to be the smallest $\epsilon >0$ such that there exits a test $\cT$ for which
$$
\sup_{f \in \cC \cap \cS_{s}}\bP_f(\cT \text{ rejects } H_0) + 
\sup_{f \in \cS_{s}: \Phi_2(f) \geq \epsilon}
\bP_f(\cT \text{ accepts } H_0) \; \leq\; p.
$$
In \cite{juditsky2002}, which considers the Gaussian white noise model, it is shown that if $s \geq 2$,  then $\cR_n^* = C n^{-s/(2s+1)}$, which is of the same order as the minimax rate of estimating $f \in \cS_s$ itself. If $s < 2$, based on the remark following \cite{juditsky2002}[Theorem 1], $\cR_n^*=C n^{-((2/5) \wedge (2s/(4s+1)))}$. Specifically, for $f \in \cS_s$, denote $f^*$ the projection of $f$ onto $\cC$, i.e. $f^* \in \cC \cap \cS_s$ such that $\Phi_2(f) = \|f-f^*\|_2 =  \|f\|_2 - \|f^*\|_2$; here, $\|f\|_2$ can be estimated at the rate of $n^{-2s/(4s+1)}$ \cite{lepski1999estimation}, and $\|f^*\|_2$ at the rate of $n^{-2/5}$ \cite{10.1214/20-EJS1714}[Corollary 4.2] (see also \cite{guntuboyina2015global}[Theorem 6.1]).

For $d \geq 2$, to the best of our knowledge, the minimax separation rate is not known in the literature. In view of the results for $d=1$, for smooth functions  (say H\"{o}lder continuous with smoothness parameter $s \geq 2$ \cite{belloni2015some}[Section 3.1]), it is reasonable to \textit{conjecture} that $\cR_n^* = C n^{-s/(2s+d)}$, the minimax rate of estimating $f$ itself. For $s < 2$, although an upper bound on the rate of recovering $\|f^*\|_2$ is established recently in \cite{10.1214/20-EJS1714}[Corollary 4.2] for $d \geq 2$,  the rate of estimating $\|f\|_2$ seems unknown for $d \geq 2$.\\

\noindent \underline{\it FS test \cite{fang2020projection}.}
\cite{fang2020projection} proposes a projection framework for testing shape restrictions including concavity. Specifically, \cite{fang2020projection} proposes to  first estimate the regression function $f$ by \textit{unconstrained, nonparametric} methods, say, by sieved B-splines, and then evaluate and calibrate the $L^2$ distance between the estimator  and the space of concave functions.

Denote $k_n$ the number of B-spline terms of order $s_0$ used to estimate $f$ without constraint \cite{fang2020projection}. There are two important components in the FS test \cite{fang2020projection}: under-smoothing and strong approximation. For H\"{o}lder continuous functions with smoothness parameter $s$, under-smoothing requires that
$k_n \gg  n^{d/(2(s\wedge s_0)+d)}$, while the strong approximation requires $k_n \ll n^{-1/5}$ \cite{belloni2015some}[Theorem 4.4] (see also \cite{chernozhukov2013intersection}[Example 5]).

Thus if $s \wedge s_0 \leq 2d$, then there is \textit{no choice} of $k_n$ that can simultaneously  meet these two requirements. In particular, if $d=2$, it requires the smoothness parameter $s > 4$ \textit{and} the order of B-splines $s_0 > 4$. This explains why the FS test fails to control  the size properly for concave, piecewise affine functions in Section \ref{sec:simulation}.

On the other hand, if $s \wedge s_0 > 2d$ and
the under-smoothing is moderate (say by a logarithmic amount), then the FS test achieves the separation rate $n^{-s/(2s+d)}$ \cite{fang2020projection}[Theorem 3.2] (ignoring logarithmic factors), which is minimax for smooth functions for $d=1$ and may be minimax for $d \geq 2$ as discussed above.


\subsection{An algorithm without stratification} \label{subsec:without_partitioning}

In this section, we present an algorithm to compute the test statistics $U_{n,N}'(h_v^{(*)})$ over $\sV_n$ without stratification. It has a similar computational complexity as Algorithm \ref{alg:test_stat_gamma_B} in theory with $d$ fixed,  but it is not computationally feasible since the multiplicative constant is of order $2^{dr}$. We adopt the notations in Section \ref{sec:local_simplex_stat}.

Without partitioning, there is a single sampling plan $\{Z_{\iota}: \iota \in I_{n,r}\}$.  The key insight is that for $\iota \in I_{n,r}$ and $v \in \sV_n$,
\[
h_v(X_{\iota}) = 0, \text{ if } V_j \not\in \cN(V_{j'},2b_n) \text{ for some } j,j' \in \iota.
\]
As a result, it suffices to focus on those $r$-tuples such that their feature vectors are within $2b_n$-neighbourhood of each other (in $\|\cdot\|_{\infty}$).

The pseudocode  is listed in Algorithm \ref{alg:test_stat_no_partition}, whose  computational complexity is
\begin{align*}
n *\; \left( \binom{n (2b_n)^{d} }{r-1} \frac{n^{\kappa} b_n^{-dr}}{|I_{n,r}|} \log(n)\right) *\; |\sV_n| b_n^{d}\;\lesssim \; 2^{dr} |\sV_n|
 n^{\kappa} \log(n) b_n.
\end{align*}
The factor $2^{dr}$ is due to the fact that we need to focus on a $2b_n$-neighbourhood, instead of a $b_n$-neighbourhood as in Algorithm \ref{alg:test_stat_gamma_B}.

\begin{algorithm}[htbp!]
  \KwIn{Observations $\{X_i = (V_i, Y_i) \in \bR^{d+1}: i \in [n]\}$, budget $N$, kernel $L(\cdot)$, bandwidth $b_n$,  query points $\sV_n$.}
  
  \KwOut{ ${U}_{n,N}'$ a list of length $|\sV_n|$}
  
   \KwSty{Initialization:} $p_n = N/\binom{n}{r}$, $\widehat{N} = 0$, ${U}_{n,N}'$  set  zero \;
   Compute $A$, a  length $n$ list, where $A[i] = \{j > i: V_j \in \cN(V_i, 2b_n)\}$ for $i \in [n]$ \;
   Compute $B$, a  length $n$ list, where $B[i] = \{v \in \sV_n: v \in \cN(V_i, b_n)\}$ for $i \in [n]$ \;

\For{$i \leftarrow 1$ \KwTo $n$ }{
 Generate $T_{1} \; \sim \; \text{Binomial}(\binom{|A[i]|}{r-1}, \;p_n), \quad  T_{2} \; \sim \; \text{Binomial}(\binom{n-i}{r-1}-\binom{|A[i] |}{r-1}, \;p_n)$\;
$\widehat{N} \leftarrow \widehat{N} + T_{1} + T_{2}$\;
Sample $T_{1}$ terms $\{\iota_\ell': 1 \leq \ell \leq T_{1}\}$ without replacement from  $\left\{(s_1,\ldots,s_{r-1}): s_j \in A[i] \text{ for each } j \in [r-1] \right\}$.\;
	\For{$\ell \leftarrow 1$ \KwTo $T_1$}{	
	    $\iota \leftarrow \{i\} \cup \iota_{\ell}'$\;
		\For{$ v\in B[i]$}{
			 ${U}_{n,N}'[v] \leftarrow {U}_{n,N}'[v] +h_v^*(X_{\iota})$\;
		}
	}
}
${U}_{n,N}' \leftarrow {U}_{n,N}'/\widehat{N}$ \tcc{Operations are element-wise}
   \caption{Algorithm to compute $U_{n,N}'$  for the concavity test without partitioning.}
   \label{alg:test_stat_no_partition}
\end{algorithm}

\end{appendix}

\end{document}